\documentclass[a4paper,final,12pt]{article}

\usepackage{amsmath,amssymb,amsthm,amsfonts,
geometry,listings,graphicx,showkeys,bbm,hyperref,
enumerate,nicefrac,mathtools}

\newtheorem{lemma}{Lemma}[section]
\newtheorem{theorem}[lemma]{Theorem}
\newtheorem{proposition}[lemma]{Proposition}
\newtheorem{corollary}[lemma]{Corollary}

\newcommand{\XX}[1]{\ensuremath{\mathbf{X}_{#1}}}

\newcommand{\XXMN}[1]{\ensuremath{\mathbf{X}_{#1}^{M,I}}}
\newcommand{\Y}[1]{\ensuremath{\mathcal{X}_{#1}}}
\newcommand{\Yb}[1]{\ensuremath{\bar{\mathcal{X}}_{#1}}}

\newcommand{\YMN}[1]{\ensuremath{\mathcal{X}_{#1}^{M,I}}}

\newcommand{\YMNp}[1]{\ensuremath{\tilde{\mathcal{X}}_{#1}^{M,I}}}

\newcommand{\YMNN}[1]{\ensuremath{\mathcal{X}_{#1}^{M,N}}}

\renewcommand{\O}[1]{\ensuremath{\mathcal{O}_{#1}}}

\newcommand{\OMN}[1]{\ensuremath{\mathcal{O}_{#1}^{M,I}}}

\newcommand{\OMNN}[1]{\ensuremath{\mathcal{O}_{#1}^{M,N}}}
\newcommand{\OMNNp}[1]{\ensuremath{\tilde{\mathcal{O}}_{#1}^{M,N}}}

\renewcommand{\S}[0]{\ensuremath{\mathcal{S}}}

\newcommand{\D}[0]{\ensuremath{\mathcal{D}}}

\newcommand{\Id}[0]{\ensuremath{\textup{Id}}}
\renewcommand{\P}[0]{\ensuremath{\mathbb{P}}}
\newcommand{\Pow}[0]{\ensuremath{\mathcal{P}}}

\renewcommand{\H}[0]{\ensuremath{\mathbb{H}}}

\newcommand{\R}[0]{\ensuremath{\mathbb{R}}}
\newcommand{\N}[0]{\ensuremath{\mathbb{N}}}
\newcommand{\E}[0]{\ensuremath{\mathbb{E}}}
\newcommand{\B}[0]{\ensuremath{\mathcal{B}}}

\newcommand{\M}[0]{\ensuremath{\mathcal{M}}}
\newcommand{\C}[0]{\ensuremath{\mathcal{C}}}
\newcommand{\F}[0]{\ensuremath{\mathcal{F}}}
\renewcommand{\L}[0]{\ensuremath{\mathcal{L}}}
\newcommand{\fl}[2][T\!/\!M]{\ensuremath{ \lfloor #2 \rfloor_{#1} }}
\newcommand{\cl}[2][T\!/\!M]{\ensuremath{ \lceil #2 \rceil_{#1} }}

\newcommand{\Vm}[0]{\ensuremath{\mathcal{V}}}
\newcommand{\one}[0]{\ensuremath{\mathbbm{1}}}

\hypersetup{
  linktocpage=true
}

\usepackage{courier}
\title{Strong convergence rates for explicit \\
space-time discrete numerical approximations \\
of stochastic Allen-Cahn equations}
\author{
Sebastian Becker,
Benjamin Gess, 
Arnulf Jentzen,
and Peter E. Kloeden
}

\begin{document}

\maketitle

\begin{abstract}
The scientific literature contains a 
number of numerical approximation results
for stochastic partial differential equations (SPDEs)
with superlinearly growing nonlinearities but,
to the best of our knowledge, none of them prove
strong or weak convergence rates for full-discrete numerical approximations
of space-time white noise driven SPDEs
with superlinearly growing nonlinearities.
In particular, in the scientific
literature there exists neither a result which 
proves strong convergence rates 
nor a result which proves weak convergence rates
for full-discrete numerical approximations of stochastic Allen-Cahn equations.
In this article we bridge this gap and establish 
strong convergence rates for full-discrete numerical approximations
of space-time white noise driven SPDEs
with superlinearly growing nonlinearities
such as stochastic Allen-Cahn equations.
Moreover, we also establish lower bounds for 
strong temporal and spatial approximation errors
which demonstrate that our strong convergence rates
are essentially sharp and can, in general, not be improved.
\end{abstract}

\tableofcontents

\section{Introduction}
\label{sec:intro}
In this article we are interested
in strong convergence rates for full-discrete 
numerical approximations of
space-time white noise driven SPDEs
with superlinearly growing nonlinearities
such as stochastic Allen-Cahn equations.
The literature contains a number of 
numerical approximation results for SPDEs
with superlinearly growing nonlinearities
(cf., e.g., 
Gy{\"o}ngy \& Millet~\cite{gm05},
Gy{\"o}ngy, Sabanis, \& {{\v S}i{\v s}ka}~\cite{GoengySabanisS2014},
Jentzen \& Pu{\v s}nik~\cite{jp2015},
Kov{\'a}cs, Larsson, \& Lindgren~\cite{kll2015},
Becker \& Jentzen~\cite{BeckerJentzen2016},
Hutzenthaler, Jentzen, \& Salimova~\cite{HutzenthalerJentzenSalimova2016},
Jentzen \& Pu{\v s}nik~\cite{JentzenPusnik2016},
Furihata et al.~\cite{FurihataKovacsLarssonLindgreen2016}, and
Bl{\"o}mker \& Kamrani~\cite{BloemkerKamrani2017}). 
The articles~\cite{gm05,GoengySabanisS2014,
FurihataKovacsLarssonLindgreen2016,HutzenthalerJentzenSalimova2016}
establish strong convergence of numerical
approximations for such SPDEs with no information
on the speed of strong convergence and the 
papers~\cite{jp2015,kll2015,BeckerJentzen2016}
prove strong convergence rates for numerical
approximations of such SPDEs. 
To be more specific, the article~\cite{BeckerJentzen2016}
establishes strong convergence rates for semi-discrete
temporal numerical approximations of space-time white noise 
driven SPDEs with superlinearly growing nonlinearities
such as stochastic Allen-Cahn equations.
The papers~\cite{jp2015,kll2015} prove strong
convergence rates for full-discrete 
(temporal and spatial discrete) numerical approximations
for SPDEs with superlinearly growing nonlinearities
in the case of the more regular trace class noise.
To the best of our knowledge,
there exists no result in the scientific literature
which establishes strong or weak convergence rates for 
a full-discrete numerical approximation scheme of a 
space-time white noise driven SPDE with a superlinearly 
growing nonlinearity such as the stochastic Allen-Cahn
equation. A key difficulty in the case of full-discrete
numerical approximations for
space-time white noise driven SPDEs with superlinearly 
growing nonlinearities is to derive appropriate uniform
a priori moment bounds for the numerical approximation
processes. 

In this article we overcome this difficulty
(cf.\ \eqref{eq:intro_lyapunov}--\eqref{eq:intro_lyapunov_2}
below for our approach to this challenge)
and establish
essentially sharp strong convergence rates for 
full-discrete numerical approximations of
space-time white noise driven SPDEs with
superlinearly growing nonlinearities
such as stochastic Allen-Cahn equations;
see Theorem~\ref{thm:main} in Section~\ref{sec:main_result} 
below for the main
convergence rate result in this work.
To illustrate Theorem~\ref{thm:main}, we now present in
Theorem~\ref{thm:intro} below the specialization of
Theorem~\ref{thm:main} to the case of stochastic Allen-Cahn equations.

\begin{theorem}
\label{thm:intro}
Let $ T \in (0,\infty) $,
$
  ( H, \langle \cdot, \cdot \rangle_H, \left\| \cdot \right\|\!_H )
  =
  ( 
    L^2( (0,1); \R ), 
    \langle \cdot, \cdot \rangle_{ L^2( (0,1); \R ) },
    \left\| \cdot \right\|\!_{ L^2( (0,1) ; \R ) }
  )
$,
$ a_0, a_1 $, $ a_2 \in \R $,
$ a_3 \in (-\infty,0] $,
$ (e_n)_{ n \in \N } \subseteq H $,
$ ( P_n )_{ n \in \N } \subseteq L( H ) $,
$ F \colon L^{6}( (0,1) ; \R ) \rightarrow H $
satisfy for all 
$ n \in \N $, $ v \in L^{6}( (0,1) ; \R ) $
that 
$ e_n(\cdot) = \sqrt{2} \sin( n\pi (\cdot) ) $,
$ F(v) = \sum_{ k=0 }^3 a_k v^k $,
$ P_n(v) = \sum_{ k=1 }^n \langle e_k, v \rangle_H \, e_k $,
and 
$ a_2 \one_{[0,\infty)}(a_3) = 0 $,
let $ A \colon D(A) $ $ \subseteq H \rightarrow H $ be the 
Laplacian with Dirichlet boundary conditions on 
$ H $, 
let $ ( \Omega, \F, \P ) $ be a probability space, 
let $ ( W_t )_{ t \in [0,T] } $ be an $ \Id_H $-cylindrical
Wiener process, 
let
$ \xi \in D( (-A)^{\nicefrac{1}{2}} ) $,
$ \gamma \in (\nicefrac{1}{6}, \nicefrac{1}{4}) $,
$ 
  \chi 
  \in 
  (0,  
    \nicefrac{\gamma}{3}
    -
    \nicefrac{1}{18}
  ]
$,
let
$ \OMNN{} \colon [0,T] \times \Omega \rightarrow P_N(H) $,
$ M,N \in \N $,
and $ \YMNN{} \colon [0,T] \times \Omega \rightarrow P_N( H ) $,
$ M,N \in \N $, be stochastic processes which satisfy that for all
$ M,N \in \N $, $ m \in \{ 0, 1, 2, \ldots, M-1 \} $, 
$ t \in (\nicefrac{mT}{M}, \nicefrac{(m+1)T}{M} ] $ 
we have $ \P $-a.s.\ that
\begin{align}
  \OMNN{0} = \YMNN{0} = P_N \, \xi,
  \qquad
  \OMNN{t} 
  =
  e^{(t-\nicefrac{mT}{M})A}
  \left[ 
    \OMNN{\nicefrac{mT}{M}}
    +  
    \int_{\nicefrac{mT}{M}}^t
    P_{ N } \, dW_s
  \right],
\end{align}
and
\begin{align}
\label{eq:intro_scheme}
\begin{split}
  \YMNN{t} 
 &=
  e^{(t-\nicefrac{mT}{M})A} \,
  \YMNN{\nicefrac{mT}{M}}
  +
  \OMNN{t}
  -
  e^{(t-\nicefrac{mT}{M})A} \,
  \OMNN{\nicefrac{mT}{M}}
\\&\quad+
  P_N 
  A^{-1} 
  ( e^{(t-\nicefrac{mT}{M})A} - \Id_H ) \,
  \one_{ 
    \{
      \| (-A)^{\gamma} \YMNN{\nicefrac{mT}{M}} \|_{ H }
      +
      \| (-A)^{\gamma} \OMNN{\nicefrac{mT}{M}} \|_{ H }
      \leq
      (M/T)^{\chi}
    \}
  } \,
  F( \YMNN{\nicefrac{mT}{M}} ) .
\end{split}
\end{align}
Then 
\begin{enumerate}[(i)]
\item\label{it:intro_1} we have that
there exists an up to indistinguishability
unique stochastic process
$
  X \colon [0,T] \times \Omega \rightarrow L^{6}( (0,1); \R )
$
with continuous sample paths
which satisfies for all $ t \in [0,T] $, $ p \in (0,\infty) $ 
that 
$ 
  \sup_{ s \in [0,T] } 
  \E\big[ \| X_s \|_{ L^{6}( (0,1); \R ) }^p \big]
  <
  \infty
$
and
\begin{align}
\label{eq:intro_allen_cahn}
  \P\!\left(
    X_t 
    =
    e^{tA} \xi 
    +
    \smallint_0^t e^{(t-s)A} F(X_s) \, ds
    +
    \smallint_0^t
    e^{(t-s)A} \, dW_s
  \right)
  =
  1,
\end{align}
\item\label{it:intro_2}
we have for all $ p \in (0,\infty) $
that
$
  \sup_{ r \in (-\infty, \gamma] }
  \sup_{ M, N \in \N }
  \sup_{ t \in [0,T] }
  \E\big[ \| (-A)^r \YMNN{t} \|_{ H }^p \big]
  <
  \infty
$, 
and
\item\label{it:intro_3}
we have for all 
$ p, \varepsilon \in (0,\infty) $
that there exists a real number
$ C \in \R $ such that for all
$ M,N \in \N $
it holds that
\begin{align}
  \sup_{ t \in [0,T] }
  \left(
    \E\big[
      \|
        X_t - \YMNN{t}
      \|_{ H }^p
    \big]
  \right)^{\!\nicefrac{1}{p}}
  \leq
  C
  (
    M^{(\varepsilon-\nicefrac{1}{4})}
    +
    N^{(\varepsilon-\nicefrac{1}{2})}
  ) .
\end{align}
\end{enumerate}
\end{theorem}

Theorem~\ref{thm:intro} follows from Corollary~\ref{cor:Ginzburg3}
which, in turn, follows from our main result, Theorem~\ref{thm:main} below.
Theorem~\ref{thm:main} also proves strong convergence rates for
full-discrete numerical approximations of a more general class
of SPDEs than Theorem~\ref{thm:intro} above.
Next we would like to point out that the numerical
approximation scheme~\eqref{eq:intro_scheme} has been
proposed in Hutzenthaler, Jentzen, \& Salimova~\cite{HutzenthalerJentzenSalimova2016}
and has there been referred to as a
nonlinearity-truncated approximation scheme
(cf.\ \cite[(3) in Section~1]{HutzenthalerJentzenSalimova2016}
and, e.g., \cite{HutzenthalerJentzenKloeden2012,
WangGan2013,
HutzenthalerJentzenKloeden2013,
HutzenthalerJentzen2012,
TretyakovZhang2013,
Sabanis2013,
Sabanis2013E,
GoengySabanisS2014,
jp2015,JentzenPusnik2016} for further research articles on 
explicit approximation schemes for stochastic differential
equations with superlinearly growing nonlinearities).
Moreover, note that Theorem~\ref{thm:intro} 
demonstrates that the full-discrete
numerical approximations in~\eqref{eq:intro_scheme}
converge for every $ \varepsilon \in (0, \infty) $ 
strongly to the solution of the stochastic Allen-Cahn 
equation~\eqref{eq:intro_allen_cahn} with the spatial 
rate of convergence $ \nicefrac{1}{2} - \varepsilon $
and the temporal rate of convergence $ \nicefrac{1}{4} - \varepsilon $.
We also would like to point out that the 
strong convergence rates established in Theorem~\ref{thm:intro}
can, in general,
not essentially be improved. More formally, 
Corollary~\ref{cor:lower_bounds} below proves
in the case where $ \sum_{i=0}^3 |a_i| = 0 $ 
and $ \xi = 0 $ in the framework 
of Theorem~\ref{thm:intro} that
there exist real numbers $ c, C \in (0,\infty) $
such that for all $ M, N \in \N $
we have that
\begin{align}
\label{eq:intro_conv_bounds_1}
\begin{split}
 &c \,
  M^{-\nicefrac{1}{4}}
  \leq 
  \adjustlimits\lim_{ n \rightarrow \infty }
  \sup_{ t \in [0,T] }
  \left(
    \E\big[ 
      \|
        X_t - \mathcal{X}_t^{M,n}
      \|_{ H }^p 
    \big] 
  \right)^{\!\nicefrac{1}{p}}
  \leq
  C
  M^{-\nicefrac{1}{4}}
\end{split}
\end{align}
and 
\begin{align}
\label{eq:intro_conv_bounds_2}
\begin{split}
 &c \,
  N^{-\nicefrac{1}{2}}
  \leq 
  \adjustlimits\lim_{ m \rightarrow \infty }
  \sup_{ t \in [0,T] }
  \left(
    \E\big[
      \|
        X_t - \mathcal{X}_t^{m,N}
      \|_{ H }^p
    \big] 
  \right)^{\!\nicefrac{1}{p}}
  \leq
  C 
  N^{-\nicefrac{1}{2}} .
\end{split}
\end{align}
Inequalities~\eqref{eq:intro_conv_bounds_1} 
and~\eqref{eq:intro_conv_bounds_2} thus
show that the spatial rate $ \nicefrac{1}{2} - \varepsilon $
and the temporal rate $ \nicefrac{1}{4} - \varepsilon $ 
established in Theorem~\ref{thm:intro} can essentially
not be improved. Further related lower bounds 
for strong approximation errors in the linear case 
$ \sum_{i=0}^3 |a_i| = 0 $ can, e.g., be found in
M{\"u}ller-Gronbach, Ritter, \& Wagner~\cite[Theorem~1]{mgritterwager2006},
M{\"u}ller-Gronbach \& Ritter~\cite[Theorem~1]{mgritter2007},
M{\"u}ller-Gronbach, Ritter, \& Wagner~\cite[Theorem~4.2]{mgritterwagner2008},
Conus, Jentzen, \& Kurniawan~\cite[Lemma~6.2]{ConusJentzenKurniawan2014},
and~Jentzen \& Kurniawan~\cite[Corollary~9.4]{JentzenKurniawan2015}.

Finally, we would like to add some comments
on the proof of Theorem~\ref{thm:intro} above 
and Theorem~\ref{thm:main} below, respectively.
The main difficulty to prove Theorem~\ref{thm:intro} 
is to obtain uniform
a priori moment bounds for the space-time discrete 
numerical approximations~\eqref{eq:intro_scheme}
(see Section~\ref{sec:a_priori_pw} and 
Section~\ref{sec:a_priori_moment} below).
Once the uniform a priori moment bounds have been established, we 
exploit the fact that the nonlinearity of the 
stochastic Allen-Cahn equation satisfies a
global monotonicity property to prevent
that the local discretization errors 
accumulate too quickly.
It thus remains to sketch our procedure 
to establish
uniform a priori moment bounds for 
the numerical approximations.
We first subtract 
the noise process from~\eqref{eq:intro_scheme}
as it is often done in the literature.
The key idea that we use to derive uniform a priori bounds for the subtracted 
equation is then to employ a 
suitable path-dependent Lyapunov-type function
which on the one hand incorporates the dissipative dynamics 
of the stochastic Allen-Cahn equation~\eqref{eq:intro_allen_cahn} 
and which on the other hand respects the spatial spectral
Galerkin approximations used for the spatial discretization 
of~\eqref{eq:intro_allen_cahn}.
More formally, a key contribution of this work is to reveal that there exists
a suitable $ \B( \C([0,T], L^{\infty}( (0,1); \R ) ) ) / \B( [0,\infty) ) $-measurable mapping
$ \phi \colon \C([0,T], L^{\infty}( (0,1); \R ) ) \rightarrow [0,\infty) $ 
such that for every $ N \in \N $ we have that the mapping
\begin{align}
\label{eq:intro_lyapunov}
\begin{split}
 &P_N(H) \times \C( [0,T], L^{\infty}( (0,1); \R ) )
  \ni 
  (v,w) \mapsto 
  \| (-A)^{\nicefrac{1}{2}} v \|_H^2 
  +
  \phi(w) \| v \|_H^2
  \in \R
\end{split}
\end{align}
is an appropriate path dependent Lyapunov-type function for the 
system of the $ N $-dimensional spatial spectral Galerkin approximation
of the subtracted equation associated to the stochastic
Allen-Cahn equation~\eqref{eq:intro_allen_cahn} (variable 
$ v \in P_N(H) $) and the $ N $-dimensional spatial spectral Galerkin approximation
of the Ornstein-Uhlenbeck process (variable $ w \in \C([0,T], L^{\infty}( (0,1); \R ) ) $).
It is crucial that the Lyapunov-type 
function~\eqref{eq:intro_lyapunov} does 
not only depend on $ w_T $ but on
the whole path $ w_t $, $ t \in [0,T] $, of $ w $. Applying the fundamental theorem
of calculus to~\eqref{eq:intro_lyapunov} results,
roughly speaking, in the 
coercivity type condition that there exist real numbers
$ \epsilon \in [0,1) $, $ c \in (0,\infty) $ and 
$ \B( \C([0,T], L^{\infty}( (0,1); \R ) ) ) / \B( [0,\infty) ) $-measurable mappings
$ \phi, \Phi \colon \C([0,T], L^{\infty}( (0,1); \R ) ) \rightarrow [0,\infty) $
such that for every
$ N \in \N $, $ v \in P_N(H) $, $ w \in \C( [0,T], L^{\infty}( (0,1); \R ) ) $
we have that
\begin{align}
\label{eq:intro_lyapunov_2}
\begin{split}
 &\sup_{ t \in [0,T] }
  \left(
  \langle
    (-A)^{\nicefrac{1}{2}} v, (-A)^{\nicefrac{1}{2}} P_N F(v+w_t)
  \rangle_{ H }
  +
  \phi(w)
  \langle
    v, F(v+w_t)
  \rangle_H
  \right)
\\&
  \leq
  \epsilon
  \| A v \|_{ H }^2
  +
  (c + \phi(w))
  \| (-A)^{\nicefrac{1}{2}} v \|_{ H }^2
  +
  c \phi(w)
  \| v \|_H^2
  +
  \Phi(w) .
\end{split}
\end{align}
Essentially, the coercivity type condition~\eqref{eq:intro_lyapunov_2}
appears as one of our assumptions of
Theorem~\ref{thm:main} below
(see~\eqref{eq:main_coercivity} in
Section~\ref{sec:main_result_setting} below for details).
Our proposal for this specific Lyapunov-type function is partially
inspired by the arguments in Section~4 in Bianchi, Bl{\"o}mker, \& 
Schneider~\cite{BianchiBloemkerSchenider2017}
(cf. \cite[Theroem~4.1 and Lemma~4.4]{BianchiBloemkerSchenider2017}).

The remainder of this article is structured as follows. 
Section~\ref{sec:a_priori_pw} establishes suitable
a priori bounds for the numerical approximations.
In Section~\ref{sec:pw_section} the error analysis for
the considered nonlinearity-truncated approximation schemes is
carried out in the pathwise sense and in Section~\ref{sec:lp_section}
we perform the error analysis for these numerical schemes
in the strong $ L^p $-sense.
In Section~\ref{sec:main_result} we combine the results from
Section~\ref{sec:lp_section} with appropriate uniform a priori moment
bounds for the numerical approximation processes 
(see Section~\ref{sec:a_priori_pw}) to establish
Theorem~\ref{thm:main} which is the main result of this article.
Section~\ref{sec:examples} makes sure that the assumptions of
Theorem~\ref{thm:main} are satisfied for stochastic Allen-Cahn
equations and finally, in Section~\ref{sec:lower_bounds_section}
we prove lower and upper bounds for strong approximation errors
of numerical approximations of linear stochastic heat equations.

\subsection{Notation}
\label{sec:notation}
Throughout this article the following notation is used.
For every measurable space $ (A, \mathcal{A}) $
and every measurable space $ (B, \mathcal{B}) $ we denote by
$ \M( \mathcal{A}, \mathcal{B} ) $ the
set of all $ \mathcal{A} / \mathcal{B} $-measurable functions.
For every set $ A $ we denote by 
$ \#_A  \in \{ 0, 1, 2, \ldots \} \cup \{ \infty \} $
the number of elements of $ A $,
we denote by $ \Pow(A) $
the power set of $ A $, 
and we denote by 
$ \Pow_0(A) $ the set given by
$ \Pow_0(A) = \{ B \in \Pow(A) \colon \#_B < \infty \} $.
For every set $ A $ and every set $ \mathcal{A} $
with $ \mathcal{A} \subseteq \mathcal{P}(A) $
we denote by $ \sigma_A( \mathcal{A} ) $ the smallest 
sigma-algebra on $ A $ which contains $ \mathcal{A} $.
For every topological space $ (X, \tau) $ we denote by 
$ \B(X) $ the set given by 
$ \B(X) = \sigma_X( \tau ) $.
For every natural number $ d \in \N $ and every
set $ A \in \B(\R^d) $
we denote by $ \lambda_A \colon \B(A) \rightarrow [0, \infty] $
the Lebesgue-Borel measure on $ A $.
We denote by $ \fl[h]{\cdot} \colon \R \rightarrow \R $, 
$ h \in (0,\infty) $,
the functions which satisfy
for all $ h \in (0,\infty) $, $ t \in \R $ that 
$
  \fl[h]{t}
  =
  \max\!\left(
    \{ 0, h, -h, 2h, -2h, \ldots \} \cap (-\infty, t]
  \right)
$. 
For every measure space
$ (\Omega, \F, \nu) $,
every measurable space 
$ (S, \S) $,
every set $ R $,
and every
function $ f \colon \Omega \rightarrow R $
we denote by $ [f]_{\nu,\S} $ the set given by
$ 
  [f]_{\nu,\S} 
  =
  \{
    g \in \M(\F,\S)
    \colon
    ( 
      \exists \, A \in \F
      \colon
      \nu(A) = 0
      \text{ and }
      \{
        \omega \in \Omega
        \colon
        f(\omega)
        \neq
        g(\omega)
      \}
      \subseteq A
    )
  \}
$.
For every set $ \Omega $ and every set $ A $ we 
denote by $ \one_A^{\Omega} \colon \Omega \rightarrow \R $
the function which satisfies for all 
$ x \in \Omega $ that
\begin{align}
 &\one_A^{\Omega}(x)
  =
  \begin{cases}
    1 & \colon x \in A \\
    0 & \colon x \notin A .
  \end{cases}
\end{align}

\subsection{Acknowledgments}
We thank Dirk Bl{\"o}mker 
for fruitful discussions and for pointing out
his instructive paper
Bianchi, Bl{\"o}mker, \& Schneider~\cite{BianchiBloemkerSchenider2017}
to us.
This work has been partially supported
through the SNSF-Research project
200021\_156603
``Numerical approximations of nonlinear
stochastic ordinary and partial differential
equations''.

\section{A priori bounds for the numerical approximation}
\label{sec:a_priori_pw}

\begin{lemma}
\label{lem:Y_a_priori}
Consider the notation in Section~\ref{sec:notation}, 
let $ ( H, \left< \cdot, \cdot \right>_H, \left\| \cdot \right\|_H ) $
be a separable $ \R $-Hilbert space,
let $ \H \subseteq H $ be a non-empty orthonormal basis of $ H $,
let $ T, \varphi, c \in (0,\infty) $, $ C \in [0,\infty) $,
$ \epsilon, \kappa, \rho \in [0,1) $, $ \gamma \in (\rho,1) $,
$ 
  \chi 
  \in 
  (0,\nicefrac{(\gamma-\rho)}{(1+\nicefrac{\varphi}{2})}] 
  \cap 
  (0,\nicefrac{(1-\rho)}{(1+\varphi)}]
$, 
$ M \in \N $,
$ \mu \colon \H \rightarrow \R $ satisfy $ \sup_{ h \in \H } \mu_h < 0 $,
let $ A \colon D(A) \subseteq H \rightarrow H $ be the linear operator
which satisfies
$ 
  D(A) 
  = 
  \{ 
    v \in H 
    \colon 
    \sum_{ h \in \H } 
    | 
      \mu_h 
      \langle h, v \rangle_H 
    |^2
    <
    \infty
  \}
$
and 
$
  \forall \, v \in D(A)
  \colon
  Av
  =
  \sum_{ h \in \H }
  \mu_h
  \langle h,v \rangle_H h
$,
let $ (H_r, \langle \cdot, \cdot \rangle_{H_r}, \left\| \cdot \right\|\!_{H_r} ) $,
$ r \in \R $, be a family of interpolation spaces associated to $ -A $
(cf., e.g., \cite[Section~3.7]{sy02}),
let $ I \in \Pow_0(\H) $, $ P \in L(H) $
satisfy for all $ v \in H $ that
$ 
  P(v) 
  = 
  \sum_{ h \in I } 
  \langle h,v \rangle_H h 
$,
and let 
$ \Y{} \colon [0,T] \rightarrow P(H) $,
$ \O{} \in \C( [0,T], P(H) ) $,
$ F \in \C( P(H),H ) $,
$ \phi, \Phi \colon \C([0,T], P(H)) \rightarrow [0,\infty) $
satisfy for all 
$ u,v \in P(H) $, 
$ w \in \C([0,T], P(H)) $, 
$ t \in [0,T] $ that
\begin{equation}
  \| F(u) \|_H^2
  \leq 
  C 
  \max\{ 1, \| u \|_{H_{\gamma}}^{(2+\varphi)} \},
\end{equation}
\begin{equation}
  \|
    F(u) - F(v)
  \|_H^2
  \leq
  C
  \max\{ 1, \| u \|_{ H_{\gamma} }^{\varphi} \}
  \| u-v \|_{ H_{\rho} }^2
  +
  C
  \| u-v \|_{ H_{\rho} }^{(2+\varphi)},
\end{equation}
\begin{align}
\begin{split}
 &\langle v, P F(v+w_t) \rangle_{ H_{\nicefrac{1}{2}} }
  +
  \phi(w)
  \langle v, F(v+w_t) \rangle_H
\\&\leq
  \epsilon \| v \|_{H_1}^2
  +
  (c + \phi(w))
  \| v \|_{ H_{\nicefrac{1}{2}} }^2
  +
  \kappa c
  \phi(w)
  \| v \|_H^2
  +
  \Phi(w),
\end{split}
\end{align}
\begin{equation}
  \text{and}
  \qquad
  \Y{t}
  =
  \int_0^t
  P
  e^{(t-s)A}
  \one_{ 
    [ 0, (\nicefrac{M}{T})^{\chi} ]
  }^{\R}
  ( \| \Y{\fl{s}} \|_{ H_{\gamma} } + \| \O{\fl{s}} \|_{ H_{\gamma} } ) \,
  F( \Y{\fl{s}} ) \, ds
  +
  \O{t}.
\end{equation}
Then
\begin{enumerate}[(i)]
\item\label{it:Y_a_priori_1}
  we have that the function 
  $
    [0,T] \ni t
    \mapsto
    \Y{t} - \O{t}
    \in P(H)
  $
  is continuous and
\item\label{it:Y_a_priori_2}
  we have that 
\begin{align}
\begin{split}
 &\sup_{ t \in [0,T] }
  \big(
    \| \Y{t} - \O{t} \|_{H_{\nicefrac{1}{2}}}^2
    +
    \phi(\O{})
    \| \Y{t} - \O{t} \|_{H}^2
  \big)
\\&\leq
  \frac{ 
    e^{2cT}  
  }{ c }
  \left(
    \Phi(\O{})
    +    
    \frac{
      \max\{1, \phi(\O{})\}
      C(c+1)
    }{2(1-\epsilon)(1-\kappa)c}
    \left[ 
      \tfrac{\max\{1,T\} (1+\sqrt{C})}{(1-\rho)}
    \right]^{\!(2+\varphi)}
  \right) .
\end{split}
\end{align}
\end{enumerate}
\end{lemma}
\begin{proof}[Proof of Lemma~\ref{lem:Y_a_priori}]
Throughout this proof
assume w.l.o.g.\ that $ I \neq \emptyset $
and let
$ \Yb{} \colon [0,T] \rightarrow P(H) $
and
$ Z \colon [0,T] \rightarrow \{0,1\} $
be the functions which
satisfy for all
$ s \in [0,T] $ that
\begin{equation}
  \Yb{s} = \Y{s} - \O{s}
  \qquad 
  \text{and}
  \qquad
  Z_s
  =
  \one_{ 
    [ 0, (\nicefrac{M}{T})^{\chi} ]
  }^{\R}
  ( \| \Y{s} \|_{ H_{\gamma} } + \| \O{s} \|_{ H_{\gamma} } ).
\end{equation}
Observe that, e.g., Lemma~2.4 in~\cite{JentzenSalimovaWelti2016}
(with 
$ V = P(H) $,
$ 
  \left\| \cdot \right\|\!_V
  =
  P(H) \ni v \mapsto \left\| v \right\|\!_H^2 \in [0,\infty)
$,
$ T = T $, $ \eta = 0 $,
$ A = P(H) \ni v \mapsto Av \in P(H) $,
$ 
  \mathbb{V} 
  = 
  P(H) \ni v \mapsto 
  \| v \|_{H_{\nicefrac{1}{2}}}^2 + \phi(\O{}) \| v \|_H^2 \in \R 
$,
$ 
  Z 
  = 
  [0,T] \times \Omega 
  \ni
  (t, \omega)
  \mapsto 
  Z_t(\omega)
  \in 
  \R
$, 
$ Y = \Y{} $, $ O = \O{} $, $ \mathbb{O} = \O{} $,
$ F = P(H) \ni v \mapsto PF(v) \in P(H) $, 
$ \phi = V \ni v \mapsto 2c \in \R $,
$ f = V \ni v \mapsto 0 \in \R $, $ h = \nicefrac{T}{M} $ in the notation of
Lemma~2.4 in~\cite{JentzenSalimovaWelti2016})
implies that~\eqref{it:Y_a_priori_1} holds
and that for all $ t \in [0,T] $ 
it holds that
\begin{align}
\begin{split}
&
  e^{-2ct}
  \left[
    \| \Yb{t} \|_{H_{\nicefrac{1}{2}}}^2
    +
    \phi(\O{})
    \| \Yb{t} \|_H^2
  \right]
\\ &=
  2
  \int_0^t
  e^{-2cs} \,
  \Big[
    \langle 
      \Yb{s},
      A \Yb{s}
      +
      Z_{\fl{s}}
      P F( \Yb{s} + \O{\fl{s}} )
    \rangle_{H_{\nicefrac{1}{2}}}
\\&
\quad
+
    \phi(\O{})
    \langle 
      \Yb{s},
      A \Yb{s}
      +
      Z_{\fl{s}}
      P F( \Yb{s} + \O{\fl{s}} )
    \rangle_H
  \Big] \, ds
\\&
\quad
+
  2
  \int_0^t
  e^{-2cs} \,
  Z_{\fl{s}}
  \Big[ 
    \langle 
      \Yb{s},
      P [
        F( \Y{\fl{s}} )
        -
        F( \Yb{s} + \O{\fl{s}} )
      ]
    \rangle_{H_{\nicefrac{1}{2}}}
\\&
\quad
+
    \phi(\O{})
    \langle 
      \Yb{s},
      P [
        F( \Y{\fl{s}} )
        -
        F( \Yb{s} + \O{\fl{s}} )
      ]
    \rangle_H
  \Big] \, ds
\\&
\quad
-
  2c 
  \int_0^t
  e^{-2cs}
  \left[
    \| \Yb{s} \|_{H_{\nicefrac{1}{2}}}^2
    +
    \phi(\O{})
    \| \Yb{s} \|_H^2
  \right] ds .
\end{split}
\end{align}
The fact that
$ P \in L(H) $
is symmetric
hence proves for all 
$ t \in [0,T] $ that
\begin{align}
\begin{split}
 &e^{-2ct}
  \left[
    \| \Yb{t} \|_{H_{\nicefrac{1}{2}}}^2
    +
    \phi(\O{})
    \| \Yb{t} \|_H^2
  \right]
\\&=
  -2
  \int_0^t
  e^{-2cs}
  \left[
    \langle
      \Yb{s},
      (-A) \Yb{s}
    \rangle_{H_{\nicefrac{1}{2}}}
    +
    \phi(\O{})
    \langle 
      \Yb{s},
      (-A) \Yb{s}
    \rangle_H
  \right] ds
\\&\quad+
  2
  \int_0^t
  e^{-2cs} \,
  Z_{\fl{s}}
  \left[
    \langle
      \Yb{s},
      P F( \Yb{s} + \O{\fl{s}} )
    \rangle_{H_{\nicefrac{1}{2}}}
    +
    \phi(\O{})
    \langle
      P
      \Yb{s},
      F( \Yb{s} + \O{\fl{s}} )
    \rangle_H
  \right] ds
\\&\quad+
  2
  \int_0^t
  e^{-2cs} \,
  Z_{\fl{s}}
  \Big[ 
    \langle 
      (-A)^{\nicefrac{1}{2}}
      \Yb{s},
      (-A)^{\nicefrac{1}{2}}
      P [
        F( \Y{\fl{s}} )
        -
        F( \Yb{s} + \O{\fl{s}} )
      ]
    \rangle_H
\\&\quad+
    \phi(\O{})
    \langle 
      P \Yb{s},
      F( \Y{\fl{s}} )
      -
      F( \Yb{s} + \O{\fl{s}} )
    \rangle_H
  \Big] \, ds
\\&\quad- 
  2c 
  \int_0^t
  e^{-2cs}
  \left[
    \| \Yb{s} \|_{H_{\nicefrac{1}{2}}}^2
    +
    \phi(\O{})
    \| \Yb{s} \|_H^2
  \right] ds .
\end{split}
\end{align}
The fact that
$
  \forall \,
  s \in [0,T]
  \colon
  \Yb{s} \in P(H)
$
therefore implies for all 
$ t \in [0,T] $ that
\begin{align}
\begin{split}
 &e^{-2ct}
  \left[
    \| \Yb{t} \|_{H_{\nicefrac{1}{2}}}^2
    +
    \phi(\O{})
    \| \Yb{t} \|_H^2
  \right]
\\&=
  -2
  \int_0^t
  e^{-2cs}
  \left[
    \langle
      (-A) \Yb{s},
      (-A) \Yb{s}
    \rangle_H
    +
    \phi(\O{})
    \langle 
      (-A)^{\nicefrac{1}{2}}
      \Yb{s},
      (-A)^{\nicefrac{1}{2}} \Yb{s}
    \rangle_H
  \right] ds
\\&\quad+
  2
  \int_0^t
  e^{-2cs} \,
  Z_{\fl{s}}
  \left[
    \langle
      \Yb{s},
      P F( \Yb{s} + \O{\fl{s}} )
    \rangle_{H_{\nicefrac{1}{2}}}
    +
    \phi(\O{})
    \langle
      \Yb{s},
      F( \Yb{s} + \O{\fl{s}} )
    \rangle_H
  \right] ds
\\&\quad+
  2
  \int_0^t
  e^{-2cs} \,
  Z_{\fl{s}}
  \Big[ 
    \langle 
      (-A)
      \Yb{s},
      P [
        F( \Y{\fl{s}} )
        -
        F( \Yb{s} + \O{\fl{s}} )
      ]
    \rangle_H
\\&\quad+
    \phi(\O{})
    \langle 
      \Yb{s},
      F( \Y{\fl{s}} )
      -
      F( \Yb{s} + \O{\fl{s}} )
    \rangle_H
  \Big] \, ds
\\&\quad- 
  2c 
  \int_0^t
  e^{-2cs}
  \left[
    \| \Yb{s} \|_{H_{\nicefrac{1}{2}}}^2
    +
    \phi(\O{})
    \| \Yb{s} \|_H^2
  \right] ds .
\end{split}
\end{align}
This and the Cauchy-Schwarz inequality
ensure for all $ t \in [0,T] $ that
\begin{align}
\begin{split}
 &e^{-2ct}
  \left[
    \| \Yb{t} \|_{H_{\nicefrac{1}{2}}}^2
    +
    \phi(\O{})
    \| \Yb{t} \|_H^2
  \right]
\\&\leq
  -2
  \int_0^t
  e^{-2cs}
  \left[
    \| \Yb{s} \|_{H_1}^2
    +
    (
      c
      +
      \phi(\O{})
    )
    \| \Yb{s} \|_{H_{\nicefrac{1}{2}}}^2
    +
    c \,
    \phi(\O{})
    \| \Yb{s} \|_H^2
  \right] ds
\\&\quad+
  2
  \int_0^t
  e^{-2cs} \,
  Z_{\fl{s}}
  \left[
    \langle
      \Yb{s},
      P F( \Yb{s} + \O{\fl{s}} )
    \rangle_{H_{\nicefrac{1}{2}}}
    +
    \phi(\O{})
    \langle
      \Yb{s},
      F( \Yb{s} + \O{\fl{s}} )
    \rangle_H
  \right] ds
\\&\quad+
  2
  \int_0^t
  e^{-2cs} \,
  Z_{\fl{s}}
  \Bigg[ 
    \left( 
      \sqrt{2(1-\epsilon)} \,
      \| \Yb{s} \|_{H_1}
    \right) \!
    \left(
      \tfrac{1}{\sqrt{2(1-\epsilon)}} \,
      \|
        P [
          F( \Y{\fl{s}} )
          -
          F( \Yb{s} + \O{\fl{s}} )
        ]
      \|_H
    \right)
\\&\quad+
    \phi(\O{})
    \left(
      \sqrt{2c(1-\kappa)} \,
      \| \Yb{s} \|_H
    \right) \!
    \left(
      \tfrac{1}{\sqrt{2c(1-\kappa)}} \,
      \|
        F( \Y{\fl{s}} )
        -
        F( \Yb{s} + \O{\fl{s}} )
      \|_H
    \right)
  \Bigg] \, ds .
\end{split}
\end{align}
The fact that
\begin{equation}
  \forall \,
  x,y \in \R
  \colon
  2 xy \leq x^2 + y^2
\end{equation}
hence
proves that for all $ t \in [0,T] $
we have that
\begin{align}
\begin{split}
 &e^{-2ct}
  \left[
    \| \Yb{t} \|_{H_{\nicefrac{1}{2}}}^2
    +
    \phi(\O{})
    \| \Yb{t} \|_H^2
  \right]
\\&\leq
  -2
  \int_0^t
  e^{-2cs}
  \left[
    \| \Yb{s} \|_{H_1}^2
    +
    (
      c
      +
      \phi(\O{})
    )
    \| \Yb{s} \|_{H_{\nicefrac{1}{2}}}^2
    +
    c \,
    \phi(\O{})
    \| \Yb{s} \|_H^2
  \right] ds
\\&\quad+
  2
  \int_0^t
  e^{-2cs} \,
  Z_{\fl{s}}
  \left[
    \langle
      \Yb{s},
      P F( \Yb{s} + \O{\fl{s}} )
    \rangle_{H_{\nicefrac{1}{2}}}
    +
    \phi(\O{})
    \langle
      \Yb{s},
      F( \Yb{s} + \O{\fl{s}} )
    \rangle_H
  \right] ds
\\&\quad+
  \int_0^t
  e^{-2cs} \,
  Z_{\fl{s}}
  \bigg[  
    2 (1-\epsilon) \,
    \| \Yb{s} \|_{H_1}^2
    +
    \tfrac{1}{2(1-\epsilon)} \,
    \| P \|_{ L(H) }^2 \,
    \|
      F( \Y{\fl{s}} )
      -
      F( \Yb{s} + \O{\fl{s}} )
    \|_H^2
\\&\quad+
    2c(1-\kappa) \,
    \phi(\O{})
    \| \Yb{s} \|_H^2
    +
    \tfrac{\phi(\O{})}{2c(1-\kappa)} \,
    \|
      F( \Y{\fl{s}} )
      -
      F( \Yb{s} + \O{\fl{s}} )
    \|_H^2
  \bigg] \, ds .
\end{split}
\end{align}
The fact 
$ \| P \|_{ L(H) } \leq 1 $
therefore shows
for all $ t \in [0,T] $ that
\begin{align}
\begin{split}
 &e^{-2ct}
  \left[
    \| \Yb{t} \|_{H_{\nicefrac{1}{2}}}^2
    +
    \phi(\O{})
    \| \Yb{t} \|_H^2
  \right]
\\&\leq
  -2
  \int_0^t
  e^{-2cs}
  \left[
    \epsilon
    Z_{\fl{s}}
    \| \Yb{s} \|_{H_1}^2
    +
    (
      c
      +
      \phi(\O{})
    )
    \| \Yb{s} \|_{H_{\nicefrac{1}{2}}}^2
    +
    \kappa c \,
    \phi(\O{})
    \| \Yb{s} \|_H^2
  \right] ds
\\&\quad+
  2
  \int_0^t
  e^{-2cs} \,
  Z_{\fl{s}}
  \left[
    \langle
      \Yb{s},
      P F( \Yb{s} + \O{\fl{s}} )
    \rangle_{H_{\nicefrac{1}{2}}}
    +
    \phi(\O{})
    \langle
      \Yb{s},
      F( \Yb{s} + \O{\fl{s}} )
    \rangle_H
  \right] ds
\\&\quad+
  \left[ 
    \tfrac{1}{2(1-\epsilon)}
    +
    \tfrac{\phi(\O{})}{2c(1-\kappa)}
  \right]
  \int_0^t
  e^{-2cs} \,
  Z_{\fl{s}}
  \|
    F( \Y{\fl{s}} )
    -
    F( \Yb{s} + \O{\fl{s}} )
  \|_H^2 \, ds .
\end{split}
\end{align}
Hence, we obtain that for all $ t \in [0,T] $ 
we have that
\begin{align}
\label{eq:Y_a_priori_1}
\begin{split}
 &\| \Yb{t} \|_{H_{\nicefrac{1}{2}}}^2
  +
  \phi(\O{})
  \| \Yb{t} \|_H^2
\\&\leq
  -2
  \int_0^t
  e^{2c(t-s)}
  \left[
    \epsilon
    Z_{\fl{s}}
    \| \Yb{s} \|_{H_1}^2
    +
    (
      c
      +
      \phi(\O{})
    )
    \| \Yb{s} \|_{H_{\nicefrac{1}{2}}}^2
    +
    \kappa c \,
    \phi(\O{})
    \| \Yb{s} \|_H^2
  \right] ds
\\&\quad+
  2
  \int_0^t
  e^{2c(t-s)} \,
  Z_{\fl{s}}
  \left[
    \langle
      \Yb{s},
      P F( \Yb{s} + \O{\fl{s}} )
    \rangle_{H_{\nicefrac{1}{2}}}
    +
    \phi(\O{})
    \langle
      \Yb{s},
      F( \Yb{s} + \O{\fl{s}} )
    \rangle_H
  \right] ds
\\&\quad+
  \left[ 
    \frac{1}{2(1-\epsilon)}
    +
    \frac{\phi(\O{})}{2c(1-\kappa)}
  \right]
  \int_0^t
  e^{2c(t-s)} \,
  Z_{\fl{s}}
  \|
    F( \Y{\fl{s}} )
    -
    F( \Yb{s} + \O{\fl{s}} )
  \|_H^2 \, ds .
\end{split}
\end{align}
Moreover, note that the triangle inequality 
implies that for all 
$ s \in [0,T] $ we have that
\begin{align}
\begin{split}
&
  Z_{\fl{s}}
  \| \Yb{s} - \Yb{\fl{s}} \|_{ H_{\rho} }
\\ &=
  Z_{\fl{s}}
  \| (\Y{s} - \O{s}) - (\Y{\fl{s}} - \O{\fl{s}}) \|_{ H_{\rho} }
\\&\leq 
  Z_{\fl{s}}
  \| 
    (e^{(s-\fl{s})A} - \Id_H)
    (\Y{\fl{s}} - \O{\fl{s}})
  \|_{ H_{\rho} }
\\&\quad+
  Z_{\fl{s}}
  \|
    (\Y{s} - \O{s}) 
    - 
    e^{(s-\fl{s})A}(\Y{\fl{s}} - \O{\fl{s}})  
  \|_{ H_{\rho} }
\\&\leq
  Z_{\fl{s}}
  \| 
    (-A)^{-(\gamma-\rho)}(e^{(s-\fl{s})A} - \Id_H)
  \|_{L(H)}
  \|
    \Y{\fl{s}} - \O{\fl{s}}
  \|_{ H_{\gamma} }
\\&\quad+
  Z_{\fl{s}}
  \int_{\fl{s}}^s
  \|
    P e^{(s-u)A}
    Z_{\fl{u}}
    F( \Y{\fl{u}} )
  \|_{ H_{\rho} } \, du .
\end{split}
\end{align}
The fact that
\begin{equation}
  \forall \,
  s \in (0,\infty)
  ,
  r \in [0,1]
  \colon
  \|
    (-sA)^{-r}
    (e^{sA} - \Id_H)
  \|_{ L(H) }
  \leq 
  1
  ,
\end{equation}
the triangle inequality,
the fact that
\begin{equation}
  \forall \,
  s \in [0,\infty)
  ,
  r \in [0,1]
  \colon
  \|
    (-sA)^{r}
    e^{sA}
  \|_{ L(H) }
  \leq 
  1
  ,
\end{equation}
the fact that
$
  \| P \|_{ L(H) }
  \leq 
  1
$,
and
the assumption that
\begin{equation}
  \forall \,
  v \in P(H)
  \colon 
  \| F(v) \|_H^2
  \leq 
  C 
  \max\{ 1, \| v \|_{H_{\gamma}}^{(2+\varphi)} \}
\end{equation}
hence ensure for all
$ s \in [0,T] $ that
\begin{align}
\begin{split}
 &Z_{\fl{s}}
  \| \Yb{s} - \Yb{\fl{s}} \|_{ H_{\rho} }
\\&\leq
  (s-\fl{s})^{(\gamma-\rho)}
  Z_{\fl{s}}
  \big(
    \|
      \Y{\fl{s}}
    \|_{ H_{\gamma} }
    +
    \|
      \O{\fl{s}}
    \|_{ H_{\gamma} }
  \big)
\\&\quad+
  Z_{\fl{s}}
  \int_{\fl{s}}^s
  \|
    P
  \|_{ L(H) }
  \|
    (-A)^{\rho}
    e^{(s-u)A}
  \|_{ L(H) }
  \|
    F( \Y{\fl{s}} )
  \|_{ H } \, du
\\&\leq
  |\nicefrac{T}{M}|^{(\gamma-\rho)}
  |\nicefrac{M}{T}|^{\chi}
  +
  \sqrt{C}
  \int_{\fl{s}}^s
  (s-u)^{-\rho} \,
  Z_{\fl{s}}
  \max\!\left\{
    1,
    \| \Y{\fl{s}} \|_{ H_{\gamma} }^{(1+\nicefrac{\varphi}{2})}
  \right\} du
\\&\leq
  |\nicefrac{T}{M}|^{(\gamma-\rho-\chi)}
  +
  \sqrt{C}
  \max\!\left\{
    1,
    |\nicefrac{M}{T}|^{(1+\nicefrac{\varphi}{2})\chi}
  \right\}
  \int_{\fl{s}}^s
  (s-u)^{-\rho} \, du .
\end{split}
\end{align}
This shows that for all 
$ s \in [0,T] $
we have that
\begin{align}
\label{eq:Y_a_priori_5}
\begin{split}
 &Z_{\fl{s}}
  \| \Yb{s} - \Yb{\fl{s}} \|_{ H_{\rho} }
\\&\leq
  |\nicefrac{T}{M}|^{(\gamma-\rho-\chi)}
  +
  \sqrt{C}
  \max\!\left\{
    1,
    |\nicefrac{M}{T}|^{(1+\nicefrac{\varphi}{2})\chi}
  \right\}
  \frac{
    (s-\fl{s})^{(1-\rho)}
  }{
    (1-\rho)
  }
\\&\leq
  \frac{1}{(1-\rho)}
  \left[ 
    |\nicefrac{T}{M}|^{(\gamma-\rho-\chi)}
    +
    \sqrt{C}
    \max\!\left\{
      |\nicefrac{T}{M}|^{(1-\rho)},
      |\nicefrac{T}{M}|^{(1-\rho -(1+\nicefrac{\varphi}{2})\chi)}
    \right\}
  \right]
\\&\leq
  \frac{(1+\sqrt{C})}{(1-\rho)}
  \max\!\left\{
    \nicefrac{T}{M},
    |\nicefrac{T}{M}|^{(\gamma-\rho-\chi)},
    |\nicefrac{T}{M}|^{(1-\rho -(1+\nicefrac{\varphi}{2})\chi)}
  \right\} .
\end{split}
\end{align}
Next observe that the 
assumption 
\begin{equation}
  \forall \,
  u,v \in P(H)
  \colon
  \|
    F(u) - F(v)
  \|_H^2
  \leq
  C
  \max\{ 1, \| u \|_{ H_{\gamma} }^{\varphi} \}
  \| u-v \|_{ H_{\rho} }^2
  +
  C
  \| u-v \|_{ H_{\rho} }^{(2+\varphi)}
\end{equation}
ensures for all $ s \in [0,T] $ that
\begin{align}
\begin{split}
 &Z_{\fl{s}}
  \|
    F( \Y{\fl{s}} )
    -
    F( \Yb{s} + \O{\fl{s}} )
  \|_H^2
\\&\leq
  C Z_{\fl{s}} \!
  \Big[ 
    \max\{ 1, \| \Y{\fl{s}} \|_{ H_{\gamma} }^{\varphi} \}
    \| \Yb{\fl{s}} - \Yb{s} \|_{ H_{\rho} }^2
    +
    \| \Yb{\fl{s}} - \Yb{s} \|_{ H_{\rho} }^{(2+\varphi)}
  \Big]
\\&\leq
  C Z_{\fl{s}}
  \| \Yb{\fl{s}} - \Yb{s} \|_{ H_{\rho} }^2
  \Big[ 
    \max\{ 1, |\nicefrac{M}{T}|^{\varphi\chi} \}
    +
    Z_{\fl{s}}
    \| \Yb{\fl{s}} - \Yb{s} \|_{ H_{\rho} }^{\varphi}
  \Big]
\\&\leq
  2 C Z_{\fl{s}}
  \| \Yb{\fl{s}} - \Yb{s} \|_{ H_{\rho} }^2
  \Big[ 
    \max\!\big\{ 1, |\nicefrac{M}{T}|^{\varphi\chi},
           Z_{\fl{s}} \| \Yb{\fl{s}} - \Yb{s} \|_{ H_{\rho} }^{\varphi}
         \big\}
  \Big]
\\&=
  2 C Z_{\fl{s}}
  \| \Yb{\fl{s}} - \Yb{s} \|_{ H_{\rho} }^2
  \Big[ 
    \max\!\big\{ 1, |\nicefrac{M}{T}|^{\chi},
           Z_{\fl{s}} \| \Yb{\fl{s}} - \Yb{s} \|_{ H_{\rho} }
         \big\}
  \Big]^{\varphi} .
\end{split}
\end{align}
This together with~\eqref{eq:Y_a_priori_5}
proves for all 
$ s \in [0,T] $ that
\begin{align}
\label{eq:Y_a_priori_6}
\begin{split}
 &Z_{\fl{s}}
  \|
    F( \Y{\fl{s}} )
    -
    F( \Yb{s} + \O{\fl{s}} )
  \|_H^2
\\&\leq
  \frac{2 C (1+\sqrt{C})^{(2+\varphi)}}{(1-\rho)^{(2+\varphi)}}
  \left|
    \max\!\left\{
      \nicefrac{T}{M},
      |\nicefrac{T}{M}|^{(\gamma-\rho-\chi)},
      |\nicefrac{T}{M}|^{(1-\rho -(1+\nicefrac{\varphi}{2})\chi)}
    \right\}
  \right|^2
\\&\quad\cdot
  \left|
    \max\!\left\{
      1,
      |\nicefrac{M}{T}|^{\chi},
      \nicefrac{T}{M},
      |\nicefrac{T}{M}|^{(\gamma-\rho-\chi)},
      |\nicefrac{T}{M}|^{(1-\rho -(1+\nicefrac{\varphi}{2})\chi)}
    \right\}
  \right|^{\varphi}
\\&=
  \frac{2 C (1+\sqrt{C})^{(2+\varphi)}}{(1-\rho)^{(2+\varphi)}}
  \left|
    \max\!\left\{
      \nicefrac{T}{M},
      |\nicefrac{T}{M}|^{(\gamma-\rho-\chi)},
      |\nicefrac{T}{M}|^{(1-\rho -(1+\nicefrac{\varphi}{2})\chi)}
    \right\}
  \right|^2
\\&\quad\cdot
  \left|
    \max\!\left\{
      \nicefrac{T}{M},
      |\nicefrac{M}{T}|^{\chi},
      |\nicefrac{T}{M}|^{(\gamma-\rho-\chi)},
      |\nicefrac{T}{M}|^{(1-\rho -(1+\nicefrac{\varphi}{2})\chi)}
    \right\}
  \right|^{\varphi} .
\end{split}
\end{align}
In addition, note that the assumption that
$
  \chi 
  \in 
  (0,\nicefrac{(\gamma-\rho)}{(1+\nicefrac{\varphi}{2})}] 
  \cap 
  (0,\nicefrac{(1-\rho)}{(1+\varphi)}]  
$
ensures that
\begin{align}
\label{eq:Y_a_priori_7}
\begin{split}
 &\gamma-\rho-\chi 
  \in (0,1),
  \qquad 
  1-\rho -(1+\nicefrac{\varphi}{2})\chi 
  \in (0,1) ,
\end{split}
\end{align}
and
\begin{align}
\label{eq:Y_a_priori_8}
\begin{split}
 &\min\{ 
    \gamma-\rho -(1+\nicefrac{\varphi}{2})\chi,
    1-\rho -(1+\varphi)\chi
  \}
  \in [0,1) .
\end{split}
\end{align}
This implies that for all $ h \in (0,1] $ 
we have that
\begin{align}
\label{eq:Y_a_priori_9}
\begin{split}
 &\left|
    \max\!\left\{
      h,
      h^{(\gamma-\rho-\chi)},
      h^{(1-\rho -(1+\nicefrac{\varphi}{2})\chi)}
    \right\}
  \right|^2
  \left|
    \max\!\left\{
      h,
      h^{-\chi},
      h^{(\gamma-\rho-\chi)},
      h^{(1-\rho -(1+\nicefrac{\varphi}{2})\chi)}
    \right\}
  \right|^{\varphi}
\\&=
  h^{2\min\{\gamma-\rho-\chi,1-\rho -(1+\nicefrac{\varphi}{2})\chi\}}
  h^{-\varphi\chi}
  =
  h^{2\min\{\gamma-\rho -(1+\nicefrac{\varphi}{2})\chi,1-\rho -(1+\varphi)\chi\}}
  \leq 
  1 .
\end{split}
\end{align}
Moreover, observe that~\eqref{eq:Y_a_priori_7}
shows for all $ h \in (1,\infty) $ that
\begin{align}
\label{eq:Y_a_priori_10}
\begin{split}
 &\left|
    \max\!\left\{
      h,
      h^{(\gamma-\rho-\chi)},
      h^{(1-\rho -(1+\nicefrac{\varphi}{2})\chi)}
    \right\}
  \right|^2
  \left|
    \max\!\left\{
      h,
      h^{-\chi},
      h^{(\gamma-\rho-\chi)},
      h^{(1-\rho -(1+\nicefrac{\varphi}{2})\chi)}
    \right\}
  \right|^{\varphi}
  =
  h^{(2+\varphi)} .
\end{split}
\end{align}
Combining~\eqref{eq:Y_a_priori_6}
with \eqref{eq:Y_a_priori_9} and~\eqref{eq:Y_a_priori_10}
yields that for all $ s \in [0,T] $ 
we have that
\begin{align}
\label{eq:Y_a_priori_11}
\begin{split}
 &Z_{\fl{s}}
  \|
    F( \Y{\fl{s}} )
    -
    F( \Yb{s} + \O{\fl{s}} )
  \|_H^2
  \leq
  2 C
  \left[ 
    \frac{\max\{1,T\} (1+\sqrt{C})}{(1-\rho)}
  \right]^{\!(2+\varphi)} .
\end{split}
\end{align}
Furthermore, note that 
the assumption that
$
  \forall \,
  v \in P(H)
$,
$
  w \in \C([0,T],P(H))
$,
$ 
  s \in [0,T]
  \colon
  \langle v, P F(v+w_s) \rangle_{ H_{\nicefrac{1}{2}} }
  +
  \phi(w)
  \langle v, F(v+w_s) \rangle_H
  \leq
  \epsilon \| v \|_{H_1}^2
  +
  (c + \phi(w))
  \| v \|_{ H_{\nicefrac{1}{2}} }^2
  +
  \kappa c
  \phi(w)
  \| v \|_H^2
  +
  \Phi(w)
$
ensures that for all 
$ v \in P(H) $,
$ w \in \C([0,T],P(H)) $,
$ s \in [0,T] $
we have that
\begin{multline}
  \langle v, P F(v+w_{\fl{s}}) \rangle_{ H_{\nicefrac{1}{2}} }
  +
  \phi(w)
  \langle v, F(v+w_{\fl{s}}) \rangle_H
\\
  \leq
  \epsilon \| v \|_{H_1}^2
  +
  (c + \phi(w)) 
  \| v \|_{ H_{\nicefrac{1}{2}} }^2
  +
  \kappa c 
  \phi(w)
  \| v \|_H^2
  +
  \Phi(w) .
\end{multline}
This implies that for all 
$ w \in \C([0,T],P(H)) $,
$ s \in [0,T] $
we have that
\begin{multline}
  \langle \Yb{s}, P F(\Yb{s}+w_{\fl{s}}) \rangle_{ H_{\nicefrac{1}{2}} }
  +
  \phi(w)
  \langle \Yb{s}, F(\Yb{s}+w_{\fl{s}}) \rangle_H
\\
  \leq
  \epsilon \| \Yb{s} \|_{H_1}^2
  +
  (c + \phi(w)) 
  \| \Yb{s} \|_{ H_{\nicefrac{1}{2}} }^2
  +
  \kappa c 
  \phi(w)
  \| \Yb{s} \|_H^2
  +
  \Phi(w) .
\end{multline}
The assumption that
$ 
  \O{} \in \C( [0,T], P(H) )
$
hence guarantees for all 
$ s \in [0,T] $ that
\begin{multline}
\label{eq:Y_a_priori_12}
  \langle \Yb{s}, P F(\Yb{s}+\O{\fl{s}}) \rangle_{ H_{\nicefrac{1}{2}} }
  +
  \phi(\O{})
  \langle \Yb{s}, F(\Yb{s}+\O{\fl{s}}) \rangle_H
\\
  \leq
  \epsilon \| \Yb{s} \|_{H_1}^2
  +
  (c + \phi(\O{})) 
  \| \Yb{s} \|_{ H_{\nicefrac{1}{2}} }^2
  +
  \kappa c 
  \phi(\O{})
  \| \Yb{s} \|_H^2
  +
  \Phi(\O{}) .
\end{multline}
Combining~\eqref{eq:Y_a_priori_1}
with~\eqref{eq:Y_a_priori_11} 
and~\eqref{eq:Y_a_priori_12}
demonstrates that for all 
$ t \in [0,T] $ we have that
\begin{align}
\begin{split}
 &\| \Yb{t} \|_{H_{\nicefrac{1}{2}}}^2
  +
  \phi(\O{})
  \| \Yb{t} \|_H^2
\\&\leq
  -
  2
  \int_0^t
  e^{2c(t-s)}
  \left[
    \epsilon
    Z_{\fl{s}}
    \| \Yb{s} \|_{H_1}^2
    +
    (
      c
      +
      \phi(\O{})
    )
    \| \Yb{s} \|_{ H_{\nicefrac{1}{2}} }^2
    +
    \kappa c
    \phi(\O{})
    \| \Yb{s} \|_H^2
  \right] ds
\\&\quad+
  2
  \int_0^t
  e^{2c(t-s)} \,
  Z_{\fl{s}}
  \left[
    \epsilon
    \| \Yb{s} \|_{H_1}^2
    +
    (c + \phi(\O{})) 
    \| \Yb{s} \|_{ H_{\nicefrac{1}{2}} }^2
    +
    \kappa c
    \phi(\O{})
    \| \Yb{s} \|_H^2
    +
    \Phi(\O{})
  \right] ds
\\&\quad+
  2 C
  \left[
    \frac{1}{2(1-\epsilon)}
    +
    \frac{\phi(\O{})}{2c(1-\kappa)}
  \right] \!
  \left[ 
    \frac{\max\{1,T\} (1+\sqrt{C})}{(1-\rho)}
  \right]^{\!(2+\varphi)}
  \left[ 
    \int_0^t
    e^{2c(t-s)} \, ds
  \right] .
\end{split}
\end{align}
This shows that for all 
$ t \in [0,T] $ we have that
\begin{align}
\begin{split}
 &\| \Yb{t} \|_{H_{\nicefrac{1}{2}}}^2
  +
  \phi(\O{})
  \| \Yb{t} \|_H^2
\\&\leq
  \frac{ 
    \Phi(\O{}) \left[ e^{2ct} - 1 \right]  
  }{ c }
  +
  \frac{ 
    C \left[ e^{2ct} - 1 \right]  
  }{ c }
  \left[
    \frac{1}{2(1-\epsilon)}
    +
    \frac{\phi(\O{})}{2c(1-\kappa)}
  \right] \!
  \left[ 
    \frac{\max\{1,T\} (1+\sqrt{C})}{(1-\rho)}
  \right]^{\!(2+\varphi)} .
\end{split}
\end{align}
Hence, we obtain that
\begin{align}
\begin{split}
 &\sup_{ t \in [0,T] }
  \Big[
    \| \Yb{t} \|_{H_{\nicefrac{1}{2}}}^2
    +
    \phi(\O{})
    \| \Yb{t} \|_H^2
  \Big]
\\&\leq
  \frac{ 
    \left( e^{2cT} - 1 \right)  
  }{ c }
  \left( 
    \Phi(\O{}) 
    +
    \frac{ 
      C  
    }{ 2 }
    \left[
      \frac{1}{(1-\epsilon)}
      +
      \frac{\phi(\O{})}{c(1-\kappa)}
    \right] \!
    \left[ 
      \tfrac{\max\{1,T\} (1+\sqrt{C})}{(1-\rho)}
    \right]^{\!(2+\varphi)}
  \right) 
\\&\leq
  \frac{ 
    \left( e^{2cT} - 1 \right)  
  }{ c }
  \left( 
    \Phi(\O{}) 
    +
    \frac{ 
      C
      \max\{ 1, \phi(\O{}) \} 
    }{ 2(1-\epsilon) (1-\kappa) }
    \left[
      1
      +
      \frac{1}{c}
    \right] \!
    \left[ 
      \tfrac{\max\{1,T\} (1+\sqrt{C})}{(1-\rho)}
    \right]^{\!(2+\varphi)}
  \right) 
\\&=
  \frac{ 
    \left( e^{2cT} - 1 \right)  
  }{ c }
  \left( 
    \Phi(\O{}) 
    +
    \frac{ 
      \max\{ 1, \phi(\O{}) \} 
      C ( 1 + c )
    }{ 2(1-\epsilon) (1-\kappa) c }
    \left[ 
      \tfrac{\max\{1,T\} (1+\sqrt{C})}{(1-\rho)}
    \right]^{\!(2+\varphi)}
  \right) .
\end{split}
\end{align}
The proof of Lemma~\ref{lem:Y_a_priori} is thus completed.
\end{proof}

\section{Pathwise error estimates}
\label{sec:pw_section}

\subsection{Setting}
\label{sec:pw_setting}
Consider the notation in Section~\ref{sec:notation},
let $ ( H, \left< \cdot, \cdot \right>_H, \left\| \cdot \right\|_H ) $
be a separable $ \R $-Hilbert space,
let $ \H \subseteq H $ be a non-empty orthonormal basis of $ H $,
let $ T, c, \varphi \in (0,\infty) $, $ C \in [0,\infty) $, $ M \in \N $,
$ \mu \colon \H \rightarrow \R $ satisfy $ \sup_{ h \in \H } \mu_h < 0 $,
let $ A \colon D(A) \subseteq H \rightarrow H $ be the linear operator
which satisfies
$ 
  D(A) 
  = 
  \{ 
    v \in H 
    \colon 
    \sum_{ h \in \H } 
    | 
      \mu_h 
      \langle h, v \rangle_H 
    |^2
    <
    \infty
  \}
$
and 
$
  \forall \, v \in D(A)
  \colon
  Av
  =
  \sum_{ h \in \H }
  \mu_h
  \langle h,v \rangle_H h
$,
let $ ( V, \left\| \cdot \right\|_V ) $ be an $ \R $-Banach space
with $ D(A) \subseteq V \subseteq H $ continuously and densely,
and let 
$ O, \O{}, \XX{}, \Y{} \colon [0,T] \rightarrow V $
and 
$ \Vm \colon V \times V \rightarrow [0,\infty) $
be functions,
and let
$ X \in \C( [0,T], V ) $,
$ F \in \C( V,H ) $, $ I \in \Pow_0(\H) $, 
$ P \in L(H) $
satisfy for all $ v,w \in D(A) $, $ t \in [0,T] $ that
\begin{equation}
  P(v) 
  = 
  \textstyle\sum_{ h \in I } 
  \langle h,v \rangle_H h ,
  \qquad
  \langle
    v-w,
    Av
    +
    F(v)
    -
    Aw
    -
    F(w)
  \rangle_H
  \leq
  c
  \|
    v-w
  \|_H^2 ,
\end{equation}
\begin{equation}
  \| F(v) - F(w) \|_H^2 
  \leq 
  C 
  \| 
    v-w 
  \|_V^2
  (
    1
    +
    \| v \|_V^{\varphi}
    +
    \| w \|_V^{\varphi}
  ) ,
\end{equation}
\begin{equation}
  X_t
  =
  \int_0^t
  e^{(t-s)A}
  F( X_s ) \, ds
  +
  O_t ,
  \qquad 
  \XX{t}
  =
  \int_0^t
  P
  e^{(t-s)A}
  F( \Y{\fl{s}} ) \, ds
  +
  P O_t ,
\end{equation}
\begin{equation}
  \text{and}
  \qquad
  \Y{t}
  =
  \int_0^t
  P
  e^{(t-s)A}
  \one_{ 
    [ 0, \nicefrac{M}{T} ]
  }^{\R}
  ( \Vm(\Y{\fl{s}},\O{\fl{s}}) ) \,
  F( \Y{\fl{s}} ) \, ds
  +
  \O{t} .
\end{equation}

\subsection{On the separability of a certain Banach space}

The next elementary lemma, 
Lemma~\ref{lem:separable},
ensures that the $\R$-Banach space
$ (V, \left\|\cdot\right\|\!_V) $
in Section~\ref{sec:pw_setting}
is separable. Lemma~\ref{lem:separable}
is well-known in the literature.
The proof of Lemma~\ref{lem:separable}
is given for completeness.

\begin{lemma}
\label{lem:separable}
Let $ ( V, \left\| \cdot \right\|_V ) $ be a separable
$\R$-Banach space and let $ ( W, \left\| \cdot \right\|_W ) $
be an $\R$-Banach space with $ V \subseteq W $
continuously and densely. Then
$ ( W, \left\| \cdot \right\|_W ) $ is a separable
$\R$-Banach space.
\end{lemma}
\begin{proof}[Proof of Lemma~\ref{lem:separable}]
Throughout this proof let
$ w \in W $ be a vector
and let $ \varepsilon \in (0,\infty) $
be a strictly positive real number.
Note that the assumption that
$ V \subseteq W $ continuously
and densely ensures that there exists
a $ c \in \R $ and a $ v \in V $
such that
\begin{align}
\label{eq:separable_1}
\begin{split}
 &c
  =
  \sup\!\left( 
    \left\{ 
      \frac{ \| u \|_W }{ \| u \|_V }
      \colon 
      u \in V \setminus \{0\}
    \right\} 
    \cup 
    \{1\}
  \right)
  \qquad \text{and} \qquad 
  \| v - w \|_W
  <
  \frac{\varepsilon}{2} .
\end{split}
\end{align}
In addition,
observe that the assumption that
$ ( V, \left\| \cdot \right\|_V ) $
is a separable $ \R $-Banach space shows
that there exists an at most countable
set $ A \subseteq V $ such that
\begin{align}
 &\overline{A}^V
  =
  V .
\end{align}
Hence, we obtain that there exists an 
$ a \in A $ such that
$
  \| a-v \|_V
  <
  \frac{ \varepsilon }{ 2 c }
$.
Combining this and~\eqref{eq:separable_1}
with the triangle inequality
establishes that
\begin{equation}
\begin{split}
  \| a-w \|_W
& \leq 
  \| a - v \|_W
  +
  \| v - w \|_W
  <
  \| a - v \|_W
  +
  \frac{\varepsilon}{2}
\\ &
  \leq
  c \, \| a - v \|_V
  +
  \frac{\varepsilon}{2}
  <
  c \cdot \frac{ \varepsilon }{ 2c }
  +
  \frac{\varepsilon}{2}
  =
  \varepsilon .
\end{split}
\end{equation}
The proof of Lemma~\ref{lem:separable} is thus completed.
\end{proof}

\subsection{Analysis of the error between the Galerkin
projection of the exact solution and the Galerkin projection 
of the semilinear integrated 
version of the numerical approximation}
\label{sec:pw_analysis}

In our error analysis in Lemma~\ref{lem:pw_X_XX_diff}
below we employ the following elementary result,
Lemma~\ref{lem:Y_regularity} below, on mild solutions
of certain semilinear evolution equations.

\begin{lemma}
\label{lem:Y_regularity}
Consider the notation in Section~\ref{sec:notation},
let $ (V, \left\|\cdot\right\|\!_V) $ be a separable $ \R $-Banach space,
and let $ A \in L(V) $, $ T \in (0,\infty) $,
$ Y \colon [0,T] \rightarrow V $, $ Z \in \C( [0,T], V ) $
satisfy for all $ t \in [0,T] $ that
$ Y_t = \int_0^t e^{(t-s)A} Z_s \, ds $. 
Then
\begin{enumerate}[(i)]
 \item\label{it:Y_regularity_1}
 we have that $ Y $ is continuously differentiable,
 \item\label{it:Y_regularity_2}
 we have for all $ t \in [0,T] $ that
 $ Y_t = \int_0^t A Y_s + Z_s \, ds $, and 
 \item\label{it:Y_regularity_3}
 we have for all $ t \in [0,T] $ that
 $ \frac{d}{dt} Y_t = A Y_t + Z_t $.
\end{enumerate}
\end{lemma}
\begin{proof}[Proof of Lemma~\ref{lem:Y_regularity}]
Note that Lemma~2.2 in~\cite{JentzenSalimovaWelti2016} (with $ V = V $,
$ A = A $,
$ T = T $, $ Y = Y $, 
$ Z = Z $
in the notation
of Lemma~2.2 in \cite{JentzenSalimovaWelti2016})
implies that the function
$ [0,T] \ni t \mapsto Y_t \in V $ is continuous
and that~\eqref{it:Y_regularity_2} holds.
Next observe that for all 
$ t_1, t_2 \in [0,T] $ with $ t_1 \leq t \leq t_2 $
we have that
\begin{align}
\begin{split}
 &Y_{t_2} 
  - 
  Y_{t_1}
  -
  A Y_t 
  \left(t_2 - t_1\right)
  -
  Z_t 
  \left(t_2 - t_1\right)
\\&=
  \int_{t_1}^{t_2}
  e^{(t_2-s)A}
  Z_s \, ds
  +
  e^{(t_2-t_1)A}
  \int_0^{t_1}
  e^{(t_1-s)A}
  Z_s \, ds
  -
  Y_{t_1}
  -
  A Y_t 
  \left(t_2 - t_1\right)
  -
  Z_t 
  \left(t_2 - t_1\right)
\\&=
  \left( 
    e^{(t_2-t_1)A}
    -
    \Id_V
  \right)
  Y_{t_1}
  -
  A Y_t 
  \left(t_2 - t_1\right)
  +
  \int_{t_1}^{t_2}
  \left(
    e^{(t_2-s)A}
    Z_s
    -
    Z_t
  \right) ds
\\&=
  \left( 
    e^{(t_2-t_1)A}
    -
    \Id_V
    -
    A
    (t_2 - t_1)
  \right)
  Y_t
  +
  \left( 
    e^{(t_2-t_1)A}
    -
    \Id_V
  \right)
  \left(
    Y_{ t_1 }
    -
    Y_t
  \right)
\\&\quad+
  \int_{t_1}^{t_2}
  e^{(t_2-s)A}
  \left(
    Z_s
    -
    Z_t
  \right) ds
  +
  \int_{t_1}^{t_2}
  \left(
    e^{(t_2-s)A}
    -
    \Id_V
  \right) Z_t \, ds.
\end{split}
\end{align}
This proves
for all 
$ t_1, t_2 \in [0,T] $ with $ t_1 \leq t \leq t_2 $
that
\begin{align}
\begin{split}
 &Y_{t_2} 
  - 
  Y_{t_1}
  -
  A Y_t 
  \left(t_2 - t_1\right)
  -
  Z_t 
  \left(t_2 - t_1\right)
\\&=
  \int_0^{(t_2-t_1)}
  \left( 
    e^{sA}
    -
    \Id_V
  \right)
  A Y_t \, ds
  +
  \int_0^{(t_2-t_1)}
  e^{sA}
  A
  (
    Y_{ t_1 }
    -
    Y_t
  ) \, ds
\\&\quad+
  \int_{t_1}^{t_2}
  e^{(t_2-s)A}
  \left(
    Z_s
    -
    Z_t
  \right) ds
  +
  \int_0^{(t_2-t_1)}
  \left(
    e^{sA}
    -
    \Id_V
  \right) Z_t \, ds
\\&= 
  \int_0^{(t_2-t_1)}
  \int_0^s
  e^{rA}
  A
  ( A Y_t + Z_t ) \, dr \, ds
  +
  \int_0^{(t_2-t_1)}
  e^{sA}
  A
  (
    Y_{ t_1 }
    -
    Y_t
  ) \, ds
\\&\quad+ 
  \int_{t_1}^{t_2}
  e^{(t_2-s)A}
  \left(
    Z_s
    -
    Z_t
  \right) ds .
\end{split}
\end{align}
Hence, we obtain that
for all $ t \in [0,T] $,
$ (t_1,t_2) \in ( [0,t] \times [t,T] ) \setminus \{ (t,t) \} $
we have that
\begin{align}
\begin{split}
 &\frac{ 
   \|
     Y_{t_2} 
    - 
    Y_{t_1}
    -
    A Y_t 
    \left(t_2 - t_1\right)
    -
    Z_t 
    \left(t_2 - t_1\right)
    \|_V
  }{ (t_2 - t_1) }
\\&\leq 
  (t_2 - t_1)
  \left[ 
    \sup_{ s \in (0,T) }
    \| e^{sA} \|_{ L(V) }
  \right] 
  \| A \|_{ L(V) } \,
  \| A Y_t + Z_t \|_V
\\ &
\quad
  +
  \left[ 
    \sup_{ s \in (0,T) }
    \| e^{sA} \|_{ L(V) }
  \right] 
  \| A \|_{ L(V) } \,
  \| Y_{ t_1} - Y_t \|_V
\\&\quad+ 
  \left[ 
    \sup_{ s \in (0,T) }
    \| e^{sA} \|_{ L(V) }
  \right] 
  \left[ 
    \sup_{ s \in (t_1,t_2) }
    \| Z_s - Z_t \|_V
  \right] .
\end{split}
\end{align}
The fact that 
$ Y \in \C([0,T], V) $
and the assumption that
$ Z \in \C([0,T], V) $
therefore ensure for all 
$ t \in [0,T] $ that
\begin{align}
\begin{split}
 &\limsup_{
    \substack{
      (t_1,t_2) \rightarrow (t,t) , \\
      (t_1,t_2) \in ( [0,t] \times [t,T] ) \setminus \{ (t,t) \}
    }
  }
  \left(
    \frac{ 
     \|
       Y_{t_2} 
      - 
      Y_{t_1}
      -
      A Y_t 
      \left(t_2 - t_1\right)
      -
      Z_t 
      \left(t_2 - t_1\right)
      \|_V
    }{ (t_2 - t_1) }
  \right)
  =
  0 .
\end{split}
\end{align}
This completes the proof of Lemma~\ref{lem:Y_regularity}.
\end{proof}

\begin{lemma}
\label{lem:pw_X_XX_diff}
Assume the setting in Section~\ref{sec:pw_setting}
and let $ \kappa \in (2,\infty) $, $ t \in [0,T]$.
Then
\begin{align}
\begin{split}
  \|
    P X_t - \XX{t}
  \|_H^2
 &\leq
  \frac{C e^{\kappa cT} }{
    \left( \kappa - 2 \right) c
  }
  \int_0^t
  \left[ 
    \|
      X_s - P X_s
    \|_V
    +
    \|
      \XX{s} - \Y{\fl{s}}
    \|_V 
  \right]^2
\\&\quad\cdot
  \left(
    1
    +
    \|
      X_s
    \|_V^{\varphi}
    +
    \|
      P X_s
    \|_V^{\varphi}
    +
    \|
      \XX{s}
    \|_V^{\varphi}
    +
    \|
      \Y{\fl{s}}
    \|_V^{\varphi}
  \right) ds .
\end{split}
\end{align}

\end{lemma}
\begin{proof}[Proof of Lemma~\ref{lem:pw_X_XX_diff}]
Throughout this proof assume
w.l.o.g.\ that $ t \in (0,T] $
(otherwise the proof is clear).
Observe that Lemma~\ref{lem:Y_regularity}
(with $ V = P(H) $,
$ A = ( P(H) \ni v \mapsto Av \in P(H) ) $,
$ T = T $, $ Y = ( [0,T] \ni s \mapsto P(X_s -  O_s) \in P(H) ) $, 
$ Z = ( [0,T] \ni s \mapsto P F( X_s ) \in P(H) ) $
in the notation
of Lemma~\ref{lem:Y_regularity})
shows that the function $ [0,T] \ni s \mapsto P(X_s -  O_s) \in P(H) $
is continuously differentiable and that for all 
$ s \in [0,T] $ we have that
\begin{align}
\label{eq:X_diffable}
  P(X_s -  O_s)
  =
  \int_0^s
  A P (X_u-O_u)
  +
  P F( X_u ) \, du . 
\end{align}
Next note that Lemma~2.1 in 
Hutzenthaler et al.~\cite{HutzenthalerJentzenSalimova2016} 
(with $ V = H $, $ A = A $,
$ T = T $, $ h = \nicefrac{T}{M} $, 
$ Y = ( [0,T] \ni s \mapsto \XX{s} - P O_s \in H ) $,
$ Z = ( [0,T] \ni s \mapsto P F(\Y{s}) \in H ) $ 
in the notation
of Lemma~2.1 in Hutzenthaler et al.~\cite{HutzenthalerJentzenSalimova2016})
implies that for all $ s \in [0,T] $ we have that
$ \XX{s} - P O_s \in D(A) $, that
the function $ [0,T] \ni s \mapsto \XX{s} - P O_s \in D(A) $
is continuous,
that the function 
$ [0,T]\setminus\{0, \frac{T}{M}, \frac{2T}{M}, \ldots \} \ni s \mapsto \XX{s} - P O_s \in H $
is continuously differentiable,
and that for all 
$ s \in [0,T] $ we have that
\begin{align}
  \XX{s} - P O_s
  =
  \int_0^s
  A (\XX{u}- P O_u)
  +
  P F( \Y{\fl{u}} ) \, du . 
\end{align}
This, \eqref{eq:X_diffable}, the fundamental
theorem of calculus, 
and the fact that
\begin{equation}
  \forall \, 
  s \in [0,T]
  \colon
  A P X_s, A \XX{s}
  \in 
  P(H)
\end{equation}
prove that
\begin{align}
\begin{split}
 &e^{-\kappa c t}
  \left\|
    P X_t - \XX{t}
  \right\|_H^2
  =
  e^{-\kappa c t}
  \left\|
    P (X_t - O_t) - (\XX{t} - P O_t)
  \right\|_H^2
\\&=
  2
  \int_0^t
  e^{-\kappa c s} \,
  \big<
    P X_s - \XX{s},
    A P (X_s-O_s) + P F( X_s )
    - A (\XX{s} - P O_s) - P F( \Y{\fl{s}} )
  \big>_H \, ds
\\&\quad-
  \kappa c
  \int_0^t
  e^{-\kappa c s}
  \left\|
    P X_s - \XX{s}
  \right\|_H^2 ds
\\&=
  2
  \int_0^t
  e^{-\kappa c s} \,
  \big<
    P X_s - \XX{s},
    A P X_s + P F( X_s )
    - A \XX{s} - P F( \Y{\fl{s}} )
  \big>_H \, ds
\\&\quad-
  \kappa c
  \int_0^t
  e^{-\kappa c s}
  \left\|
    P X_s - \XX{s}
  \right\|_H^2 ds
\\&=
  2
  \int_0^t
  e^{-\kappa c s} \,
  \big<
    P X_s - \XX{s},
    P \big[ A P X_s + F( X_s )
    - A \XX{s} - F( \Y{\fl{s}} ) \big]
  \big>_H \, ds
\\&\quad-
  \kappa c
  \int_0^t
  e^{-\kappa c s}
  \left\|
    P X_s - \XX{s}
  \right\|_H^2 ds .
\end{split}
\end{align}
The fact that $ P \in L(H) $
is symmetric together with the fact
that 
\begin{equation}
  \forall \, 
  s \in [0,T]
  \colon
  P X_s, \XX{s}
  \in 
  P(H)
\end{equation}
therefore ensures that
\begin{align}
\begin{split}
 &e^{-\kappa c t}
  \left\|
    P X_t - \XX{t}
  \right\|_H^2
\\&=
  2
  \int_0^t
  e^{-\kappa c s} \,
  \big<
    P X_s - \XX{s},
    A P X_s + F( X_s )
    - A \XX{s} - F( \Y{\fl{s}} )
  \big>_H \, ds
\\&\quad-
  \kappa c
  \int_0^t
  e^{-\kappa c s}
  \left\|
    P X_s - \XX{s}
  \right\|_H^2 ds
\\&=
  2
  \int_0^t
  e^{-\kappa c s} \,
  \big<
    P X_s - \XX{s},
    A P X_s + F( P X_s )
    - A \XX{s} - F( \XX{s} )
  \big>_H \, ds
\\&\quad+
  2
  \int_0^t
  e^{-\kappa c s} \,
  \big<
    P X_s - \XX{s},
    F( X_s ) - F( P X_s)
    + F( \XX{s} ) - F( \Y{\fl{s}} )
  \big>_H \, ds
\\&\quad-
  \kappa c
  \int_0^t
  e^{-\kappa c s}
  \left\|
    P X_s - \XX{s}
  \right\|_H^2 ds .
\end{split}
\end{align}
The assumption that
\begin{equation}
  \forall \,
  v, w \in D(A)
  \colon
  \left<
    v-w,
    Av + F(v) - Aw - F(w)
  \right>_H
  \leq
  c
  \| v-w \|_H^2
  ,
\end{equation}
the Cauchy-Schwarz inequality,
and the fact that 
\begin{equation}
  \forall \, x,y \in \R
  \colon 
  xy \leq \nicefrac{x^2}{2} + \nicefrac{y^2}{2}
\end{equation}
hence demonstrate that
\begin{align}
\begin{split}
 &e^{-\kappa c t}
  \left\|
    P X_t - \XX{t}
  \right\|_H^2
\\&\leq
  2 c
  \int_0^t
  e^{-\kappa c s} 
  \left\|
    P X_s - \XX{s}
  \right\|_H^2 ds
  -
  \kappa c
  \int_0^t
  e^{-\kappa c s}
  \left\|
    P X_s - \XX{s}
  \right\|_H^2 ds
\\&\quad+
  2
  \int_0^t
  e^{-\kappa c s} \,
  \|
    P X_s - \XX{s}
  \|_H \,
  \|
    F( X_s ) - F( P X_s) + F( \XX{s} ) - F( \Y{\fl{s}} )
  \|_H \, ds
\\&=
  \left(2-\kappa\right) c
  \int_0^t
  e^{-\kappa c s} 
  \left\|
    P X_s - \XX{s}
  \right\|_H^2 ds
  +
  2
  \int_0^t
  e^{-\kappa c s} 
  \left[
    \sqrt{ \left( \kappa - 2 \right) c }
    \left\|
      P X_s - \XX{s}
    \right\|_H
  \right]
\\&\quad\cdot 
  \left[ 
    \tfrac{1}{
      \sqrt{ \left( \kappa - 2 \right) c }
    } \,
    \|
      F( X_s ) - F( P X_s) + F( \XX{s} ) - F( \Y{\fl{s}} )
    \|_H 
  \right] ds
\\&\leq
  \left(2-\kappa\right) c
  \int_0^t
  e^{-\kappa c s} 
  \left\|
    P X_s - \XX{s}
  \right\|_H^2 ds
  +
  \int_0^t
  e^{-\kappa c s} \,
  \bigg[
    \left( \kappa - 2 \right) c
    \left\|
      P X_s - \XX{s}
    \right\|_H^2
\\&\quad+  
    \tfrac{1}{
      \left( \kappa - 2 \right) c
    } \,
    \|
      F( X_s ) - F( P X_s) + F( \XX{s} ) - F( \Y{\fl{s}} )
    \|_H^2 
  \bigg] \, ds .
\end{split}
\end{align}
The triangle inequality
and the fact that
\begin{equation}
  \forall \,
  v,w \in V
  \colon
  \| F(v) - F(w) \|_H^2
  \leq
  C
  \| v-w \|_V^2
  (
    1
    +
    \| v \|_V^{\varphi}
    +
    \| w \|_V^{\varphi}
  )
\end{equation}
therefore yield that
\begin{align}
\begin{split}
 &e^{-\kappa c t}
  \left\|
    P X_t - \XX{t}
  \right\|_H^2
\\&\leq
  \frac{1}{
    \left( \kappa - 2 \right) c
  }
  \int_0^t
  e^{-\kappa c s}
  \left[ 
    \big\|
      F( X_s ) - F( P X_s)
    \big\|_H 
    +  
    \big\|
      F( \XX{s} ) - F( \Y{\fl{s}} )
    \big\|_H
  \right]^2 ds
\\&\leq
  \frac{C}{
    \left( \kappa - 2 \right) c
  }
  \int_0^t
  e^{-\kappa c s} \,
  \bigg[
    \|
      X_s - P X_s
    \|_V \,
    \sqrt{
      1
      +
      \|
	X_s
      \|_V^{\varphi}
      +
      \|
	P X_s
      \|_V^{\varphi}
    }
\\&\quad+
    \|
      \XX{s} - \Y{\fl{s}}
    \|_V \,
    \sqrt{
      1
      +
      \|
	\XX{s}
      \|_V^{\varphi}
      +
      \|
	\Y{\fl{s}}
      \|_V^{\varphi}
    } \,
  \bigg]^2 \, ds .
\end{split}
\end{align}
This completes the proof of Lemma~\ref{lem:pw_X_XX_diff}.
\end{proof}

\subsection{Analysis of the error between the numerical 
approximation and the Galerkin projection
of the semilinear integrated version of the numerical
approximation}

\begin{lemma}
\label{lem:XX_YO_diff_pw}
Assume the setting in Section~\ref{sec:pw_setting}
and let $ \alpha \in (0,\infty) $, $ \rho \in [0,1) $,
$ t \in [0,T] $
satisfy 
$
  \sup_{ s \in (0,T] }
  s^{\rho}
  \| e^{sA} \|_{ L(H, V) }
  <
  \infty
$. Then
\begin{align}
\begin{split}
&
  \|
    \XX{t}
    -
    \Y{t}
  \|_V
\\
&\leq
  \frac{ T^{\alpha} }{ M^{\alpha} }
  \left[
    \sup_{ s \in (0,T) }
    s^{\rho} \,
    \big\|
      e^{sA}
    \big\|_{ L(H,V) }
  \right]
  \int_0^t
  \left(
    t-s
  \right)^{-\rho}
  \big|
    \Vm(\Y{\fl{s}},\O{\fl{s}})
  \big|^{\alpha} \, 
  \|
    F( \Y{\fl{s}} ) 
  \|_H \, ds
\\&\quad+
  \|
    P O_t
    -
    \O{t}
  \|_V .
\end{split}
\end{align}

\end{lemma}

\begin{proof}[Proof of Lemma~\ref{lem:XX_YO_diff_pw}]
Throughout this proof we assume w.l.o.g.\ that $ t \in (0,T] $
(otherwise the proof is clear).
Observe that 
\begin{align}
\begin{split}
 &\|
    \XX{t}
    -
    \Y{t}
  \|_V
\\&\leq
  \int_0^t
  \big\|
    P
    e^{(t-s)A}
    \big[
      1
      -
      \one_{ 
	[ 0, \nicefrac{M}{T} ]
      }^{\R}
      ( \Vm(\Y{\fl{s}},\O{\fl{s}}) ) 
    \big]
    F( \Y{\fl{s}} )
  \big\|_V \, ds
  + 
  \|
    P O_t
    -
    \O{t}
  \|_V
\\&\leq
  \left[
    \sup_{ s \in (0,T) }
    s^{\rho} \,
    \big\|
      e^{sA}
    \big\|_{ L(H,V) }
  \right]
  \int_0^t
  \left(
    t-s
  \right)^{-\rho}
  \big\|
    \one_{ 
      ( \nicefrac{M}{T}, \infty )
    }^{\R}
    ( \Vm(\Y{\fl{s}},\O{\fl{s}}) ) \,
    P
    F( \Y{\fl{s}} ) 
  \big\|_H \, ds
\\&\quad+ 
  \|
    P O_t
    -
    \O{t}
  \|_V
\\&\leq
  \left[
    \sup_{ s \in (0,T) }
    s^{\rho} \,
    \big\|
      e^{sA}
    \big\|_{ L(H,V) }
  \right]
  \int_0^t
  \left(
    t-s
  \right)^{-\rho}
  \one_{ 
    ( \nicefrac{M}{T}, \infty )
  }^{\R}
  ( \Vm(\Y{\fl{s}},\O{\fl{s}}) )
  \left[
    \tfrac{ T }{ M }
    \Vm(\Y{\fl{s}},\O{\fl{s}})
  \right]^{\alpha}
\\&\quad\cdot
  \| P \|_{ L(H) } \,
  \|
    F( \Y{\fl{s}} ) 
  \|_H \, ds 
  + 
  \|
    P O_t
    -
    \O{t}
  \|_V .
\end{split}
\end{align}
This and the fact that $ \| P \|_{ L(H) } \leq 1 $
complete the proof of Lemma~\ref{lem:XX_YO_diff_pw}.
\end{proof}

\subsection{Temporal regularity for the 
Galerkin projection of the
semilinear integrated version of the numerical approximation}

\begin{lemma}
\label{lem:XX_XX_diff_pw}
Assume the setting in Section~\ref{sec:pw_setting}
and let $ \rho \in [0,1) $, $ \varrho \in [0, 1-\rho) $,
$ t_1 \in [0,T) $, $ t_2 \in (t_1, T] $
satisfy 
$
  \sup_{ s \in (0,T] }
  s^{\rho}
  \| e^{sA} \|_{ L(H, V) }
  <
  \infty
$. Then
\begin{align}
\begin{split}
 &\|
    \XX{t_1}
    -
    \XX{t_2}
  \|_V
  \leq
  \left[
    \sup_{ s \in (0,T) }
    s^{\rho} \,
    \big\|
      e^{sA}
    \big\|_{ L(H,V) }
  \right]
  \Bigg[
    \int_{t_1}^{t_2}
    \left(
      t_2 - s
    \right)^{-\rho}
    \|
      F( \Y{\fl{s}} )
    \|_H \, ds
\\&\quad+
    2^{(\rho+\varrho)}
    \left(
      t_2 - t_1
    \right)^{\varrho}
    \int_0^{t_1}
    \left(
      t_1-s
    \right)^{-(\rho+\varrho)}
    \|
      F( \Y{\fl{s}} )
    \|_H \, ds
  \Bigg]
  +
  \|
    P ( O_{t_2} - O_{t_1} ) 
  \|_V .
\end{split}
\end{align}
\end{lemma}

\begin{proof}[Proof of Lemma~\ref{lem:XX_XX_diff_pw}]
Observe that
\begin{align}
\begin{split}
 &\|
    \XX{t_2}
    -
    \XX{t_1}
  \|_V
\\&\leq
  \int_{t_1}^{t_2}
  \big\|
    P
    e^{(t_2-s)A}
    F( \Y{\fl{s}} )
  \big\|_V \, ds
  +
  \int_0^{t_1}
  \big\|
    P
    \big(
      e^{(t_2-s)A}
      -
      e^{(t_1-s)A}
    \big)
    F( \Y{\fl{s}} )
  \big\|_V \, ds
\\ &
\quad
  +
  \big\|
    P ( O_{t_2} - O_{t_1} ) 
  \big\|_V
\\&\leq
  \left[
    \sup_{ s \in (0,T) }
    s^{\rho} \,
    \big\|
      e^{sA}
    \big\|_{ L(H,V) }
  \right]
  \int_{t_1}^{t_2}
  \left(
    t_2 - s
  \right)^{-\rho}
  \big\|
    P
    F( \Y{\fl{s}} )
  \big\|_H \, ds
  +
  \big\|
    P ( O_{t_2} - O_{t_1} ) 
  \big\|_V
\\&\quad+
  \left[
    \sup_{ s \in (0,T) }
    s^{\rho} \,
    \big\|
      e^{sA}
    \big\|_{ L(H,V) }
  \right]
  \int_0^{t_1}
  \left[
    \frac{2}{t_1-s}
  \right]^{\rho}
  \big\|
    P
    e^{\frac{1}{2}(t_1-s)A}
    \left(
      e^{(t_2-t_1)A}
      -
      \Id_H
    \right)
    F( \Y{\fl{s}} )
  \big\|_H \, ds .
\end{split}
\end{align}
This, the fact that 
\begin{equation}
  \forall \,
  s \in [0,\infty) 
  ,
  r \in [0,1]
  \colon 
  \| (-sA)^r e^{sA} \|_{ L(H) } 
  \leq 1
  ,
\end{equation}
the fact that 
\begin{equation}
  \forall \,
  s \in (0,\infty) 
  ,
  r \in [0,1]
  \colon 
  \| (-sA)^{-r} (e^{sA} - \Id_H ) \|_{ L(H) } 
  \leq 1
  ,
\end{equation}
and the fact that $ \| P \|_{ L(H) } \leq 1 $
prove that
\begin{align}
\begin{split}
 &\|
    \XX{t_2}
    -
    \XX{t_1}
  \|_V
\\&\leq
  \left[
    \sup_{ s \in (0,T) }
    s^{\rho} \,
    \big\|
      e^{sA}
    \big\|_{ L(H,V) }
  \right]
  \int_{t_1}^{t_2}
  \left(
    t_2 - s
  \right)^{-\rho}
  \big\|
    P
  \big\|_{ L(H) }
  \big\|
    F( \Y{\fl{s}} )
  \big\|_H \, ds
\\ &
\quad
  +
  \big\|
    P ( O_{t_2} - O_{t_1} ) 
  \big\|_V
\\&\quad+
  2^{\rho}
  \left[
    \sup_{ s \in (0,T) }
    s^{\rho} \,
    \big\|
      e^{sA}
    \big\|_{ L(H,V) }
  \right]
  \int_0^{t_1}
  \left(
    t_1-s
  \right)^{-\rho}
  \big\|
    (-A)^{\varrho} \,
    e^{\frac{1}{2}(t_1-s)A}
  \big\|_{ L(H) }
\\&\quad\cdot
  \big\|
    (-A)^{-\varrho}
    \left(
      e^{(t_2-t_1)A}
      -
      \Id_H
    \right) \!
  \big\|_{ L(H) }
  \big\|
    P
  \big\|_{ L(H) }
  \big\|
    F( \Y{\fl{s}} )
  \big\|_H \, ds 
\\&\leq
  \left[
    \sup_{ s \in (0,T) }
    s^{\rho} \,
    \big\|
      e^{sA}
    \big\|_{ L(H,V) }
  \right]
  \int_{t_1}^{t_2}
  \left(
    t_2 - s
  \right)^{-\rho}
  \big\|
    F( \Y{\fl{s}} )
  \big\|_H \, ds
  +
  \big\|
    P ( O_{t_2} - O_{t_1} ) 
  \big\|_V
\\&\quad+
  2^{(\rho+\varrho)}
  \left( t_2 - t_1 \right)^{\varrho}
  \left[
    \sup_{ s \in (0,T) }
    s^{\rho} \,
    \big\|
      e^{sA}
    \big\|_{ L(H,V) }
  \right]
  \int_0^{t_1}
  \left(
    t_1-s
  \right)^{-(\rho+\varrho)}
  \big\|
    F( \Y{\fl{s}} )
  \big\|_H \, ds .
\end{split}
\end{align}
The proof of Lemma~\ref{lem:XX_XX_diff_pw}
is thus completed.
\end{proof}

\section{Strong error estimates}
\label{sec:lp_section}

\subsection{Setting}
\label{sec:lp_setting}
Consider the notation in Section~\ref{sec:notation},
let $ ( H, \left< \cdot, \cdot \right>_H, \left\| \cdot \right\|_H ) $
be a separable $ \R $-Hilbert space,
let $ \H \subseteq H $ be a non-empty orthonormal basis of $ H $,
let $ T, c, C, \varphi \in (0,\infty) $,
$ \D \subseteq \Pow_0(\H) $,
$ \mu \colon \H \rightarrow \R $ satisfy
$ \sup_{ h \in \H } \mu_h < 0 $,
let $ A \colon D(A) \subseteq H \rightarrow H $ be the linear operator
which satisfies
$ 
  D(A) 
  = 
  \{ 
    v \in H 
    \colon 
    \sum_{ h \in \H } 
    | 
      \mu_h 
      \langle h, v \rangle_H 
    |^2
    <
    \infty
  \}
$
and 
$
  \forall \, v \in D(A)
  \colon
  Av
  =
  \sum_{ h \in \H }
  \mu_h
  \langle h,v \rangle_H h
$,
let $ ( V, \left\| \cdot \right\|_V ) $ be an $ \R $-Banach space
with $ D(A) \subseteq V \subseteq H $ continuously and densely,
and let 
$ \Vm \in \M\big( \B( V \times V ), \B( [0,\infty) ) \big) $,
$ F \in \C( V,H ) $, $ ( P_I )_{ I \in \Pow(\H) } \subseteq L(H) $
satisfy for all $ v,w \in D(A) $,
$ I \in \Pow(\H) $ that
\begin{equation}
  \langle
    v-w,
    Av
    +
    F(v)
    -
    Aw
    -
    F(w)
  \rangle_H
  \leq
  c
  \|
    v-w
  \|_H^2,
\end{equation}
\begin{equation}
  \| F(v) - F(w) \|_H^2 
  \leq 
  C 
  \| 
    v-w 
  \|_V^2
  (
    1
    +
    \| v \|_V^{\varphi}
    +
    \| w \|_V^{\varphi}
  ),
\end{equation}
and
$ 
  P_I(v) 
  = 
  \sum_{ h \in I } 
  \langle h,v \rangle_H h 
$,
let $ ( \Omega, \F, \P ) $
be a probability space,
let 
$ 
  X
  \colon
  [0,T] \times \Omega
  \rightarrow V
$
be a stochastic process
with continuous sample paths,
and let
$ 
  O
  \colon
  [0,T] \times \Omega
  \rightarrow V
$
and
$
  \XXMN{},
  \YMN{},
  \OMN{}
  \colon
  [0,T] \times \Omega
  \rightarrow V
$,
$ M \in \N $,
$ I \in \D $,
be stochastic processes
which satisfy
for all $ t \in [0,T] $,
$ M \in \N $, $ I \in \D $
that
\begin{equation}
  X_t
  =
  \int_0^t
  e^{(t-s)A}
  F( X_s ) \, ds
  +
  O_t ,
  \qquad
  \XXMN{t}
  =
  \int_0^t
  P_I
  e^{(t-s)A}
  F( \YMN{\fl{s}} ) \, ds
  +
  P_I O_t ,
\end{equation}
\begin{equation}
  \text{and}
  \qquad
  \YMN{t}
  =
  \int_0^t
  P_I
  e^{(t-s)A}
  \one_{ 
    \{
      \Vm(\YMN{\fl{s}},\OMN{\fl{s}})
      \leq
      M/T
    \}
  }^{\Omega} \,
  F( \YMN{\fl{s}} ) \, ds
  +
  \OMN{t} .
\end{equation}

\subsection{Analysis of the error between the Galerkin
projection of the exact solution and the Galerkin 
projection of the semilinear
integrated version of the numerical approximation}

\begin{lemma}
\label{lem:lp_X_XX_diff}
Assume the setting in Section~\ref{sec:lp_setting}
and let $ \kappa \in (2,\infty) $, $ t \in [0,T]$,
$ p \in [2,\infty) $, $ M \in \N $, $ I \in \D $
satisfy
$
  \sup_{ s \in (0,T) }
  \E\big[  
    \| X_s \|_V^{p\varphi}
    +
    \| P_I X_s \|_V^{p\varphi}
    +
    \| \XXMN{s} \|_V^{p\varphi}
    +
    \| \YMN{\fl{s}} \|_V^{p\varphi}
  \big] 
  <
  \infty
$.
Then
\begin{multline}
  \|
    P_I X_t - \XXMN{t}
  \|_{ \L^p(\P; H) }
  \leq
  \frac{\sqrt{ C T e^{\kappa c T} } }{
    \sqrt{ \left( \kappa - 2 \right) c }
  }
  \left[ 
    \sup_{ s \in (0,T) }
    \left(
      \|
        P_{ \H \setminus I } X_s 
      \|_{ \L^{2p}(\P;V) }
      +
      \|
        \XXMN{s} - \YMN{\fl{s}}
      \|_{ \L^{2p}(\P;V) }
    \right)
  \right]
\\
  \cdot
  \left[
    1
    +
    \sup_{ s \in (0,T) }
    \left(
      \|
        X_s
      \|_{ \L^{p\varphi}(\P;V) }^{\nicefrac{\varphi}{2}}
      +
      \|
        P_I X_s
      \|_{ \L^{p\varphi}(\P;V) }^{\nicefrac{\varphi}{2}}
      +
      \|
        \XXMN{s}
      \|_{ \L^{p\varphi}(\P;V) }^{\nicefrac{\varphi}{2}}
      +
      \|
        \YMN{s}
      \|_{ \L^{p\varphi}(\P;V) }^{\nicefrac{\varphi}{2}}
    \right)
  \right] .
\end{multline}
\end{lemma}

\begin{proof}[Proof of Lemma~\ref{lem:lp_X_XX_diff}]
Note that Lemma~\ref{lem:pw_X_XX_diff}
and H{\"o}lder's inequality ensure that
\begin{align}
\begin{split}
&
  \|
    P_I X_t - \XXMN{t}
  \|_{ \L^p(\P; H) }^2
\\ &=
  \big\|
    \|
      P_I X_t - \XXMN{t}
    \|_H^2
  \big\|_{ \L^{\nicefrac{p}{2}}(\P; \R) }
\\&\leq
  \frac{C e^{\kappa cT} }{
    \left( \kappa - 2 \right) c
  }
  \int_0^t
  \bigg\| \!
    \left[ 
      \|
	X_s - P_I X_s
      \|_V
      +
      \|
	\XXMN{s} - \YMN{\fl{s}}
      \|_V
    \right]^2
\\&\quad\cdot
    \left(
      1
      +
      \|
	X_s
      \|_V^{\varphi}
      +
      \|
	P_I X_s
      \|_V^{\varphi}
      +
      \|
	\XXMN{s}
      \|_V^{\varphi}
      +
      \|
	\YMN{\fl{s}}
      \|_V^{\varphi}
    \right) \!
  \bigg\|_{ \L^{\nicefrac{p}{2}}(\P;\R) } \, ds
\\&\leq
  \frac{ C e^{\kappa cT} }{
    \left( \kappa - 2 \right) c
  }
  \int_0^t
  \big\| 
    \|
      X_s - P_I X_s
    \|_V
    +
    \|
      \XXMN{s} - \YMN{\fl{s}}
    \|_V 
  \big\|_{ \L^{2p}(\P;\R) }^2
\\&\quad\cdot
  \big\|
    1
    +
    \|
      X_s
    \|_V^{\varphi}
    +
    \|
      P_I X_s
    \|_V^{\varphi}
    +
    \|
      \XXMN{s}
    \|_V^{\varphi}
    +
    \|
      \YMN{\fl{s}}
    \|_V^{\varphi}
  \big\|_{ \L^p(\P;\R) } \, ds .
\end{split}
\end{align}
This shows that
\begin{align}
\begin{split}
 &\|
    P_I X_t - \XXMN{t}
  \|_{ \L^p(\P; H) }^2
\\&\leq
  \frac{C e^{\kappa cT} }{
    \left( \kappa - 2 \right) c
  }
  \int_0^t
  \left[ 
    \|
      X_s - P_I X_s
    \|_{ \L^{2p}(\P;V) }
    +
    \|
      \XXMN{s} - \YMN{\fl{s}}
    \|_{ \L^{2p}(\P;V) }
  \right]^2
\\&\quad\cdot
  \left[
    1
    +
    \|
      X_s
    \|_{ \L^{p\varphi}(\P;V) }^{\varphi}
    +
    \|
      P_I X_s
    \|_{ \L^{p\varphi}(\P;V) }^{\varphi}
    +
    \|
      \XXMN{s}
    \|_{ \L^{p\varphi}(\P;V) }^{\varphi}
    +
    \|
      \YMN{\fl{s}}
    \|_{ \L^{p\varphi}(\P;V) }^{\varphi}
  \right] ds
\\&\leq
  \frac{ C T e^{\kappa cT} }{
    \left( \kappa - 2 \right) c
  }
  \left[ 
    \sup_{ s \in (0,T) }
    \left(
      \|
        P_{ \H \setminus I } X_s 
      \|_{ \L^{2p}(\P;V) }
      +
      \|
        \XXMN{s} - \YMN{\fl{s}}
      \|_{ \L^{2p}(\P;V) }
    \right)
  \right]^{\!2}
\\&\quad\cdot
  \left[
    1
    +
    \sup_{ s \in (0,T) }
    \left(
      \|
        X_s
      \|_{ \L^{p\varphi}(\P;V) }^{\varphi}
      \|
        P_I X_s
      \|_{ \L^{p\varphi}(\P;V) }^{\varphi}
      +
      \|
        \XXMN{s}
      \|_{ \L^{p\varphi}(\P;V) }^{\varphi}
      +
      \|
        \YMN{\fl{s}}
      \|_{ \L^{p\varphi}(\P;V) }^{\varphi}
    \right)
  \right] .
\end{split}
\end{align}
Combining this with
the fact that 
\begin{equation}
  \forall \, n \in \N
  ,
  x_1, \ldots, x_n \in [0,\infty)
  \colon
  \sqrt{ x_1 + \ldots + x_n }
  \leq
  \sqrt{x_1} + \ldots + \sqrt{x_n}
\end{equation}
completes
the proof of Lemma~\ref{lem:lp_X_XX_diff}.
\end{proof}

\subsection{Analysis of the error between the numerical 
approximation and the Galerkin projection of the 
semilinear integrated version of the numerical 
approximation}

\begin{lemma}
\label{lem:XX_YO_diff_lp}
Assume the setting in Section~\ref{sec:lp_setting}
and let $ \alpha \in (0,\infty) $, $ \rho \in [0,1) $,
$ t \in [0,T] $, $ p \in [1,\infty) $, $ M \in \N $,
$ I \in \D $ satisfy
$
  \sup_{ s \in (0,T] }
  s^{\rho}
  \| e^{sA} \|_{ L(H,V) } 
  +
  \sup_{ s \in [0,T) }
  \E\big[ 
    | \Vm(\YMN{s},\OMN{s}) |^{2p\alpha}
    +
    \| F( \YMN{s} ) \|_H^{2p}
  \big]
  <
  \infty
$. Then
\begin{align}
\begin{split}
 &\|
    \XXMN{t}
    -
    \YMN{t}
  \|_{ \L^p(\P; V) } 
\\&\leq
  \frac{ T^{(1+\alpha-\rho)} }{ \left(1-\rho\right) M^{\alpha} }
  \left[
    \sup_{ s \in (0,T) }
    s^{\rho} \,
    \|
      e^{sA}
    \|_{ L(H,V) }
  \right] \!
  \left[ 
    \sup_{ s \in [0,T) }
    \left(
      \|
        \Vm(\YMN{s},\OMN{s})
      \|_{ \L^{2p\alpha}( \P; \R) }^{\alpha}
      \|
        F( \YMN{s} )  
      \|_{ \L^{2p}( \P; H) }
    \right)
  \right]
\\&\quad+
  \|
    P_I O_t
    -
    \OMN{t}
  \|_{ \L^p(\P; V) } .
\end{split}
\end{align}
\end{lemma}
\begin{proof}[Proof of Lemma~\ref{lem:XX_YO_diff_lp}]
Note that Lemma~\ref{lem:XX_YO_diff_pw}
and H{\"o}lder's inequality imply that
\begin{align}
\begin{split}
 &\|
    \XXMN{t}
    -
    \YMN{t}
  \|_{ \L^p(\P; V) }
\\&\leq
  \frac{ T^{\alpha} }{ M^{\alpha} }
  \left[
    \sup_{ s \in (0,T) }
    s^{\rho} \,
    \|
      e^{sA}
    \|_{ L(H,V) }
  \right]
\\ &
\quad
\cdot 
  \left[  
  \int_0^t
  \left(
    t-s
  \right)^{-\rho}
  \big\|
    |
      \Vm(\YMN{\fl{s}},\OMN{\fl{s}})
    |^{\alpha} \, 
    \|
      F( \YMN{\fl{s}} ) 
    \|_H 
  \big\|_{ \L^p( \P; \R) } \, ds
  \right]
\\&\quad+
  \|
    P_I O_t
    -
    \OMN{t}
  \|_{ \L^p(\P; V) } 
\\&\leq
  \frac{ T^{\alpha} }{ M^{\alpha} }
  \left[
    \sup_{ s \in (0,T) }
    s^{\rho} \,
    \|
      e^{sA}
    \|_{ L(H,V) }
  \right]
\\&
\quad 
\cdot
  \left[
  \int_0^t
  \left(
    t-s
  \right)^{-\rho}
  \|
    \Vm(\YMN{\fl{s}},\OMN{\fl{s}})
  \|_{ \L^{2p\alpha}( \P; \R) }^{\alpha} 
  \,
  \|
    F( \YMN{\fl{s}} )  
  \|_{ \L^{2p}( \P; H) } \, ds
  \right]
\\&\quad+
  \|
    P_I O_t
    -
    \OMN{t}
  \|_{ \L^p(\P; V) } .
\end{split}
\end{align}
This and the fact that 
$
  \int_0^t
  (t-s)^{-\rho} \, ds
  =
  \frac{ t^{(1-\rho)} }{ (1-\rho) }
$
complete the proof of Lemma~\ref{lem:XX_YO_diff_lp}.
\end{proof}

\subsection{Temporal regularity for the 
Galerkin projection of the
semilinear integrated version of the numerical approximation}

\begin{lemma}
\label{lem:XX_XX_diff_lp}
Assume the setting in Section~\ref{sec:lp_setting}
and let $ \rho \in [0,1) $, $ \varrho \in [0, 1-\rho) $,
$ t_1 \in [0,T) $, $ t_2 \in (t_1, T] $, $ p \in [1,\infty) $, 
$ M \in \N $, $ I \in \D $ satisfy 
$
  \sup_{ s \in (0,T] }
  s^{\rho}
  \| e^{sA} \|_{ L(H,V) }
  <
  \infty
$. Then
\begin{multline}
  \|
    \XXMN{t_1}
    -
    \XXMN{t_2}
  \|_{ \L^p( \P; V ) }
  \leq
  \|
    P_I ( O_{t_1} - O_{t_2} ) 
  \|_{ \L^p( \P; V ) }
\\+
  \frac{ 
    3 \, T^{ (1 - \rho - \varrho) }
    \left( t_2 - t_1 \right)^{\varrho}
  }{
    \left( 1 - \rho - \varrho \right)
  }
  \left[
    \sup_{ s \in (0,T) }
    s^{\rho} \,
    \|
      e^{sA}
    \|_{ L(H,V) }
  \right]
  \left[ 
    \sup_{ s \in [0,T) }
    \|
      F( \YMN{s} ) 
    \|_{ \L^p( \P; H ) }
  \right]  .
\end{multline}
\end{lemma}

\begin{proof}[Proof of Lemma~\ref{lem:XX_XX_diff_lp}]
Observe that Lemma~\ref{lem:XX_XX_diff_pw} proves that
\begin{align}
\begin{split}
 &\|
    \XXMN{t_1}
    -
    \XXMN{t_2}
  \|_{ \L^p( \P; V ) }
\\&\leq
  \left[
    \sup_{ s \in (0,T) }
    s^{\rho} \,
    \|
      e^{sA}
    \|_{ L(H,V) }
  \right]
  \left[ 
    \sup_{ s \in [0,T) }
    \|
      F( \YMN{s} ) 
    \|_{ \L^p( \P; H ) }
  \right] 
  \Bigg[
    \int_{t_1}^{t_2}
    \left(
      t_2 - s
    \right)^{-\rho} ds
\\&\quad+
    2^{(\rho+\varrho)}
    \left(
      t_2 - t_1
    \right)^{\varrho}
    \int_0^{t_1}
    \left(
      t_1-s
    \right)^{-(\rho+\varrho)} ds
  \Bigg]
  +
  \|
    P_I ( O_{t_2} - O_{t_1} ) 
  \|_{ \L^p( \P; V ) } 
\\&=
  \left[
    \sup_{ s \in (0,T) }
    s^{\rho} \,
    \|
      e^{sA}
    \|_{ L(H,V) }
  \right]
  \left[ 
    \sup_{ s \in [0,T) }
    \|
      F( \YMN{s} ) 
    \|_{ \L^p( \P; H ) }
  \right] 
\\&\quad\cdot
  \Bigg[
    \frac{ 
      \left( t_2 - t_1 \right)^{(1-\rho)}
    }{
      \left( 1 - \rho \right)
    }
    +
    \frac{
      2^{(\rho+\varrho)}
      \left( t_2 - t_1 \right)^{\varrho}
      | t_1 |^{ (1-\rho-\varrho) }
    }{
      \left( 1 - \rho - \varrho \right)
    }
  \Bigg]
  +
  \|
    P_I ( O_{t_2} - O_{t_1} ) 
  \|_{ \L^p( \P; V ) } .
\end{split}
\end{align}
This completes the proof of Lemma~\ref{lem:XX_XX_diff_lp}.
\end{proof}

\subsection{Analysis of the error 
between the exact solution and the numerical
approximation}

\begin{proposition}
\label{prop:X_Y_diff_lp}
Assume the setting in Section~\ref{sec:lp_setting} and let
$ \alpha \in (0,\infty) $, 
$ \rho \in [0,1) $, $ \varrho \in [0,1-\rho) $,
$ \kappa \in (2,\infty) $,
$ p \in [\max\{2,\nicefrac{1}{\varphi}\}, \infty) $,
$ M \in \N $, $ I \in \D $. Then
\begin{align}
\begin{split}
 &\sup_{ t \in [0,T] }
  \|
    X_t - \YMN{t}
  \|_{ \L^p( \P; H ) }
\\&\leq
  \frac{
    4^{(2+\varphi)}
    \max\{ 1, T^{(\nicefrac{3}{2} + 
                 \alpha - \rho + \nicefrac{\varphi}{2} - \nicefrac{\rho\varphi}{2} )}
        \}
    \max\{ 1, C^{(1+\nicefrac{\varphi}{4})} \}
    \sqrt{  e^{\kappa c T } }
  }{
    \min\{ 1, \sqrt{ c \left( \kappa - 2 \right) } \}
    \left( 1 - \rho - \varrho \right)^{(1+\nicefrac{\varphi}{2})}
  }
  \Bigg[
    \sup_{ t \in [0,T] }
    \|
      P_{ \H \setminus I } X_t 
    \|_{ \L^{2p}(\P;V) }
    +
    M^{-\min\{\alpha,\varrho\} }
\\&+
    \sup_{ t \in (0,T) }
    \|
      P_I ( O_{t} - O_{\fl{t}} ) 
    \|_{ \L^{2p}( \P; V ) }
    +
    \sup_{ t \in [0,T] }
    \|
      P_I O_t
      -
      \OMN{t}
    \|_{ \L^{2p}(\P; V) }
  \Bigg] 
  \max\!\left\{
    1,
    \sup_{ v \in V \setminus \{0\} }
    \frac{ \| v \|_H }{ \| v \|_V }
  \right\}
\\&\cdot
  \Bigg[
    1
    +
    \sup_{ s \in (0,T) }
    \|
      X_s
    \|_{ \L^{p\varphi}(\P;V) }^{\nicefrac{\varphi}{2}}
    +
    \sup_{ s \in (0,T) }
    \|
      P_I X_s
    \|_{ \L^{p\varphi}(\P;V) }^{\nicefrac{\varphi}{2}}
    +
    \sup_{ s \in [0,T) }
    \|
      \YMN{s}
    \|_{ \L^{p\varphi}(\P;V) }^{\nicefrac{\varphi}{2}}
    +
    \sup_{ s \in (0,T) }
    \|
      P_I O_s
    \|_{ \L^{p\varphi}( \P; V ) }^{\nicefrac{\varphi}{2}}
  \Bigg]
\\&\cdot
  \left[ 
    1
    +
    \sup_{ s \in [0,T) }
    \|
      \Vm(\YMN{s},\OMN{s})
    \|_{ \L^{4p\alpha}( \P; \R) }^{\alpha}
  \right]
  \max\!\left\{
    1,
    \sup_{ s \in (0,T) }
    \left[
      s^{\rho} \,
      \|
	e^{sA}
      \|_{ L(H,V) }
    \right]^{(1+\nicefrac{\varphi}{2})}
  \right\} 
\\&\cdot
  \left[ 
    \max\!\left\{ 
      1,
      \sup_{ s \in [0,T) }
      \|
	\YMN{s}
      \|_{ \L^{p(1+\nicefrac{\varphi}{2})\max\{4,\varphi\}}( \P; V ) }^{[(1+\nicefrac{\varphi}{2})^2]}
    \right\} 
    +
    \| F(0) \|_H^{ (1+\nicefrac{\varphi}{2}) } 
  \right] .
\end{split}
\end{align}
\end{proposition}

\begin{proof}[Proof of Proposition~\ref{prop:X_Y_diff_lp}]
Throughout this proof assume w.l.o.g.\ that
$
  \sup_{ s \in (0,T) }
  \big(
    s^{\rho}
    \| e^{sA} \|_{ L(H,V) }
    +
    \E\big[ 
      \| P_I O_s \|_V^{p\varphi}
      +
      \| P_I X_s \|_V^{p\varphi}
      +
      \| X_s \|_V^{p\varphi}
    \big]
  \big)
  +
  \sup_{ s \in [0,T) }
  \E\big[
    | \Vm(\YMN{s},\OMN{s}) |^{4p\alpha}
    +
    \| \YMN{s} \|_V^{ p(1+\nicefrac{\varphi}{2})\max\{4,\varphi\} }
  \big]
  <
  \infty
$.
Note that the fact that
\begin{equation}
  \forall \, v,w \in V
  \colon
  \| F(v) - F(w) \|_H^2
  \leq
  C
  \| v-w \|_V^2
  (
    1
    +
    \| v \|_V^{\varphi}
    +
    \| w \|_V^{\varphi}
  )  
\end{equation}
and the fact that
\begin{equation}
  \forall \, x,y \in [0, \infty)
  \colon
  \sqrt{ x + y}
  \leq
  \sqrt{x} + \sqrt{y} 
\end{equation}
imply for all $ s \in [0,T] $ that
\begin{align}
\begin{split}
\|
    F( \YMN{s} )
  \|_H
&\leq
  \|
    F( \YMN{s} )
    -
    F(0)
  \|_H
  +
  \| F(0) \|_H
\\ &
  \leq
  \sqrt{ 
    C \,
    \|
      \YMN{s}
    \|_V^2
    \big(
      1
      +
      \|
        \YMN{s}
      \|_V^{\varphi}
    \big)
  } 
  +
  \| F(0) \|_H
\\&\leq
  \sqrt{ 
    C
  } \,
  \|
    \YMN{s}
  \|_V
  \big(
    1
    +
    \|
      \YMN{s}
    \|_V^{\nicefrac{\varphi}{2}}
  \big) 
  +
  \| F(0) \|_H
\\ &
=
  \sqrt{C}
  \left(
    \|
      \YMN{s}
    \|_V
    +
    \|
      \YMN{s}
    \|_V^{(1+\nicefrac{\varphi}{2})}
  \right) 
  +
  \| F(0) \|_H .
\end{split}
\end{align}
Hence, we obtain
for all $ q \in [1,\infty) $ that
\begin{align}
\label{eq:prop_1}
\begin{split}
  \sup_{ s \in [0,T) }
  \|
    F( \YMN{s} )
  \|_{ \L^q( \P; H ) }
 &\leq
  2 \, \sqrt{C}
  \max\!\left\{ 
    1,
    \sup_{ s \in [0,T) }
    \|
      \YMN{s}
    \|_{ \L^{q(1+\nicefrac{\varphi}{2})}( \P; V ) }^{(1+\nicefrac{\varphi}{2})}
  \right\}
  +
  \|
    F(0)
  \|_H .
\end{split}
\end{align}
Next observe that the triangle inequality,
Lemma~\ref{lem:XX_YO_diff_lp},
Lemma~\ref{lem:XX_XX_diff_lp},
and H{\"o}lder's inequality
imply that
\begin{align}
\label{eq:prop_2}
\begin{split}
 &\sup_{ s \in (0,T) }
  \|
    \XXMN{s} - \YMN{\fl{s}}
  \|_{ \L^{2p}(\P; V) }
\\&\leq
  \sup_{ s \in (0,T) }
  \|
    \XXMN{s} - \XXMN{\fl{s}}
  \|_{ \L^{2p}(\P; V) }
  +
  \sup_{ s \in [0,T) }
  \|
    \XXMN{s} - \YMN{s}
  \|_{ \L^{2p}(\P; V) }
\\&\leq
  \frac{ 
    3 \, T^{ (1 - \rho) }
  }{
    \left( 1 - \rho - \varrho \right)
    M^{\varrho}
  }
  \left[
    \sup_{ s \in (0,T) }
    s^{\rho} \,
    \|
      e^{sA}
    \|_{ L(H,V) }
  \right]
  \left[ 
    \sup_{ s \in [0,T) }
    \|
      F( \YMN{s} ) 
    \|_{ \L^{2p}( \P; H ) }
  \right]
\\&+
  \frac{ T^{(1+\alpha-\rho)} }{ \left(1-\rho\right) M^{\alpha} }
  \left[
    \sup_{ s \in (0,T) }
    s^{\rho} \,
    \|
      e^{sA}
    \|_{ L(H,V) }
  \right] 
\\ &
  \cdot 
  \left[ 
    \sup_{ s \in [0,T) }
    \|
      \Vm(\YMN{s},\OMN{s})
    \|_{ \L^{4p\alpha}( \P; \R) }^{\alpha}
  \right] \!
  \left[ 
    \sup_{ s \in [0,T) }
    \|
      F( \YMN{s} )  
    \|_{ \L^{4p}( \P; H) }
  \right]
\\&+
  \sup_{ s \in (0,T) }
  \|
    P_I ( O_{s} - O_{\fl{s}} ) 
  \|_{ \L^{2p}( \P; V ) }
  +
  \sup_{ s \in [0,T) }
  \|
    P_I O_s
    -
    \OMN{s}
  \|_{ \L^{2p}(\P; V) }
\\&\leq
  \frac{  
    3 \max\{ 1, T^{(1+\alpha-\rho)} \} 
  }{ 
    \left(1-\rho- \varrho\right) 
    M^{\min\{\alpha,\varrho\}} 
  }
  \left[
    \sup_{ s \in (0,T) }
    s^{\rho} \,
    \|
      e^{sA}
    \|_{ L(H,V) }
  \right] \!
  \left[ 
    1
    +
    \sup_{ s \in [0,T) }
    \|
      \Vm(\YMN{s},\OMN{s})
    \|_{ \L^{4p\alpha}( \P; \R) }^{\alpha}
  \right]
\\&\cdot
  \left[ 
    \sup_{ s \in [0,T) }
    \|
      F( \YMN{s} )  
    \|_{ \L^{4p}( \P; H) }
  \right]
  +
  \sup_{ s \in (0,T) }
  \|
    P_I ( O_{s} - O_{\fl{s}} ) 
  \|_{ \L^{2p}( \P; V ) }
\\ &
  +
  \sup_{ s \in [0,T) }
  \|
    P_I O_s
    -
    \OMN{s}
  \|_{ \L^{2p}(\P; V) }  .
\end{split}
\end{align}
Moreover,
note that~\eqref{eq:prop_1}
and the fact that 
$ 
  \| P_I \|_{ L(H) } 
  \leq 1
$ 
yields that
\begin{align}
\begin{split}
 &\sup_{ s \in (0,T) }
  \|
    \XXMN{s}
  \|_{ \L^{p\varphi}( \P; V ) }
\\&\leq
  \sup_{ s \in (0,T) }
  \left\|
    \int_0^s 
    P_I
    e^{(s-u)A}
    F( \YMN{\fl{u}} ) \, du
  \right\|_{ \L^{p\varphi}( \P; V ) }
  +
  \sup_{ s \in (0,T) }
  \|
    P_I O_s
  \|_{ \L^{p\varphi}( \P; V ) }
\\&\leq
  \sup_{ s \in (0,T) }
  \int_0^s
  \|
    e^{(s-u)A}
  \|_{ L(H,V) }
  \| P_I \|_{ L(H) }
  \|
    F( \YMN{\fl{u}} )
  \|_{ \L^{p\varphi}( \P; H ) } \, du
  +
  \sup_{ s \in (0,T) }
  \|
    P_I O_s
  \|_{ \L^{p\varphi}( \P; V ) }
\\&\leq
  \left[ 
    \sup_{ s \in (0,T) }
    s^{\rho} \,
    \|
      e^{sA}
    \|_{ L(H,V) }
  \right] 
  \left[ 
    \sup_{ s \in [0,T) }
    \|
      F( \YMN{s} )
    \|_{ \L^{p\varphi}( \P; H ) }
  \right] 
  \left[ 
    \sup_{ s \in (0,T) }
    \int_0^s
    \left( s-u \right)^{-\rho} du
  \right]
\\&\quad+
  \sup_{ s \in (0,T) }
  \|
    P_I O_s
  \|_{ \L^{p\varphi}( \P; V ) } 
\\&\leq
  \frac{ 
    T^{(1-\rho)}
  }{
    \left( 1-\rho\right)
  }
  \left[ 
    \sup_{ s \in (0,T) }
    s^{\rho} \,
    \|
      e^{sA}
    \|_{ L(H,V) }
  \right] 
  \left[ 
    \sup_{ s \in [0,T) }
    \|
      F( \YMN{s} )
    \|_{ \L^{p\varphi}( \P; H ) }
  \right] 
\\ &
\quad
  +
  \sup_{ s \in (0,T) }
  \|
    P_I O_s
  \|_{ \L^{p\varphi}( \P; V ) }
  <
  \infty.  
\end{split}
\end{align}
Combining this, \eqref{eq:prop_2},
and the fact that
\begin{equation}
  \forall \, x,y \in [0,\infty)
  \colon
  (x+y)^{\nicefrac{\varphi}{2}}
  \leq
  2^{\max\{0,\nicefrac{\varphi}{2}-1\}}
  (
    x^{\nicefrac{\varphi}{2}}
    +
    y^{\nicefrac{\varphi}{2}}
  )
\end{equation}
with Lemma~\ref{lem:lp_X_XX_diff}
proves that
\begin{align}
\label{eq:prop_3}
\begin{split}
 &\sup_{ t \in [0,T] }
  \|
    P_I X_t - \XXMN{t}
  \|_{ \L^p(\P; H) }
\\&\leq
  \frac{
    2^{\max\{0,\nicefrac{\varphi}{2}-1\}}
    \sqrt{ 
      C T 
      e^{\kappa c T}
    }
  }{
    \sqrt{ \left( \kappa - 2 \right) c }
  }
  \Bigg( 
    \sup_{ s \in (0,T) }
    \|
      P_{ \H \setminus I } X_s 
    \|_{ \L^{2p}(\P;V) }
    +
    \sup_{ s \in (0,T) }
    \|
      P_I ( O_{s} - O_{\fl{s}} ) 
    \|_{ \L^{2p}( \P; V ) }
\\&+
    \sup_{ s \in [0,T) }
    \|
      P_I O_s
      -
      \OMN{s}
    \|_{ \L^{2p}(\P; V) }
    +
    \frac{  
      3 \max\{ 1, T^{(1+\alpha-\rho)} \} 
    }{ 
      \left(1-\rho- \varrho\right) 
      M^{\min\{\alpha,\varrho\}} 
    }
    \left[
      \sup_{ s \in (0,T) }
      s^{\rho} \,
      \|
	e^{sA}
      \|_{ L(H,V) }
    \right]
\\&\cdot
    \left[ 
      1
      +
      \sup_{ s \in [0,T) }
      \|
	\Vm(\YMN{s},\OMN{s})
      \|_{ \L^{4p\alpha}( \P; \R) }^{\alpha}
    \right] \!
    \left[ 
      \sup_{ s \in [0,T) }
      \|
	F( \YMN{s} )  
      \|_{ \L^{4p}( \P; H) }
    \right]
  \Bigg)
\\&\cdot
  \Bigg(
    1
    +
    \sup_{ s \in (0,T) }
    \|
      X_s
    \|_{ \L^{p\varphi}(\P;V) }^{\nicefrac{\varphi}{2}}
    +
    \sup_{ s \in (0,T) }
    \|
      P_I X_s
    \|_{ \L^{p\varphi}(\P;V) }^{\nicefrac{\varphi}{2}}
    +
    \sup_{ s \in [0,T) }
    \|
      \YMN{s}
    \|_{ \L^{p\varphi}(\P;V) }^{\nicefrac{\varphi}{2}}
\\&+
    \frac{ 
      T^{(1-\rho)\nicefrac{\varphi}{2}}
    }{
      \left( 1-\rho\right)^{\nicefrac{\varphi}{2}}
    }
    \left[ 
      \sup_{ s \in (0,T) }
      s^{\rho} \,
      \|
	e^{sA}
      \|_{ L(H,V) }
    \right]^{\nicefrac{\varphi}{2}}
    \left[ 
      \sup_{ s \in [0,T) }
      \|
	F( \YMN{s} )
      \|_{ \L^{p\varphi}( \P; H ) }
    \right]^{\nicefrac{\varphi}{2}}
\\ &
    +
    \sup_{ s \in (0,T) }
    \|
      P_I O_s
    \|_{ \L^{p\varphi}( \P; V ) }^{\nicefrac{\varphi}{2}}
  \Bigg) .
\end{split}
\end{align}
Hence, we obtain
that
\begin{align}
\label{eq:prop_4}
\begin{split}
 &\sup_{ t \in [0,T] }
  \|
    P_I X_t - \XXMN{t}
  \|_{ \L^p(\P; H) }
\\&\leq
  \frac{
    3 \cdot
    2^{\max\{0,\nicefrac{\varphi}{2}-1\}}
    \max\{ 1, T^{(\nicefrac{3}{2} + 
                 \alpha - \rho + \nicefrac{\varphi}{2} - \nicefrac{\rho\varphi}{2} )}
        \}
    \sqrt{ 
      C
      e^{\kappa c T}
    }
  }{
    \left( 
      1 - \rho - \varrho 
    \right)^{(1+\nicefrac{\varphi}{2})}
    \sqrt{ \left( \kappa - 2 \right) c }
  }
  \Bigg[
    \sup_{ s \in (0,T) }
    \|
      P_{ \H \setminus I } X_s 
    \|_{ \L^{2p}(\P;V) }
\\&\quad+
    \sup_{ s \in (0,T) }
    \|
      P_I ( O_{s} - O_{\fl{s}} ) 
    \|_{ \L^{2p}( \P; V ) }
    +
    \sup_{ s \in [0,T) }
    \|
      P_I O_s
      -
      \OMN{s}
    \|_{ \L^{2p}(\P; V) }
    +
    M^{-\min\{\alpha,\varrho\}} 
  \Bigg] 
\\&\cdot
  \Bigg[
    2
    +
    \sup_{ s \in (0,T) }
    \|
      X_s
    \|_{ \L^{p\varphi}(\P;V) }^{\nicefrac{\varphi}{2}}
    +
    \sup_{ s \in (0,T) }
    \|
      P_I X_s
    \|_{ \L^{p\varphi}(\P;V) }^{\nicefrac{\varphi}{2}}
    +
    \sup_{ s \in [0,T) }
    \|
      \YMN{s}
    \|_{ \L^{p\varphi}(\P;V) }^{\nicefrac{\varphi}{2}}
\\&\quad+
    \sup_{ s \in (0,T) }
    \|
      P_I O_s
    \|_{ \L^{p\varphi}( \P; V ) }^{\nicefrac{\varphi}{2}}
  \Bigg]
  \left[ 
    1
    +
    \sup_{ s \in [0,T) }
    \|
      \Vm(\YMN{s},\OMN{s})
    \|_{ \L^{4p\alpha}( \P; \R) }^{\alpha}
  \right]
\\&\quad\cdot
  \max\!\left\{
    1,
    \sup_{ s \in (0,T) }
    \left[
      s^{\rho} \,
      \|
	e^{sA}
      \|_{ L(H,V) }
    \right]^{(1+\nicefrac{\varphi}{2})}
  \right\}
  \max\!\left\{ 
    1,
    \sup_{ s \in [0,T) }
    \|
      F( \YMN{s} )  
    \|_{ \L^{p\max\{4,\varphi\}}( \P; H) }^{(1+\nicefrac{\varphi}{2})}
  \right\} .
\end{split}
\end{align}
In the next step observe that the triangle inequality
implies that
\begin{align}
\begin{split}
 &\sup_{ t \in [0,T] }
  \|
    X_t - \YMN{t}
  \|_{ \L^p( \P; H ) }
\\&\leq
  \sup_{ t \in [0,T] }
  \left[ 
    \|
      X_t - P_I X_t
    \|_{ \L^p( \P; H ) }
    +
    \|
      P_I X_t - \XXMN{t}
    \|_{ \L^p( \P; H ) }
    +
    \|
      \XXMN{t} - \YMN{t}
    \|_{ \L^p( \P; H ) }
  \right]
\\&\leq
  \sup_{ t \in [0,T] } 
  \|
    P_I X_t - \XXMN{t}
  \|_{ \L^p( \P; H ) }
\\&\quad+
  \left[ 
    \sup_{ v \in V \setminus \{0\} }
    \frac{ \| v \|_H }{ \| v \|_V }
  \right] 
  \sup_{ t \in [0,T] }
  \left[ 
    \|
      P_{ \H \setminus I } X_t
    \|_{ \L^p( \P; V ) }
    +
    \|
      \XXMN{t} - \YMN{t}
    \|_{ \L^p( \P; V ) }
  \right] 
\\&\leq
  \left[
    \sup_{ t \in [0,T] }
    \|
      P_I X_t - \XXMN{t}
    \|_{ \L^p( \P; H ) }
    +
    \sup_{ t \in [0,T] }
    \left[ 
      \|
        P_{ \H \setminus I } X_t
      \|_{ \L^{2p}( \P; V ) }
      +
      \|
        \XXMN{t} - \YMN{t}
      \|_{ \L^p( \P; V ) }  
    \right]
  \right]
\\&\quad\cdot
  \max\!\left\{ 
    1,
    \sup_{ v \in V \setminus \{0\} }
    \frac{ \| v \|_H }{ \| v \|_V }
  \right\}
.
\end{split}
\end{align}
This, \eqref{eq:prop_4}, and Lemma~\ref{lem:XX_YO_diff_lp}
prove that
\begin{align}
\label{eq:prop_5}
\begin{split}
 &\sup_{ t \in [0,T] }
  \|
    X_t - \YMN{t}
  \|_{ \L^p( \P; H ) }
\\&\leq
  \frac{
    \left[ 
      3 \cdot
      2^{\max\{1,\nicefrac{\varphi}{2}\}}
      +
      1
    \right]
    \max\{ 1, T^{(\nicefrac{3}{2} + 
                 \alpha - \rho + \nicefrac{\varphi}{2} - \nicefrac{\rho\varphi}{2} )}
        \}
    \max\{ 1, \sqrt{ C e^{\kappa c T } } \}
  }{
    \min\{ 1, \sqrt{ \left( \kappa - 2 \right) c } \}
    \left( 1 - \rho - \varrho \right)^{(1+\nicefrac{\varphi}{2})}
  }
  \Bigg[
    \sup_{ t \in [0,T] }
    \|
      P_{ \H \setminus I } X_t 
    \|_{ \L^{2p}(\P;V) }
\\&+
    M^{-\min\{\alpha,\varrho\}}
    +
    \sup_{ t \in (0,T) }
    \|
      P_I ( O_{t} - O_{\fl{t}} ) 
    \|_{ \L^{2p}( \P; V ) }
    +
    \sup_{ t \in [0,T] }
    \|
      P_I O_t
      -
      \OMN{t}
    \|_{ \L^{2p}(\P; V) }
  \Bigg] 
\\&\cdot
  \Bigg[
    1
    +
    \sup_{ s \in (0,T) }
    \|
      X_s
    \|_{ \L^{p\varphi}(\P;V) }^{\nicefrac{\varphi}{2}}
    +
    \sup_{ s \in (0,T) }
    \|
      P_I X_s
    \|_{ \L^{p\varphi}(\P;V) }^{\nicefrac{\varphi}{2}}
\\ &
    +
    \sup_{ s \in [0,T) }
    \|
      \YMN{s}
    \|_{ \L^{p\varphi}(\P;V) }^{\nicefrac{\varphi}{2}}
    +
    \sup_{ s \in (0,T) }
    \|
      P_I O_s
    \|_{ \L^{p\varphi}( \P; V ) }^{\nicefrac{\varphi}{2}}
  \Bigg]
\\&\cdot
  \left[ 
    1
    +
    \sup_{ s \in [0,T) }
    \|
      \Vm(\YMN{s},\OMN{s})
    \|_{ \L^{4p\alpha}( \P; \R) }^{\alpha}
  \right]
  \max\!\left\{
    1,
    \sup_{ s \in (0,T) }
    \left[
      s^{\rho} \,
      \|
	e^{sA}
      \|_{ L(H,V) }
    \right]^{(1+\nicefrac{\varphi}{2})}
  \right\} 
\\&\cdot
  \max\!\left\{
    1,
    \sup_{ s \in [0,T) }
    \|
      F( \YMN{s} )
    \|_{ \L^{p\max\{4,\varphi\}}( \P; H ) }^{(1+\nicefrac{\varphi}{2})}
  \right\} 
  \max\!\left\{
    1,
    \sup_{ v \in V \setminus \{0\} }
    \frac{ \| v \|_H }{ \| v \|_V }
  \right\} .
\end{split}
\end{align}
Next note that~\eqref{eq:prop_1}
and the fact that 
\begin{equation}
  \forall \, x,y \in [0,\infty)
  \colon
  (x+y)^{\nicefrac{(1+\varphi}{2})}
  \leq
  2^{\nicefrac{\varphi}{2}}
  (
    x^{\nicefrac{(1+\varphi}{2})}
    +
    y^{\nicefrac{(1+\varphi}{2})}
  )
\end{equation}
ensure that
\begin{align}
\begin{split}
 &\max\!\left\{ 
    1,
    \sup_{ s \in [0,T) }
    \|
      F( \YMN{s} )
    \|_{ \L^{p\max\{4,\varphi\}}( \P; H ) }^{(1+\nicefrac{\varphi}{2})}
  \right\} 
\\&\leq
  \left[ 
    2
    \max\{ 1, \sqrt{C} \}
    \max\!\left\{ 
      1,
      \sup_{ s \in [0,T) }
      \|
	\YMN{s}
      \|_{ \L^{p(1+\nicefrac{\varphi}{2})\max\{4,\varphi\}}( \P; V ) }^{(1+\nicefrac{\varphi}{2})}
    \right\}
    +
    \| F(0) \|_H
  \right]^{ (1+\nicefrac{\varphi}{2}) }
\\&\leq
  2^{(1+\varphi)}
  \max\{ 1, C^{(\nicefrac{1}{2}+\nicefrac{\varphi}{4})} \}
  \left[ 
    \max\!\left\{ 
      1,
      \sup_{ s \in [0,T) }
      \|
	\YMN{s}
      \|_{ \L^{p(1+\nicefrac{\varphi}{2})\max\{4,\varphi\}}( \P; V ) }^{[(1+\nicefrac{\varphi}{2})^2]}
    \right\} 
    +
    \| F(0) \|_H^{ (1+\nicefrac{\varphi}{2}) } 
  \right] 
\\&<
  \infty .
\end{split}
\end{align}
Combining this with~\eqref{eq:prop_5}
and the fact that
\begin{align}
\begin{split}
  \left[ 
    3 \cdot 
    2^{\max\{ 1, \nicefrac{\varphi}{2} \}}
    +
    1
  \right]
  \cdot 
  2^{(1+\varphi)}
& \leq
  3 \cdot 
  2^{\max\{ 1, \nicefrac{\varphi}{2} \}}
  \left[ 
    1 + \nicefrac{1}{6}
  \right] 
  \cdot 
  2^{(1+\varphi)}
\\ & =
  7
  \cdot 
  2^{ \max\{ 1+\varphi, \nicefrac{3\varphi}{2} \} }
  \leq
  4^{ (2+\varphi) }
\end{split}
\end{align}
completes the proof of Proposition~\ref{prop:X_Y_diff_lp}.
\end{proof}

\section{Main result}
\label{sec:main_result}
\subsection{Setting}
\label{sec:main_result_setting}
Consider the notation in Section~\ref{sec:notation},
let $ ( H , \left< \cdot, \cdot \right>_H, \left\| \cdot \right\|_H ) $
be a separable $ \R $-Hilbert space,
let $ \H \subseteq H $ be a non-empty orthonormal basis of $ H $,
let $ T, c ,\varphi \in (0, \infty) $,
$ \epsilon \in [0,1) $,
$ \rho \in [0,\nicefrac{1}{2}) $, 
$ \gamma \in (\rho, \nicefrac{1}{2}] $,
$ 
  \chi 
  \in 
  (0, \nicefrac{(\gamma-\rho)}{(1+\nicefrac{\varphi}{2})}] 
  \cap 
  (0, \nicefrac{(1-\rho)}{(1+\varphi)}] 
$,
$ \D \subseteq \Pow_0(\H) \setminus \{ \emptyset \} $,
$ \mu \colon \H \rightarrow \R $
satisfy 
$
  \sup_{ h \in \H }
  \mu_h
  <
  0
$,
let $ A \colon D(A) \subseteq H \rightarrow H $ be the linear operator
which satisfies 
that
$ 
  D(A) 
  = 
  \{ 
    v \in H 
    \colon 
    \sum_{ h \in \H } 
    | 
      \mu_h 
      \langle h, v\rangle_H 
    |^2
    <
    \infty
  \}
$
and
$
  \forall \,
  v \in D(A)
  \colon
  Av
  =
  \sum_{ h \in \H }
  \mu_h
  \langle h,v \rangle_H \, h
$, 
let $ ( H_r, \langle \cdot, \cdot \rangle_{ H_r }, \left\| \cdot \right\|\!_{ H_r } ) $,
$ r \in \R $, be a family of interpolation spaces associated to $ -A $
(cf., e.g., \cite[Section~3.7]{sy02}),
let $ ( V, \left\| \cdot \right\|\!_V ) $ be an $ \R $-Banach space
with $ H_{ \rho } \subseteq V \subseteq H $ continuously and densely,
let $ \phi, \Phi \colon \C([0,T], H_{1}) \rightarrow [0,\infty) $
be $ \B(\C([0,T], H_{1})) / \B( [0,\infty) ) $-measurable functions,
let $ F \in \C( V, H ) $,
$ ( P_I )_{ I \in \Pow_0(\H) } \subseteq L(H) $
satisfy for all 
$ I \in \Pow_0(\H) $,
$ u \in H $, $ v,w \in P_I(H) $,
$ x \in \C([0,T], H_1) $
that
\begin{equation}
  P_I(u) 
  = 
  \textstyle\sum_{ h \in I } 
  \langle h,u \rangle_H \, h,
  \qquad
  \langle
    v-w,
    Av
    +
    F(v)
    -
    Aw
    -
    F(w)
  \rangle_H
  \leq
  c
  \|
    v-w
  \|_H^2,
\end{equation}
\begin{align}
\label{eq:main_coercivity}
\begin{split}
 &\sup_{ t \in [0,T] }
  \big(
    \langle
      v, P_I F(v+x_t)
    \rangle_{ H_{ \nicefrac{1}{2} } }
    +
    \phi(x)
    \langle
      v, F(v+x_t)
    \rangle_H
  \big)
\\&\leq
  \epsilon
  \| v \|_{ H_1 }^2
  +
  (c + \phi(x))
  \| v \|_{ H_{ \nicefrac{1}{2} } }^2
  +
  c
  \phi(x)
  \| v \|_H^2
  +
  \Phi(x),
\end{split}
\end{align}
\begin{align}
  \text{and}
  \qquad
  \|
    F(v)
    -
    F(w)
  \|_H^2
  \leq 
  c
  \|
    v-w
  \|_{ V }^2
  (
    1
    +
    \|
      v
    \|_{ V }^{\varphi}
    +
    \|
      w
    \|_{ V }^{\varphi}
  ),
\end{align}
let $ ( \Omega, \F, \P ) $ be a probability space,
let $ X,O \colon [0,T] \times \Omega \rightarrow V $
and $ \OMN{} \colon [0,T] \times \Omega \rightarrow H_1 $,
$ M \in \N $, $ I \in \D $,
be stochastic processes with continuous sample paths,
let $ \YMN{} \colon [0,T] \times \Omega \rightarrow H_{ \gamma } $,
$ M \in \N $, $ I \in \D $,
be stochastic processes,
and assume for all
$ t \in [0,T] $, $ M \in \N $, $ I \in \D $ that
$
  \OMN{}([0,T] \times \Omega)
  \subseteq 
  P_I(H)
$
and
\begin{align}
\begin{split}
 &\P\!\left( 
    X_t
    =
    \smallint_0^t
    e^{(t-s)A}
    F( X_s ) \, ds
    +
    O_t
  \right)
\\&=
  \P\!\left(
    \YMN{t}
    =
    \smallint_0^t
    P_I \,
    e^{(t-s)A} \,
    \one_{ 
      \{
	\| \YMN{\fl{s}} \|_{ H_{\gamma} }
	+
	\| \OMN{\fl{s}} \|_{ H_{\gamma} }
	\leq
	(M/T)^{\chi}
      \}
    }^{\Omega} \,
    F( \YMN{\fl{s}} ) \, ds
    +
    \OMN{t}
  \right)
  =
  1 .
\end{split}
\end{align}

\subsection{Comments on the setting}

In the next two results,
Lemma~\ref{lem:span_dense} and Lemma~\ref{lem:P_I_dense} below,
we establish a few elementary consequences of the framework
in Section~\ref{sec:main_result_setting}.

\begin{lemma}
\label{lem:span_dense}
Assume the setting in Section~\ref{sec:main_result_setting}
and let $ r \in [0,\infty) $.
Then 
$
  \overline{\operatorname{span}(\H)}^{H_r}
  =
  H_r
$. 
\end{lemma}
\begin{proof}[Lemma~\ref{lem:span_dense}]
Throughout this proof let $ u \in H_r $,
let $ \H_n \subseteq \H $, $ n \in \N $, 
be a non-decreasing sequence of finite
subsets of $ \H $ which satisfies $ \cup_{ n \in \N } \, \H_n = \H $,
and let $ (u_n)_{n\in\N} \subseteq \operatorname{span}(\H) $
satisfy for all $ n \in \N $ that 
\begin{equation}
  u_n
  =
  \sum_{ h \in \H_n }
  \langle h, u \rangle_H \, h .
\end{equation}
Note that the fact that 
$ \H \subseteq H $ is orthogonal in $ H_r $
and the fact that
\begin{equation}
  \limsup_{ n \rightarrow \infty }
  \sum_{ h \in \H_n }
  | \langle h, u \rangle_H |^2 \,
  | \mu_h |^{2r}
  =
  \sum_{ h \in \H }
  | \langle h, u \rangle_H |^2 \,
  | \mu_h |^{2r}
  <
  \infty
\end{equation}
show that
\begin{align}
\begin{split}
 &\inf_{ N \in \N }
  \sup_{ \substack{ m, n \in \N, \\ m \geq n \geq N } }
  \| u_m - u_n \|_{ H_r }^2
\\&=
  \inf_{ N \in \N }
  \sup_{ \substack{ m, n \in \N, \\ m \geq n \geq N } }
  \Bigg\|
    \sum_{ h \in \H_m }
    \langle h, u \rangle_H \, h
    -
    \sum_{ h \in \H_n }
    \langle h, u \rangle_H \, h
  \Bigg\|_{H_r}^2   
\\ &
  =
  \inf_{ N \in \N }
  \sup_{ \substack{ m, n \in \N, \\ m \geq n \geq N } }
  \Bigg\|
    \sum_{ h \in \H_m \setminus \H_n }
    \langle h, u \rangle_H \, h
  \Bigg\|_{H_r}^2
\\&=
  \inf_{ N \in \N }
  \sup_{ \substack{ m, n \in \N, \\ m \geq n \geq N } }
  \left[
    \sum_{ h \in \H_m \setminus \H_n }
    \|
      \langle h, u \rangle_H \, h
    \|_{ H_r }^2
  \right]
  =
  \inf_{ N \in \N }
  \sup_{ \substack{ m, n \in \N, \\ m \geq n \geq N } }
  \left[
    \sum_{ h \in \H_m \setminus \H_n }
    | \langle h, u \rangle_H |^2 \,
    \| (-A)^r h \|_H^2
  \right]
\\&=
  \inf_{ N \in \N }
  \sup_{ \substack{ m, n \in \N, \\ m \geq n \geq N } }
  \left[
    \sum_{ h \in \H_m \setminus \H_n }
    | \langle h, u \rangle_H |^2 \,
    | \mu_h |^{2r}
  \right]
  \leq
  \inf_{ N \in \N }
  \sup_{ \substack{ m, n \in \N, \\ m \geq n \geq N } }
  \left[
    \sum_{ h \in \H \setminus \H_n }
    | \langle h, u \rangle_H |^2 \,
    | \mu_h |^{2r}
  \right]
\\&=
  \inf_{ N \in \N }
  \sup_{ \substack{ n \in \N, \\ n \geq N } }
  \left[
    \sum_{ h \in \H \setminus \H_n }
    | \langle h, u \rangle_H |^2 \,
    | \mu_h |^{2r}
  \right]
  =
  \limsup_{ n \rightarrow \infty }
  \left[ 
    \sum_{ h \in \H \setminus \H_n }
    | \langle h, u \rangle_H |^2 \,
    | \mu_h |^{2r}
  \right]
  =
  0 .
\end{split}
\end{align}
Hence, we obtain that
$ (u_n)_{ n \in \N } $
is a Cauchy-sequence in $ H_r $.
This together with the fact that 
$ H_r $ is complete implies that
there exists a vector $ \tilde{u} \in H_r $
such that
\begin{align} 
\label{eq:P_I_dense_x}
  \limsup_{ n \rightarrow \infty }
  \| \tilde{u} - u_n \|_{ H_r }
  =
  0 .
\end{align}
Moreover, observe that
the fact that 
$ \H \subseteq H $ is an orthonormal basis of $ H $
shows that
\begin{align}
\label{eq:P_I_dense_xx}
 &\limsup_{ n \rightarrow \infty }
  \left\|
    u
    -
    u_n
  \right\|_H
  =
  \limsup_{ n \rightarrow \infty }
  \left\|
    u
    -
    \sum_{ h \in \H_n }
    \langle h, u \rangle_H \, h
  \right\|_H
  =
  0 .
\end{align}
Combining~\eqref{eq:P_I_dense_x} with~\eqref{eq:P_I_dense_xx}
and the fact that $ H_r \subseteq H $
continuously
proves that $ u = \tilde{u} $.
This completes the proof of Lemma~\ref{lem:span_dense}.
\end{proof}

\begin{lemma}
\label{lem:P_I_dense}
Assume the setting in Section~\ref{sec:main_result_setting}. 
Then 
\begin{enumerate}[(i)]
 \item\label{it:P_I_dense_1} 
 we have that 
 $
   \operatorname{span}(\H) 
   = 
   \cup_{ I \in \Pow_0(\H) } P_I(H)
 $,
 \vspace{-0.15cm}
 \item\label{it:P_I_dense_2} 
 we have for all $ r \in [0,\infty) $ that
 $
   \overline{\cup_{ I \in \Pow_0(\H) }
   P_I(H)}^{H_r}
   =
   H_r
 $,
 \item\label{it:P_I_dense_3} 
 we have for all $ v,w \in H_1 $ that
 $
  \langle
    v-w,
    Av
    +
    F(v)
    -
    Aw
    -
    F(w)
  \rangle_H
  \leq
  c
  \|
    v-w
  \|_H^2
 $,
 and 
 \item\label{it:P_I_dense_4}
 we have for all $ v,w \in V $ that
 $
  \|
    F(v)
    -
    F(w)
  \|_H^2
  \leq 
  c
  \|
    v-w
  \|_{ V }^2
  (
    1
    +
    \|
      v
    \|_{ V }^{\varphi}
    +
    \|
      w
    \|_{ V }^{\varphi}
  )
 $.
\end{enumerate}
\end{lemma}
\begin{proof}[Proof of Lemma~\ref{lem:P_I_dense}]
Note for all 
$ I,J \in \Pow_0(\H) $, 
$ v \in P_I(H) $, $ w \in P_J(H) $,
$ \alpha, \beta \in \R $
that
\begin{equation}
  \alpha v + \beta w 
  \in 
  P_{ I \cup J }(H) 
  \subseteq 
  (\cup_{ K \in \Pow_0(\H) } P_K(H) ).
\end{equation}
This implies that 
$
 \cup_{ I \in \Pow_0(\H) } P_I(H)
$
is an $ \R $-vector space.
Moreover, observe that for all $ h \in \H $
we have that
\begin{equation}
  h 
  \in 
  P_{ \{h\} }(H) 
  \subseteq 
  (\cup_{ I \in \Pow_0(\H) } P_I(H) ).
\end{equation}
Hence, we obtain that 
$ \H \subseteq ( \cup_{ I \in \Pow_0(\H) } P_I(H) ) $.
This together with the fact that 
$
 \cup_{ I \in \Pow_0(\H) } P_I(H)
$
is an $ \R $-vector space
proves that
\begin{align}
\label{eq:P_I_dense_1}
  \operatorname{span}(\H) 
  \subseteq 
  ( \cup_{ I \in \Pow_0(\H) } P_I(H) ).
\end{align}
In addition, note that the fact that
$ 
  \forall \, I \in \Pow_0(\H) 
$,
$
  v \in P_I(H)
  \colon
  v \in \operatorname{span}(\H)
$
implies that
\begin{align}
  ( \cup_{ I \in \Pow_0(\H) } P_I(H) ) 
  \subseteq 
  \operatorname{span}(\H).
\end{align}
Combining this with \eqref{eq:P_I_dense_1}
establishes~\eqref{it:P_I_dense_1}.
Furthermore, observe that~\eqref{it:P_I_dense_2} is
an immediate consequence of~\eqref{it:P_I_dense_1}
and Lemma~\ref{lem:span_dense}.
In the next step note
that the assumption that
\begin{equation}
  \forall \, I \in \Pow_0(\H)
  ,
  v,w \in P_I(H)
  \colon
  \langle v-w, Av + F(v) - Aw - F(w) \rangle_H
  \leq
  c
  \| v-w \|_H^2
\end{equation}
ensures that for all 
$ (v_k)_{k\in\N} \subseteq (\cup_{ I \in \Pow_0(\H) } P_I(H) ) $, 
$ (w_k)_{k\in\N} \subseteq (\cup_{ I \in \Pow_0(\H) } P_I(H) ) $,
$ n \in \N $
we have that
\begin{align}
\label{eq:P_I_dense_2}
  \langle v_n-w_n, Av_n + F(v_n) - Aw_n - F(w_n) \rangle_H
  \leq
  c
  \| v_n-w_n \|_H^2 . 
\end{align}
Combining this and the fact that
$
  H_1 \ni v \mapsto Av \in H
$
is continuous with the fact that
$
  F_{| H_1 } \in \C( H_1, H ) 
$
proves that for all 
$ v_0,w_0 \in H_1 $,
$ (v_n)_{n\in\N} \subseteq (\cup_{ I \in \Pow_0(\H) } P_I(H) ) $,
$ (w_n)_{n\in\N} \subseteq (\cup_{ I \in \Pow_0(\H) } P_I(H) ) $ 
with 
\begin{equation}
  \limsup_{ n\rightarrow\infty } 
  \| v_n - v_0 \|_{H_1} 
  =
  \limsup_{ n\rightarrow\infty } 
  \| w_n - w_0 \|_{H_1}
  =
  0
\end{equation}
we have that 
\begin{align}
\label{eq:P_I_dense_2b}
\begin{split}
 &\langle v_0-w_0, Av_0 + F(v_0) - Aw_0 - F(w_0) \rangle_H 
\\&=
  \limsup_{ n \rightarrow \infty }
  \langle v_n-w_n, Av_n + F(v_n) - Aw_n - F(w_n) \rangle_H
\\ &
  \leq
  c
  \limsup_{ n \rightarrow \infty }
  \| v_n - w_n \|_H^2
  =
  c
  \| v_0 - w_0 \|_H^2 .
\end{split}
\end{align}
Moreover, observe
that~\eqref{it:P_I_dense_2} ensures that 
for every $ v_0,w_0 \in H_1 $ there
exist sequences
$ (v_n)_{n\in\N} \subseteq (\cup_{ I \in \Pow_0(\H) } P_I(H) ) $ 
and 
$ (w_n)_{n\in\N} \subseteq (\cup_{ I \in \Pow_0(\H) } P_I(H) ) $ 
which satisfy
that 
\begin{align}
  \limsup_{ n\rightarrow\infty } 
  \| v_n - v_0 \|_{ H_1 }
  = 
  \limsup_{ n\rightarrow\infty } 
  \| w_n - w_0 \|_{ H_1 }
  =
  0 .
\end{align}
This and~\eqref{eq:P_I_dense_2b} 
establish~\eqref{it:P_I_dense_3}.
Next note 
that the assumption that
\begin{equation}
  \forall \, I \in \Pow_0(\H)
  ,
  v,w \in P_I(H)
  \colon
  \| F(v) 
  - F(w) \|_H^2
  \leq
  c
  \| v-w \|_V^2
  (
    1
    +
    \| v \|_V^{\varphi}
    +
    \| w \|_V^{\varphi}
  )  
\end{equation}
ensures that for all 
$ (v_k)_{k\in\N} \subseteq (\cup_{ I \in \Pow_0(\H) } P_I(H) ) $, 
$ (w_k)_{k\in\N} \subseteq (\cup_{ I \in \Pow_0(\H) } P_I(H) ) $,
$ n \in \N $
we have that
\begin{align}
\label{eq:P_I_dense_3}
  \| F(v_n) - F(w_n) \|_H^2
  \leq
  c
  \| v_n-w_n \|_V^2
  (
    1
    +
    \| v_n \|_V^{\varphi}
    +
    \| w_n \|_V^{\varphi}
  ) . 
\end{align}
This and the assumption that
$ F \in \C(V,H) $
imply for all 
$ v_0,w_0 \in V $,
$ (v_n)_{n\in\N} \subseteq (\cup_{ I \in \Pow_0(\H) } P_I(H) ) $,
$ (w_n)_{n\in\N} \subseteq (\cup_{ I \in \Pow_0(\H) } P_I(H) ) $ 
with 
\begin{equation}
  \limsup_{ n\rightarrow\infty } 
  \| v_n - v_0 \|_V
  =
  \limsup_{ n\rightarrow\infty } 
  \| w_n - w_0 \|_V
  =
  0
\end{equation}
that
\begin{align}
\label{eq:P_I_dense_3b}
\begin{split}
 &\|
    F(v_0) - F(w_0)
  \|_H^2
\\&=
  \limsup_{ n \rightarrow \infty }
  \|
    F(v_n) - F(w_n)
  \|_H^2
  \leq
  c
  \limsup_{ n \rightarrow \infty }
  \Big[ 
    \| v_n-w_n \|_V^2
    (
      1
      +
      \| v_n \|_V^{\varphi}
      +
      \| w_n \|_V^{\varphi}
    )
  \Big]
\\&=
  c
  \| v_0-w_0 \|_V^2
  (
    1
    +
    \| v_0 \|_V^{\varphi}
    +
    \| w_0 \|_V^{\varphi}
  ) .
\end{split}
\end{align}
In addition, observe
that~\eqref{it:P_I_dense_2}
together with the assumption that
$
  H_{ \rho }
  \subseteq 
  V
$
continuously and densely
guarantees that for every $ v_0,w_0 \in V $ there
exist sequences
$ (v_n)_{n\in\N} \subseteq (\cup_{ I \in \Pow_0(\H) } P_I(H) ) $ 
and 
$ (w_n)_{n\in\N} \subseteq (\cup_{ I \in \Pow_0(\H) } P_I(H) ) $ 
which satisfy
that 
\begin{equation} 
  \limsup_{ n\rightarrow\infty } 
  \| v_n - v_0 \|_V
  =
  \limsup_{ n\rightarrow\infty } 
  \| w_n - w_0 \|_V
  =
  0 .
\end{equation}
Combining this with~\eqref{eq:P_I_dense_3b}
establishes~\eqref{it:P_I_dense_4}.
The proof of Lemma~\ref{lem:P_I_dense} is thus completed.
\end{proof}

\subsection{On the measurability of a certain function}

In our proof of Theorem~\ref{thm:main} (the main result
of this article) we employ the following well-known result.

\begin{lemma}
\label{lem:measurable}
Consider the notation in Section~\ref{sec:notation},
let $ (V, \left\| \cdot \right\|_V ) $ be a
separable $ \R $-Banach space, let
$ (W, \left\| \cdot \right\|_W ) $ be an
$ \R $-Banach space with $ V \subseteq W $
continuously and densely, let $ (S, \mathcal{S}) $ be a
measurable space, let $ s \in S $,
let $ \psi \colon V \rightarrow \mathcal{S} $
be a $ \B(V) / \mathcal{S} $-measurable function,
and let $ \Psi \colon W \rightarrow S $
be the function which satisfies for all 
$ v \in W $ that
\begin{align}
 &\Psi(v)
  =
  \begin{cases}
    \psi(v)  & \colon v \in V \\
    s       & \colon v \in W \setminus V .
  \end{cases}
\end{align}
Then we have that $ \Psi \colon W \rightarrow \mathcal{S} $
is a $ \B(W) / \mathcal{S} $-measurable function.
\end{lemma}
\begin{proof}[Proof of Lemma~\ref{lem:measurable}]
First, observe that Lemma~\ref{lem:separable} ensures that
$ (W, \left\| \cdot \right\|_W ) $ is a separable
$ \R $-Banach space. This and, e.g., Lemma~2.2 in
Andersson et al.~\cite{AnderssonJentzenKurniawan2015}
(with $ V_0 = W $, $ V_1 = V $ in the notation of
Lemma~2.2 in
Andersson et al.~\cite{AnderssonJentzenKurniawan2015})
ensure that
\begin{align}
\label{eq:measurable_1}
 &V \in \B(V) \subseteq \B(W) . 
\end{align}
The assumption that $ \psi \colon V \rightarrow \mathcal{S} $
is a $ \B(V) / \mathcal{S} $-measurable function
hence ensures that for all $ A \in \mathcal{S} $ with $ s \notin A $
we have that
\begin{align}
\label{eq:measurable_2}
\begin{split}
  \Psi^{-1}(A)
 &=
  \{ v \in W \colon \Psi(v) \in A \}
  =
  \{ v \in V \colon \Psi(v) \in A \}
  \cup
  \{ v \in (W \setminus V ) \colon \Psi(v) \in A \}
\\&=
  \{ v \in V \colon \Psi(v) \in A \}
  =
  \{ v \in V \colon \psi(v) \in A \}
  =
  \psi^{-1}(A)
  \in 
  \B(V) \subseteq \B(W) .
\end{split}
\end{align}
Next note that~\eqref{eq:measurable_1} and the
assumption that $ \psi \colon V \rightarrow \mathcal{S} $
is a $ \B(V) / \mathcal{S} $-measurable function
prove that for all $ A \in \mathcal{S} $ with $ s \in A $
we have that
\begin{align}
\label{eq:measurable_3}
\begin{split}
  \Psi^{-1}(A)
 &=
  \{ v \in W \colon \Psi(v) \in A \}
  =
  \{ v \in V \colon \Psi(v) \in A \}
  \cup
  ( W \setminus V )
\\&=
  \{ v \in V \colon \psi(v) \in A \}
  \cup
  (W \setminus V)
  = \!\!\!\!\!
  \underbrace{\psi^{-1}(A)}_{ \in \, \B(V) \, \subseteq \, \B(W) } \!\!\!\!
  \cup \,
  \underbrace{ ( W \setminus V ) }_{ \in \, \B(W) }
  \in \B(W) .
\end{split}
\end{align}
Combining~\eqref{eq:measurable_2}
and~\eqref{eq:measurable_3} demonstrates
that for all $ A \in \mathcal{S} $ we have
that $ \Psi^{-1}(A) \in \B(W) $. This completes
the proof of Lemma~\ref{lem:measurable}.
\end{proof}

\subsection{A priori moment bounds for the numerical approximation}
\label{sec:a_priori_moment}
\begin{lemma}
\label{lem:Y_a_priori_mom}
Assume the setting in Section~\ref{sec:main_result_setting},
let $ p \in [1, \infty) $, $ \sigma \in [0,\gamma] $, and assume that
\begin{align}
  \sup_{ M \in \N }
  \sup_{ I \in \D }
  \sup_{ t \in [0,T] }
  \E\big[
    \|
      \OMN{t}
    \|_{ H_{\sigma} }^{2p}
    +
    |
      \Phi( \OMN{} )
    |^p
    +
    |
      \phi( \OMN{} )
    |^p
  \big]
  <
  \infty .
\end{align}
Then
\begin{equation}
  \sup_{ M \in \N }
  \sup_{ I \in \D }
  \sup_{ t \in [0,T] }
  \E\big[
    \|
      \YMN{t}
    \|_{H_{ \sigma }}^{2p}
  \big]
  <
  \infty .
\end{equation}
\end{lemma}
\begin{proof}[Proof of Lemma~\ref{lem:Y_a_priori_mom}]
Throughout this proof let 
$ \kappa \in (0,1) $ be a real number, 
let
$ \YMNp{} \colon [0,T] \times \Omega \rightarrow P_I(H) $,
$ M \in \N $, $ I \in \D $,
be the functions which satisfy
for all  
$ M \in \N $, $ I \in \D $,
$ t \in [0,T] $
that
\begin{align}
\label{eq:Y_a_priori_mom_def}
 &\YMNp{t}
  =
  \int_0^t
  P_I \,
  e^{(t-s)A} \,
  \one_{ 
    \{
      \| \YMNp{\fl{s}} \|_{ H_{\gamma} }
      +
      \| \OMN{\fl{s}} \|_{ H_{\gamma} }
      \leq
      (M/T)^{\chi}
    \}
  }^{\Omega} \,
  F( \YMNp{\fl{s}} ) \, ds
  +
  \OMN{t} ,
\end{align}
and let $ C,K \in [0, \infty] $  
be the extended real numbers given by
\begin{align}
  K
 &=
  \sqrt{c}
  \max\!\left\{
    1,
    \sup_{ v \in (V \cap H_{\rho}) \setminus \{0\} }
    \frac{ 
      \| v \|_V^{(1+\nicefrac{\varphi}{2})} 
    }{ 
      \| v \|_{H_{\rho}}^{(1+\nicefrac{\varphi}{2})} 
    }
  \right\} 
\end{align}
and 
\begin{align}
\begin{split}
  C
  =
  \max\!\Bigg\{
   &3 K^2
    ( 1+ 2^{\max\{0,\varphi-1\}} )
    \left[ 
      1
      +
      \sup_{ v \in H_{\gamma}\setminus\{0\} }
      \frac{ \| v \|_{H_{\rho}}^{\varphi} }{ \| v \|_{H_{\gamma}}^{\varphi} }
    \right] ,
\\
   &( 8 K^2 + 2 \| F(0) \|_H^2 )
    \max\!\left\{ 
      1,
      \sup_{ v \in H_{\gamma}\setminus\{0\} }
      \frac{ \| v \|_{H_{\rho}}^{(2+\varphi)} }{ \| v \|_{H_{\gamma}}^{(2+\varphi)} }
    \right\}
  \Bigg\} .
\end{split}
\end{align}
Note that the fact that 
$ 
  H_{\gamma}
  \subseteq
  H_{\rho} 
  \subseteq 
  V
$ 
continuously
ensures that $ C,K \in [0,\infty) $.
In the next step observe that
Lemma~\ref{lem:P_I_dense}~\eqref{it:P_I_dense_4}
and the fact that
\begin{equation}
  \forall \, x,y,z \in [0,\infty)
  \colon 
  \sqrt{x+y+z}
  \leq 
  \sqrt{x} + \sqrt{y} + \sqrt{z}
\end{equation}
imply for all $ v,w \in H_{\gamma} $
that
\begin{align} 
\begin{split}
 &\|
    F(v)
    -
    F(w)
  \|_H
  \leq 
  \sqrt{ 
    c
    \|
      v-w
    \|_V^2
    \left(
      1
      +
      \left\| v \right\|_V^{\varphi}
      +
      \left\| w \right\|_V^{\varphi}
    \right)
  }
\\&\leq
  \sqrt{c}
  \left[ 
    \sup_{ u \in H_{\rho} \setminus \{0\} }
    \frac{ 
      \| u \|_V 
    }{ 
      \| u \|_{H_{\rho}} 
    }
  \right]
  \|
    v-w
  \|_{ H_{\rho} } \!
  \left( 
    1
    +
    \left[ 
      \sup_{ u \in H_{\rho} \setminus \{0\} }
      \frac{ 
        \| u \|_V^{\varphi} 
      }{ 
        \| u \|_{H_{\rho}}^{\varphi} 
      }
    \right]
    \Big[ 
      \|
      v
      \|_{ H_{\rho} }^{\varphi}
      +
      \|
	w
      \|_{ H_{\rho} }^{\varphi}
    \Big]
  \right)^{\!\!\nicefrac{1}{2}}
\\&\leq
  K
  \|
    v-w
  \|_{ H_{\rho} }
  \left(
    1
    +
    \|
      v
    \|_{ H_{\rho} }^{\nicefrac{\varphi}{2}}
    +
    \|
      w
    \|_{ H_{\rho} }^{\nicefrac{\varphi}{2}}
  \right) .
\end{split}
\end{align}
Combining this with Lemma~2.4 in Hutzenthaler et
al.~\cite{HutzenthalerJentzenSalimova2016}
(with $ (V, \left\|\cdot\right\|\!_{V}) = (H_{\gamma}, \left\|\cdot\right\|\!_{H_{\gamma}}) $, 
$ (\mathcal{V}, \left\|\cdot\right\|\!_{\mathcal{V}}) = (H_{\rho}, \left\|\cdot\right\|\!_{H_{\rho}}) $, 
$ (W, \left\|\cdot\right\|\!_{W}) = (H, \left\|\cdot\right\|\!_{H}) $,
$ (\mathcal{W}, \left\|\cdot\right\|\!_{\mathcal{W}}) = (H, \left\|\cdot\right\|\!_{H}) $, 
$ \epsilon = K $, $ \theta = C $,
$ \varepsilon = \nicefrac{\varphi}{2} $, $ \vartheta = \varphi $,
$ F = H_{\gamma} \ni v \mapsto F(v) \in H $
in the notation of Lemma~2.4 in Hutzenthaler et
al.~\cite{HutzenthalerJentzenSalimova2016})
implies that for all $ u,v \in H_{\gamma} $ 
we have that
\begin{align}
\label{eq:Y_a_priori_mom_growth}
\begin{split}
 &\| F(u) \|_H^2
  \leq 
  C
  \max\{ 1, \| u \|_{H_{\gamma}}^{(2+\varphi)} \}
\end{split}
\end{align}
and
\begin{align}
\label{eq:Y_a_priori_mom_lip}
\begin{split}
 &\|
    F(u) - F(v)
  \|_H^2
  \leq
  C
  \max\{ 1, \| u \|_{ H_{\gamma} }^{\varphi} \}
  \| u-v \|_{ H_{\rho} }^2
  +
  C
  \| u-v \|_{ H_{\rho} }^{(2+\varphi)} .
\end{split}
\end{align}
Moreover, observe that
the assumption that
for all
$
  I \in \Pow_0(\H)
$,
$
  v \in P_I(H)
$,
$
  w \in \C([0,T], H_1)
$
we have that
\begin{align}
\begin{split}
 &\sup_{ t \in [0,T] }
  \big(
    \langle
      v, P_I F(v+w_t)
    \rangle_{ H_{ \nicefrac{1}{2} } }
    +
    \phi(w)
    \langle
      v, F(v+w_t)
    \rangle_H
  \big)
\\&\leq
    \epsilon
    \| v \|_{ H_1 }^2
    +
    (c + \phi(w))
    \| v \|_{ H_{ \nicefrac{1}{2} } }^2
    +
    c
    \phi(w)
    \| v \|_H^2
    +
    \Phi(w)
\end{split}
\end{align}
ensures that for all 
$ I \in \D $,
$ v \in P_I(H) $,
$ w \in \C([0,T], P_I(H)) $,
$ t \in [0,T] $
we have that
\begin{align}
\begin{split}
 &\langle
    v, P_I F(v+w_t)
  \rangle_{ H_{ \nicefrac{1}{2} } }
  +
  \phi([0,T] \ni s \mapsto w_s \in H_1)
  \langle
    v, F(v+w_t)
  \rangle_H
\\&\leq
  \epsilon
  \| v \|_{ H_1 }^2
  +
  (\tfrac{c}{\kappa} + \phi([0,T] \ni s \mapsto w_s \in H_1))
  \| v \|_{ H_{ \nicefrac{1}{2} } }^2
  +
  c
  \phi([0,T] \ni s \mapsto w_s \in H_1)
  \| v \|_H^2
\\&\quad+
  \Phi([0,T] \ni s \mapsto w_s \in H_1) .
\end{split}
\end{align}
This together with
\eqref{eq:Y_a_priori_mom_def},
\eqref{eq:Y_a_priori_mom_growth},
and~\eqref{eq:Y_a_priori_mom_lip}
allows us to apply Lemma~\ref{lem:Y_a_priori}
(with $ H = H $,
$ \H = \H $, $ T = T $,
$ \varphi = \varphi $,
$ c = \nicefrac{c}{\kappa} $, $ C = C $,
$ \epsilon = \epsilon $,
$ \kappa = \kappa $,
$ \rho = \rho $,
$ \gamma = \gamma $,
$ \chi = \chi $,
$ M = M $, $ \mu = \mu $,
$ A = A $,
$ I = I $, $ P = P_I $,
$ \Y{} = [0,T] \ni t \mapsto \YMNp{t}(\omega) \in P_I(H) $,
$ \O{} = [0,T] \ni t \mapsto \OMN{t}(\omega) \in P_I(H) $,
$ F = P_I(H) \ni v \mapsto F(v) \in H $,
$ \phi = \C([0,T], P_I(H)) \ni v \mapsto \phi( [0,T] \ni t \mapsto v(t) \in H_1) \in [0,\infty) $, 
$ \Phi = \C([0,T], P_I(H)) \ni v \mapsto \Phi( [0,T] \ni t \mapsto v(t) \in H_1) \in [0,\infty) $
for $ I \in \D $, $ M \in \N $,
$ \omega \in \Omega $
in the notation of Lemma~\ref{lem:Y_a_priori})
to obtain that for every $ M \in \N $,
$ I \in \D $, $ \omega \in \Omega $
we have that the function
$
  [0,T] \ni t \mapsto 
  \YMNp{t}(\omega)
  -
  \OMN{t}(\omega)
  \in 
  P_I(H)
$
is continuous and that
\begin{align}
\label{eq:Y_a_priori_mom_2}
\begin{split}
 &\sup_{ t \in [0,T] }
  \big(
    \|
      \YMNp{t}(\omega)
      -
      \OMN{t}(\omega)
    \|_{ H_{\nicefrac{1}{2}} }^2
    +
    \|
      \YMNp{t}(\omega)
      -
      \OMN{t}(\omega)
    \|_{ H }^2
  \big)
\\&\leq
  \frac{ \kappa }{ c }
  \exp\!\left(\frac{2cT}{\kappa}\right)  
  \left(
    \Phi(\OMN{}(\omega))
    +  
    \frac{\max\{1, \phi(\OMN{}(\omega))\}  
          C(c+\kappa)}{2(1-\epsilon)(1-\kappa)c}
    \left[ 
      \tfrac{\max\{1,T\} (1+\sqrt{C})}{(1-\rho)}
    \right]^{\!(2+\varphi)}
  \right) .
\end{split}
\end{align}
This, in particular, implies for all
$ M \in \N $,
$ I \in \D $
that 
\begin{align}
  \left(
    \Omega \ni \omega 
    \mapsto
    \sup\nolimits_{ t \in [0,T] }
    \| \YMNp{t}(\omega) - \OMN{t}(\omega) \|_{ H_{\nicefrac{1}{2}} }
    \in
    \R
  \right)
  \in 
  \M( \F, \B(\R) ) .
\end{align}
The assumption that $ p \geq 1 $, 
the fact that 
\begin{equation}
  \forall \,
  x,y \in [0,\infty)
  \colon 
  \sqrt{ x + y }
  \leq
  \sqrt{x} + \sqrt{y}
  ,
\end{equation}
and~\eqref{eq:Y_a_priori_mom_2}
hence ensure for 
all $ M \in \N $, $ I \in \D $ that
\begin{align}
\label{eq:Y_a_priori_mom_7}
\begin{split}
 &\big\| 
    \sup\nolimits_{ t \in [0,T] }
    \|
      \YMNp{t}
      -
      \OMN{t}
    \|_{ H_{\nicefrac{1}{2}} }
  \big\|_{ \L^{2p}( \P; \R) }
\\&\leq
  \left\|
    \frac{ \kappa }{ c }
    \exp\!\left(\frac{2cT}{\kappa}\right) 
    \left(
      \Phi(\OMN{})
      +    
      \frac{\max\{1, \phi(\OMN{})\}
            C(c+\kappa)}{2(1-\epsilon)(1-\kappa)c}
      \left[ 
	\tfrac{\max\{1,T\} (1+\sqrt{C})}{(1-\rho)}
      \right]^{\!(2+\varphi)}
    \right)
  \right\|_{ \L^{p}(\P; \R) }^{\nicefrac{1}{2}}
\\&\leq
  \frac{ 
    \sqrt{\kappa} 
  }{ \sqrt{c} }
  \exp\!\left(\frac{cT}{\kappa}\right) 
\\ &
\quad
\cdot 
  \left(
    \|
      \Phi(\OMN{})
    \|_{ \L^{p}(\P; \R) }
    +
    \frac{\left(
      \|
	\phi(\OMN{})
      \|_{ \L^{p}(\P; \R) }
      +
      1
    \right) C(c+\kappa)}{2(1-\epsilon)(1-\kappa)c}
    \left[ 
      \tfrac{\max\{1,T\} (1+\sqrt{C})}{(1-\rho)}
    \right]^{\!(2+\varphi)}
  \right)^{\!\nicefrac{1}{2}}
\\&\leq
  \frac{ 
    \sqrt{\kappa}   
  }{ \sqrt{c} }
  \exp\!\left(\frac{cT}{\kappa}\right)
  \|
    \Phi(\OMN{})
  \|_{ \L^{p}(\P; \R) }^{\nicefrac{1}{2}}
\\&\quad+
  \left(
    \|
      \phi(\OMN{})
    \|_{ \L^{p}(\P; \R) }^{\nicefrac{1}{2}}
    +
    1
  \right)
  \frac{
    \sqrt{C(c+\kappa) }
    \exp\!\left(\frac{cT}{\kappa}\right)
  }{c\sqrt{2(1-\epsilon)(\nicefrac{1}{\kappa} - 1)}}
  \left[ 
    \tfrac{\max\{1,T\} (1+\sqrt{C})}{(1-\rho)}
  \right]^{\!(1+\nicefrac{\varphi}{2})} .  
\end{split}
\end{align}
In addition, observe
that the fact that for every 
$ r \in \R $, $ I \in \D $ we have that
$ P_I(H) \subseteq H_{r} $ continuously 
and, e.g., Andersson et al.~\cite[Lemma~2.2]{AnderssonJentzenKurniawan2015}
(with $ V_0 = H_{r} $, 
$ \left\|\cdot\right\|\!_{V_0} = \left\|\cdot\right\|\!_{H_{r}} $,
$ V_1 = P_I(H) $,
$ \left\|\cdot\right\|\!_{V_1} = P_I(H) \ni v \mapsto \|v\|_H \in [0,\infty) $ 
for $ I \in \D $, $ r \in \R $
in the notation of Andersson et al.~\cite[Lemma~2.2]{AnderssonJentzenKurniawan2015})
prove that for all $ I \in \D $ we have that
\begin{align}
\label{eq:Y_a_priori_mom_ii}
  \B( P_I(H) ) 
  =
  \big\{ 
    S \in \Pow(P_I(H))
    \colon
    \big(
      \exists \,
      B \in \B( H_{r} )
      \colon
      S = B \cap P_I(H)
    \big)
  \big\}
  \subseteq \B( H_{r} ) . 
\end{align}
The hypothesis that
$ \OMN{} \colon [0,T] \times \Omega \rightarrow H_1 $,
$ M \in \N $, $ I \in \D $, are stochastic processes therefore
demonstrates that for every $ M \in \N $, $ I \in \D $
we have that $ \YMNp \colon [0,T] \times \Omega \rightarrow P_I(H) $
is a stochastic process.
Combining this with~\eqref{eq:Y_a_priori_mom_ii}
shows that for all $ B \in \B( H_{\gamma} ) $,
$ t \in [0,T] $, $ M \in \N $, $ I \in \D $ we have that
\begin{align}
\begin{split}
 &\left( 
    \Omega \ni \omega \mapsto \YMNp{t}(\omega) \in H_{\gamma} 
  \right)^{-1}
  \!\!(B)
  =
  \left( 
    \Omega \ni \omega \mapsto \YMNp{t}(\omega) \in H_{\gamma}   
  \right)^{-1}
  \!\!\big(
    B \cap P_I(H) 
  \big)
\\&=
  \left( 
    \Omega \ni \omega \mapsto \YMNp{t}(\omega) \in P_I(H)  
  \right)^{-1}
  \!\!\big(
    B \cap P_I(H) 
  \big)
  =
  \left( 
    \YMNp{t}  
  \right)^{-1}
  \!\!\big(
    \underbrace{B \cap P_I(H)}_{ \in \B( P_I(H) ) } 
  \big)
  \in 
  \F .
\end{split}
\end{align}
Hence, we obtain for all $ t \in [0,T] $, $ M \in \N $, $ I \in \D $
that
\begin{align}
 \big\{ 
   \omega \in \Omega
   \colon 
   \YMN{t}(\omega)
   =
   \YMNp{t}(\omega)
 \big\}
 \in \F .
\end{align}
The assumption that for 
every $ t \in [0,T] $,
$ M \in \N $, $ I \in \D $
we have that
\begin{align}
  \P\!\left(
    \YMN{t}
    =
    \int_0^t
    P_I \,
    e^{(t-s)A} \,
    \one_{ 
      \{
	\| \YMN{\fl{s}} \|_{ H_{\gamma} }
	+
	\| \OMN{\fl{s}} \|_{ H_{\gamma} }
	\leq
	(M/T)^{\chi}
      \}
    }^{\Omega} \,
    F( \YMN{\fl{s}} ) \, ds
    +
    \OMN{t}
  \right)
  =
  1
\end{align}
therefore implies that for all 
$ t \in [0,T] $,
$ M \in \N $,
$ I \in \D $
we have that
\begin{equation}
  \P\big( 
    \YMN{t} = \YMNp{t} 
  \big)
  =
  1 .
\end{equation}
This and the triangle inequality
assure that
\begin{align}
\begin{split}
 &\sup_{ M \in \N }
  \sup_{ I \in \D }
  \sup_{ t \in [0,T] }
  \|
    \YMN{t}
  \|_{ \L^{2p}( \P; H_{ \sigma } ) }
  =
  \sup_{ M \in \N }
  \sup_{ I \in \D }
  \sup_{ t \in [0,T] }
  \|
    \YMNp{t}
  \|_{ \L^{2p}( \P; H_{ \sigma } ) }
\\&\leq
  \sup_{ M \in \N }
  \sup_{ I \in \D }
  \sup_{ t \in [0,T] }
  \left[ 
    \|
      \YMNp{t} - \OMN{t}
    \|_{ \L^{2p}( \P; H_{ \sigma } ) }
    +
    \|
      \OMN{t}
    \|_{ \L^{2p}( \P;  H_{ \sigma } ) }
  \right]
\\&\leq
  \left[
    \sup_{ v \in H_{\nicefrac{1}{2}} \setminus \{ 0 \} }
    \frac{ 
      \| v \|_{  H_{ \sigma } } 
    }{
      \| v \|_{ H_{\nicefrac{1}{2}} }
    }
  \right] 
  \left[
    \sup_{ M \in \N }
    \sup_{ I \in \D }
    \big\|
      \sup\nolimits_{ t \in [0,T] }
      \|
        \YMNp{t} - \OMN{t}
      \|_{ H_{ \nicefrac{1}{2} } }
    \big\|_{ \L^{2p}( \P ; \R) }
  \right] 
\\&\quad+
  \sup_{ M \in \N }
  \sup_{ I \in \D }
  \sup_{ t \in [0,T] }
  \|
    \OMN{t}
  \|_{ \L^{2p}( \P; H_{\sigma} ) } .
\end{split}
\end{align}
The assumption that
\begin{equation}
  \sup_{ M \in \N }
  \sup_{ I \in \D }
  \sup_{ t \in [0,T] }
  \E\big[
    \|
      \OMN{t}
    \|_{ H_{\sigma} }^{2p}
    +
    | \Phi( \OMN{} ) |^p
    +
    | \phi( \OMN{} ) |^p
  \big]
  <
  \infty
  ,
\end{equation}
the fact that
$ H_{\nicefrac{1}{2}} \subseteq H_{ \sigma } $ continuously,
and \eqref{eq:Y_a_priori_mom_7} hence prove that
\begin{align}
\begin{split}
 &\sup_{ M \in \N }
  \sup_{ I \in \D }
  \sup_{ t \in [0,T] }
  \|
    \YMN{t}
  \|_{ \L^{2p}( \P; H_{ \sigma } ) }
\\&\leq
  \left[
    \sup_{ v \in H_{\nicefrac{1}{2}} \setminus \{ 0 \} }
    \frac{ 
      \| v \|_{  H_{ \sigma } } 
    }{
      \| v \|_{ H_{\nicefrac{1}{2}} }
    }
  \right]
  \Bigg[
    \frac{ 
      \sqrt{\kappa}  
    }{ \sqrt{c} }
    \exp\!\left(\frac{cT}{\kappa}\right)
    \sup_{ M \in \N }
    \sup_{ I \in \D }
    \|
      \Phi( \OMN{} )
    \|_{ \L^{p}(\P; \R) }^{\nicefrac{1}{2}}
\\&\quad+
    \left(
      \sup_{ M \in \N }
      \sup_{ I \in \D }
      \|
	\phi(\OMN{})
      \|_{ \L^{p}(\P; \R) }^{\nicefrac{1}{2}}
      +
      1
    \right)
    \frac{
      \sqrt{C(c+\kappa)}
      \exp\!\left(\frac{cT}{\kappa}\right)
    }{c\sqrt{2(1-\epsilon)(\nicefrac{1}{\kappa}-1)}}
    \left[ 
      \tfrac{\max\{1,T\} (1+\sqrt{C})}{(1-\rho)}
    \right]^{\!(1+\nicefrac{\varphi}{2})}
  \Bigg]
\\&\quad+
  \sup_{ M \in \N }
  \sup_{ I \in \D }
  \sup_{ t \in [0,T] }
  \|
    \OMN{t}
  \|_{ \L^{2p}( \P; H_{\sigma} ) }
  <
  \infty .
\end{split}
\end{align}
This completes the proof
of Lemma~\ref{lem:Y_a_priori_mom}.
\end{proof}

\subsection{Main result}

\begin{theorem}
\label{thm:main}
Assume the setting in Section~\ref{sec:main_result_setting},
let $ \vartheta \in (0,\infty) $,
$ p \in [ \max\{ 2, \nicefrac{1}{\varphi} \}, \infty ) $,
$ \varrho \in [0,1-\rho) $,
and assume that
\begin{align}
\label{eq:main_assumption}
\begin{split}
 &\sup_{ t \in [0,T] }
  \sup_{ I \in \D }
  \E\big[
    \|
      X_t
    \|_{ V }^{p(1+\nicefrac{\varphi}{2})\max\{2,\varphi\}}
    +
    \| P_I O_t \|_{ V }^{p\varphi}
  \big]
\\&+ 
  \sup_{ M \in \N }
  \sup_{ I \in \D }
  \sup_{ t \in [0,T] }
  \E\!\left[
    \big|
      \|
        \OMN{t}
      \|_{ H_{\gamma} }^2
      +
      \Phi( \OMN{} )
      +
      \phi( \OMN{} )
    \big|^{p\max\{2\vartheta, 2+\varphi, (1+\nicefrac{\varphi}{2})\nicefrac{\varphi}{2} \}}
  \right]
  <
  \infty .
\end{split}
\end{align}
Then we have
\begin{enumerate}[(i)]
 \item\label{it:thm_main_1} that
 $
  \sup_{ M \in \N }
  \sup_{ I \in \D }
  \sup_{ t \in [0,T] }
  \E\big[
    \|
      \YMN{t}
    \|_{ H_{\gamma} }^{p\max\{4\vartheta, 4+2\varphi, \varphi(1+\nicefrac{\varphi}{2}) \}}
  \big]
  <
  \infty
$
and
\item\label{it:thm_main_2}
that there exists
a real number $ C \in (0,\infty) $ such that
for all $ M \in \N $, $ I \in \D $
it holds that
\begin{align}
\label{eq:main_assert_2}
\begin{split}
 &\sup_{ t \in [0,T] }
  \left(
    \E\big[
      \|
        X_t
        -
        \YMN{t}
      \|_{ H }^p
    \big] 
  \right)^{\!\nicefrac{1}{p}}
  \leq
  C
  \bigg[
    M^{-\min\{\vartheta\chi,\varrho\}}
    +
    \|
      (-A)^{-\varrho}
      ( \Id_H - P_I )
    \|_{ L(H) }
\\&\quad+
    \sup_{ t \in [0,T] }
    \left(
      \E\big[
        \|
          ( \Id_H - P_I )
          O_t
        \|_{ V }^{2p}
        +
        \|
          P_I O_{t} 
          -
          \OMN{t}
        \|_{ V }^{2p}
        +
        \|
          P_I( 
            O_{t}
            - 
            O_{\fl{t}} 
          )
        \|_{ V }^{2p}
      \big]
    \right)^{\!\nicefrac{1}{2p}}
  \bigg] .
\end{split}
\end{align}
\end{enumerate}

\end{theorem}
\begin{proof}[Proof of Theorem~\ref{thm:main}]
Throughout this proof let $ \kappa \in (2,\infty) $ be a real number, 
let $ (\mathfrak{P}_I )_{ I \in \Pow(\H) } \subseteq L(H) $
be the linear operators which satisfy for all $ I \in \Pow(\H) $, $ v \in H $
that
\begin{equation}
  \mathfrak{P}_I(v) 
  =
  \sum_{ h \in I }
  \langle h, v \rangle_H h ,
\end{equation}
let $ \Vm \colon V \times V \rightarrow [0,\infty) $ be the function
which satisfies for all $ v,w \in V $ that
\begin{align}
\label{eq:main_thm_V_1}
\begin{split}
 &\Vm(v,w)
  =
  \begin{cases}
    ( \| v \|_{ H_{\gamma} } + \| w \|_{ H_{\gamma} } )^{\nicefrac{1}{\chi}}
    & \colon (v,w) \in H_{\gamma} \times H_{\gamma} \\
    0 & \colon (v,w) \in (V \times V) \setminus (H_{\gamma} \times H_{\gamma}) ,
  \end{cases}
\end{split}
\end{align}
let $ \YMNp{} \colon [0,T] \times \Omega \rightarrow H_{\gamma} $,
$ M \in \N $, $ I \in \D $,
be the functions which satisfy
for all $ M \in \N $, $ I \in \D $, $ t \in [0,T] $ that
\begin{align}
\label{eq:main_proof_0}
 &\YMNp{t}
  =
  \int_0^t
  P_I \,
  e^{(t-s)A} \,
  \one_{ 
    \{
      \| \YMNp{\fl{s}} \|_{ H_{\gamma} }
      +
      \| \OMN{\fl{s}} \|_{ H_{\gamma} }
      \leq
      (M/T)^{\chi}
    \}
  }^{\Omega} \,
  F( \YMNp{\fl{s}} ) \, ds
  +
  \OMN{t},
\end{align}
let $ \tilde{\Omega} \subseteq \Omega $
be the set given by
\begin{align}
\label{eq:main_thm_Omega_1}
\begin{split}
  \tilde{\Omega}
 &=
  \left\{
    \forall \, 
    t \in [0,T]
    \colon
    X_t
    =
    \smallint\nolimits_0^t
    e^{(t-s)A}
    F(X_s) \, ds
    +
    O_t
  \right\}
\\&=
  \left\{
    \forall \, 
    t \in [0,T] \cap \mathbb{Q}
    \colon
    X_t
    =
    \smallint\nolimits_0^t
    e^{(t-s)A}
    F(X_s) \, ds
    +
    O_t
  \right\},
\end{split}
\end{align}
and let
$
  \tilde{X}
  \colon 
  [0,T] \times \Omega
  \rightarrow
  V
$
and
$
  \tilde{O}
  \colon 
  [0,T] \times \Omega
  \rightarrow
  V
$
be the functions which satisfy
for all $ t \in [0,T] $ that
$ 
  \tilde{X}_t 
  = 
  X_t
  \one_{\tilde{\Omega}}^{\Omega}
$
and
\begin{align}
  \tilde{O}_t
  = 
  O_t
  \one_{\tilde{\Omega}}^{\Omega}
  -
  \left[ 
    \int_0^t
    e^{(t-s)A}
    F(0) \, ds
  \right] 
  \one_{\Omega\backslash\tilde{\Omega}}^{\Omega}
  =
  O_t
  \one_{\tilde{\Omega}}^{\Omega}
  +
  A^{-1}
  ( \Id_H - e^{tA} ) F(0)
  \one_{\Omega\backslash\tilde{\Omega}}^{\Omega} .
\end{align}
Observe that for all 
$ I \in \Pow_0(\H) \subseteq \Pow(\H) $
we have that
\begin{align}
 &\mathfrak{P}_I
  =
  P_I 
  \qquad 
  \text{and}
  \qquad
  \mathfrak{P}_{\H\setminus I}
  =
  \Id_H
  -
  P_I .
\end{align}
Next note that the fact that
$ H_1 \subseteq H_{\gamma} $
continuously and, e.g.,
Andersson et al.~\cite[Lemma~2.2]{AnderssonJentzenKurniawan2015}
(with $ V_0 = H_{\gamma} $,
$ \left\|\cdot\right\|\!_{ V_0 } = \left\|\cdot\right\|\!_{ H_{\gamma} } $,
$ V_1 = H_1 $,
$ \left\|\cdot\right\|\!_{ V_1 } = \left\|\cdot\right\|\!_{ H_{1} } $
in the notation of 
Andersson et al.~\cite[Lemma~2.2]{AnderssonJentzenKurniawan2015})
ensure that
\begin{align} 
 &\B(H_1)
  =
  \left\{
    S \in \Pow(H_1) 
    \colon 
    \big( 
      \exists \,
      B \in \B( H_{\gamma} )
      \colon 
      S
      =
      B \cap H_1
    \big)
  \right\}
  \subseteq 
  \B( H_{\gamma} ) .
\end{align}
The hypothesis that
$ \OMN{} \colon [0,T] \times \Omega \rightarrow H_1 $,
$ M \in \N $, $ I \in \D $,
are stochastic processes therefore demonstrates that
for every $ M \in \N $, $ I \in \D $
we have that
$ \YMNp{} \colon [0,T] \times \Omega \rightarrow H_{\gamma} $
is a stochastic process. Hence,
we obtain for all $ t \in [0,T] $, $ M \in \N $, $ I \in \D $
that
\begin{align}
 \big\{ 
   \omega \in \Omega
   \colon 
   \YMN{t}(\omega)
   =
   \YMNp{t}(\omega)
 \big\}
 \in \F .
\end{align}
The assumption that for 
every $ t \in [0,T] $,
$ M \in \N $, $ I \in \D $
we have that
\begin{align}
  \P\!\left(
    \YMN{t}
    =
    \int_0^t
    P_I \,
    e^{(t-s)A} \,
    \one_{ 
      \{
	\| \YMN{\fl{s}} \|_{ H_{\gamma} }
	+
	\| \OMN{\fl{s}} \|_{ H_{\gamma} }
	\leq
	(M/T)^{\chi}
      \}
    }^{\Omega} \,
    F( \YMN{\fl{s}} ) \, ds
    +
    \OMN{t}
  \right)
  =
  1
\end{align}
therefore implies that for all 
$ t \in [0,T] $,
$ M \in \N $,
$ I \in \D $
we have that
\begin{equation}
  \P\big( 
    \YMN{t} = \YMNp{t} 
  \big)
  =
  1 .
\end{equation}
Combining this
and the assumption that
\begin{align}
  \sup_{ M \in \N }
  \sup_{ I \in \D }
  \sup_{ t \in [0,T] }
  \E\big[
    |
      \|
        \OMN{t}
      \|_{ H_{\gamma} }^2
      +
      \Phi( \OMN{} )
      +
      \phi( \OMN{} )
    |^{p\max\{2\vartheta, 2+\varphi, (1+\nicefrac{\varphi}{2})\nicefrac{\varphi}{2} \}}
  \big]
  <
  \infty
\end{align}
with Lemma~\ref{lem:Y_a_priori_mom}
demonstrates that
\begin{align}
\label{eq:main_proof_1}
\begin{split}
 &\sup_{ M \in \N }
  \sup_{ I \in \D }
  \sup_{ t \in [0,T] }
  \|
    \YMNp{t}
  \|_{ \L^{p\max\{4\vartheta,
                       4+2\varphi,
                       \varphi(1+\nicefrac{\varphi}{2})
                     \}}( \P; H_{\gamma} ) }
\\&=
  \sup_{ M \in \N }
  \sup_{ I \in \D }
  \sup_{ t \in [0,T] }
  \|
    \YMN{t}
  \|_{ \L^{p\max\{4\vartheta,
                       4+2\varphi,
                       \varphi(1+\nicefrac{\varphi}{2})
                     \}}( \P; H_{\gamma} ) }
  <
  \infty .
\end{split}
\end{align}
This establishes~\eqref{it:thm_main_1}.
Moreover, note that
the assumption that
$
  \forall \,
  t \in [0,T]
  \colon
  \P( 
    X_t
    =
    \int_0^t
    e^{(t-s)A}
    F( X_s ) \, ds
$
$
    +
    O_t
  )
  =
  1
$
yields that
$
  \P( \tilde{\Omega} ) = 1
$.
Hence, we obtain
for all $ t \in [0,T] $ that
$
  \P( \tilde{X}_t = X_t )
  \geq
  \P( \tilde{\Omega} ) = 1
$.
Combining this with the assumption that
$
  \sup_{ t \in [0,T] }
  \E\big[ 
    \| X_t \|_V^{p(1+\nicefrac{\varphi}{2})\max\{2,\varphi\}}
  \big]
  <
  \infty
$
ensures that
\begin{align}
\label{eq:main_proof_6}
  \sup_{ t \in [0,T] }
  \|
    \tilde{X}_t
  \|_{ \L^{p(1+\nicefrac{\varphi}{2})\max\{2,\varphi\}}( \P; V ) }
  =
  \sup_{ t \in [0,T] }
  \|
    X_t
  \|_{ \L^{p(1+\nicefrac{\varphi}{2})\max\{2,\varphi\}}( \P; V ) }
  <
  \infty .
\end{align}
Furthermore, observe that 
the fact that
\begin{equation}
  \forall \,
  t \in [0, \infty)
  ,
  r \in [0,1]
  \colon
  \| (-tA)^{r} \, e^{tA} \|_{ L(H) }
  \leq 1
  ,
\end{equation}
the triangle inequality,
and Lemma~\ref{lem:P_I_dense}~\eqref{it:P_I_dense_4}
show for all $ t \in (0,T] $ that
\begin{align}
\begin{split}
 &\int_0^t
  \| 
    (-A)^{(\rho+\varrho)}
    e^{(t-s)A}
    F( \tilde{X}_s )
  \|_{ \L^{p\max\{2,\varphi\}}( \P; H ) } \, ds
\\&\leq
  \int_0^t
  \| 
    (-A)^{(\rho+\varrho)}
    e^{(t-s)A}
  \|_{ L(H) } \,
  \|
    F( \tilde{X}_s )
  \|_{ \L^{p\max\{2,\varphi\}}( \P; H ) } \, ds
\\&\leq
  \int_0^t
  \left(t-s\right)^{-(\rho+\varrho)}
  \left[ 
    \|
      F( \tilde{X}_s ) - F(0)
    \|_{ \L^{p\max\{2,\varphi\}}( \P; H ) }
    +
    \| F(0) \|_H
  \right] ds
\\&\leq
  \int_0^t
  \left(t-s\right)^{-(\rho+\varrho)}
  \left[
    \Big\|
      \sqrt{c} \,
      \|
        \tilde{X}_s
      \|_V
      \big(
        1
        +
        \|
          \tilde{X}_s
        \|_V^{\nicefrac{\varphi}{2}}
      \big)
    \Big\|_{ \L^{p\max\{2,\varphi\}}( \P; \R ) }
    +
    \| F(0) \|_H
  \right] ds 
\\&\leq
  \int_0^t
  \left(t-s\right)^{-(\rho+\varrho)}
  \left[
    \sqrt{c} \,
    \Big(
      \|
	\tilde{X}_s
      \|_{ \L^{p\max\{2,\varphi\}}( \P; V ) }
      +
      \|
	\tilde{X}_s
      \|_{ \L^{p(1+\nicefrac{\varphi}{2})\max\{2,\varphi\}}( \P; V ) }^{\nicefrac{(1+\varphi}{2})}
    \Big)
    +
    \| F(0) \|_H
  \right] ds .
\end{split}
\end{align}
Inequality~\eqref{eq:main_proof_6} 
hence implies for all $ t \in (0,T] $ that
\begin{align}
\begin{split}
 &\int_0^t
  \| 
    (-A)^{(\rho+\varrho)}
    e^{(t-s)A}
    F( \tilde{X}_s )
  \|_{ \L^{p\max\{2,\varphi\}}( \P; H ) } \, ds
\\&\leq
  \int_0^t
  \left(t-s\right)^{-(\rho+\varrho)}
  \left[
    2 \sqrt{c}
    \max\!\left\{ 
      1,
      \|
        \tilde{X}_s
      \|_{ \L^{p(1+\nicefrac{\varphi}{2})\max\{2,\varphi\}}( \P; V ) }^{(1+\nicefrac{\varphi}{2})}
    \right\}
    +
    \| F(0) \|_H
  \right] ds
\\&\leq
  \left[
    2 \sqrt{c}
    \max\!\left\{ 
      1,
      \sup_{ s \in (0,t) }
      \|
        \tilde{X}_s
      \|_{ \L^{p(1+\nicefrac{\varphi}{2})\max\{2,\varphi\}}( \P; V ) }^{(1+\nicefrac{\varphi}{2})}
    \right\}
    +
    \| F(0) \|_H
  \right]
  \left[ 
    \int_0^t
    \left(t-s\right)^{-(\rho+\varrho)} ds
  \right]
\\&\leq
  \left[
    2 \sqrt{c}
    \max\!\left\{ 
      1,
      \sup_{ s \in (0,T) }
      \|
        \tilde{X}_s
      \|_{ \L^{p(1+\nicefrac{\varphi}{2})\max\{2,\varphi\}}( \P; V ) }^{(1+\nicefrac{\varphi}{2})}
    \right\}
    +
    \| F(0) \|_H
  \right]
  \frac{ T^{(1-\rho-\varrho)} }{ (1-\rho-\varrho) }
  <
  \infty .
\end{split}
\end{align} 
This and the fact that
\begin{equation}
  \forall \,
  t \in [0,T]
  \colon
  \tilde{X}_t
  =
  \int_0^t
  e^{(t-s)A}
  F( \tilde{X}_s ) \, ds
  +
  \tilde{O}_t
\end{equation}
prove that
\begin{align}
\label{eq:main_proof_xo}
\begin{split}
 &\sup_{ t \in [0,T] }
  \|
    \tilde{X}_t
    -
    \tilde{O}_t
  \|_{ \L^{p\max\{2,\varphi\}}( \P; H_{(\rho+\varrho)} ) }
\\ &
  \leq
  \sup_{ t \in [0,T] }
  \left(
    \int_0^t
    \| 
      e^{(t-s)A}
      F( \tilde{X}_s )
    \|_{ \L^{p\max\{2,\varphi\}}( \P; H_{(\rho+\varrho)} ) } \, ds
  \right)
  <
  \infty . 
\end{split}
\end{align}
In addition, observe that the triangle inequality
assures for all $ I \in \D $ that
\begin{align}
\label{eq:main_proof_10x}
\begin{split}
 &\sup_{ t \in [0,T] }
  \|
    \mathfrak{P}_{ \H \setminus I } \tilde{X}_t 
  \|_{ \L^{2p}(\P;V) }
  \leq
  \sup_{ t \in [0,T] }
  \left[ 
    \|
      \mathfrak{P}_{ \H \setminus I } (\tilde{X}_t - \tilde{O}_t)
    \|_{ \L^{2p}(\P;V) }
    +
    \|
      \mathfrak{P}_{ \H \setminus I } \tilde{O}_t 
    \|_{ \L^{2p}(\P;V) }
  \right]
\\&\leq
  \left[ 
    \sup_{ v \in H_{\rho} \setminus \{0\} }
    \frac{ \| v \|_V }{ \| v \|_{ H_{\rho} } }
  \right]
  \left[ 
    \sup_{ t \in [0,T] }  
    \|
      \mathfrak{P}_{ \H \setminus I } (\tilde{X}_t - \tilde{O}_t)
    \|_{ \L^{2p}(\P;H_{\rho}) }
  \right]
  +
  \sup_{ t \in [0,T] }  
  \|
    \mathfrak{P}_{ \H \setminus I } \tilde{O}_t 
  \|_{ \L^{2p}(\P;V) }
\\&=
  \left[ 
    \sup_{ v \in H_{\rho} \setminus \{0\} }
    \frac{ \| v \|_V }{ \| v \|_{ H_{\rho} } }
  \right]
  \left[ 
    \sup_{ t \in [0,T] }  
    \|
      (-A)^{-\varrho} \,
      \mathfrak{P}_{ \H \setminus I }
      (-A)^{(\rho + \varrho)}
      (\tilde{X}_t - \tilde{O}_t)
    \|_{ \L^{2p}(\P;H) } 
  \right]
\\&\quad+
  \sup_{ t \in [0,T] }  
  \|
    \mathfrak{P}_{ \H \setminus I } \tilde{O}_t 
  \|_{ \L^{2p}(\P;V) } .
\end{split}
\end{align}
The fact that 
\begin{equation}
  \forall \,
  I \in \Pow(\H)
  \colon
  \| \mathfrak{P}_{ \H \setminus I } \|_{ L(H) }
  \leq 
  1
\end{equation}
hence guarantees 
for all $ I \in \D $ that
\begin{align}
\label{eq:main_proof_10}
\begin{split}
 &\sup_{ t \in [0,T] }
  \|
    \mathfrak{P}_{ \H \setminus I } \tilde{X}_t 
  \|_{ \L^{2p}(\P;V) }
\\&\leq
  \left[ 
    \sup_{ v \in H_{\rho} \setminus \{0\} }
    \frac{ \| v \|_V }{ \| v \|_{ H_{\rho} } }
  \right]
  \left[ 
    \sup_{ t \in [0,T] }  
    \|
      \mathfrak{P}_{ \H \setminus I } (\tilde{X}_t - \tilde{O}_t)
    \|_{ \L^{2p}(\P;H_{(\rho+\varrho)}) } \,
  \right]
  \|
    (-A)^{-\varrho} \, \mathfrak{P}_{ \H \setminus I }
  \|_{ L(H) }
\\&\quad+
  \sup_{ t \in [0,T] }  
  \|
    \mathfrak{P}_{ \H \setminus I } \tilde{O}_t 
  \|_{ \L^{2p}(\P;V) } 
\\&\leq
  \left[ 
    \sup_{ v \in H_{\rho} \setminus \{0\} }
    \frac{ \| v \|_V }{ \| v \|_{ H_{\rho} } }
  \right]
  \left[ 
    \sup_{ t \in [0,T] }  
    \|
      \tilde{X}_t - \tilde{O}_t
    \|_{ \L^{2p}(\P;H_{(\rho+\varrho)}) } \,
  \right]
  \|
    (-A)^{-\varrho} \, \mathfrak{P}_{ \H \setminus I }
  \|_{ L(H) }
\\&\quad+
  \sup_{ t \in [0,T] }  
  \|
    \mathfrak{P}_{ \H \setminus I } \tilde{O}_t 
  \|_{ \L^{2p}(\P;V) } .
\end{split}
\end{align}
Furthermore, observe that
the triangle inequality,
the fact that 
\begin{equation}
  \forall \,
  I \in \Pow_0(\H)
  \colon
  \| P_I \|_{ L(H) }
  \leq 
  1
  ,
\end{equation}
the fact that
$
  H_{(\rho+\varrho)}
  \subseteq 
  V
$
continuously,
the fact that
\begin{equation}
  \forall \, t \in [0,T] 
  \colon
  \P( O_t = \tilde{O}_t )
  \geq 
  \P( \tilde{\Omega} )
  =
  1
  ,
\end{equation}
the assumption that
\begin{equation}
  \sup_{ t \in [0,T] }
  \sup_{ I \in \D }
  \E\big[
    \|
      P_I O_t
    \|_V^{p\varphi}
  \big]
  <
  \infty
  ,
\end{equation}
and~\eqref{eq:main_proof_xo}
imply that
\begin{align}
\label{eq:main_proof_xo_2}
\begin{split}
 &\sup_{ t \in [0,T] }
  \sup_{ I \in \D }
  \|
    P_I \tilde{X}_t
  \|_{ \L^{p\varphi}(\P; V) }
\\&\leq 
  \sup_{ t \in [0,T] }
  \sup_{ I \in \D }
  \|
    P_I \tilde{X}_t
    -
    P_I \tilde{O}_t
  \|_{ \L^{p\varphi}(\P; V) }
  +
  \sup_{ t \in [0,T] }
  \sup_{ I \in \D }
  \|
    P_I \tilde{O}_t
  \|_{ \L^{p\varphi}(\P; V) }
\\&\leq
  \left[ 
    \sup_{ v \in H_{(\rho+\varrho)} \setminus \{0\} }
    \frac{ \| v \|_V }{ \| v \|_{ H_{(\rho+\varrho)} } } 
  \right] 
  \left[ 
    \sup_{ t \in [0,T] }
    \sup_{ I \in \D }
    \|
      P_I 
      (
        \tilde{X}_t
        - 
        \tilde{O}_t
      )
    \|_{ \L^{p\varphi}(\P; H_{(\rho+\varrho)}) }
  \right]
  +
  \sup_{ t \in [0,T] }
  \sup_{ I \in \D }
  \|
    P_I \tilde{O}_t
  \|_{ \L^{p\varphi}(\P; V) }
\\&\leq
  \left[ 
    \sup_{ v \in H_{(\rho+\varrho)} \setminus \{0\} }
    \frac{ \| v \|_V }{ \| v \|_{ H_{(\rho+\varrho)} } } 
  \right] 
  \left[ 
    \sup_{ t \in [0,T] }
    \|
      \tilde{X}_t
      - 
      \tilde{O}_t
    \|_{ \L^{p\varphi}(\P; H_{(\rho+\varrho)}) }
  \right]
  +
  \sup_{ t \in [0,T] }
  \sup_{ I \in \D }
  \|
    P_I O_t
  \|_{ \L^{p\varphi}(\P; V) }
  <
  \infty .
\end{split}
\end{align}
In the next step we note that 
the hypothesis that $ H_{\rho} \subseteq V $ continuously
and the fact that
\begin{equation}
  \forall \,
  t \in [0,\infty) 
  ,
  r \in [0,1] 
  \colon
  \|
    (-tA)^r
    e^{tA}
  \|_{ L(H) }
  \leq 
  1
\end{equation}
ensure that
\begin{align}
\label{eq:main_proof_2}
\begin{split}
  \sup_{ t \in (0,T) } 
  t^{\rho}
  \| e^{tA} \|_{ L(H,V) }
 &=
  \sup_{ t \in (0,T) }
  \sup_{ v \in H\backslash\{0\} }
  \left[
    \frac{ 
      t^{\rho}
      \| e^{tA} v \|_V
    }{ \| v \|_H }
  \right] 
\\&\leq
  \left[ 
    \sup_{ v \in H_{\rho}\backslash\{0\} }
    \frac{ 
      \| v \|_V
    }{
      \| v \|_{ H_{\rho} }
    }
  \right]
  \left[
    \sup_{ t \in (0,T) }
    \sup_{ v \in H\backslash\{0\} }
    \frac{ 
      \| (-tA)^{\rho} e^{tA} v \|_H
    }{ \| v \|_H }
  \right]
\\&=
  \left[ 
    \sup_{ v \in H_{\rho}\backslash\{0\} }
    \frac{ 
      \| v \|_V
    }{
      \| v \|_{ H_{\rho} }
    }
  \right]
  \left[
    \sup_{ t \in (0,T) }
    \|
      (-tA)^{\rho}
      e^{tA}
    \|_{ L(H) }
  \right]
\\ & \leq
  \left[ 
    \sup_{ v \in H_{\rho}\backslash\{0\} }
    \frac{ 
      \| v \|_V
    }{
      \| v \|_{ H_{\rho} }
    }
  \right]
  <
  \infty .
\end{split}
\end{align}
Moreover, observe that, e.g., 
Lemma~\ref{lem:measurable}
(with $ V = H_{\gamma} \times H_{\gamma} $,
$ W = V \times V $, $ S = [0,\infty) $,
$ \mathcal{S} = \B( [0,\infty) ) $,
$ s = 0 $, 
$ 
  \psi 
  = 
  H_{\gamma} \times H_{\gamma}
  \ni (v,w)
  \mapsto
  ( 
    \| v \|_{ H_{\gamma} }
    +
    \| w \|_{ H_{\gamma} }
  )^{ \nicefrac{1}{\chi} }
  \in 
  [0,\infty)
$,
$ \Psi = \Vm $
in the notation of Lemma~\ref{lem:measurable})
establishes that
\begin{align}
 &\Vm \in \M\big( \B( V \times V ) , \B( [0,\infty) ) \big).
\end{align}
The fact that
$
  \forall \,
  t \in [0,T]
  \colon
  \tilde{X}_t
  =
  \int_0^t
  e^{(t-s)A}
  F( \tilde{X}_s ) \, ds
  +
  \tilde{O}_t
$,
\eqref{eq:main_proof_0},
\eqref{eq:main_proof_2},
Lemma~\ref{lem:P_I_dense}~\eqref{it:P_I_dense_3},
Lemma~\ref{lem:P_I_dense}~\eqref{it:P_I_dense_4},
and, e.g., Andersson et al.~\cite[Lemma~2.2]{AnderssonJentzenKurniawan2015}
hence allow us to
apply Proposition~\ref{prop:X_Y_diff_lp}
(with 
$ H = H $, $ \H = \H $, $ T = T $,
$ c = c $, $ \varphi = \varphi $,
$ C = c $, $ \D = \D $, $ \mu = \mu $,
$ A = A $, $ V = V $, $ \Vm = \Vm $,
$ F = F $, 
$ 
  (P_J)_{ J \in \Pow(\H) } \subseteq L(H) 
  = 
  (\mathfrak{P}_J)_{ J \in \Pow(\H) } \subseteq L(H) 
$, 
$ (\Omega, \F, \P ) = ( \Omega, \F, \P ) $,
$ X_t(\omega) = \tilde{X}_t(\omega) $, 
$ O_t(\omega) = \tilde{O}_t(\omega) $,
$
  \XXMN{t}(\omega) 
  =
  \int_0^t 
  P_I e^{(t-s)A} F( \YMNp{\fl{s}}(\omega) ) \, ds
  +
  P_I ( \tilde{O}_t(\omega) )
$,
$ \YMN{t}(\omega) = \YMNp{t}(\omega) $,
$ \OMN{t}(\omega) = \OMN{t}(\omega) $,
$ \alpha = \vartheta\chi $,
$ \rho = \rho $, $ \varrho = \varrho $, 
$ \kappa = \kappa $, 
$ p = p $, $ M = M $, $ I = I $
for $ M \in \N $, $ I \in \D $, $ t \in [0,T] $,
$ \omega \in \Omega $
in the notation of Proposition~\ref{prop:X_Y_diff_lp})
to obtain that for all $ M \in \N $, $ I \in \D $ we have that
\begin{align}
\label{eq:main_proof__6}
\begin{split}
 &\sup_{ t \in [0,T] }
  \|
    \tilde{X}_t - \YMNp{t}
  \|_{ \L^p( \P; H ) }
\\&\leq
  \frac{
    4^{(2+\varphi)}
    \max\{1, T^{(\nicefrac{3}{2}+\vartheta\chi-\rho + \nicefrac{\varphi}{2} - \nicefrac{\rho\varphi}{2}) } \}
    \max\{ 1, c^{(1+\nicefrac{\varphi}{4})} \}
    \sqrt{ e^{\kappa c T } }
  }{
    \min\{ 1, \sqrt{ c \left( \kappa - 2 \right) } \}
    \left( 1 - \rho - \varrho \right)^{(1+\nicefrac{\varphi}{2})}
  }
  \Bigg[
    \sup_{ t \in [0,T] }
    \|
      \mathfrak{P}_{ \H \setminus I } \tilde{X}_t 
    \|_{ \L^{2p}(\P;V) }
    +
    M^{-\min\{\vartheta\chi,\varrho\} }
\\&+
    \sup_{ t \in (0,T) }
    \|
      P_I ( \tilde{O}_{t} - \tilde{O}_{\fl{t}} ) 
    \|_{ \L^{2p}( \P; V ) }
    +
    \sup_{ t \in [0,T] }
    \|
      P_I \tilde{O}_t
      -
      \OMN{t}
    \|_{ \L^{2p}(\P; V) }
  \Bigg] 
  \max\!\left\{
    1,
    \sup_{ v \in V \setminus \{0\} }
    \frac{ \| v \|_H }{ \| v \|_V }
  \right\}
\\&\cdot 
  \Bigg[
    1
    +
    \sup_{ s \in (0,T) }
    \|
      \tilde{X}_s
    \|_{ \L^{p\varphi}(\P;V) }^{\nicefrac{\varphi}{2}}
    +
    \sup_{ s \in (0,T) }
    \|
      P_I \tilde{X}_s
    \|_{ \L^{p\varphi}(\P;V) }^{\nicefrac{\varphi}{2}}
    +
    \sup_{ s \in [0,T) }
    \|
      \YMNp{s}
    \|_{ \L^{p\varphi}(\P;V) }^{\nicefrac{\varphi}{2}}
    +
    \sup_{ s \in (0,T) }
    \|
      P_I \tilde{O}_s
    \|_{ \L^{p\varphi}( \P; V ) }^{\nicefrac{\varphi}{2}}
  \Bigg]
\\&\cdot
  \left[ 
    1
    +
    \sup_{ s \in [0,T) }
    \|
      \Vm( \YMNp{s}, \OMN{s} )
    \|_{ \L^{4p\vartheta\chi}( \P; \R) }^{\vartheta\chi}
  \right]
  \max\!\left\{ 
    1,
    \sup_{ s \in (0,T) }
    \left[ 
      s^{\rho}
      \|
	e^{sA}
      \|_{ L(H,V) }
    \right]^{(1+\nicefrac{\varphi}{2})}
  \right\}
\\&\cdot
  \left[ 
    \max\!\left\{ 
      1,
      \sup_{ s \in [0,T) }
      \|
	\YMNp{s}
      \|_{ \L^{p(1+\nicefrac{\varphi}{2})\max\{4,\varphi\}}( \P; V ) }^{[(1+\nicefrac{\varphi}{2})^2]}
    \right\} 
    +
    \| F(0) \|_H^{ (1+\nicefrac{\varphi}{2}) } 
  \right] .
\end{split}
\end{align}
The fact that 
$ 
  \forall \, t \in [0,T] 
  \colon
  \P( X_t = \tilde{X}_t )
  \geq 
  \P( \tilde{\Omega} )
  =
  1
$,
the fact that
$ 
  \forall \, t \in [0,T] 
  \colon
  \P( O_t = \tilde{O}_t )
  \geq 
  \P( \tilde{\Omega} )
  =
  1
$,
and~\eqref{eq:main_thm_V_1}
hence prove that for all 
$ M \in \N $, $ I \in \D $
we have that
\begin{align}
\label{eq:main_proof__6_B}
\begin{split}
 &\sup_{ t \in [0,T] }
  \|
    X_t - \YMNp{t}
  \|_{ \L^p( \P; H ) }
  =
  \sup_{ t \in [0,T] }
  \|
    \tilde{X}_t - \YMNp{t}
  \|_{ \L^p( \P; H ) }
\\&\leq
  \frac{
    4^{(2+\varphi)}
    \max\{1, T^{(\nicefrac{3}{2}+\vartheta\chi-\rho + \nicefrac{\varphi}{2} - \nicefrac{\rho\varphi}{2}) } \}
    \max\{ 1, c^{(1+\nicefrac{\varphi}{4})} \}
    \sqrt{ e^{\kappa c T } }
  }{
    \min\{ 1, \sqrt{ c \left( \kappa - 2 \right) } \}
    \left( 1 - \rho - \varrho \right)^{(1+\nicefrac{\varphi}{2})}
  }
  \Bigg[
    \sup_{ t \in [0,T] }
    \|
      \mathfrak{P}_{ \H \setminus I } \tilde{X}_t 
    \|_{ \L^{2p}(\P;V) }
    +
    M^{-\min\{\vartheta\chi,\varrho\} }
\\&+
    \sup_{ t \in (0,T) }
    \|
      P_I ( O_{t} - O_{\fl{t}} ) 
    \|_{ \L^{2p}( \P; V ) }
    +
    \sup_{ t \in [0,T] }
    \|
      P_I O_t
      -
      \OMN{t}
    \|_{ \L^{2p}(\P; V) }
  \Bigg] 
  \max\!\left\{
    1,
    \sup_{ v \in V \setminus \{0\} }
    \frac{ \| v \|_H }{ \| v \|_V }
  \right\}
\\&\cdot 
  \Bigg[
    1
    +
    \sup_{ s \in (0,T) }
    \|
      \tilde{X}_s
    \|_{ \L^{p\varphi}(\P;V) }^{\nicefrac{\varphi}{2}}
    +
    \sup_{ s \in (0,T) }
    \|
      P_I \tilde{X}_s
    \|_{ \L^{p\varphi}(\P;V) }^{\nicefrac{\varphi}{2}}
    +
    \sup_{ s \in [0,T) }
    \|
      \YMNp{s}
    \|_{ \L^{p\varphi}(\P;V) }^{\nicefrac{\varphi}{2}}
    +
    \sup_{ s \in (0,T) }
    \|
      P_I O_s
    \|_{ \L^{p\varphi}( \P; V ) }^{\nicefrac{\varphi}{2}}
  \Bigg]
\\&\cdot
  \left[ 
    1
    +
    \sup_{ s \in [0,T) }
    \left\{
      \|
        \YMNp{s}
      \|_{ \L^{4p\vartheta}(\P; H_{\gamma}) }
      +
      \|
        \OMN{s}
      \|_{ \L^{4p\vartheta}(\P; H_{\gamma}) }
    \right\}^{\vartheta}
  \right]
  \max\!\left\{ 
    1,
    \sup_{ s \in (0,T) }
    \left[ 
      s^{\rho}
      \|
	e^{sA}
      \|_{ L(H,V) }
    \right]^{(1+\nicefrac{\varphi}{2})}
  \right\}
\\&\cdot
  \left[ 
    \max\!\left\{ 
      1,
      \sup_{ s \in [0,T) }
      \|
	\YMNp{s}
      \|_{ \L^{p(1+\nicefrac{\varphi}{2})\max\{4,\varphi\}}( \P; V ) }^{[(1+\nicefrac{\varphi}{2})^2]}
    \right\} 
    +
    \| F(0) \|_H^{ (1+\nicefrac{\varphi}{2}) } 
  \right] .
\end{split}
\end{align}
The fact that 
\begin{equation}
  \forall \,
  M \in \N
  ,
  I \in \D
  ,
  t \in [0,T]
  \colon
  \P( \YMN{t} = \YMNp{t} )
  =
  1
  ,
\end{equation}
the fact that 
\begin{equation}
  \forall \,
  t \in [0,T]
  \colon
  \P( O_t = \tilde{O}_t )
  \geq 
  \P( \tilde{\Omega} )
  =
  1
  ,
\end{equation}
the fact that 
$ 
  \forall \,
  I \in \D
  \colon
  \| P_I \|_{ L(H) }
  \leq 1
$,
\eqref{eq:main_proof_10},
and~\eqref{eq:main_proof_2} therefore
assure that
for all $ M \in \N $, $ I \in \D $
we have that
\begin{align}
\label{eq:main_proof_7}
\begin{split}
 &\sup_{ t \in [0,T] }
  \|
    X_t - \YMN{t}
  \|_{ \L^p( \P; H ) }
  =
  \sup_{ t \in [0,T] }
  \|
    X_t - \YMNp{t}
  \|_{ \L^p( \P; H ) }
\\&\leq
  \frac{
    4^{(2+\varphi)}
    \max\{1, T^{(\nicefrac{3}{2}+\vartheta\chi-\rho + \nicefrac{\varphi}{2} - \nicefrac{\rho\varphi}{2}) } \}
    \max\{ 1, c^{(1+\nicefrac{\varphi}{4})} \}
    \sqrt{ e^{\kappa c T } }
  }{
    \min\{ 1, \sqrt{ c \left( \kappa - 2 \right) } \}
    \left( 1 - \rho - \varrho \right)^{(1+\nicefrac{\varphi}{2})}
  }
  \Bigg[
    \sup_{ t \in [0,T] }
    \|
       \mathfrak{P}_{ \H \setminus I } O_t
    \|_{ \L^{2p}(\P;V) }
\\&+
    \sup_{ t \in [0,T] }
    \|
      \tilde{X}_t - \tilde{O}_t
    \|_{ \L^{2p}(\P;H_{(\rho+\varrho)}) } \,
    \|
      (-A)^{-\varrho} \,
      \mathfrak{P}_{ \H \setminus I }
    \|_{ L(H) }
    +
    M^{-\min\{\vartheta\chi,\varrho\} }
\\&+
    \sup_{ t \in (0,T) }
    \|
      P_I ( O_{t} - O_{\fl{t}} ) 
    \|_{ \L^{2p}( \P; V ) }
    +
    \sup_{ t \in [0,T] }
    \|
      P_I O_t
      -
      \OMN{t}
    \|_{ \L^{2p}(\P; V) }
  \Bigg] 
  \max\!\left\{
    1,
    \sup_{ v \in V \setminus \{0\} }
    \frac{ \| v \|_H }{ \| v \|_V }
  \right\}
\\&\cdot
  \Bigg[
    1
    +
    \sup_{ s \in (0,T) }
    \|
      \tilde{X}_s
    \|_{ \L^{p\varphi}(\P;V) }^{\nicefrac{\varphi}{2}}
    +
    \sup_{ s \in (0,T) }
    \|
      P_I \tilde{X}_s
    \|_{ \L^{p\varphi}(\P;V) }^{\nicefrac{\varphi}{2}}
    +
    \sup_{ s \in [0,T) }
    \|
      \YMNp{s}
    \|_{ \L^{p\varphi}(\P;H_{\gamma}) }^{\nicefrac{\varphi}{2}}
    +
    \sup_{ s \in (0,T) }
    \|
      P_I O_s
    \|_{ \L^{p\varphi}( \P; V ) }^{\nicefrac{\varphi}{2}}
  \Bigg]
\\&\cdot
  \left[ 
    1
    +
    \sup_{ s \in [0,T) }
    \left\{
      \|
        \YMNp{s}
      \|_{ \L^{4p\vartheta}(\P; H_{\gamma}) }
      +
      \|
        \OMN{s}
      \|_{ \L^{4p\vartheta}(\P; H_{\gamma}) }
    \right\}^{\vartheta}
  \right]
  \max\!\left\{
    1,
    \sup_{ v \in H_{\gamma} \setminus \{0\} }
    \frac{ 
      \| v \|_V^{ [(1+\nicefrac{\varphi}{2})^2 + \nicefrac{\varphi}{2} ] } 
    }{ 
      \| v \|_{ H_{\gamma} }^{ [(1+\nicefrac{\varphi}{2})^2 + \nicefrac{\varphi}{2}] } 
    }
  \right\}
\\&\cdot
  \left[ 
    \max\!\left\{ 
      1,
      \sup_{ s \in [0,T) }
      \|
	\YMNp{s}
      \|_{ 
        \L^{p(1+\nicefrac{\varphi}{2})\max\{4,\varphi\}}( \P; H_{\gamma} ) 
      }^{[(1+\nicefrac{\varphi}{2})^2]}
    \right\} 
    +
    \| F(0) \|_H^{ (1+\nicefrac{\varphi}{2}) } 
  \right] 
  \max\!\left\{
    1,
    \sup_{ v \in H_{\rho} \setminus \{0\} }
    \frac{ 
      \| v \|_V^{ (2+\nicefrac{\varphi}{2}) } 
    }{ 
      \| v \|_{ H_{\rho} }^{ (2+\nicefrac{\varphi}{2}) } 
    }
  \right\} .
\end{split}
\end{align}
The fact that $ H_{\gamma} \subseteq H_{\rho} \subseteq V \subseteq H $
continuously,
\eqref{eq:main_proof_1},
\eqref{eq:main_proof_6},
\eqref{eq:main_proof_xo},
\eqref{eq:main_proof_xo_2},
and the fact that
\begin{equation}
  \sup_{ M \in \N }
  \sup_{ I \in \D }
  \sup_{ t \in [0,T] }
  \big[ 
    \|
      \OMN{t}
    \|_{ \L^{4p\vartheta}(\P; H_{\gamma}) }
    +
    \| P_I O_t \|_{ \L^{p\varphi}( \P; V ) }
  \big]  
  <
  \infty
\end{equation}
therefore establish~\eqref{it:thm_main_2}.
The proof
of Theorem~\ref{thm:main}
is thus completed. 
\end{proof}

\begin{corollary}
\label{cor:main}
Assume the setting in Section~\ref{sec:main_result_setting},
let $ \theta \in [0,\infty) $, $ \vartheta \in (0,\infty) $,
$ p \in [ \max\{ 2, \nicefrac{1}{\varphi} \}, \infty ) $,
$ \varrho \in [0,1-\rho) $,
and assume that
\begin{align}
\label{eq:cor_main_assumption}
\begin{split}
 &\sup_{ t \in [0,T] }
  \E\big[
    \|
      X_t
    \|_V^{p(1+\nicefrac{\varphi}{2})\max\{2,\varphi\}}
  \big]
\\&+
  \sup_{ M \in \N }
  \sup_{ I \in \D }
  \sup_{ t \in [0,T] }
  \left(
    M^{\theta} \!
    \left\{
      \left(
        \E\big[
          \|
          P_I( O_{t} - O_{\fl{t}} )
          \|_V^{2p}
        \big]
      \right)^{\!\frac{1}{2p}}
      +
      \left(
        \E\big[
          \|
            P_I O_{t} 
            -
            \OMN{t}
          \|_V^{p\max\{2,\varphi\}}  
        \big]
      \right)^{\!\frac{1}{p\max\{2,\varphi\}}}
    \right\}
  \right)
\\&+
  \sup_{ M \in \N }
  \sup_{ I \in \D }
  \sup_{ t \in [0,T] }
  \E\!\left[ 
    \big|
      \|
        \OMN{t}
      \|_{ H_{\gamma} }^2
      +
      \Phi( \OMN{} )
      +
      \phi( \OMN{} )
    \big|^{p\max\{2\vartheta, 2+\varphi, (1+\nicefrac{\varphi}{2})\nicefrac{\varphi}{2} \} }
  \right]
  <
  \infty .
\end{split}
\end{align}
Then we have
\begin{enumerate}[(i)]
 \item\label{it:cor_main_1} that
 $
  \sup_{ M \in \N }
  \sup_{ I \in \D }
  \sup_{ t \in [0,T] }
  \E\big[
    \|
      \YMN{t}
    \|_{H_{\gamma}}^{p\max\{4\vartheta, 4+2\varphi, \varphi(1+\nicefrac{\varphi}{2}) \}}
  \big]
  <
  \infty
$
and
\item\label{it:cor_main_2}
that there exists
a real number $ C \in (0,\infty) $ such that
for all $ M \in \N $, $ I \in \D $
it holds that
\begin{align}
\label{eq:cor_main_assert_2}
\begin{split}
 &\sup_{ t \in [0,T] }
  \left(
    \E\big[
      \|
        X_t - \YMN{t}
      \|_H^p
    \big]
  \right)^{\!\nicefrac{1}{p}}
\\&\leq
  C
  \left[
    M^{-\min\{\vartheta\chi,\varrho,\theta\}}
    +
    \|
      (-A)^{-\varrho}
      (\Id_H - P_I)
    \|_{ L(H) }
    +
    \sup_{ t \in [0,T] }
    \|
      (\Id_H - P_I)
      O_t
    \|_{ \L^{2p}( \P; V ) }
  \right] .
\end{split}
\end{align}
\end{enumerate}
\end{corollary}
\begin{proof}[Proof of Corollary~\ref{cor:main}]
Note that the triangle inequality, \eqref{eq:cor_main_assumption},
and the fact that $ H_{\gamma} \subseteq V $ continuously
yield that
\begin{align}
\label{eq:cor_main_proof_1}
\begin{split}
  \sup_{ I \in \D }
  \sup_{ s \in (0,T) }
  \| P_I O_s \|_{ \L^{p\varphi}(\P; V) }
 &\leq
  \sup_{ M \in \N }
  \sup_{ I \in \D }
  \sup_{ s \in (0,T) }
  \left[ 
    \| P_I O_s - \OMN{s} \|_{ \L^{p\varphi}(\P; V) }
    +
    \| \OMN{s} \|_{ \L^{p\varphi}(\P; V) }
  \right]
\\&\leq
  \sup_{ M \in \N }
  \sup_{ I \in \D }
  \sup_{ s \in (0,T) }
  \| P_I O_s - \OMN{s} \|_{ \L^{p\varphi}(\P; V) }
\\&\quad+
  \left[ 
    \sup_{ v \in H_{\gamma} \setminus \{0\} }
    \frac{ \| v \|_V }{ \| v \|_{H_{\gamma}} }
  \right] 
  \left[ 
    \sup_{ M \in \N }
    \sup_{ I \in \D }
    \sup_{ s \in (0,T) }
    \| \OMN{s} \|_{ \L^{p\varphi}(\P; H_{\gamma}) }
  \right]
  <
  \infty .
\end{split}
\end{align}
This together with~\eqref{eq:cor_main_assumption}
allows us to apply Theorem~\ref{thm:main}
to obtain that~\eqref{it:cor_main_1} holds
and that there exists
a real number $ K \in (0,\infty) $ such that
for all $ M \in \N $, $ I \in \D $
we have that
\begin{align}
\label{eq:cor_main_proof_3}
\begin{split}
 &\sup_{ t \in [0,T] }
  \|
    X_t
    -
    \YMN{t}
  \|_{ \L^p( \P; H) }
  \leq
  K
  \bigg(
    M^{-\min\{\vartheta\chi,\varrho\}}
    +
    \|
      (-A)^{-\varrho}
      (\Id_H - P_I)
    \|_{ L(H) }
\\&\quad+
    \sup_{ t \in [0,T] }
    \left[ 
      \|
	(\Id_H - P_I)
	O_t
      \|_{ \L^{2p}( \P; V ) }
      +
      \|
	P_I O_{t} 
	-
	\OMN{t}
      \|_{ \L^{2p}( \P; V ) }
      +
      \|
	P_I( 
	  O_{t}
	  -
	  O_{\fl{t}} 
	)
      \|_{ \L^{2p}( \P; V ) }
    \right]
  \bigg) .
\end{split}
\end{align}
Hence, we obtain that
for all $ M \in \N $, $ I \in \D $ 
we have that
\begin{align}
\label{eq:cor_main_proof_4}
\begin{split}
 &\sup_{ t \in [0,T] }
  \|
    X_t
    -
    \YMN{t}
  \|_{ \L^p( \P; H) } 
\\&\leq
  K
  \Bigg(
    M^{-\min\{\vartheta\chi,\varrho\}}
    +
    \|
      (-A)^{-\varrho}
      (\Id_H - P_I)
    \|_{ L(H) }
    +
    \sup_{ t \in [0,T] }
    \|
      (\Id_H - P_I)
      O_t
    \|_{ \L^{2p}( \P; V ) }
\\&\quad+
    M^{-\theta}
    \left\{ 
      \sup_{ t \in [0,T] }
      \left( 
        M^{\theta}
        \left[
          \|
            P_I O_{t} 
            -
            \OMN{t}
          \|_{ \L^{2p}( \P; V ) }
          +
          \|
            P_I( O_{t} - O_{\fl{t}} )
          \|_{ \L^{2p}( \P; V ) }
        \right] 
      \right)
    \right\}
  \Bigg) .
\end{split}
\end{align}
This implies that for all $ M \in \N $, $ I \in \D $
we have that
\begin{align}
\label{eq:cor_main_proof_5}
\begin{split}
 &\sup_{ t \in [0,T] }
  \|
    X_t
    -
    \YMN{t}
  \|_{ \L^p( \P; H) } 
\\&\leq
  K
  \Bigg(
    M^{-\min\{\vartheta\chi,\varrho, \theta\}}
    +
    \|
      (-A)^{-\varrho}
      (\Id_H - P_I)
    \|_{ L(H) }
    +
    \sup_{ t \in [0,T] }
    \|
      (\Id_H - P_I)
      O_t
    \|_{ \L^{2p}( \P; V ) }
\\&
+
    M^{-\min\{\vartheta\chi,\varrho, \theta\}}
    \left\{ 
      \sup_{ t \in [0,T] }
      \left( 
        M^{\theta}
        \left[
          \|
            P_I O_{t} 
            -
            \OMN{t}
          \|_{ \L^{2p}( \P; V ) }
          +
          \|
            P_I( O_{t} - O_{\fl{t}} )
          \|_{ \L^{2p}( \P; V ) }
        \right] 
      \right)
    \right\}
  \Bigg)
\\&=
  K
  \Bigg(
    \|
      (-A)^{-\varrho}
      (\Id_H - P_I)
    \|_{ L(H) }
    +
    \sup_{ t \in [0,T] }
    \|
      (\Id_H - P_I)
      O_t
    \|_{ \L^{2p}( \P; V ) }
    +
    M^{-\min\{\vartheta\chi,\varrho, \theta\}}
\\&
\cdot    
    \left\{ 
      1
      +
      \sup_{ t \in [0,T] }
      \left( 
        M^{\theta}
        \left[
          \|
            P_I O_{t} 
            -
            \OMN{t}
          \|_{ \L^{2p}( \P; V ) }
          +
          \|
            P_I( O_{t} - O_{\fl{t}} )
          \|_{ \L^{2p}( \P; V ) }
        \right] 
      \right)
    \right\}
  \Bigg) .
\end{split}
\end{align}
Therefore, we obtain that
for all $ M \in \N $, $ I \in \D $
we have that
\begin{align}
\label{eq:cor_main_proof_6}
\begin{split}
 &\sup_{ t \in [0,T] }
  \|
    X_t
    -
    \YMN{t}
  \|_{ \L^p( \P; H) } 
\\&\leq
  K
  \Bigg(
    \|
      (-A)^{-\varrho}
      (\Id_H - P_I)
    \|_{ L(H) }
    +
    \sup_{ t \in (0,T) }
    \|
      (\Id_H - P_I)
      O_t
    \|_{ \L^{2p}( \P; V ) }
    +
    M^{-\min\{\vartheta\chi,\varrho,\theta\}}
\\&\quad\cdot 
    \left[ 
      1
      +
      \sup_{ N \in \N }
      \sup_{ J \in \D }
      \sup_{ t \in [0,T] }
      \left(
        N^{\theta}
        \left[
          \|
            P_J( O_{t} - O_{\fl[T/N]{t}} )
          \|_{ \L^{2p}( \P; V ) }
          +
          \|
            P_J O_{t} - \mathcal{O}^{N,J}_t
          \|_{ \L^{2p}( \P; V ) }
        \right]
      \right)
    \right]
  \Bigg) 
\\&\leq
  K
  \left[ 
    1
    +
    \sup_{ N \in \N }
    \sup_{ J \in \D }
    \sup_{ t \in [0,T] }
    \left(
      N^{\theta}
      \left[
        \|
          P_J( O_{t} - O_{\fl[T/N]{t}} )
        \|_{ \L^{2p}( \P; V ) }
        +
        \|
          P_J O_{t} - \mathcal{O}^{N,J}_t
        \|_{ \L^{2p}( \P; V ) }
      \right]
    \right)
  \right]
\\&\quad\cdot
  \left[
    M^{-\min\{\vartheta\chi,\varrho,\theta\}}
    +
    \|
      (-A)^{-\varrho}
      (\Id_H - P_I)
    \|_{ L(H) }
    +
    \sup_{ t \in (0,T) }
    \|
      (\Id_H - P_I)
      O_t
    \|_{ \L^{2p}( \P; V ) }
  \right] .
\end{split}
\end{align}
Combining this with~\eqref{eq:cor_main_assumption}
establishes~\eqref{it:cor_main_2}.
The proof of Corollary~\ref{cor:main} is thus completed.
\end{proof}

\section{Stochastic Allen-Cahn equations}
\label{sec:examples}
\subsection{Setting}
\label{sec:examples_setting}
Consider the notation in Section~\ref{sec:notation},
let $ T, \nu \in (0,\infty) $,
$ a_0, a_1, a_2 \in \R $,
$ a_3 \in (-\infty,0] $,
$ 
  ( H, \langle \cdot, \cdot \rangle_H, \left\| \cdot \right\|\!_H ) 
$
$
  =
  ( 
    L^2(\lambda_{(0,1)}; \R), 
    \langle \cdot, \cdot \rangle_{ L^2(\lambda_{(0,1)}; \R) },
    \left\| \cdot \right\|\!_{ L^2(\lambda_{(0,1)}; \R) }
  )
$,
$ (e_n)_{ n \in \N } \subseteq H $,
$ F \colon L^{6}(\lambda_{(0,1)}; \R) \rightarrow H $,
$ (P_n)_{ n \in \N } \subseteq L(H) $
satisfy for all $ n \in \N $,
$ v \in L^{6}(\lambda_{(0,1)}; \R) $ that
$ a_2 \, \one_{\{0\}}^{\R}(a_3) = 0 $,
$ e_n = [ (\sqrt{2}\sin(n \pi x))_{ x \in (0,1) } ]_{ \lambda_{(0,1)}, \B(\R) } $,
$ 
  F(v) 
  = 
  \sum_{ k=0 }^{3} 
  a_k v^k  
$,
$ 
  P_n(v) 
  = 
  \sum_{ k=1 }^n 
  \langle e_k, v \rangle_H \, e_k 
$,
let $ A \colon D(A) \subseteq H \rightarrow H $
be the Laplacian with Dirichlet boundary conditions on $ H $
times the real number $ \nu $,
let $ ( H_r, \langle \cdot, \cdot \rangle_{H_r}, \left\| \cdot \right\|\!_{H_r} ) $,
$ r \in \R $, be a family of interpolation spaces associated to $ -A $
(cf., e.g., \cite[Section~3.7]{sy02}),
let $ ( \Omega, \F, \P ) $ be a probability space,
let $ (W_t)_{ t \in [0,T] } $ be an $ \Id_H $-cylindrical 
Wiener process, 
and let
$ 
  (\underline{\cdot})
  \colon
  \big\{
    [v]_{\lambda_{(0,1)},\B(\R)} \in H
    \colon
    ( v \in \C( (0,1), \R )
      \text{ is a uniformly continuous function}
    )
  \big\}
  \rightarrow
  \C( (0,1), \R)
$
be the function which satisfies
for all uniformly continuous functions
$ v \colon (0,1) \rightarrow \R $ 
that $ \underline{[v]_{\lambda_{(0,1)}, \B(\R)}} = v $.

\subsection{Properties of the nonlinearities
of stochastic Allen-Cahn equations}
\label{sec:AllenCahn}

\begin{lemma}
\label{lem:coercivity_mix}
Assume the setting in Section~\ref{sec:examples_setting} 
and let $ \epsilon \in (0,1) $,
$ c \in [\frac{32}{\epsilon}\max\{\frac{|a_2|^2}{|a_3|+\one_{\{0\}}^{\R}(a_3)},|a_3|\},\infty) $. 
Then there exists a real number $ C \in (0,\infty) $ such that
for all $ N \in \N $, $ v \in P_N(H) $,
$ w \in \C([0,T],H_1) $, $ t \in [0,T] $
we have that
\begin{align}
\begin{split}
 &\langle 
    v, 
    P_N F(v + w_t)
  \rangle_{H_{\nicefrac{1}{2}}}
  +
  c
  \big[
    \textstyle\sup_{ s \in [0,T] }
    \| w_s \|_{ L^{\infty}( \lambda_{(0,1)} ; \R) }^4
    +
    1
  \big]
  \langle 
    v, 
    F(v + w_t)
  \rangle_H
\\&\leq
  \epsilon
  \| v \|_{ H_1 }^2
  +
  \Big(
    | a_1 |
    +
    \tfrac{ 
      |a_2|^2 
    }{
      3|a_3| + \one_{\{0\}}^{\R}(a_3)
    }
  \Big)
  \| v \|_{ H_{\nicefrac{1}{2}} }^2
\\&\quad+
  c
  \big[
    \textstyle\sup_{ s \in [0,T] }
    \| w_s \|_{ L^{\infty}( \lambda_{(0,1)} ; \R) }^4
    +
    1
  \big]
  \Big(
    | a_0 |
    +
    \tfrac{3 |a_1|}{2}
  \Big)
  \| v \|_H^2
\\&\quad+
  C
  \big[
    \textstyle\sup_{ s \in [0,T] }
    \| w_s \|_{ L^{\infty}( \lambda_{(0,1)} ; \R) }^{8}
    +
    1
  \big] .
\end{split}
\end{align}
\end{lemma}
\begin{proof}[Proof of Lemma~\ref{lem:coercivity_mix}]
Throughout this proof
let $ \eta \colon \N_0 \times \N_0 \rightarrow (0,\infty) $
be a function which satisfies that
\begin{align}
\begin{split}
 &\one_{(-\infty,0)}^{\R}(a_3)
  \left(
    a_3
    +
    \frac{1}{2}
    \sum_{ k = 0 }^3
    \sum_{ j = 0 }^{ \min\{k,2\} }
    \binom{k}{j}
    (j+1) \, | a_k | \,
    | \eta(k,j) |^{\frac{4}{(j+1)}}
  \right)
  \leq 
  0 .
\end{split}
\end{align}
Observe that 
the fact that for every $ N \in \N $
we have that $ P_N $ is symmetric
implies that for all $ N \in \N $,
$ v \in P_N(H) $, $ w \in \C([0,T], H_1) $,
$ t \in [0,T] $ we have that
\begin{align}
\begin{split}
 &\langle 
    v, 
    \textstyle\sum_{ k = 0 }^3
    a_k P_N [v + w_t]^k
  \rangle_{H_{\nicefrac{1}{2}}}
\\&=
  \sum_{ k = 0 }^3
  a_k
  \langle 
    (-A)^{\nicefrac{1}{2}}
    v,
    (-A)^{\nicefrac{1}{2}}
    P_N [v + w_t]^k
  \rangle_H
  =
  \sum_{ k = 0 }^3
  a_k
  \langle 
    (-A)
    v,
    P_N [v + w_t]^k
  \rangle_H
\\&=
  -
  \sum_{ k = 0 }^3
  a_k
  \langle 
    A
    P_N
    v,
    [v + w_t]^k
  \rangle_H
  =
  -
  \sum_{ k = 0 }^3
  a_k
  \langle 
    A
    v,
    [v + w_t]^k
  \rangle_H .
\end{split}
\end{align}
This shows that
for all $ N \in \N $,
$ v \in P_N(H) $, $ w \in \C([0,T], H_1) $,
$ t \in [0,T] $ 
we have that
\begin{align}
\label{eq:coercivity_mix_1}
\begin{split}
 &\langle 
    v, 
    \textstyle\sum_{ k = 0 }^3
    a_k P_N [v + w_t]^k
  \rangle_{H_{\nicefrac{1}{2}}}
\\&=
  -\nu
  \sum_{ k = 0 }^3
  a_k
  \langle 
    v'',
    [v + w_t]^k
  \rangle_H
  =
  -\nu
  \sum_{ k = 0 }^3
  \sum_{ j = 0 }^k
  \binom{k}{j}
  a_k
  \langle 
    v'',    
    v^j (w_t)^{(k-j)}
  \rangle_H
\\&=
  -\nu
  \sum_{ k=1 }^3
  \binom{k}{k}
  a_k
  \langle 
    v'',
    v^k
  \rangle_H
  -
  \nu
  \sum_{ k = 0 }^3
  \sum_{ j = 0 }^{\max\{0, k-1\}}  
  \binom{k}{j}
  a_k
  \langle 
    v'',    
    v^j (w_t)^{(k-j)}
  \rangle_H .
\end{split}
\end{align}
Moreover, note that 
integration by parts
and the fact that
\begin{equation}
  \forall \, x,y \in \R,
  r \in (0,\infty)
  \colon 
  x \sqrt{2r}
  \cdot
  \tfrac{y}{\sqrt{2r}}
  \leq r x^2 + \tfrac{y^2}{4r}
\end{equation}
prove that for all 
$ N \in \N $,
$ v \in P_N(H) $
we have that
\begin{align}
\label{eq:coercivity_mix_2}
\begin{split}
 &-\nu
  \sum_{ k=1 }^3
  a_k
  \langle 
    v'',
    v^k
  \rangle_H
  =
  \nu
  \sum_{ k=1 }^3
  k
  a_k
  \langle 
    v',
    v^{(k-1)}
    v'
  \rangle_H
\\&\leq
  3\nu
  a_3
  \int_0^1
  [ \underline{v}'(x) ]^2 \,
  [ \underline{v}(x) ]^2 \, dx
  +
  2\nu
  |a_2|
  \int_0^1
  [ \underline{v}'(x) ]^2 \,
  | \underline{v}(x) | \, dx
  +
  \nu
  |a_1| \,
  \| v' \|_H^2
\\&\leq
  3\nu
  a_3
  \int_0^1
  [ \underline{v}'(x) ]^2 \,
  [ \underline{v}(x) ]^2 \, dx
  +
  \nu
  \int_0^1
  [ \underline{v}'(x) ]^2
  \left(
    3 |a_3|
    [ \underline{v}(x) ]^2
    +
    \frac{ 
      4 |a_2|^2 
    }{
      4\big( 3 |a_3| + \one_{\{0\}}^{\R}(a_3) \big)
    } 
  \right) dx
\\&\quad+
  \nu
  |a_1|
  \| v' \|_H^2
\\&=
  \nu  
  \left(
    |a_1|
    +
    \frac{ 
      |a_2|^2 
    }{
      3 |a_3| + \one_{\{0\}}^{\R}(a_3)
    }
  \right)
  \| v' \|_H^2
  =
  \left(
    |a_1|
    +
    \frac{ 
      |a_2|^2 
    }{
      3 |a_3| + \one_{\{0\}}^{\R}(a_3)
    }
  \right)
  \| v \|_{ H_{\nicefrac{1}{2}} }^2 .
\end{split}
\end{align}
Furthermore,
observe that the fact that
\begin{equation}
  \forall \, x,y \in \R
  \colon 
  xy \leq \epsilon x^2 + \tfrac{y^2}{4\epsilon}
\end{equation}
shows that for all $ N \in \N $,
$ v \in P_N(H) $, $ w \in \C([0,T], H_1) $,
$ t \in [0,T] $ we have that
\begin{align}
\begin{split}
 &-\nu
  \sum_{ k = 0 }^3
  \sum_{ j = 0 }^{\max\{0, k-1\}}   
  \binom{k}{j}
  a_k
  \langle 
    v'',    
    v^j (w_t)^{(k-j)}
  \rangle_H
\\&\leq
  \nu
  \sum_{ k = 0 }^3
  \sum_{ j = 0 }^{\max\{0, k-1\}}    
  \binom{k}{j}
  | a_k |
  \int_0^1
  | \underline{v}''(x) | \,
  | \underline{v}(x) |^j \,
  | \underline{w_t}(x) |^{ (k-j) } \, dx
\\&=
  \int_0^1
  \nu \,
  | \underline{v}''(x) | 
  \left[ 
    \sum_{ k = 0 }^3
    \sum_{ j = 0 }^{\max\{0, k-1\}}    
    \binom{k}{j}
    | a_k | \,
    | \underline{v}(x) |^j \,
    | \underline{w_t}(x) |^{ (k-j) }
  \right] dx
\\&\leq
  \epsilon
  \nu^2
  \| v'' \|_H^2
  +
  \frac{1}{4\epsilon}
  \int_0^1
  \left[ 
    \sum_{ k = 0 }^3
    \sum_{ j = 0 }^{\max\{0, k-1\}}     
    \binom{k}{j}
    | a_k | \,
    | \underline{v}(x) |^j \,
    | \underline{w_t}(x) |^{ (k-j) }
  \right]^{\!2} dx .
\end{split}
\end{align}
The fact that
\begin{equation}
  \forall  \,
  x_1, x_2, \ldots, x_7 \in \R
  \colon
  [x_1 + x_2 + \ldots + x_7]^2
  \leq
  7([x_1]^2 + [x_2]^2 + \ldots + [x_7]^2)
\end{equation}
hence assures that
for all $ N \in \N $,
$ v \in P_N(H) $, $ w \in \C([0,T], H_1) $,
$ t \in [0,T] $ 
we have that
\begin{align}
\begin{split}
 &-\nu
  \sum_{ k = 0 }^3
  \sum_{ j = 0 }^{\max\{0, k-1\}}    
  \binom{k}{j}
  a_k
  \langle 
    v'',    
    v^j (w_t)^{(k-j)}
  \rangle_H
\\&\leq
  \epsilon
  \nu^2
  \| v'' \|_{H}^2
  +
  \frac{7}{4\epsilon}
  \sum_{ k = 0 }^3
  \sum_{ j = 0 }^{\max\{0, k-1\}}   
  \left[ 
    \textstyle\binom{k}{j}
    |a_k|
  \right]^2
  \int_0^1
  | \underline{v}(x) |^{2j} \,
  | \underline{w_t}(x) |^{ 2(k-j) } \, dx
\\&=
  \epsilon
  \| v \|_{H_1}^2
  +
  \frac{7}{4\epsilon}
  \sum_{ k = 0 }^3 
  |a_k|^2
  \int_0^1
  | \underline{w_t}(x) |^{ 2k } \, dx
\\&\quad+
  \frac{7}{4\epsilon}
  \sum_{ k = 2 }^3
  \sum_{ j = 1 }^{k-1} 
  \left[ 
    \textstyle\binom{k}{j}
    |a_k|
  \right]^2
  \int_0^1
  | \underline{v}(x) |^{2j} \,
  | \underline{w_t}(x) |^{ 2(k-j) } \, dx .
\end{split}
\end{align}
This and the fact that
\begin{equation}
  \forall \, x, y \in \R
  \colon
  xy \leq \tfrac{x^2}{2} + \tfrac{y^2}{2}
\end{equation}
imply that for all $ N \in \N $,
$ v \in P_N(H) $, $ w \in \C([0,T], H_1) $,
$ t \in [0,T] $ we have that
\begin{align}
\begin{split}
 &-\nu
  \sum_{ k = 0 }^3
  \sum_{ j = 0 }^{\max\{0, k-1\}}
  \binom{k}{j}
  a_k
  \langle 
    v'',    
    v^j (w_t)^{(k-j)}
  \rangle_H
\\&\leq
  \epsilon
  \| v \|_{H_1}^2
  +
  \frac{7}{\epsilon}
  \left[ 
    \max_{ k \in \{0,1,2,3\} }
    | a_k |^2
  \right]
  \int_0^1
  \max\{ | \underline{w_t}(x) |^{ 6 },1 \} \, dx
\\&+
  \frac{7}{4\epsilon}
  \left[ 
    \max_{ k \in \{2,3\},j \in \{1,2\} }
    \textstyle\binom{k}{j}
    | a_k |
  \right]^{\!2}
  \int_0^1
  \Big[ 
    | \underline{v}(x) |^{2} \,
    | \underline{w_t}(x) |^{ 2 }
    +
    | \underline{v}(x) |^{2} \,
    | \underline{w_t}(x) |^{ 4 }
    +
    | \underline{v}(x) |^{4} \,
    | \underline{w_t}(x) |^{ 2 }
  \Big] \, dx
\\&\leq
  \epsilon
  \| v \|_{H_1}^2
  +
  \frac{7}{\epsilon}
  \left[ 
    \max_{ k \in \{0,1,2,3\} }
    | a_k |^2
  \right]
  \left(
    \| w_t \|_{ L^6( \lambda_{(0,1)} ; \R) }^6
    +
    1
  \right)
\\&+
  \frac{7}{4\epsilon}
  \left[ 
    \max_{ k \in \{2,3\},j \in \{1,2\} }
    \textstyle\binom{k}{j}
    | a_k |
  \right]^{\!2}
  \int_0^1
  \Big[ 
    | \underline{v}(x) |^{4} \,
    +
    | \underline{w_t}(x) |^{ 4 }
    +
    | \underline{v}(x) |^{4}
    | \underline{w_t}(x) |^{ 4 }
  \Big] \, dx .
\end{split}
\end{align}
H{\"o}lder's inequality therefore 
ensures that for all $ N \in \N $,
$ v \in P_N(H) $, $ w \in \C([0,T], H_1) $,
$ t \in [0,T] $ we have that
\begin{align}
\label{eq:coercivity_mix_3}
\begin{split}
 &-\nu
  \sum_{ k = 0 }^3
  \sum_{ j = 0 }^{\max\{0, k-1\}}
  \binom{k}{j}
  a_k
  \langle 
    v'',    
    v^j (w_t)^{(k-j)}
  \rangle_H
\\&\leq
  \epsilon
  \| v \|_{H_1}^2
  +
  \frac{7}{\epsilon}
  \left[ 
    \max_{ k \in \{0,1,2,3\} }
    | a_k |^2
  \right]
  \left(
    \| w_t \|_{ L^6( \lambda_{(0,1)} ; \R) }^6
    +
    1
  \right)
\\&\quad+
  \frac{63}{4\epsilon}
  \left[ 
    \max_{ k \in \{2,3\} }
    | a_k |^2
  \right]
  \int_0^1
  | \underline{w_t}(x) |^{ 4 } \, dx
\\&\quad+
  \frac{63}{4\epsilon}
  \left[ 
    \max_{ k \in \{2,3\} }
    | a_k |^2
  \right]
  \int_0^1
  | \underline{v}(x) |^{4}
  \left(
      1
      +
      | \underline{w_t}(x) |^{ 4 }
  \right) dx
\\&\leq
  \epsilon
  \| v \|_{H_1}^2
  +
  \frac{23}{\epsilon}
  \left[ 
    \max_{ k \in \{0,1,2,3\} }
    | a_k |^2
  \right]
  \left(
    \| w_t \|_{ L^6( \lambda_{(0,1)} ; \R) }^6
    +
    1
  \right)
\\&\quad+
  \frac{16}{\epsilon}
  \left[ 
    \max_{ k \in \{2,3\} }
    | a_k |^2
  \right]
  \| v \|_{ L^4( \lambda_{(0,1)} ; \R) }^4
  \left(
    \| w_t \|_{ L^{\infty}( \lambda_{(0,1)} ; \R) }^4
    +
    1
  \right) .
\end{split}
\end{align}
In the next step we combine~\eqref{eq:coercivity_mix_1}
with~\eqref{eq:coercivity_mix_2}
and~\eqref{eq:coercivity_mix_3}
to obtain that for all $ N \in \N $,
$ v \in P_N(H) $, $ w \in \C([0,T], H_1) $,
$ t \in [0,T] $ we have that
\begin{align}
\label{eq:coercivity_mix_4}
\begin{split}
 &\langle 
    v, 
    \textstyle\sum_{ k = 0 }^3
    a_k P_N [v + w_t]^k
  \rangle_{H_{\nicefrac{1}{2}}}
\\&\leq
  \epsilon
  \| v \|_{H_1}^2
  +
  \left(
    |a_1|
    +
    \frac{ 
      |a_2|^2 
    }{
      3 |a_3| + \one_{\{0\}}^{\R}(a_3)
    }
  \right)
  \| v \|_{ H_{\nicefrac{1}{2}} }^2
  +
  \left[ 
    \max_{ k \in \{0,1,2,3\} }
    \tfrac{ 5 | a_k | }{ \sqrt{\epsilon} }
  \right]^{\!2}
  \left[
    \| w_t \|_{ L^6( \lambda_{(0,1)} ; \R) }^6
    +
    1
  \right]
\\&\quad+
  \frac{c |a_3|}{2}
  \left[
    \| w_t \|_{ L^{\infty}( \lambda_{(0,1)} ; \R) }^4
    +
    1
  \right]
  \| v \|_{ L^4( \lambda_{(0,1)} ; \R) }^4 .
\end{split}
\end{align}
In addition, note that for all 
$ N \in \N $,
$ v \in P_N(H) $, $ w \in \C([0,T], H_1) $,
$ t \in [0,T] $ 
we have that
\begin{align}
\begin{split}
 &\langle 
    v, 
    \textstyle\sum_{ k = 0 }^3
    a_k [v + w_t]^k
  \rangle_H
\\&=
  \sum_{ k = 0 }^3
  a_k
  \langle 
    v, 
    [v + w_t]^k
  \rangle_H
  =
  \sum_{ k = 0 }^3
  \sum_{ j = 0 }^k
  \binom{k}{j}
  a_k
  \langle 
    v, 
    v^j (w_t)^{(k-j)}
  \rangle_H
\\&=
  a_3
  \| v \|_{ L^4( \lambda_{(0,1)} ; \R ) }^4
  +
  \sum_{ k = 0 }^3
  \sum_{ j = 0 }^{ \min\{k,2\} }
  \binom{k}{j}
  a_k
  \langle 
    v, 
    v^j (w_t)^{(k-j)}
  \rangle_H
\\&\leq
  a_3
  \| v \|_{ L^4( \lambda_{(0,1)} ; \R ) }^4
  +
  \sum_{ k = 0 }^3
  \sum_{ j = 0 }^{ \min\{k,2\} }
  \binom{k}{j}
  | a_k |
  \int_0^1
  | \underline{v}(x) |^{(j+1)} \,
  | \underline{w_t}(x) |^{(k-j)} \, dx .
\end{split}
\end{align}
Young's inequality hence demonstrates that 
for all 
$ N \in \N $,
$ v \in P_N(H) $, $ w \in \C([0,T], H_1) $,
$ t \in [0,T] $, $ r \in [0,\infty) $ 
we have that
\begin{align}
\begin{split}
 &r
  \langle 
    v, 
    \textstyle\sum_{ k = 0 }^3
    a_k [v + w_t]^k
  \rangle_H
\\&\leq
  r
  a_3
  \| v \|_{ L^4( \lambda_{(0,1)} ; \R ) }^4
  +
  r
  \sum_{ k = 0 }^3
  \sum_{ j = 0 }^{ \min\{k,2\} }
  \binom{k}{j}
  | a_k |
  \int_0^1
  | \eta(k,j) | \,
  | \underline{v}(x) |^{(j+1)} \,
  \frac{ 
    | \underline{w_t}(x) |^{(k-j)} 
  }{ | \eta(k,j) | } \, dx
\\&\leq
  r
  a_3
  \| v \|_{ L^4( \lambda_{(0,1)} ; \R ) }^4
\\&\quad+
  r
  \sum_{ k = 0 }^3
  \sum_{ j = 0 }^{ \min\{k,2\} }
  \binom{k}{j}
  | a_k |
  \int_0^1
  \left[ 
    \tfrac{ (j+1) }{ 4 }
    | \eta(k,j) |^{\frac{4}{(j+1)}}
    | \underline{v}(x) |^{4}
    +
    \frac{ 
      (3-j) | \underline{w_t}(x) |^{\frac{ 4(k-j) }{ (3-j) } } 
    }{ 
      4 | \eta(k,j) |^{\frac{4}{(3-j)} } 
    } 
  \right] dx .
\end{split}
\end{align}
Therefore, we obtain
that for all 
$ N \in \N $,
$ v \in P_N(H) $, $ w \in \C([0,T], H_1) $,
$ t \in [0,T] $, $ r \in [0,\infty) $ 
we have that
\begin{align}
\label{eq:coercivity_mix_5}
\begin{split}
 &\one_{ (-\infty,0) }^{ \R }(a_3) \,
  r
  \langle 
    v, 
    \textstyle\sum_{ k = 0 }^3
    a_k [v + w_t]^k
  \rangle_H
\\&\leq
  r
  \| v \|_{ L^4( \lambda_{(0,1)} ; \R ) }^4
  \one_{ (-\infty,0) }^{ \R }(a_3) \!
  \left[ 
    a_3
    +
    \frac{1}{4}
    \sum_{ k = 0 }^3
    \sum_{ j = 0 }^{ \min\{k,2\} }
    \binom{k}{j}
    (j+1) \, | a_k | \,
    | \eta(k,j) |^{\frac{4}{(j+1)}}
  \right]
\\&\quad+
  r
  \sum_{ k = 0 }^3
  \sum_{ j = 0 }^{ \min\{k,2\} }
  \binom{k}{j}
  \frac{ 
    (3-j) | a_k | 
  }{ 
    4 | \eta(k,j) |^{\frac{4}{(3-j)} } 
  }
  \int_0^1
  \max\{ 
    1,
    | \underline{w_t}(x) |^4
  \} \, dx
\\&\leq
  \frac{ r a_3 }{ 2 }
  \| v \|_{ L^4( \lambda_{(0,1)} ; \R ) }^4
  +
  r
  \sum_{ k = 0 }^3
  \sum_{ j = 0 }^{ \min\{k,2\} }
  \binom{k}{j}
  \frac{ 
    (3-j) | a_k | 
  }{ 
    4 | \eta(k,j) |^{\frac{4}{(3-j)} } 
  }
  \left[ 
    \| w_t \|_{ L^4( \lambda_{(0,1)} ; \R ) }^4
    +
    1
  \right] . 
\end{split}
\end{align}
Moreover, note that
the fact that
\begin{equation}
  \forall \,
  x, y \in [0,\infty)
  \colon
  xy
  \leq 
  \tfrac{x^2}{2}
  +
  \tfrac{y^2}{2}
\end{equation}
ensures that for all 
$ N \in \N $,
$ v \in P_N(H) $, $ w \in \C([0,T], H_1) $,
$ t \in [0,T] $, $ r \in [0,\infty) $ 
we have that
\begin{align}
\begin{split}
 &r
  \langle 
    v, 
    \textstyle\sum_{ k = 0 }^1
    a_k [v + w_t]^k
  \rangle_H
\\&\leq
  r
  \int_0^1
  | a_0 | \,
  | \underline{v}(x) |
  +
  | a_1 | \,
  | \underline{v}(x) |^2
  +
  | a_1 | \,
  | \underline{v}(x) | \,
  | \underline{w_t}(x) | \, dx
\\&\leq
  r
  \int_0^1
  | a_0 | 
  \left( 1 + | \underline{v}(x) |^2 \right)
  +
  | a_1 | \,
  | \underline{v}(x) |^2
  +
  \frac{| a_1 |}{2}
  \left(
    | \underline{v}(x) |^2
    +
    | \underline{w_t}(x) |^2
  \right) dx
\\&=
  r
  \left(
    | a_0 |
    +
    \tfrac{3 | a_1 |}{2} 
  \right)
  \| v \|_H^2
  +
  r
  \left(
    | a_0 |
    +
    \tfrac{| a_1 |}{2}
    \| w_t \|_H^2
  \right)
\\&\leq
  r
  \left(
    | a_0 |
    +
    \tfrac{3 | a_1 |}{2} 
  \right)
  \| v \|_H^2
  +
  r
  \left(
    | a_0 |
    +
    \tfrac{| a_1 |}{2}
  \right)
  \left[ 
    \| w_t \|_{ L^4( \lambda_{(0,1)} ; \R ) }^4
    +
    1
  \right] .
\end{split}
\end{align}
Hence, we obtain that
for all 
$ N \in \N $,
$ v \in P_N(H) $, $ w \in \C([0,T], H_1) $,
$ t \in [0,T] $,
$ 
  r 
  \in 
  [ 
    \| w_t \|_{ L^{\infty}( \lambda_{(0,1)} ; \R ) }^4
    +
    1,
    \infty
  )
$
we have that
\begin{align}
\begin{split}
 &\langle 
    v, 
    {\textstyle\sum_{ k = 0 }^3}
    a_k P_N [v + w_t]^k
  \rangle_{H_{\nicefrac{1}{2}}}
  +
  c
  r
  \langle 
    v, 
    \textstyle\sum_{ k = 0 }^3
    a_k [v + w_t]^k
  \rangle_H
\\&=
  \langle 
    v, 
    {\textstyle\sum_{ k = 0 }^3}
    a_k P_N [v + w_t]^k
  \rangle_{H_{\nicefrac{1}{2}}}
  +
  c
  r
  \langle 
    v, 
    \textstyle\sum_{ k = 0 }^3
    a_k [v + w_t]^k
  \rangle_H
  \one_{ (-\infty,0) }^{ \R }(a_3)
\\&\quad+
  c
  r
  \langle 
    v, 
    \textstyle\sum_{ k = 0 }^3
    a_k [v + w_t]^k
  \rangle_H
  \one_{ \{0\} }^{ \R }(a_3)
\\&=
  \langle 
    v, 
    {\textstyle\sum_{ k = 0 }^3}
    a_k P_N [v + w_t]^k
  \rangle_{H_{\nicefrac{1}{2}}}
  +
  c
  r
  \langle 
    v, 
    \textstyle\sum_{ k = 0 }^3
    a_k [v + w_t]^k
  \rangle_H
  \one_{ (-\infty,0) }^{ \R }(a_3)
\\&\quad+
  c
  r
  \langle 
    v, 
    \textstyle\sum_{ k = 0 }^1
    a_k [v + w_t]^k
  \rangle_H
  \one_{ \{0\} }^{ \R }(a_3)
\\&\leq
  \langle 
    v, 
    {\textstyle\sum_{ k = 0 }^3}
    a_k P_N [v + w_t]^k
  \rangle_{H_{\nicefrac{1}{2}}}
  +
  c
  r
  \langle 
    v, 
    \textstyle\sum_{ k = 0 }^3
    a_k [v + w_t]^k
  \rangle_H
  \one_{ (-\infty,0) }^{ \R }(a_3)
\\&\quad+ 
  cr
  \left(
    | a_0 |
    +
    \tfrac{3 | a_1 |}{2} 
  \right)
  \| v \|_H^2
  +
  cr
  \left(
    | a_0 |
    +
    \tfrac{| a_1 |}{2}
  \right)
  \left[ 
    \| w_t \|_{ L^4( \lambda_{(0,1)} ; \R ) }^4
    +
    1
  \right] .
\end{split}
\end{align}
Combining this with
\eqref{eq:coercivity_mix_4}
and~\eqref{eq:coercivity_mix_5}
assures that for all 
$ N \in \N $,
$ v \in P_N(H) $, $ w \in \C([0,T], H_1) $,
$ t \in [0,T] $,
$ 
  r 
  \in 
  [ 
    \| w_t \|_{ L^{\infty}( \lambda_{(0,1)} ; \R ) }^4
    +
    1,
    \infty
  )
$
we have that
\begin{align}
\begin{split}
 &\langle 
    v, 
    {\textstyle\sum_{ k = 0 }^3}
    a_k P_N [v + w_t]^k
  \rangle_{H_{\nicefrac{1}{2}}}
  +
  c
  r
  \langle 
    v, 
    \textstyle\sum_{ k = 0 }^3
    a_k [v + w_t]^k
  \rangle_H
\\&\leq
  \epsilon
  \| v \|_{H_1}^2
  +
  \left(
    |a_1|
    +
    \frac{ 
      |a_2|^2 
    }{
      3 |a_3| + \one_{\{0\}}^{\R}(a_3)
    }
  \right)
  \| v \|_{ H_{\nicefrac{1}{2}} }^2
  +
  c r
  \left(
    | a_0 |
    +
    \tfrac{3 | a_1 |}{2} 
  \right)
  \| v \|_H^2
\\&\quad+ 
  \frac{cr}{2}
  \| v \|_{ L^4( \lambda_{(0,1)} ; \R) }^4
  \left( 
    a_3
    +
    |a_3|
  \right)
  +
  \left[ 
    \max_{ k \in \{0,1,2,3\} }
    \tfrac{ 5 | a_k | }{ \sqrt{\epsilon} }
  \right]^{\!2}
  \left[
    \| w_t \|_{ L^6( \lambda_{(0,1)} ; \R) }^6
    +
    1
  \right]
\\&\quad+
  c
  r
  \left[ 
    | a_0 |
    +
    \frac{| a_1 |}{2}
    +
    \sum_{ k = 0 }^3
    \sum_{ j = 0 }^{ \min\{k,2\} }
    \binom{k}{j}
    \frac{ 
      (3-j) | a_k | 
    }{ 
      4 | \eta(k,j) |^{\frac{4}{(3-j)} } 
    }
  \right]
  \left[ 
    \| w_t \|_{ L^4( \lambda_{(0,1)} ; \R ) }^4
    +
    1
  \right] .
\end{split}
\end{align}
H{\"o}lder's inequality
and the fact that
\begin{equation}
  a_3 + |a_3| = a_3 - a_3 = 0 
\end{equation}  
therefore prove that
for all 
$ N \in \N $,
$ v \in P_N(H) $, $ w \in \C([0,T], H_1) $,
$ t \in [0,T] $,
$ 
  r 
  \in 
  [ 
    \| w_t \|_{ L^{\infty}( \lambda_{(0,1)} ; \R ) }^4
    +
    1,
    \infty
  )
$
we have that
\begin{align}
\begin{split}
 &\langle 
    v, 
    {\textstyle\sum_{ k = 0 }^3}
    a_k P_N [v + w_t]^k
  \rangle_{H_{\nicefrac{1}{2}}}
  +
  c
  r
  \langle 
    v, 
    \textstyle\sum_{ k = 0 }^3
    a_k [v + w_t]^k
  \rangle_H
\\&\leq
  \epsilon
  \| v \|_{H_1}^2
  +
  \left(
    |a_1|
    +
    \frac{ 
      |a_2|^2 
    }{
      3 |a_3| + \one_{\{0\}}^{\R}(a_3)
    }
  \right)
  \| v \|_{ H_{\nicefrac{1}{2}} }^2
  +
  c r
  \left(
    | a_0 |
    +
    \tfrac{3 | a_1 |}{2} 
  \right)
  \| v \|_H^2
\\&+ 
  \left[ 
    \max_{ k \in \{0,1,2,3\} }
    \tfrac{ 5 | a_k | }{ \sqrt{\epsilon} }
  \right]^{\!2} 
  \left[
    \| w_t \|_{ L^{\infty}( \lambda_{(0,1)} ; \R) }^6
    +
    1
  \right]
\\ &
  +
  c
  r^2 \!
  \left[ 
    | a_0 |
    +
    \frac{| a_1 |}{2}
    +
    \sum_{ k = 0 }^3
    \sum_{ j = 0 }^{ \min\{k,2\} } \!
    \binom{k}{j} 
    \frac{ 
      (3-j) | a_k | 
    }{ 
      4 | \eta(k,j) |^{\frac{4}{(3-j)} } 
    }
  \right] .
\end{split}
\end{align}
The fact that
\begin{equation}
  \forall \, x, y \in (0,\infty)
  \colon
  (x+y)^2
  \leq
  2(x^2+y^2)
\end{equation}
hence
implies that
for all 
$ N \in \N $,
$ v \in P_N(H) $, $ w \in \C([0,T], H_1) $,
$ t \in [0,T] $
we have that
\begin{align}
\begin{split}
 &\langle 
    v, 
    {\textstyle\sum_{ k = 0 }^3}
    a_k P_N [v + w_t]^k
  \rangle_{H_{\nicefrac{1}{2}}}
  +
  c
  \left[
    \sup_{ s \in [0,T] }
    \| w_s \|_{ L^{\infty}( \lambda_{(0,1)} ; \R) }^4
    +
    1
  \right]
  \langle 
    v, 
    \textstyle\sum_{ k = 0 }^3
    a_k [v + w_t]^k
  \rangle_H
\\&\leq
  \epsilon
  \| v \|_{H_1}^2
  +
  \left(
    |a_1|
    +
    \frac{ 
      |a_2|^2 
    }{
      3 |a_3| + \one_{\{0\}}^{\R}(a_3)
    }
  \right)
  \| v \|_{ H_{\nicefrac{1}{2}} }^2
\\&+
  c
  \left[
    \sup_{ s \in [0,T] }
    \| w_s \|_{ L^{\infty}( \lambda_{(0,1)} ; \R) }^4
    +
    1
  \right]
  \left(
    | a_0 |
    +
    \tfrac{3 | a_1 |}{2} 
  \right)
  \| v \|_H^2
\\ &
  +
  \left[ 
    \max_{ k \in \{0,1,2,3\} }
    \tfrac{ 5 | a_k | }{ \sqrt{\epsilon} }
  \right]^{\!2}
  \left[
    \| w_t \|_{ L^{\infty}( \lambda_{(0,1)} ; \R) }^8
    +
    2
  \right]
\\&+
  c
  \left[ 
    2 | a_0 |
    +
    | a_1 |
    +
    \sum_{ k = 0 }^3
    \sum_{ j = 0 }^{ \min\{k,2\} }
    \binom{k}{j}
    \frac{ 
      (3-j) | a_k | 
    }{ 
      2 | \eta(k,j) |^{\frac{4}{(3-j)} } 
    }
  \right]
  \left[
    \sup_{ s \in [0,T] }
    \| w_s \|_{ L^{\infty}( \lambda_{(0,1)} ; \R) }^8
    +
    1
  \right] .
\end{split}
\end{align}
Hence, we obtain that
for all $ N \in \N $,
$ v \in P_N(H) $, $ w \in \C([0,T], H_1) $,
$ t \in [0,T] $ we have that
\begin{align}
\begin{split}
 &\langle 
    v, 
    {\textstyle\sum_{ k = 0 }^3}
    a_k P_N [v + w_t]^k
  \rangle_{H_{\nicefrac{1}{2}}}
  +
  c
  \left[
    \sup_{ s \in [0,T] }
    \| w_s \|_{ L^{\infty}( \lambda_{(0,1)} ; \R) }^4
    +
    1
  \right]
  \langle 
    v, 
    \textstyle\sum_{ k = 0 }^3
    a_k [v + w_t]^k
  \rangle_H
\\&\leq
  \epsilon
  \| v \|_{H_1}^2
  +
  \left(
    | a_1 |
    +
    \frac{ 
      |a_2|^2 
    }{
      3 |a_3| + \one_{\{0\}}^{\R}(a_3)
    }
  \right)
  \| v \|_{ H_{\nicefrac{1}{2}} }^2
\\&\quad+
  c
  \left[
    \sup_{ s \in [0,T] }
    \| w_s \|_{ L^{\infty}( \lambda_{(0,1)} ; \R) }^4
    +
    1
  \right]
  \left(
    | a_0 |
    +
    \frac{3 |a_1|}{2}
  \right)
  \| v \|_H^2
\\&\quad+
  \left[ 
    \left[ 
      \max_{ k \in \{0,1,2,3\} }
      \tfrac{ 8 | a_k | }{ \sqrt{\epsilon} }
    \right]^{\!2} 
    +
    2c | a_0 |
    +
    c| a_1 |
    +
    c
    \sum_{ k = 0 }^3 \!
    \sum_{ j = 0 }^{ \min\{k,2\} } 
    \binom{k}{j}
    \frac{ 
      (3-j) | a_k | 
    }{ 
      2 | \eta(k,j) |^{\frac{4}{(3-j)} } 
    }
  \right]
\\&\quad\cdot
  \left[
    \sup_{ s \in [0,T] }
    \| w_s \|_{ L^{\infty}( \lambda_{(0,1)} ; \R) }^8
    +
    1
  \right] .
\end{split}
\end{align}
The proof of Lemma~\ref{lem:coercivity_mix} is thus completed.
\end{proof}

The next elementary lemma, Lemma~\ref{lem:F_lipschitz} below,
establishes a local Lipschitz estimate for the nonlinearity
$ F $ in Section~\ref{sec:examples_setting}.
Lemma~\ref{lem:F_lipschitz} is a slightly
modified version of Lemma~6.8 in~\cite{BeckerJentzen2016}.

\begin{lemma}
\label{lem:F_lipschitz}
Assume the setting in Section~\ref{sec:examples_setting}
and let $ q \in [6, \infty) $,
$ v, w \in L^q( \lambda_{ (0,1) }; \R ) $.
Then 
\begin{align}
\begin{split}
 &\|
    F(v) - F(w)
  \|_{ H }^2
  \leq
  36
  \left[ 
    \max_{j\in\{1,2,3\}}
    \left| a_j \right|
  \right]^2
  \|
    v-w
  \|_{ L^q( \lambda_{ (0,1) }; \R ) }^2
  \left(
    1
    +
    \| v \|_{ L^q( \lambda_{ (0,1) }; \R ) }^{4}
    +
    \| w \|_{ L^q( \lambda_{ (0,1) }; \R ) }^{4}
  \right).
\end{split}
\end{align}
\end{lemma}
\begin{proof}[Proof of Lemma~\ref{lem:F_lipschitz}]
Observe that the 
fundamental theorem of
calculus and Jensen's inequality ensure for all 
$ k \in \N $, $ x,y \in \R $
that
\begin{align}
\begin{split}
  |
    x^k - y^k 
  | 
 &=
  \left|
    \int_0^1
    k
    \left( y + r\left( x-y \right) \right)^{(k-1)}
    \left( x-y \right) dr
  \right|
\\ &
  \leq
  k
  \left| x-y \right|
  \int_0^1
  \left|
    rx
    +
    \left(1-r\right)y
  \right|^{(k-1)} dr
\\&\leq
  k
  \left| x-y \right|
  \int_0^1
  \left(
    r \,
    | x |^{(k-1)}
    +
    \left(1-r\right)
    | y |^{(k-1)}
  \right) dr .
\end{split}
\end{align}
Combining this and H{\"o}lder's inequality
implies that
\begin{align}
\begin{split}
 &\left\|
    F(v) - F(w)
  \right\|_H
  =
  \left\|
    \textstyle\sum_{ k=0 }^3
    a_k
    \left( v^k - w^k \right)
  \right\|_H
  \leq 
  \sum_{ k=1 }^{ 3 }
  \left| a_k \right|
  \| v^k - w^k \|_H
\\&\leq
  \sum_{ k=1 }^3
  k \left| a_k \right|
  \int_0^1
  \left\|
    | v-w |
    \left(
      r \, | v |^{(k-1)}
      +
      \left(1-r\right)
      | w |^{(k-1)}
    \right)
  \right\|_H dr
\\&\leq
  \sum_{ k=1 }^3
  k \left| a_k \right|
  \int_0^1
  \|
    v-w
  \|_{ L^q( \lambda_{ (0,1) }; \R ) }
  \left\|
    r \, | v |^{(k-1)}
    +
    \left(1-r\right)
    | w |^{(k-1)}
  \right\|_{ L^{\nicefrac{2q}{(q-2)}}(\lambda_{ (0,1) }; \R ) } dr
\\&\leq
  \|
    v-w
  \|_{ L^q( \lambda_{ (0,1) }; \R ) }
  \Bigg[
    \left| a_1 \right|
\\&\quad+
    \sum_{ k=2 }^3
    k \left| a_k \right|
    \int_0^1
    \left(
      r \,
      \| v \|_{ L^{\nicefrac{2q(k-1)}{(q-2)}}(\lambda_{ (0,1) }; \R ) }^{(k-1)}
      +
      \left(1-r\right)
      \| w \|_{ L^{\nicefrac{2q(k-1)}{(q-2)}}(\lambda_{ (0,1) }; \R ) }^{(k-1)}
    \right) dr
  \Bigg] .
\end{split}
\end{align}
Again H{\"o}lder's inequality therefore demonstrates that
\begin{align}
\begin{split}
 &\|
    F(v) - F(w)
  \|_H
\\&\leq
  \|
    v-w
  \|_{ L^q( \lambda_{ (0,1) }; \R ) }
  \Bigg[
    \left| a_1 \right|
    +
    \frac{1}{2}
    \sum_{ k=2 }^3
    k \left| a_k \right|
    \left(
      \| v \|_{ L^{\nicefrac{2q(k-1)}{(q-2)}}(\lambda_{ (0,1) }; \R ) }^{(k-1)}
      +
      \| w \|_{ L^{\nicefrac{2q(k-1)}{(q-2)}}(\lambda_{ (0,1) }; \R ) }^{(k-1)}
    \right)
  \Bigg]
\\&\leq
  \|
    v-w
  \|_{ L^q( \lambda_{ (0,1) }; \R ) }
  \Bigg[
    \left| a_1 \right|
    +
    \frac{1}{2}
    \sum_{ k=2 }^3
    k \left| a_k \right|
    \left(
      \| v \|_{ L^q( \lambda_{ (0,1) }; \R ) }^{(k-1)}
      +
      \| w \|_{ L^q( \lambda_{ (0,1) }; \R ) }^{(k-1)}
    \right)
  \Bigg]
\\&\leq
  \frac{1}{2}
  \|
    v-w
  \|_{ L^q( \lambda_{ (0,1) }; \R ) }
  \left[
    \max_{ j \in \{1,2,3\} }
    \left| a_j \right|
  \right]
  \sum_{ k=1 }^3
  k
  \left(
    \| v \|_{ L^q( \lambda_{ (0,1) }; \R ) }^{(k-1)}
    +
    \| w \|_{ L^q( \lambda_{ (0,1) }; \R ) }^{(k-1)}
  \right)
\\&\leq
  3
  \|
    v-w
  \|_{ L^q( \lambda_{ (0,1) }; \R ) }
  \left[
    \max_{ j \in \{1,2,3\} }
    \left| a_j \right|
  \right]
  \left(
    \max\!\left\{ 
      1, \| v \|_{ L^q( \lambda_{ (0,1) }; \R ) }^{2}
    \right\}
    +
    \max\!\left\{ 
      1, \| w \|_{ L^q( \lambda_{ (0,1) }; \R ) }^{2}
    \right\}
  \right)
\\&\leq
  6
  \|
    v-w
  \|_{ L^q( \lambda_{ (0,1) }; \R ) }
  \left[
    \max_{ j \in \{1,2,3\} }
    \left| a_j \right|
  \right]
  \max\!\left\{ 
    1, 
    \| v \|_{ L^q( \lambda_{ (0,1) }; \R ) }^{2}, 
    \| w \|_{ L^q( \lambda_{ (0,1) }; \R ) }^{2}
  \right\} .
\end{split}
\end{align}
This completes the proof of Lemma~\ref{lem:F_lipschitz}.
\end{proof}

\subsection{Properties of 
linear stochastic heat equations}
In this subsection we present 
a few elementary regularity
and approximation results for 
linear stochastic heat equations;
see Lemmas~\ref{lem:O_existence}--\ref{lem:O_prop_3}
and Corollary~\ref{cor:O_properties} below.
Similar regularity
and approximation results for 
linear stochastic heat equations
can, e.g., be found in 
Hutzenthaler et al.~\cite[Lemma~5.6, Corollary~5.8, and Lemma~5.9]{HutzenthalerJentzenSalimova2016}.
The next lemma, Lemma~\ref{lem:normal_variables} below,
presents a well-known fact on centered and normally
distributed random variables. Lemma~\ref{lem:normal_variables}
is used in the proofs of 
Lemma~\ref{lem:O_prop_1}, Lemma~\ref{lem:O_sobo},
and Lemma~\ref{lem:O_prop_3} below.

\begin{lemma}
\label{lem:normal_variables}
Let $ p \in [0,\infty) $,
let $ (\Omega, \F, \P) $ be a probability space,
let $ Y \colon \Omega \rightarrow \R $ be a standard
normal random variable, and let
$ Z \colon \Omega \rightarrow \R $ be a centered and 
normally distributed
random variable. Then
\begin{align}
  \E\big[ 
    | Z |^p
  \big]
  =
  \E\big[ 
    | Y |^p
  \big]
  \big(
    \E\big[
      | Z |^2
    \big]
  \big)^{\nicefrac{p}{2}} .
\end{align}
\end{lemma}
\begin{proof}[Proof of Lemma~\ref{lem:normal_variables}]
Throughout this proof assume w.l.o.g.\
that $ \E[ |Z|^2 ] > 0 $ 
(otherwise the proof is clear).
Note that
\begin{align}
\begin{split}
 &\E\big[ 
    | Z |^p
  \big]
\\&=
  \E\Bigg[ 
    \left|
      \frac{ 
        Z
      }{
        \big( 
          \E\big[
            | Z |^2
          \big]
        \big)^{\!\nicefrac{1}{2}}
      } \,
      \big( 
        \E\big[
          | Z |^2
        \big]
      \big)^{\!\nicefrac{1}{2}}
    \right|^p
  \Bigg] 
  =
  \E\Bigg[ 
    \left|
      \frac{ 
        Z
      }{
        \big( 
          \E\big[
            | Z |^2
          \big]
        \big)^{\!\nicefrac{1}{2}}
      }
    \right|^p
  \Bigg]
  \big( 
    \E\big[
      | Z |^2
    \big]
  \big)^{\nicefrac{p}{2}}
\\&=
  \E\big[ 
    | Y |^p
  \big]
  \big(
    \E\big[
      | Z |^2
    \big]
  \big)^{\nicefrac{p}{2}} .
\end{split}
\end{align}
This completes the proof of Lemma~\ref{lem:normal_variables}.
\end{proof}

\begin{lemma}
\label{lem:O_existence}
Assume the setting in Section~\ref{sec:examples_setting},
let $ \gamma \in [0, \nicefrac{1}{4}) $, $ \beta \in (\nicefrac{1}{4},\nicefrac{1}{2}-\gamma) $,
$ B \in HS( H, H_{-\beta} ) $,
and let 
$ 
  \varphi 
  \colon 
  [0,T] \rightarrow [0,T]
$
be a $ \B([0,T]) / \B([0,T]) $-measurable 
function which 
satisfies for all $ t \in [0,T] $
that $ \varphi(t) \leq t $. Then there exists an
up to indistinguishability unique
stochastic process $ O \colon [0,T] \times \Omega \rightarrow H_{\gamma} $
with continuous sample paths which satisfies
for all $ t \in [0,T] $ that
\begin{equation}
  [ O_t ]_{ \P, \B(H) } 
  =
  \int_0^t
  e^{(t-\varphi(s))A}
  B \, dW_s .
\end{equation}
\end{lemma}
\begin{proof}[Proof of Lemma~\ref{lem:O_existence}]
Throughout this proof let 
$ \varepsilon \in (0, \nicefrac{1}{2}-\gamma-\beta) $,
$ p \in (\nicefrac{1}{\varepsilon}, \infty) $ be real numbers.
Note for all $ s \in [0,T) $, $ t \in (s, T] $ that 
\begin{align}
\begin{split}
 &\int_0^t
  \left\|
    e^{(t-\varphi(u))A} B 
    -
    \one_{ (-\infty,s) }^{\R}(u) \,
    e^{(s-\varphi(u))A} B
  \right\|_{ HS( H, H_{\gamma} ) }^2 du  
\\&=
  \int_s^t
  \left\|
    e^{(t-\varphi(u))A} B
  \right\|_{ HS( H, H_{\gamma} ) }^2 du
  +
  \int_0^s
  \left\|
    \left(
      e^{(t-\varphi(u))A}
      -
      e^{(s-\varphi(u))A}
    \right) \! B
  \right\|_{ HS(H,H_{\gamma} ) }^2 du
\\&\leq
  \left\|
    B
  \right\|_{ HS(H, H_{-\beta}) }^2
  \bigg[
    \int_s^t
    \left\|
      e^{(t-\varphi(u))A}
    \right\|_{ L( H_{-\beta}, H_{\gamma} ) }^2 du
\\&+
    \int_0^s
    \left\|
      e^{(s-\varphi(u))A}
      \left(
        e^{(t-s)A}
        -
        \Id_H
      \right)
    \right\|_{ L( H_{-\beta}, H_{\gamma} ) }^2 du
  \bigg]
\\&\leq
  \left\|
    B
  \right\|_{ HS(H, H_{-\beta}) }^2
  \bigg[
    \int_s^t
    \left\|
      (-A)^{(\gamma+\beta)}
      e^{(t-u)A}
    \right\|_{ L(H) }^2
    \left\|
      e^{(u-\varphi(u))A}
    \right\|_{ L(H) }^2 du
\\&+
    \int_0^s
    \left\|
      (-A)^{(\gamma+\beta+\varepsilon)}
      e^{ (s-u)A }
    \right\|_{ L(H) }^2
    \left\|
      e^{ (u-\varphi(u))A }
    \right\|_{ L(H) }^2
    \left\|
      (-A)^{-\varepsilon}
      (
	e^{(t-s)A}
	-
	\Id_H
      )
    \right\|_{ L(H) }^2 du
  \bigg] .
\end{split}
\end{align}
The fact that 
$
  \forall \,
  t \in [0,\infty)
$,
$ 
  r \in [0,1]
  \colon
  \|
    (-tA)^r
    e^{tA}
  \|_{ L(H) }
  \leq
  1
$
and the fact that
$
  \forall \,
  t \in (0,\infty)
$,
$ 
  r \in [0,1]
  \colon
  \|
    (-tA)^{-r}
    (
    e^{tA} - \Id_H )
  \|_{ L(H) }
  \leq
  1
$
hence prove for all $ s \in [0,T) $, $ t \in (s,T] $ that
\begin{align}
\label{eq:O_existence_1}
\begin{split}
 &\int_0^t
  \left\|
    e^{(t-\varphi(u))A} B 
    -
    \one_{ (-\infty,s) }^{\R}(u) \,
    e^{(s-\varphi(u))A} B
  \right\|_{ HS( H, H_{\gamma} ) }^2 du  
\\&\leq
  \left\|
    B
  \right\|_{ HS(H, H_{-\beta}) }^2
  \bigg[
    \int_s^t
    \left(t-u\right)^{-2(\gamma+\beta)} du
    +
    \int_0^s
    \left( s-u \right)^{-2(\gamma+\beta+\varepsilon)}
    \left( t-s \right)^{2\varepsilon} du
  \bigg]
\\&=
  \left\|
    B
  \right\|_{ HS(H, H_{-\beta}) }^2  
  \bigg[
    \frac{(t-s)^{(1-2\gamma-2\beta)} }{ (1-2\gamma-2\beta) }
    +
    \frac{ s^{(1-2\gamma-2\beta-2\varepsilon)} (t-s)^{2\varepsilon}}{ (1-2\gamma-2\beta-2\varepsilon) }
  \bigg]
\\&\leq
  \frac{ 
    2
    \left(t-s\right)^{2\varepsilon}
    \max\{ 1, T \}
  }{
    (1-2\gamma-2\beta-2\varepsilon) 
  }
  \left\|
    B
  \right\|_{ HS(H, H_{-\beta}) }^2  
  <
  \infty .
\end{split}
\end{align}
This implies that for all $ t \in [0,T] $
we have that
$
  \int_0^t
  \|
    e^{(t-\varphi(u))A} B 
  \|_{ HS( H, H_{\gamma} ) }^2 \, du
  <
  \infty 
$. 
Hence, there exits
a stochastic process
$ \tilde{O} \colon [0,T] \times \Omega \rightarrow H_{\gamma} $
which satisfies for all $ t \in [0,T] $ that
\begin{align}
  [
    \tilde{O}_t
  ]_{ \P, \B(H) }
  =
  \int_0^t
  e^{(t-\varphi(s))A} B \, dW_s .
\end{align}
Moreover, note that
the Burkholder-Davis-Gundy type inequality
in Lemma~7.7 in Da Prato \& Zabcyk~\cite{dz92}
and~\eqref{eq:O_existence_1} demonstrate that 
for all $ s \in [0,T) $, $ t \in (s,T] $ 
we have that
\begin{align}
\begin{split}
 &\|
    \tilde{O}_t
    -
    \tilde{O}_s
  \|_{ \L^p(\P; H_{\gamma} ) }^2
\\&=
  \left\|
    \int_0^t
    \Big[ 
      e^{(t-\varphi(u))A} 
      -
      \one_{ (-\infty,s) }^{\R}(u) \,
      e^{(s-\varphi(u))A} 
    \Big] B \, dW_u
  \right\|_{ L^p(\P; H_{\gamma} ) }^2
\\&\leq
  \frac{p(p-1)}{2}
  \int_0^t
  \left\|
    e^{(t-\varphi(u))A} B
    -
    \one_{ (-\infty,s) }^{\R}(u) \,
    e^{(s-\varphi(u))A} B
  \right\|_{ HS( H, H_{\gamma} ) }^2 du  
\\&\leq
  \frac{ 
    p(p-1)
    \left(t-s\right)^{2\varepsilon}
    \max\{ 1, T \}
  }{
    (1-2\gamma-2\beta-2\varepsilon) 
  }
  \left\|
    B
  \right\|_{ HS(H, H_{-\beta}) }^2  
  <
  \infty .
\end{split}
\end{align}
The Kolmogorov-Chentsov theorem and the 
fact that $ p\varepsilon > 1 $
therefore imply that there exists an 
up to indistinguishability unique stochastic
process $ O \colon [0,T] \times \Omega \rightarrow H_{\gamma} $
with continuous sample paths which satisfies
for all $ t \in [0,T] $ that
\begin{equation}
  [
    O_t
  ]_{ \P, \B(H) }
  =
  \int_0^t
  e^{(t-\varphi(s))A} B \, dW_s .
\end{equation}
The proof of Lemma~\ref{lem:O_existence} is thus completed. 
\end{proof}

\begin{lemma}
\label{lem:O_prop_1}
Assume the setting in Section~\ref{sec:examples_setting}
and let $ p,q \in [2,\infty) $, 
$ \theta \in [\nicefrac{1}{4}-\nicefrac{1}{2q},\nicefrac{1}{4}) $, 
$ \xi \in \L^p( \P; H_{2\theta} ) $.
Then there exists a stochastic process
$ O \colon [0,T] \times \Omega \rightarrow L^q(\lambda_{(0,1)};\R) $
with continuous sample paths which satisfies 
\begin{enumerate}[(i)]
 \item\label{it:O_prop_1_1} that for all $ t \in [0,T] $
 we have that
 $
  [ 
    O_t - e^{tA} \xi
  ]_{ \P, \B(H) }
  =
  \int_0^t
  e^{(t-s)A} \, dW_s
 $ and
 \item\label{it:O_prop_1_2} that
 \begin{align} 
 \begin{split}
  &\adjustlimits
   \sup_{ N \in \N }
   \sup_{ 0 \leq s < t \leq T }
   \left(
     \frac{ 
       \| 
         P_{N} (O_t - O_{s} )      
       \|_{ \L^p( \P; L^{q}( \lambda_{(0,1)} ; \R ) ) }
     }{ (t-s)^{\theta} }
   \right)
\\&+
   \adjustlimits
   \sup_{ N \in \N }
   \sup_{ t \in [0,T] }
   \big(
     N^{2\theta} \,
     \| 
       O_t
       -
       P_{N}
       O_t
     \|_{ \L^p( \P; L^{q}( \lambda_{(0,1)} ; \R ) ) }
   \big)
   <
   \infty .
 \end{split}
 \end{align}
\end{enumerate}
\end{lemma}
\begin{proof}[Proof of Lemma~\ref{lem:O_prop_1}]
Throughout this proof let $ Y \colon \Omega \rightarrow \R $ be a standard 
normal random variable and let $ \tilde{p} = \max\{ p, q \} $,
$ \beta \in (\nicefrac{1}{4},\nicefrac{1}{2}-\theta) $,
$ (\mu_k)_{k\in\N} \subseteq \R $ satisfy
for all $ k \in \N $ that
$
  \mu_k = \nu k^2 \pi^2
$.
Note that the fact that 
\begin{equation}
  \forall \,
  t \in [0,\infty),
  r \in [0,1]
  \colon
  \|
    (-tA)^r
    e^{tA}
  \|_{ L(H) }
  \leq
  1 ,
\end{equation}
the fact that
\begin{equation}
  \forall \,
  t \in (0,\infty),
  r \in [0,1]
  \colon
  \|
    (-tA)^{-r}
    (
    e^{tA} - \Id_H )
  \|_{ L(H) }
  \leq
  1 ,
\end{equation}
and the assumption that
$ \xi \in \L^p( \P; H_{2\theta} ) $
assure that for all $ s \in [0,T) $, $ t \in (s,T] $ 
we have that
\begin{align}
\label{eq:O_prop_1_xi}
\begin{split}
  \|
    e^{tA} \xi
    -
    e^{sA} \xi
  \|_{ \L^p( \P; H_{\theta} ) }
 &\leq
  \|
    e^{sA}
  \|_{ L(H) } \,
  \|
    (
      e^{(t-s)A}
      -
      \Id_H
    ) \xi
  \|_{ \L^p( \P; H_{\theta} ) }
\\&\leq
  \|
    (-A)^{-\theta}
    (
      e^{(t-s)A}
      -
      \Id_H
    )
  \|_{ L(H) } \,
  \|
    \xi
  \|_{ \L^p( \P; H_{2\theta} ) }
\\&\leq
  (t-s)^{\theta} \,
  \|
    \xi
  \|_{ \L^p( \P; H_{2\theta } ) }
  <
  \infty .
\end{split}
\end{align}
In addition, observe that the fact that
$ A \colon D(A) \subseteq H \rightarrow H $
is the generator of a strongly continuous semigroup
and the fact that
\begin{equation}
  \forall \, t \in [0,\infty)
  \colon
  \| e^{tA} \|_{L(H)} \leq 1
\end{equation}
prove that for all $ \omega \in \Omega $, $ t \in [0,T] $ 
we have that
\begin{align}
\label{eq:O_prop_1_xi_2}
\begin{split}
 &\limsup_{
    \substack{ (t_1,t_2)\rightarrow(t,t) , \\ (t_1,t_2) \in [0,t]\times[t,T] } 
  }
  \|
    e^{t_2 A} \xi(\omega)
    -
    e^{t_1 A} \xi(\omega)
  \|_{ H_{\theta} }
\\&\leq
  \limsup_{
    \substack{ (t_1,t_2)\rightarrow(t,t) , \\ (t_1,t_2) \in [0,t]\times[t,T] } 
  }
  \left[ 
  \| e^{t_1 A} \|_{L(H)} \,
  \|
    ( e^{(t_2-t_1)A} - \Id_H) \xi(\omega)
  \|_{ H_{\theta} }
  \right]
\\&\leq
  \limsup_{
    \substack{ (t_1,t_2)\rightarrow(t,t) , \\ (t_1,t_2) \in [0,t]\times[t,T] } 
  }
  \|
    ( e^{(t_2-t_1)A} - \Id_H) (-A)^{\theta} \xi(\omega)
  \|_{ H }
  =
  0 .
\end{split}
\end{align}
Moreover, observe that the fact that $ 4\beta > 1 $ shows that
\begin{align}
\label{eq:O_prop_1_Id}
\begin{split}
  \sum_{ k=1 }^{ \infty }
  \|
    e_k
  \|_{ H_{-\beta} }^2
 &=
  \sum_{ k=1 }^{ \infty }
  \|
    (-A)^{-\beta}
    e_k
  \|_{ H }^2
  =
  \sum_{ k=1 }^{ \infty }
  |
    (\nu \pi^2 k^2)^{-\beta}
  |^2
  =
  \sum_{ k=1 }^{ \infty }
  \frac{ 1 }{ \left( \sqrt{\nu} k \pi\right)^{4\beta} }
\\&=
  \frac{ 1 }{ \left( \pi \sqrt{\nu} \right)^{4\beta} }
  \left[ 
    \sum_{ k=1 }^{ \infty }
    \frac{1}{k^{4\beta}}
  \right] 
  <
  \infty .
\end{split}
\end{align}
This allows us to apply Lemma~\ref{lem:O_existence}
(with $ \gamma = \theta $, $ \beta = \beta $,
$ B = ( H \ni v \mapsto v \in H_{-\beta} ) $,
$ \varphi = [0,T] \ni t \mapsto t \in [0,T] $
in the notation of Lemma~\ref{lem:O_existence})
to obtain that there exists an up to indistinguishability unique 
stochastic process $ \tilde{O} \colon [0,T] \times \Omega \rightarrow H_{\theta} $
with continuous sample paths which satisfies for all $ t \in [0,T] $ that 
\begin{align}
\label{eq:O_helper_1}
  [ \tilde{O}_t ]_{ \P, \B(H) } = \int_0^t e^{(t-s)A} \, dW_s .
\end{align}
Next note that~\eqref{eq:O_prop_1_xi_2}
and the fact that
\begin{equation}
\label{eq:O_prop_1_sobolev}
  H_{\theta}
  \subseteq
  W^{2\theta,2}( (0,1), \R )
  \subseteq 
  W^{0,q}( (0,1), \R )
  =
  L^q( \lambda_{ (0,1) }; \R )
\end{equation}
continuously
(cf., e.g., Da Prato \& Zabcyk~\cite[(A.46) in Section~A.5.2]{dz92}
and Lunardi~\cite{l09})
ensure that there exists a stochastic process
$ O \colon [0,T] \times \Omega \rightarrow L^q( \lambda_{ (0,1) }; \R ) $
with continuous sample paths which satisfies 
for all $ t \in [0,T] $, $ \omega \in \Omega $
that
\begin{align}
\label{eq:O_helper_2}
  O_t(\omega) = e^{tA} \xi(\omega) + \tilde{O}_t(\omega) .
\end{align}
Combining~\eqref{eq:O_helper_2} with~\eqref{eq:O_helper_1}
demonstrates that for all $ t \in [0,T] $ we have that
\begin{equation}
\label{eq:O_helper_3}
  [ O_t - e^{tA} \xi ]_{ \P, \B(H) } = \int_0^t e^{(t-s)A} \, dW_s .
\end{equation}
In addition, note that 
H{\"o}lder's inequality,
Fubini's theorem,
and, e.g., Lemma~\ref{lem:normal_variables}
show for all
$ s \in [0,T) $, $ t \in (s,T] $,
$ N \in \N $ that
\begin{align}
\begin{split}
  \E\!\left[
    \| P_N (\tilde{O}_t-\tilde{O}_s) \|_{ L^q( \lambda_{(0,1)}; \R ) }^{\tilde{p}}
  \right]
&=
  \E\!\left[
    \left|
      \int_0^1
      \big|
	\underline{P_N \tilde{O}_t}(x)
	-
	\underline{P_N \tilde{O}_s}(x)
      \big|^q \, dx
    \right|^{ \nicefrac{\tilde{p}}{q} }
  \right]
\\ &
  \leq
  \E\!\left[
    \int_0^1
    \big|
      \underline{P_N \tilde{O}_t}(x)
      -
      \underline{P_N \tilde{O}_s}(x)
    \big|^{\tilde{p}} \, dx
  \right]
\\&=
  \int_0^1
  \E\!\left[
    \left|
      \underline{P_N \tilde{O}_t}(x)
      -
      \underline{P_N \tilde{O}_s}(x)
    \right|^{\tilde{p}}
  \right] dx
\\ &
  =
  \E\!\left[
    \left| Y \right|^{\tilde{p}}
  \right]
  \int_0^1
  \left(
    \E\!\left[
      \left|
	\underline{P_N \tilde{O}_t}(x)
	-
	\underline{P_N \tilde{O}_s}(x)
      \right|^2    
    \right]
  \right)^{ \!\nicefrac{{\tilde{p}}}{2} } dx .
\end{split}
\end{align}
This implies for all
$ s \in [0,T) $, $ t \in (s,T] $,
$ N \in \N $ that
\begin{align}
\label{eq:PIO_diff_1}
\begin{split}
 &\E\!\left[
    \| P_N (\tilde{O}_t-\tilde{O}_s) \|_{ L^q( \lambda_{(0,1)}; \R ) }^{\tilde{p}}
  \right]
\\&\leq 
  \E\!\left[
    \left| Y \right|^{\tilde{p}}
  \right]
\\&\cdot
  \int\limits_0^1
  \left(
    \E\!\left[
      \left|
        \sum_{ k=1 }^N
	\underline{e_k}(x)
	\left( 
	  \int_0^t
	  e^{-\mu_k(t-u)} \left<e_k,dW_u\right>_H
	  -
	  \int_0^s
	  e^{-\mu_k(s-u)} \left<e_k,dW_u\right>_H
	\right)
      \right|^2    
    \right]
  \right)^{ \!\!\nicefrac{\tilde{p}}{2} } dx
\\&=
  \E\!\left[
    \left| Y \right|^{\tilde{p}}
  \right]
\\&\cdot
  \int\limits_0^1
  \left(
    \sum_{ k=1 }^N
    \left| \underline{e_k}(x) \right|^2
    \E\!\left[
      \left| 
	\int_0^t
	e^{-\mu_k(t-u)} \left<e_k,dW_u\right>_H
	-
	\int_0^s
	e^{-\mu_k(s-u)} \left<e_k,dW_u\right>_H
      \right|^2    
    \right]
  \right)^{ \!\!\nicefrac{\tilde{p}}{2} } dx .
\end{split}
\end{align}
Moreover, note that
It{\^o}'s isometry yields 
for all
$ s \in [0,T) $, $ t \in (s,T] $,
$ k \in \N $ that
\begin{align}
\begin{split}
 &\E\!\left[
    \left| 
      \int_0^t
      e^{-\mu_k(t-u)} \left<e_k,dW_u\right>_H
      -
      \int_0^s
      e^{-\mu_k(s-u)} \left<e_k,dW_u\right>_H
    \right|^2    
  \right]
\\&=
  \E\!\left[
    \left| 
      \int_s^t
      e^{-\mu_k(t-u)} \left<e_k,dW_u\right>_H
      +
      \left(
	e^{-\mu_k(t-s)}
	-
	1
      \right)
      \int_0^s
      e^{-\mu_k(s-u)} \left<e_k,dW_u\right>_H
    \right|^2    
  \right]
\\&=
  \E\!\left[
    \left| 
      \int_s^t
      e^{-\mu_k(t-u)} \left<e_k,dW_u\right>_H
    \right|^2    
  \right]
\\ &
\quad
  +
  \left(
    e^{-\mu_k(t-s)}
    -
    1
  \right)^2
  \E\!\left[
    \left| 
      \int_0^s
      e^{-\mu_k(s-u)} \left<e_k,dW_u\right>_H
    \right|^2    
  \right]
\\&=
  \int_s^t
  e^{-2\mu_k(t-u)} \, du
  +
  \left(
    e^{-\mu_k(t-s)}
    -
    1
  \right)^2
  \int_0^s
  e^{-2\mu_k(s-u)} \, du .
\end{split}
\end{align}
The fact that
\begin{align}
\begin{split}
  \sup_{ x \in (0,\infty) }
  \left(
    x^{-1} 
    \left( 
      1 - e^{-x }
    \right)
  \right)
 &=
  \sup_{ x \in (0,\infty) }
  \left(
    x^{-1}
    \int_0^x
    e^{-s} \, ds
  \right)
\\&\leq
  \sup_{ x \in (0,\infty) }
  \left(
    x^{-1}
    \int_0^x \, ds
  \right)
  =
  1
\end{split}
\end{align}
hence implies for all 
$ s \in [0,T) $, $ t \in (s,T] $,
$ k \in \N $ that
\begin{align}
\begin{split}
 &\E\!\left[
    \left| 
      \int_0^t
      e^{-\mu_k(t-u)} \left<e_k,dW_u\right>_H
      -
      \int_0^s
      e^{-\mu_k(s-u)} \left<e_k,dW_u\right>_H
    \right|^2    
  \right]
\\&=
  \frac{ 
    \left( 
      1 - e^{-2\mu_k (t-s) }
    \right)
  }{ 2 \mu_k }
  +
  \left(
    1
    -
    e^{-\mu_k(t-s)}
  \right)^2
  \frac{ 
    \left( 
      1 - e^{-2\mu_k s }
    \right)
  }{ 2 \mu_k }
\\&\leq
  \frac{ 
    \left( 
      1 - e^{-2\mu_k (t-s) }
    \right)
  }{ 2\mu_k }
  +
  \frac{ 
    \left( 
      1 - e^{-\mu_k (t-s) }
    \right)
  }{ 2\mu_k }
  \leq
  \frac{ 
    \left( 
      1 - e^{-2\mu_k (t-s) }
    \right)
  }{ \mu_k }
\\&=
  2^{2\theta}
  \left(t-s\right)^{2\theta}
  \left[ 
    \frac{ 
      \left( 
	1 - e^{-2\mu_k (t-s) }
      \right)
    }{ 2 \mu_k (t-s) }
  \right]^{2\theta}
  \left[ 
    \frac{ 
      \left( 
	1 - e^{-2\mu_k (t-s) }
      \right)
    }{ \mu_k }
  \right]^{(1-2\theta)}
\\&\leq
  \sqrt{2}
  \left(t-s\right)^{2\theta}
  \left[
    \sup_{ x \in (0,\infty) }
    \frac{ 
      \left( 
	1 - e^{-x }
      \right)
    }{ x }
  \right]^{2\theta}
  (\mu_k)^{(2\theta-1)}
  \leq 
  \sqrt{2}
  \left(t-s\right)^{2\theta}
  (\mu_k)^{(2\theta-1)} .
\end{split}
\end{align}
Combining this with~\eqref{eq:PIO_diff_1}
proves that for all
$ s \in [0,T) $, $ t \in (s,T] $,
$ N \in \N $ we have that
\begin{align}
\label{eq:O_prop_1_t_s}
\begin{split}
 &\E\!\left[
    \| P_N (\tilde{O}_t-\tilde{O}_s) \|_{ L^q( \lambda_{(0,1)}; \R ) }^{\tilde{p}}
  \right]
\\&\leq 
  \E\!\left[
    \left| Y \right|^{\tilde{p}}
  \right]
  \int_0^1
  \left[
    \sum_{ k=1 }^N
    \left| \underline{e_k}(x) \right|^2
    \sqrt{2}
    \left(t-s\right)^{2\theta}
    \mu_k^{(2\theta-1)}
  \right]^{ \!\nicefrac{\tilde{p}}{2} } dx
\\&\leq
  \left(t-s\right)^{\tilde{p}\theta}
  \E\!\left[
    \left| Y \right|^{\tilde{p}}
  \right]
  \left[
    \sqrt{8}
    \sum_{ k=1 }^N
    \left(\sqrt{\nu} k \pi\right)^{-(2-4\theta)}
  \right]^{ \!\nicefrac{\tilde{p}}{2} } .
\end{split}
\end{align}
Furthermore, observe that the triangle inequality,
H{\"o}lder's inequality, and~\eqref{eq:O_prop_1_sobolev}
show that
\begin{align}
\begin{split}
 &\sup_{ N \in \N }
  \sup_{ 0 \leq s < t \leq T }
  \left( 
    \frac{ 
      \| 
        P_N (O_t - O_{s} ) 
      \|_{ \L^p( \P; L^q( \lambda_{(0,1)}; \R ) ) } 
    }{ (t-s)^{\theta} }
  \right)
\\&\leq
  \sup_{ N \in \N }
  \sup_{ 0 \leq s < t \leq T }
  \left[ 
    \frac{ 
      \| P_N (e^{tA} - e^{sA} ) \xi  \|_{ \L^p( \P; L^q( \lambda_{(0,1)}; \R ) ) } 
    }{ (t-s)^{\theta} }  
    +
    \frac{ 
      \| P_N (\tilde{O}_t - \tilde{O}_{s} ) \|_{ \L^p( \P; L^q( \lambda_{(0,1)}; \R ) ) } 
    }{ (t-s)^{\theta} }
  \right] 
\\&\leq
  \left[ 
    \sup_{ v \in H_{\theta} \setminus \{0\} }
    \frac{ \| v \|_{ L^q( \lambda_{(0,1)} ; \R ) } }{ \| v \|_{ H_{\theta} } }
  \right] 
  \left[ 
    \sup_{ N \in \N }
    \sup_{ 0 \leq s < t \leq T }
    \frac{ 
      \| P_N \|_{ L(H) } \,
      \| (e^{tA} - e^{sA} ) \xi  \|_{ \L^p( \P; H_{\theta} ) }
    }{ (t-s)^{\theta} }
  \right]
\\&\quad+
  \sup_{ N \in \N }
  \sup_{ 0 \leq s < t \leq T }
  \frac{ 
    \| 
      P_N (\tilde{O}_t - \tilde{O}_{s} ) 
    \|_{ \L^{\tilde{p}}( \P; L^q( \lambda_{(0,1)}; \R ) ) }
  }{ (t-s)^{\theta} } . 
\end{split}
\end{align}
Combining this with~\eqref{eq:O_prop_1_xi},
\eqref{eq:O_prop_1_sobolev}, \eqref{eq:O_prop_1_t_s}, 
the fact that
$
  \forall \, N \in \N
  \colon 
  \| P_N \|_{ L(H) }
  \leq 
  1
$,
the assumption that
$ \xi \in \L^p( \P; H_{2\theta} ) $,
and the fact that $ 2 - 4\theta > 1 $
implies that
\begin{align}
\label{eq:O_prop_1_conti}
\begin{split}
 &\sup_{ N \in \N }
  \sup_{ 0 \leq s < t \leq T }
  \left(
    \frac{ 
      \| 
        P_N (O_t - O_{s} ) 
      \|_{ \L^p( \P; L^q( \lambda_{(0,1)}; \R ) ) } 
    }{ (t-s)^{\theta} }
  \right)
\\&\leq
  \left[ 
    \sup_{ v \in H_{\theta} \setminus \{0\} }
    \frac{ \| v \|_{ L^q( \lambda_{(0,1)} ; \R ) } }{ \| v \|_{ H_{\theta} } }
  \right] 
  \| \xi \|_{ \L^p(\P; H_{2\theta}) }
\\&\quad+
  \| Y \|_{ \L^{\tilde{p}}( \P; \R ) }
  \left[ 
    \sqrt{8}
    \sum_{ k=1 }^{\infty}
    \left(\sqrt{\nu} k \pi\right)^{-(2-4\theta)}
  \right]^{ \!\nicefrac{1}{2} }
  <
  \infty .
\end{split}
\end{align}
In the next step observe that
the fact that
$
  \forall \, t \in [0,\infty)
  \colon
  \| e^{tA} \|_{ L(H) }
  \leq 
  1
$
and the assumption that
$ \xi \in \L^p( \P; H_{2\theta} ) $
yield that
\begin{align}
\label{eq:O_prop_1_xi_spatial}
\begin{split}
 &\sup_{ N \in \N }
  \sup_{ t \in [0,T] }
  \left(
    N^{2\theta} \,
    \| 
      e^{tA}
      \xi
      -
      P_{N}
      e^{tA}
      \xi
    \|_{ \L^p( \P; H_{\theta} ) }
  \right)
\\&\leq
  \sup_{ N \in \N }
  \sup_{ t \in [0,T] }
  \left(
    N^{2\theta} \,
    \|
      e^{tA}
    \|_{ L(H ) } \,
    \| 
      (-A)^{\theta}      
      (
        \Id_H
        -
        P_N
      )
      \xi
    \|_{ \L^p( \P; H ) }
  \right)
\\&\leq
  \sup_{ N \in \N }
  \left(
    N^{2\theta} \,
    \| 
      (-A)^{-\theta}      
      (
        \Id_H
        -
        P_N
      )
    \|_{ L(H) } \,
    \|
      \xi
    \|_{ \L^p( \P; H_{2\theta} ) }
  \right)
\\&=
  \left[ 
    \sup_{ N \in \N }
    \frac{ N^{2\theta} }{ \nu^{\theta} \pi^{2\theta} (N+1)^{2\theta} }    
  \right]
  \|
    \xi
  \|_{ \L^p( \P; H_{2\theta} ) }
\\&\leq 
  \nu^{-\theta}
  \|
    \xi
  \|_{ \L^p( \P; H_{2\theta} ) }
  <
  \infty .
\end{split}
\end{align}
In addition, note that
H{\"o}lder's inequality,
Fubini's theorem,
and, e.g., Lemma~\ref{lem:normal_variables}
guarantee that for all
$ t \in [0,T] $,
$ M, N \in \N $
with $ M \geq N $
we have that
\begin{align}
\begin{split}
 &\E\!\left[
    \| P_M \tilde{O}_t-P_N \tilde{O}_t \|_{ L^q( \lambda_{(0,1)}; \R ) }^{\tilde{p}}
  \right]
\\&=
  \E\!\left[
    \left|
      \int_0^1
      \big|
	\underline{P_M \tilde{O}_t}(x)
	-
	\underline{P_N \tilde{O}_t}(x)
      \big|^q \, dx
    \right|^{ \nicefrac{\tilde{p}}{q} }
  \right]
\\ &
  \leq
  \E\!\left[
    \int_0^1
    \big|
      \underline{P_M \tilde{O}_t}(x)
      -
      \underline{P_N \tilde{O}_t}(x)
    \big|^{\tilde{p}} \, dx
  \right]
\\&=
  \int_0^1
  \E\!\left[
    \left|
      \underline{P_M \tilde{O}_t}(x)
      -
      \underline{P_N \tilde{O}_t}(x)
    \right|^{\tilde{p}}
  \right] dx
\\&=
  \E\!\left[
    \left| Y \right|^{\tilde{p}}
  \right]
  \int_0^1
  \left(
    \E\!\left[
      \left|
	\underline{P_M \tilde{O}_t}(x)
	-
	\underline{P_N \tilde{O}_t}(x)
      \right|^2    
    \right]
  \right)^{ \!\nicefrac{{\tilde{p}}}{2} } dx
\\&=
  \E\!\left[
    \left| Y \right|^{\tilde{p}}
  \right]
  \int_0^1
  \left(
    \E\!\left[
      \left|
        \sum_{ k=N+1 }^{M}
	\underline{e_k}(x)
	\int_0^t
	e^{-\mu_k(t-s)} \left<e_k,dW_s\right>_H
      \right|^2    
    \right]
  \right)^{ \!\!\nicefrac{\tilde{p}}{2} } dx .
\end{split}
\end{align}
It{\^o}'s isometry hence ensures for all 
$ t \in [0,T] $, $ M, N \in \N $ with $ M \geq N $ that
\begin{align}
\label{eq:O_prop_1_spatial}
\begin{split}
 &\E\!\left[
    \| P_M \tilde{O}_t-P_N \tilde{O}_t \|_{ L^q( \lambda_{(0,1)}; \R ) }^{\tilde{p}}
  \right]
\\&\leq
  \E\!\left[
    \left| Y \right|^{\tilde{p}}
  \right]
  \int_0^1
  \left(
    \sum_{ k=N+1 }^{M}
    \left| \underline{e_k}(x) \right|^2
    \E\!\left[
      \left|
	\int_0^t
	e^{-\mu_k(t-s)} \left<e_k,dW_s\right>_H
      \right|^2    
    \right]
  \right)^{ \!\nicefrac{\tilde{p}}{2} } dx
\\&=
  \E\!\left[
    \left| Y \right|^{\tilde{p}}
  \right]
  \int_0^1
  \left(
    \sum_{ k=N+1 }^{M}
    \left| \underline{e_k}(x) \right|^2
    \int_0^t
    e^{-2\mu_k(t-s)} \, ds
  \right)^{ \!\nicefrac{\tilde{p}}{2} } dx
\\&=
  \E\!\left[
    \left| Y \right|^{\tilde{p}}
  \right]
  \int_0^1
  \left(
    \sum_{ k=N+1 }^{M}
    \left| \underline{e_k}(x) \right|^2
    \frac{ (1-e^{-2\mu_k t} ) }{ 2\mu_k }
  \right)^{ \!\nicefrac{\tilde{p}}{2} } dx . 
\end{split}
\end{align}
Therefore, we obtain
for all $ t \in [0,t] $, $ M, N \in \N $
with $ M \geq N $ that
\begin{align}
\label{eq:O_prop_1_spatial_2}
\begin{split}
 &\E\!\left[
    \| P_M \tilde{O}_t-P_N \tilde{O}_t \|_{ L^q( \lambda_{(0,1)}; \R ) }^{\tilde{p}}
  \right]
\\&\leq
  \E\!\left[
    \left| Y \right|^{\tilde{p}}
  \right]
  \left(
    \sum_{ k=N+1 }^{M}
    \frac{ (1-e^{-2\mu_k t} ) }{ \mu_k }
  \right)^{ \!\nicefrac{\tilde{p}}{2} }
\\&\leq
  \E\!\left[
    \left| Y \right|^{\tilde{p}}
  \right]
  \left(
    \sum_{ k=N+1 }^{\infty}
    \frac{ 1 }{ \nu \pi^2 k^2 }
  \right)^{ \!\nicefrac{\tilde{p}}{2} } .
\end{split}
\end{align}
This implies that for all $ t \in [0,T] $,
$ N \in \N $ we have that
\begin{align}
\begin{split}
 &\E\!\left[
    \| \tilde{O}_t-P_N \tilde{O}_t \|_{ L^q( \lambda_{(0,1)}; \R ) }^{\tilde{p}}
  \right]
\\&=
  \limsup_{M\rightarrow\infty}
  \E\!\left[
    \| P_M \tilde{O}_t-P_N \tilde{O}_t \|_{ L^q( \lambda_{(0,1)}; \R ) }^{\tilde{p}}
  \right]
\\&\leq
  \E\!\left[
    \left| Y \right|^{\tilde{p}}
  \right]
  \left(
    \sum_{ k=N+1 }^{\infty}
    \frac{ 1 }{ \nu \pi^2 k^2 }
  \right)^{ \!\nicefrac{\tilde{p}}{2} } .
\end{split}
\end{align}
Hence, we obtain for all $ t \in [0,T] $,
$ N \in \N $ that
\begin{align}
\begin{split}
 &\| \tilde{O}_t-P_N \tilde{O}_t \|_{ \L^{\tilde{p}}( \P; L^q( \lambda_{ (0,1) }; \R ) ) }
\\&\leq 
  \| Y \|_{ \L^{\tilde{p}}( \P; \R ) }
  \left[ 
    \sum_{ k=N+1 }^{\infty}
    \frac{ 1 }{ \nu \pi^2 k^2 }
  \right]^{\!\nicefrac{1}{2}}
  =
  \frac{ \| Y \|_{ \L^{\tilde{p}}( \P; \R ) } }{ \pi \sqrt{\nu} }
  \left[ 
    \sum_{ k=N+1 }^{\infty}
    \frac{ 1 }{ k^{4\theta} k^{(2-4\theta)} }
  \right]^{\!\nicefrac{1}{2}}
\\&\leq
  \frac{ \| Y \|_{ \L^{\tilde{p}}( \P; \R ) } }{ N^{2\theta} \pi \sqrt{\nu} }
  \left[ 
    \sum_{ k=N+1 }^{\infty}
    \frac{ 1 }{ k^{(2-4\theta)} }
  \right]^{\!\nicefrac{1}{2}}
  \leq
  \frac{ \| Y \|_{ \L^{\tilde{p}}( \P; \R ) } }{ N^{2\theta} \pi \sqrt{\nu} }
  \left[ 
    \sum_{ k=1 }^{\infty}
    \frac{ 1 }{ k^{(2-4\theta)} }
  \right]^{\!\nicefrac{1}{2}} . 
\end{split}
\end{align}
Combining this with the fact that $ 2 - 4\theta > 1 $
demonstrates that
\begin{align}
\label{eq:O_prop_1_Z1}
\begin{split}
 &\sup_{ N \in \N }
  \sup_{ t \in [0,T] }
  \big(
    N^{2\theta} \,
    \| \tilde{O}_t-P_N \tilde{O}_t \|_{ \L^{\tilde{p}}( \P; L^q( \lambda_{ (0,1) }; \R ) ) }
  \big)
\\&\leq
  \frac{ \| Y \|_{ \L^{\tilde{p}}( \P; \R ) } }{ \pi \sqrt{\nu} }
  \left[ 
    \sum_{ k=1 }^{\infty}
    \frac{ 1 }{ k^{(2-4\theta)} }
  \right]^{\!\nicefrac{1}{2}} 
  < \infty .
\end{split}
\end{align}
The hypothesis that $ \xi \in \L^p( \P; H_{2\theta} ) $,
\eqref{eq:O_prop_1_sobolev}, \eqref{eq:O_prop_1_xi_spatial}
the triangle inequality, and the H{\"o}lder inequality
hence assure that
\begin{align}
\begin{split}
 &\sup_{ N \in \N }
  \sup_{ t \in [0,T] }
  \big(
    N^{2\theta} \,
    \| O_t - P_N O_t \|_{ \L^p( \P; L^q( \lambda_{ (0,1) }; \R ) ) }
  \big)
\\&\leq
  \sup_{ N \in \N }
  \sup_{ t \in [0,T] }
  \Big(
    N^{2\theta}
    \left[ 
      \| e^{tA} \xi - P_N e^{tA} \xi \|_{ \L^p( \P; L^q( \lambda_{ (0,1) }; \R ) ) }
      +
      \| \tilde{O}_t - P_N \tilde{O}_t \|_{ \L^p( \P; L^q( \lambda_{ (0,1) }; \R ) ) }
    \right]
  \Big)
\\&\leq
  \left[ 
    \sup_{ v \in H_{\theta} \setminus \{0\} }
    \frac{ \| v \|_{ L^q( \lambda_{(0,1)} ; \R ) } }{ \| v \|_{ H_{\theta} } }
  \right] 
  \left[ 
    \sup_{ N \in \N }
    \sup_{ t \in [0,T] }
    \big(
      N^{2\theta} \,
      \| e^{tA} \xi - P_N e^{tA} \xi \|_{ \L^p( \P; H_{\theta} ) }
    \big)
  \right] 
\\&\quad+ 
  \sup_{ N \in \N }
  \sup_{ t \in [0,T] }
  \big(
    N^{2\theta} \,
    \| \tilde{O}_t - P_N \tilde{O}_t \|_{ \L^{\tilde{p}}( \P; L^q( \lambda_{ (0,1) }; \R ) ) } 
  \big)
\\&\leq
  \left[ 
    \sup_{ v \in H_{\theta} \setminus \{0\} }
    \frac{ \| v \|_{ L^q( \lambda_{(0,1)} ; \R ) } }{ \| v \|_{ H_{\theta} } }
  \right] 
  \frac{ \| \xi \|_{ \L^p( \P; H_{2\theta} ) } }{ \nu^{\theta} }
\\&\quad+ 
  \sup_{ N \in \N }
  \sup_{ t \in [0,T] }
  \big(
    N^{2\theta} \,
    \| \tilde{O}_t-P_N \tilde{O}_t \|_{ \L^{\tilde{p}}( \P; L^q( \lambda_{ (0,1) }; \R ) ) }
  \big)
\\&< \infty .
\end{split}
\end{align}
Combining this, \eqref{eq:O_prop_1_conti}, and~\eqref{eq:O_helper_3}
with the fact that 
$ O \colon [0,T] \times \Omega \rightarrow L^q( \lambda_{(0,1)} ; \R ) $
is a stochastic process with continuous sample paths completes the proof 
of Lemma~\ref{lem:O_prop_1}.
\end{proof}

\begin{lemma}
\label{lem:O_prop_2}
Assume the setting in Section~\ref{sec:examples_setting}
and let $ p \in [2,\infty) $, $ \theta \in [0,\nicefrac{1}{4}) $,
$ \xi \in \L^p( \P; H_{\theta} ) $.
Then there exist stochastic processes
$ \OMNN{} \colon [0,T] \times \Omega \rightarrow P_N(H) $,
$ M, N \in \N $,
with continuous sample paths which satisfy 
\begin{enumerate}[(i)]
 \item\label{it:O_prop_2_1} that for all $ t \in [0,T] $,
 $ M, N \in \N $
 we have that
 $
    [ 
      \OMNN{t} - P_{N} e^{tA} \xi
    ]_{ \P, \B(H) }
    =
    \int_0^t
    P_{N}
    e^{(t-\fl{s})A} \, dW_s
 $ and
 \item\label{it:O_prop_2_2} that
$  
  \sup_{ \gamma \in [0,\theta] }
  \sup_{ M,N \in \N }
  \sup_{ t \in [0,T] }
  \E\big[
    \| 
      \OMNN{t}
    \|_{ H_{\gamma} }^p
   \big]
   <
   \infty
 $.
\end{enumerate}

\end{lemma}
\begin{proof}[Proof of Lemma~\ref{lem:O_prop_2}]
Throughout this proof let
$ \beta \in (\nicefrac{1}{4},\nicefrac{1}{2}-\theta) $.
Observe that the fact that
\begin{equation}
  [0,\infty)
  \ni
  t
  \mapsto
  (
    H_{\theta}
    \ni
    v
    \mapsto 
    e^{tA} v
    \in
    H_{\theta}
  )
  \in 
  L( H_{\theta} )
\end{equation}
is a strongly continuous semigroup,
the fact that
$
  \forall \, N \in \N
  \colon 
  \| P_N \|_{ L(H) }
  \leq 
  1
$,
and the fact that
$
  \forall \, t \in [0,\infty)
  \colon
  \| e^{tA} \|_{L(H)} \leq 1
$
prove that for all 
$ \omega \in \Omega $, $ t \in [0,T] $, $ N \in \N $ 
we have that
\begin{align}
\label{eq:O_prop_2_xi}
\begin{split}
 &\limsup_{
    \substack{ (t_1,t_2)\rightarrow(t,t) \\ (t_1,t_2) \in [0,t]\times[t,T] } 
  }
  \|
    P_N
    e^{t_2 A} \xi(\omega)
    -
    P_N
    e^{t_1 A} \xi(\omega)
  \|_{ H_{\theta} }
\\&\leq
  \limsup_{
    \substack{ (t_1,t_2)\rightarrow(t,t) \\ (t_1,t_2) \in [0,t]\times[t,T] } 
  }
  \left(
    \| P_N \|_{L(H)} \,
    \| e^{t_1 A} \|_{L(H)} \,
    \|
      ( e^{(t_2-t_1)A} - \Id_H) \xi(\omega)
    \|_{ H_{\theta} }
  \right)
\\&\leq
  \limsup_{
    \substack{ (t_1,t_2)\rightarrow(t,t) \\ (t_1,t_2) \in [0,t]\times[t,T] } 
  }
  \|
    ( e^{(t_2-t_1)A} - \Id_H) \xi(\omega)
  \|_{ H_{\theta} }
  =
  0 .
\end{split}
\end{align}
In addition, note that
the fact that
$
  4\beta > 1
$
shows that
\begin{align}
\label{eq:O_prop_2_Id}
\begin{split}
  \sup_{ N \in \N }
  \|
    P_N
  \|_{ HS(H, H_{-\beta}) }^2
 &=
  \sup_{ N \in \N }
  \left[
    \sum_{ k=1 }^{ N }
    \|
      e_k
    \|_{ H_{-\beta} }^2
  \right]
  =
  \sum_{ k=1 }^{ \infty }
  \|
    (-A)^{-\beta}
    e_k
  \|_{ H }^2
\\&=
  \sum_{ k=1 }^{ \infty }
  |
    (\nu \pi^2 k^2)^{-\beta}
  |^2
  =
  \sum_{ k=1 }^{ \infty }
  \frac{ 1 }{ ( \sqrt{\nu} \pi k )^{4\beta} }
  <
  \infty .
\end{split}
\end{align}
We can hence apply Lemma~\ref{lem:O_existence}
(with $ \gamma = \theta $, $ \beta = \beta $,
$ B = H \ni v \mapsto P_N(v) \in H_{-\beta} $,
$ \varphi = [0,T] \ni t \mapsto \fl{t} \in [0,T] $
for $ M,N \in \N $ in the notation of Lemma~\ref{lem:O_existence})
to obtain that there exist up to modifications unique
stochastic processes $ \OMNNp{} \colon [0,T] \times \Omega \rightarrow H_{\theta} $,
$ M,N \in \N $, with continuous sample paths which satisfy
for all $ t \in [0,T] $, $ M,N \in \N $ that
\begin{equation}
\label{eq:O_prop_2_nrA}
  [ \OMNNp{t} ]_{ \P, \B(H) } = \int_0^t P_N e^{(t-\fl{s})A} \, dW_s .
\end{equation}
Inequality~\eqref{eq:O_prop_2_xi} therefore
shows that there exists stochastic processes 
$ \OMNN{} \colon [0,T] \times \Omega \rightarrow P_N(H) $,
$ M,N \in \N $, with continuous sample paths which satisfy
for all $ t \in [0,T] $, $ N,M \in \N $ that
\begin{equation}
\label{eq:O_prop_2_nrB} 
  \OMNN{t} = P_N e^{tA} \xi + P_N \OMNNp{t} . 
\end{equation}
Next note that~\eqref{eq:O_prop_2_nrA}
ensures that for all $ t \in [0,T] $,
$ N, M \in \N $ we have that
\begin{equation}
  [ P_N \OMNNp{t} ]_{ \P, \B(H) } = \int_0^t P_N e^{(t-\fl{s})A} \, dW_s .
\end{equation}
Combining this with~\eqref{eq:O_prop_2_nrB}
demonstrates that for all 
$ t \in [0,T] $,
$ N, M \in \N $ we have that
\begin{equation}
\label{eq:O_prop_2_res_i}
    [ 
      \OMNN{t} - P_{N} e^{tA} \xi
    ]_{ \P, \B(H) }
    =
    \int_0^t
    P_{N}
    e^{(t-\fl{s})A} \, dW_s .
\end{equation}
In the next step observe that, e.g.,
the Burkholder-Davis-Gundy type inequality
in Lemma~7.7 in Da Prato \& Zabcyk~\cite{dz92}, 
the fact that 
\begin{equation}
  \forall \,
  t \in [0,\infty) ,
  r \in [0,1]
  \colon
  \|
    (-tA)^r
    e^{tA}
  \|_{ L(H) }
  \leq
  1,
\end{equation}
and~\eqref{eq:O_prop_2_Id} imply that for all 
$ t \in [0,T] $, $ N,M \in \N $ we have that
\begin{align}
\begin{split}
 &\left\|
    \int_0^t
    P_{N} 
    e^{(t-\fl{s})A} \, dW_s
  \right\|_{ L^{p}(\P; H_{\theta} ) }^2
  \leq 
  \frac{p(p-1)}{2}
  \int_0^t
  \| 
    P_{N} 
    e^{(t-\fl{s})A}
  \|_{ HS( H, H_{\theta} ) }^2 \, ds
\\&\leq
  \frac{p(p-1)}{2}
  \| (-A)^{-\beta} P_N \|_{ HS(H) }^2
  \int_0^t
  \| (-A)^{\beta} e^{(t-\fl{s})A} \|_{ L(H, H_{\theta}) }^2 \, ds
\\&\leq
  \frac{p(p-1)}{2}
  \| P_N \|_{ HS(H, H_{-\beta} ) }^2
  \int_0^t
  \| (-A)^{ (\theta+ \beta) } e^{(t-s)A} \|_{ L(H) }^2 \, 
  \| e^{(s-\fl{s})A} \|_{ L(H) }^2 \, ds
\\&\leq
  \frac{p(p-1)}{2}
  \| P_N \|_{ HS(H, H_{-\beta} ) }^2
  \int_0^t
  \left( t-s \right)^{-2(\theta+\beta)} \, ds
\\&\leq
  \frac{p(p-1) \, T^{(1-2\theta-2\beta)} }{2\,(1-2\theta-2\beta)}
  \| P_N \|_{ HS(H, H_{-\beta} ) }^2
  <
  \infty.
\end{split}
\end{align}
The triangle inequality, the fact that
$
  \forall \, 
  N \in \N
  \colon 
  \| P_N \|_{ L(H) } \leq 1
$,
the fact that
\begin{equation}
  \forall \,
  t \in [0,\infty)
  \colon
  \| e^{tA} \|_{ L(H) }
  \leq
  1 ,
\end{equation}
the assumption that
$ \xi \in \L^p( \P; H_{\theta} ) $,
and~\eqref{eq:O_prop_2_Id}
hence assure that
\begin{align}
\label{eq:O_prop_2_bound}
\begin{split}
 &\sup_{ M,N \in \N }
  \sup_{ t \in [0,T] }
  \big\| 
    \OMNN{t}
  \big\|_{ \L^p( \P; H_{\theta} ) }
\\&\leq
  \sup_{ M,N \in \N }
  \sup_{ t \in [0,T] }
  \left[ 
    \|
      P_N
      e^{tA} \xi
    \|_{ \L^p( \P; H_{ \theta } ) }
    +
    \| P_N \|_{L(H)} 
    \left\|
      \int_0^t
      P_{N}
      e^{(t-\fl{u})A} \, dW_u
    \right\|_{ L^{p}(\P; H_{\theta} ) }
  \right]
\\&\leq
  \left[ 
    \sup_{ N \in \N }
    \| P_N \|_{ L(H) }
  \right]
  \left[ 
    \sup_{ t \in [0,T] }
    \| e^{tA} \|_{ L(H) }
  \right] 
  \|
    \xi
  \|_{ \L^p( \P; H_{\theta} ) }
\\&\quad+ 
  \left[ 
    \sup_{ N \in \N }
    \| P_N \|_{ HS(H, H_{-\beta} ) }  
  \right] 
  \frac{\sqrt{p(p-1)} \, T^{(\nicefrac{1}{2}-\theta-\beta)} }{\sqrt{ 2\,(1-2\theta-2\beta) }}
  <
  \infty .
\end{split}
\end{align}
Combining this with~\eqref{eq:O_prop_2_res_i} and the fact that 
\begin{align}
\begin{split}
  \sup_{ \gamma \in [0,\theta] }
  \| (-A)^{(\gamma-\theta)} \|_{L(H)}
 &=
  \sup_{ \gamma \in [0,\theta] }
  \| (-A)^{-\gamma} \|_{L(H)}
\\&=
  \sup_{ \gamma \in [0,\theta] }
  \left[ 
    (\nu \pi^2)^{-\gamma}
  \right]
  =
  \max\!\left\{ 
    \frac{ 1 }{ ( \nu \pi^2 )^{\theta} }, 1
  \right\}
  <
  \infty
\end{split}
\end{align}
completes the proof of Lemma~\ref{lem:O_prop_2}.
\end{proof}

\begin{lemma}
\label{lem:O_sobo}
Assume the setting in
Section~\ref{sec:examples_setting},
let $ p \in [1,\infty) $,
$ \xi \in \cup_{ r \in (\nicefrac{1}{4}, \infty) } \L^p( \P; H_r ) $,
and let 
$
  \OMNN{}
  \colon 
  [0,T] \times \Omega
  \rightarrow P_N(H)
$,
$
  M,N \in \N
$,
be stochastic processes
with continuous sample paths
which satisfy for all
$ M, N \in \N $,
$ t \in [0,T] $
that
\begin{equation}
  [ \OMNN{t} - P_N e^{tA} \xi ]_{ \P, \B(H) }
  =
  \int_0^t
  P_N e^{(t-\fl{s})A} \, dW_s .
\end{equation}
Then
\begin{equation}
  \sup_{ M,N \in \N }
  \E\big[
     \textstyle\sup_{ t \in [0,T] }
     \| \OMNN{t} \|_{ L^{\infty}( \lambda_{ (0,1) } ; \R ) }^p
  \big]
  <
  \infty .
\end{equation}
\end{lemma}
\begin{proof}[Proof of Lemma~\ref{lem:O_sobo}]
Throughout this proof let
$ Y \colon \Omega \rightarrow \R $ be a 
standard normal random variable,
let
$ \alpha \in (0,\nicefrac{1}{8}) $,
$ q \in (\nicefrac{2}{\alpha}, \infty) \cap [p,\infty) $,
$ (\mu_k)_{ k \in \N } \subseteq \R $
satisfy for all $ k \in \N $ that
$
  \mu_k
  =
  \nu k^2 \pi^2
$,
let
$
  \OMNNp{}
  \colon
$
$
  [0,T] \times \Omega
  \rightarrow
  P_N(H)
$,
$
  M,N \in \N
$,
be the stochastic processes which satisfy 
for all $ M,N \in \N $,
$ t \in [0,T] $
that 
\begin{equation}
  \OMNNp{t}
  =
  \OMNN{t} - P_N e^{tA} \xi ,
\end{equation}
let
$
  \left|\!\left|\!\left| \cdot \right|\!\right|\!\right|
  \colon 
  \C( (0,1) \times [0,T], \R )
  \rightarrow 
  [0,\infty]
$
be the function which
satisfies for all 
$ v \in \C( (0,1) \times [0,T], \R ) $
that
\begin{multline}
  \left|\!\left|\!\left| v \right|\!\right|\!\right|
  =
  \Bigg(
    \int_{ (0,1) \times [0,T] }
    | v(x,t) |^q \, \lambda_{ \R^2 }( dx, dt )
\\
    +
    \int_{ (0,1) \times [0,T] }
    \int_{ ( (0,1) \times [0,T] ) \setminus \{ (x_2,t_2) \} }
    \frac{ 
      | 
	v(x_1,t_1)
	-
	v(x_2,t_2)
      |^q 
    }{ 
      \| (x_1,t_1) - (x_2,t_2) \|_{\R^2}^{(2+\alpha q)}
    } \, \lambda_{ \R^2 }( dx_1, dt_1 ) \, \lambda_{ \R^2 }( dx_2, dt_2 )
  \Bigg)^{\!\!\nicefrac{1}{q}} ,
\end{multline}
and let $ C \in [0,\infty] $ be the extended 
real number given by
\begin{align}
 &C
  =
  \sup\!\left(
    \left\{
      \sup_{ x \in (0,1) }
      \sup_{ t \in [0,T] }
      | v(x,t) |
      \colon \!
      \Big[
	v \in \C( (0,1) \times [0,T], \R )
	\text{ and }
	\left|\!\left|\!\left| v \right|\!\right|\!\right|
	\leq 1
      \Big]
    \right\}
  \right) .
\end{align}
Observe that the Sobolev embedding theorem
and the fact that $ \alpha q > 2 $
imply that 
\begin{equation}
  C < \infty .
\end{equation}
Next note that for all $ M, N \in \N $ we have that
\begin{align}
\begin{split}
 &\E\!\left[
    {\textstyle \sup_{ t \in [0,T] } }
    \| \OMNNp{t} \|_{ L^{\infty}( \lambda_{ (0,1) } ; \R ) }^q
  \right]
  =
  \E\!\left[
    {\textstyle \sup_{ (x,t) \in (0,1) \times [0,T] } }
    | \underline{(\OMNNp{t})}(x) |^q
  \right]
\\&\leq
  C^q \,
  \E\!\left[
    \int_{ (0,1) \times [0,T] }
    | \underline{(\OMNNp{t})}(x) |^q \, \lambda_{ \R^2 }( dx, dt )
  \right]
\\&+
  C^q \,
  \E\!\left[
    \int_{ (0,1) \times [0,T] }
    \int_{ ( (0,1) \times [0,T] ) \setminus \{ (x_2,t_2) \} }
    \frac{ 
      | 
        \underline{(\OMNNp{t_1})}(x_1)
        -
        \underline{(\OMNNp{t_2})}(x_2)
      |^q 
    }{ 
      \| (x_1,t_1) - (x_2,t_2) \|_{\R^2}^{(2+\alpha q)}
    } \, \lambda_{ \R^2 }( dx_1, dt_1 ) \, \lambda_{ \R^2 }( dx_2, dt_2 )
  \right] .
\end{split}
\end{align}
This and, e.g., Lemma~\ref{lem:normal_variables} show 
that for all $ M, N \in \N $ we have that
\begin{align}
\label{eq:O_sobo_0}
\begin{split}
 &\E\!\left[
    {\textstyle \sup_{ t \in [0,T] } }
    \| \OMNNp{t} \|_{ L^{\infty}( \lambda_{ (0,1) } ; \R ) }^q
  \right] 
\\&\leq
  C^q \,
  \E\big[ 
    |Y|^q
  \big] 
  \int_{ (0,1) \times [0,T] }
  \left( 
    \E\Big[
      | \underline{(\OMNNp{t})}(x) |^2
    \Big]
  \right)^{\!\nicefrac{q}{2}} \lambda_{ \R^2 }( dx, dt )
\\&+
  C^q \,
  \E\big[ 
    |Y|^q
  \big]
\\&\cdot
  \int_{ (0,1) \times [0,T] }
  \int_{ ( (0,1) \times [0,T] ) \setminus \{ (x_2,t_2) \} }
  \frac{
    \left( 
      \E\Big[
	| 
	  \underline{(\OMNNp{t_1})}(x_1)
	  -
	  \underline{(\OMNNp{t_2})}(x_2)
	|^2  
      \Big]
    \right)^{\!\nicefrac{q}{2}}
  }{ 
    \| (x_1,t_1) - (x_2,t_2) \|_{\R^2}^{(2+\alpha q)}
  } \, \lambda_{ \R^2 }( dx_1, dt_1 ) \, \lambda_{ \R^2 }( dx_2, dt_2 ) .
\end{split}
\end{align}
Moreover, note that
It{\^o}'s isometry proves that
for all $ M,N \in \N $, $ t \in [0,T] $,
$ x \in (0,1) $ we have that
\begin{align}
\label{eq:O_sobo_1}
\begin{split}
 &\E\Big[
    | \underline{(\OMNNp{t})}(x) |^2
  \Big]
  =
  \E\!\left[ 
    \left|
      \sum_{ k=1 }^N
      \underline{e_k}(x)
      \int_0^t
      e^{-\mu_k(t-\fl{s})} \, \langle e_k, dW_s \rangle_H
    \right|^2
  \right] 
\\&=
  \sum_{ k=1 }^N
  | \underline{e_k}(x) |^2 \,
  \E\!\left[ 
    \left|
      \int_0^t
      e^{-\mu_k(t-\fl{s})} \, \langle e_k, dW_s \rangle_H
    \right|^2
  \right]
\\&=
  \sum_{ k=1 }^N
  | \underline{e_k}(x) |^2
  \int_0^t
  e^{-2\mu_k(t-\fl{s})} \, ds
\\&=
  \sum_{ k=1 }^N
  | \underline{e_k}(x) |^2
  \int_0^t
  e^{-2\mu_k(t-s)} \,
  e^{-2\mu_k(s-\fl{s})} \, ds
  \leq
  2
  \sum_{ k=1 }^N
  \int_0^t
  e^{-2\mu_k(t-s)} \, ds
\\&=
  2
  \sum_{ k=1 }^N
  \int_0^t
  e^{-2\mu_k s} \, ds
  =
  2
  \sum_{ k=1 }^N
  \frac{ 
    \left( 1 - e^{-2\mu_k t} \right)
  }{ 2 \mu_k }
  \leq
  \sum_{ k=1 }^N
  \frac{ 
    1
  }{ \mu_k }
  =
  \frac{ 1 }{ \nu \pi^2 }
  \left(
    \sum_{ k=1 }^N
    \frac{ 1 }{ k^2 }
  \right)
\\&=
  \frac{ 1 }{ \nu \pi^2 }
  \left(
    1
    +
    \sum_{ k=2 }^N
    \int_{ k - 1 }^k
    \frac{ 1 }{ k^2 } \, ds
  \right)
  \leq
  \frac{ 1 }{ \nu \pi^2 } 
  \left(
    1
    +
    \int_1^N 
    s^{-2} \, ds 
  \right)
\\&=
  \frac{ 1 }{ \nu \pi^2 } 
  \left( 
    1
    -
    \left[
      \nicefrac{1}{s}
    \right]_{s=1}^{s=N}
  \right)
  =
  \frac{ 1 }{ \nu \pi^2 } 
  \left[
    2
    -
    \frac{1}{N}
  \right] 
  \leq
  \frac{ 2 }{ \nu \pi^2 } .
\end{split}
\end{align}
In the next step note 
that for all $ M,N \in \N $,
$ s,t \in [0,T] $, $ x,y \in (0,1) $
we have that
\begin{align}
\begin{split}
 &\E\Big[
    | 
      \underline{(\OMNNp{t})}(x)
      -
      \underline{(\OMNNp{s})}(y)
    |^2  
  \Big]
\\&=
  \E\!\left[ 
    \left| 
      \sum_{ k=1 }^N
      \left(
        \underline{e_k}(x)
        \int_0^t
        e^{-\mu_k(t-\fl{u})}
        \langle e_k, dW_u \rangle_H
        -
        \underline{e_k}(y)
        \int_0^s
        e^{-\mu_k(s-\fl{u})}
        \langle e_k, dW_u \rangle_H
      \right)
    \right|^2
  \right] 
\\&=
  \sum_{ k=1 }^N
  \E\!\left[ 
    \left| 
      \underline{e_k}(x)
      \int_0^t
      e^{-\mu_k(t-\fl{u})}
      \langle e_k, dW_u \rangle_H
      -
      \underline{e_k}(y)
      \int_0^s
      e^{-\mu_k(s-\fl{u})}
      \langle e_k, dW_u \rangle_H
    \right|^2
  \right] .
\end{split}
\end{align}
The fact that
\begin{equation}
  \forall \,
  a,b \in \R
  \colon 
  |a + b|^2
  \leq
  2a^2 + 2b^2
\end{equation}
hence implies for 
all $ M,N \in \N $,
$ s,t \in [0,T] $, $ x,y \in (0,1) $
that
\begin{align}
\label{eq:O_sobo_2}
\begin{split}
 &\E\Big[
    | 
      \underline{(\OMNNp{t})}(x)
      -
      \underline{(\OMNNp{s})}(y)
    |^2  
  \Big]
\\&=
  \sum_{ k=1 }^N
  \E\!\Bigg[ 
    \bigg| \!
      \left(
        \underline{e_k}(x)
        -
        \underline{e_k}(y)
      \right)
      \int_0^t
      e^{-\mu_k(t-\fl{u})}
      \langle e_k, dW_u \rangle_H
\\&\quad+
      \underline{e_k}(y)
      \left(
	\int_0^t
	e^{-\mu_k(t-\fl{u})}
	\langle e_k, dW_u \rangle_H
	-
	\int_0^s
	e^{-\mu_k(s-\fl{u})}
	\langle e_k, dW_u \rangle_H
      \right) \!
    \bigg|^2
  \Bigg]
\\&\leq
  2
  \sum_{ k=1 }^N
  \left|
    \underline{e_k}(x)
    -
    \underline{e_k}(y)
  \right|^2
  \E\!\left[ 
    \left|
      \int_0^t
      e^{-\mu_k(t-\fl{u})}
      \langle e_k, dW_u \rangle_H
    \right|^2
  \right]
\\&\quad+
  2
  \sum_{ k=1 }^N
  \left|
    \underline{e_k}(y)
  \right|^2
  \E\!\left[ 
    \left|
      \int_0^t
      e^{-\mu_k(t-\fl{u})}
      \langle e_k, dW_u \rangle_H
      -
      \int_0^s
      e^{-\mu_k(s-\fl{u})}
      \langle e_k, dW_u \rangle_H
    \right|^2
  \right] .
\end{split}
\end{align}
In addition, observe that
It{\^o}'s isometry,
the fact that 
\begin{equation}
  \forall \, a,b \in \R
  \colon
  |\! \sin(a) - \sin(b) |
  \leq 
  2
  ,
\end{equation}
and the fact that
\begin{equation}
  \forall \, a,b \in \R
  \colon
  |\! \sin(a) - \sin(b) |
  \leq 
  | a - b |
\end{equation}
show for all 
$ M, N \in \N $,
$ t \in [0,T] $,
$ x,y \in (0,1) $,
$ r \in [0,2] $
that
\begin{align}
\begin{split}
 &\sum_{ k=1 }^N
  \left|
    \underline{e_k}(x)
    -
    \underline{e_k}(y)
  \right|^2
  \E\!\left[ 
    \left|
      \int_0^t
      e^{-\mu_k(t-\fl{u})}
      \langle e_k, dW_u \rangle_H
    \right|^2
  \right]
\\&=
  2
  \sum_{ k=1 }^N
  \left|
    \sin( k \pi x )
    -
    \sin( k \pi y )
  \right|^{(2-r)}
  \left|
    \sin( k \pi x )
    -
    \sin( k \pi y )
  \right|^r
  \int_0^t
  e^{-2\mu_k(t-u)} \,
  e^{-2\mu_k(u-\fl{u})} \, du
\\&\leq
  2^{(3-r)}
  | x - y |^r
  \left[
    \sum_{ k=1 }^N
    ( k \pi )^r
    \int_0^t
    e^{-2\mu_k(t-u)} \, du
  \right]
  =
  2^{(3-r)}
  | x - y |^r
  \left[
    \sum_{ k=1 }^N
    \frac{ 
      ( k \pi )^r 
      (1-e^{-2\mu_k t} )
    }{ 2 \mu_k }
  \right] .
\end{split}
\end{align}
This yields that
for all 
$ M, N \in \N $,
$ t \in [0,T] $,
$ x,y \in (0,1) $,
$ r \in [0,1) $
we have that
\begin{align}
\label{eq:O_sobo_3}
\begin{split}
 &\sum_{ k=1 }^N
  \left|
    \underline{e_k}(x)
    -
    \underline{e_k}(y)
  \right|^2
  \E\!\left[ 
    \left|
      \int_0^t
      e^{-\mu_k(t-\fl{u})}
      \langle e_k, dW_u \rangle_H
    \right|^2
  \right]
\\&\leq
  2^{(2-r)}
  | x - y |^r
  \left[
    \sum_{ k=1 }^N
    \frac{ 
      k^r \pi^r
    }{ \nu k^2 \pi^2 }
  \right]
  =
  \frac{ 
    2^{(2-r)}
    | x - y |^r
  }{ \nu \pi^{(2-r)} }
  \left[
    \sum_{ k=1 }^N
    k^{(r-2)}
  \right]
\\&\leq
  \frac{ 
    2^{(2-r)}
    | x - y |^r
  }{ \nu \pi^{(2-r)} }
  \left( 
    1
    +
    \int_1^N
    s^{(r-2)} \, ds
  \right)
  =
  \frac{ 
    2^{(2-r)}
    | x - y |^r
  }{ \nu \pi^{(2-r)} }
  \left( 
    1
    -
    \left[ 
      \frac{s^{(r-1)}}{(1-r)}
    \right]_{ s=1 }^{s=N}
  \right)
\\&=
  \frac{
    2^{(2-r)}
    | x - y |^r
  }{ \nu \pi^{(2-r)} }
  \left( 
    1
    -
    \frac{( N^{(r-1)} - 1 )}{(1-r)}
  \right)
  \leq 
  \frac{
    2^{(2-r)}
    | x - y |^r
  }{ \nu \pi^{(2-r)} }
  \left( 
    \frac{2}{(1-r)}
    -
    \frac{N^{(r-1)}}{(1-r)}
  \right)
\\&=
  \frac{
    2^{(2-r)}
    ( 2 - N^{(r-1)} )
    | x - y |^r
  }{ \nu \pi^{(2-r)} (1-r) }
  \leq
  \frac{ 
    2^{(3-r)}
    | x - y |^r
  }{ 
    \nu \pi^{(2-r)}
    (1-r)
  } .
\end{split}
\end{align}
Furthermore,
note that 
for all
$ M,N \in \N $,
$ s \in [0,T) $,
$ t \in (s,T] $
we have that
\begin{align}
\begin{split}
 &\sum_{ k=1 }^N
  \E\!\left[ 
    \left|
      \int_0^t
      e^{-\mu_k(t-\fl{u})}
      \langle e_k, dW_u \rangle_H
      -
      \int_0^s
      e^{-\mu_k(s-\fl{u})}
      \langle e_k, dW_u \rangle_H
    \right|^2
  \right]
\\&=
  \sum_{ k=1 }^N
  \E\!\left[ 
    \left|
      \int_s^t
      e^{-\mu_k(t-\fl{u})}
      \langle e_k, dW_u \rangle_H
      +
      \left(
	e^{-\mu_k(t-s)}
	-
	1
      \right)
      \int_0^s
      e^{-\mu_k(s-\fl{u})}
      \langle e_k, dW_u \rangle_H
    \right|^2
  \right]
\\&=
  \sum_{ k=1 }^N
  \E\!\left[ 
    \left|
      \int_s^t
      e^{-\mu_k(t-\fl{u})}
      \langle e_k, dW_u \rangle_H
    \right|^2
  \right]
\\&\quad+
  \sum_{ k=1 }^N
  \left(
    e^{-\mu_k(t-s)}
    -
    1
  \right)^2
  \E\!\left[ 
    \left|
      \int_0^s
      e^{-\mu_k(s-\fl{u})}
      \langle e_k, dW_u \rangle_H
    \right|^2
  \right] .
\end{split}
\end{align}
It{\^o}'s isometry
hence ensures
for all
$ M,N \in \N $,
$ s \in [0,T) $,
$ t \in (s,T] $ that
\begin{align}
\begin{split}
 &\sum_{ k=1 }^N
  \E\!\left[ 
    \left|
      \int_0^t
      e^{-\mu_k(t-\fl{u})}
      \langle e_k, dW_u \rangle_H
      -
      \int_0^s
      e^{-\mu_k(s-\fl{u})}
      \langle e_k, dW_u \rangle_H
    \right|^2
  \right]
\\&=
  \sum_{ k=1 }^N
  \int_s^t
  e^{-2\mu_k(t-u)} \,
  e^{-2\mu_k(u-\fl{u})} \, du
\\&\quad+
  \sum_{ k=1 }^N
  \left(
    e^{-\mu_k(t-s)}
    -
    1
  \right)^2
  \int_0^s
  e^{-2\mu_k(s-u)} \,
  e^{-2\mu_k(u-\fl{u})} \, du
\\&\leq
  \sum_{ k=1 }^N
  \int_s^t
  e^{-2\mu_k(t-u)} \, du
  +
  \sum_{ k=1 }^N
  \left(
    e^{-\mu_k(t-s)}
    -
    1
  \right)^2
  \int_0^s
  e^{-2\mu_k(s-u)} \, du
\\&=
  \sum_{ k=1 }^N
  \int_0^{(t-s)}
  e^{-2\mu_k u} \, du
  +
  \sum_{ k=1 }^N
  \left(
    e^{-\mu_k(t-s)}
    -
    1
  \right)^2
  \int_0^s
  e^{-2\mu_k u} \, du
\\&=
  \sum_{ k=1 }^N
  \frac{ (1-e^{-2\mu_k (t-s) } ) }{ 2 \mu_k }
  +
  \sum_{ k=1 }^N
  \left(
    e^{-\mu_k(t-s)}
    -
    1
  \right)^2
  \frac{ (1-e^{-2\mu_k s } ) }{ 2 \mu_k } .
\end{split}
\end{align}
The fact that
\begin{equation}
\begin{split}
  \sup_{ x \in (0,\infty) }
  \left(
    x^{-1}
    (1 - e^{-x} )
  \right)
& =
  \sup_{ x \in (0,\infty) }
  \left(
    x^{-1}
    \left(
      \int_0^x
      e^{-s} \, ds
    \right)
  \right)
\\ & \leq 
  \sup_{ x \in (0,\infty) }
  \left(
    x^{-1}
    \left(
      \int_0^x \, ds
    \right)
  \right)
  =
  1
\end{split}
\end{equation}
therefore implies
that for all
$ M,N \in \N $,
$ s \in [0,T) $,
$ t \in (s,T] $,
$ r \in [0,1] $ 
we have that
\begin{align}
\begin{split}
 &\sum_{ k=1 }^N
  \E\!\left[ 
    \left|
      \int_0^t
      e^{-\mu_k(t-\fl{u})}
      \langle e_k, dW_u \rangle_H
      -
      \int_0^s
      e^{-\mu_k(s-\fl{u})}
      \langle e_k, dW_u \rangle_H
    \right|^2
  \right]
\\&\leq
  \sum_{ k=1 }^N
  \frac{ (1-e^{-2\mu_k (t-s) } ) }{ 2 \mu_k }
  +
  \sum_{ k=1 }^N
  \frac{ (1-e^{-\mu_k (t-s)} ) }{ 2 \mu_k }
  \leq
  \sum_{ k=1 }^N
  \frac{ (1-e^{-2\mu_k (t-s) } ) }{ \mu_k }
\\&=
  2^r
  (t-s)^r
  \sum_{ k=1 }^N
  \left[ 
    \frac{ 
      (1-e^{-2\mu_k (t-s) } ) 
    }{ 
      2 \mu_k 
      (t-s)
    }
  \right]^r
  \left[ 
    \frac{ (1-e^{-2\mu_k (t-s) } ) }{ \mu_k }
  \right]^{(1-r)}
\\&\leq
  2^r
  (t-s)^r
  \left[
    \sup_{ x \in (0,\infty) }
    \frac{ 
      (1-e^{-x} ) 
    }{ 
      x
    }
  \right]^r
  \left[
    \sum_{ k=1 }^N
    (\mu_k)^{(r-1)}
  \right]
\\&\leq
  2^r
  (t-s)^r
  \left[
    \sum_{ k=1 }^N
    \frac{1}{(\mu_k)^{(1-r)}}
  \right] 
  =
  \frac{ 
    2^r
    (t-s)^r
  }{ \nu^{(1-r)} \pi^{2(1-r)} }
  \left[ 
    \sum_{ k=1 }^N
    k^{2(r-1)}
  \right] .
\end{split}
\end{align}
This yields that 
for all
$ M,N \in \N $,
$ s \in [0,T) $,
$ t \in (s,T] $,
$ r \in [0,\nicefrac{1}{2}) $
we have that
\begin{align}
\label{eq:O_sobo_4}
\begin{split}
 &\sum_{ k=1 }^N
  \E\!\left[ 
    \left|
      \int_0^t
      e^{-\mu_k(t-\fl{u})}
      \langle e_k, dW_u \rangle_H
      -
      \int_0^s
      e^{-\mu_k(s-\fl{u})}
      \langle e_k, dW_u \rangle_H
    \right|^2
  \right]
\\&\leq
  \frac{ 
    2^r
    (t-s)^r
  }{ \nu^{(1-r)} \pi^{2(1-r)} }
  \left(
    1
    +
    \int_1^N
    s^{2(r-1)} \, ds
  \right)
  =
  \frac{ 
    2^r
    (t-s)^r
  }{ \nu^{(1-r)} \pi^{2(1-r)} }
  \left(
    1
    -
    \left[ 
      \frac{ s^{(2r-1)} }{ (1-2r) }
    \right]_{ s=1 }^{ s = N }
  \right)
\\&=
  \frac{ 
    2^r
    (t-s)^r
  }{ \nu^{(1-r)} \pi^{2(1-r)} }
  \left(
    1
    - 
    \frac{ ( N^{(2r-1)} - 1 ) }{ (1-2r) }
  \right) 
  \leq
  \frac{ 
    2^{(1+r)}
    (t-s)^r
  }{ 
    \nu^{(1-r)} \pi^{2(1-r)}
    (1-2r)
  } .
\end{split}
\end{align}
Moreover, observe that
Jensen's inequality
proves that for all 
$ s,t,x,y \in [0,\infty) $,
$ r \in (0,2] $ 
we have that
\begin{align}
\label{eq:O_sobo_5}
\begin{split}
  \|
    (x,t)
    -
    (y,s)
  \|_{\R^2}^r
 &=
  \Big[ 
    |x-y|^2
    +
    |t-s|^2
  \Big]^{\nicefrac{r}{2}}
  =
  2^{\nicefrac{r}{2}}
  \Big[ 
    \tfrac{1}{2}
    |x-y|^2
    +
    \tfrac{1}{2}
    |t-s|^2
  \Big]^{\nicefrac{r}{2}}
\\&=
  2^{\nicefrac{r}{2}}
  \Big[ 
    \tfrac{1}{2}
    \{ |x-y|^r \}^{\nicefrac{2}{r}}
    +
    \tfrac{1}{2}
    \{ |t-s|^r \}^{\nicefrac{2}{r}}
  \Big]^{\nicefrac{r}{2}}
\\&\geq
  2^{\nicefrac{r}{2}}
  \Big[ 
    \tfrac{1}{2}
    |x-y|^r
    +
    \tfrac{1}{2}
    |t-s|^r
  \Big]
  =
  2^{(\nicefrac{r}{2}-1)}
  \Big[ 
    |x-y|^r
    +
    |t-s|^r
  \Big] .
\end{split}
\end{align}
Combining~\eqref{eq:O_sobo_2}
with~\eqref{eq:O_sobo_3}, \eqref{eq:O_sobo_4},
and~\eqref{eq:O_sobo_5}
yields that for all 
$ M,N \in \N $,
$ s,t \in [0,T] $,
$ x,y \in (0,1) $
we have that
\begin{align}
\label{eq:O_sobo_6}
\begin{split}
 &\E\Big[
    | 
      \underline{(\OMNNp{t})}(x)
      -
      \underline{(\OMNNp{s})}(y)
    |^2  
  \Big]
\\&\leq
  \frac{ 
    2^{(4-4\alpha)}
    | x - y |^{4\alpha}
  }{ 
    \nu \pi^{(2-4\alpha)}
    (1-4\alpha)
  }
  +
  \frac{ 
    2^{(3+4\alpha)}
    |t-s|^{4\alpha}
  }{ 
    \nu^{(1-4\alpha)} \pi^{2(1-4\alpha)}
    (1-8\alpha)
  }
\\&\leq
  \frac{ 
    2^{(4-4\alpha)}
  }{ 
    \min\{1,\nu\} 
    \pi^{(2-8\alpha)}
    (1-8\alpha)
  }
  \Big[ 
    | x - y |^{4\alpha}
    +
    |t-s|^{4\alpha}
  \Big]
\\&\leq
  \frac{ 
    2^{(5-6\alpha)}
    \|
      (x,t)
      -
      (y,s)
    \|_{\R^2}^{4\alpha}
  }{ 
    \min\{1,\nu\}
    \pi^{(2-8\alpha)}
    (1-8\alpha)
  }
\\&\leq
  \frac{ 
    2^3
    \pi^{2\alpha}
    \|
      (x,t)
      -
      (y,s)
    \|_{\R^2}^{4\alpha}
  }{ 
    \min\{1,\nu\}
    (1-8\alpha)
  }
  =
  \frac{
    \pi^{2\alpha}
    \|
      (x,t)
      -
      (y,s)
    \|_{\R^2}^{4\alpha}
  }{ 
    \min\{1,\nu\}
    (\nicefrac{1}{8}-\alpha)
  } .
\end{split}
\end{align}
In the next step we note 
that~\eqref{eq:O_sobo_0}
together with~\eqref{eq:O_sobo_1},
\eqref{eq:O_sobo_6},
and the fact that
$ \alpha q > 2 $
assures that for all 
$ M,N \in \N $ we have that
\begin{align}
\label{eq:O_sobo_7}
\begin{split}
 &\E\!\left[
    {\textstyle \sup_{ t \in [0,T] } }
    \| \OMNNp{t} \|_{ L^{\infty}( \lambda_{ (0,1) } ; \R ) }^q
  \right] 
\\&\leq
  \frac{  
    2^{\nicefrac{q}{2}}
    C^q \,
    \E\big[ 
      |Y|^q
    \big] 
  }{ 
    \nu^{\nicefrac{q}{2}} \pi^q  
  }
  \int_{ (0,1) \times [0,T] }
  \lambda_{ \R^2 }( dx, dt )
\\&+
  \frac{ 
    \pi^{\alpha q}
    C^q \,
    \E\big[ 
      |Y|^q
    \big]
  }{
    \min\{1,\nu^{\nicefrac{q}{2}}\}
    (\nicefrac{1}{8}-\alpha)^{\nicefrac{q}{2}}
  }
\\&\cdot
  \int_{ (0,1) \times [0,T] }
  \int_{ (0,1) \times [0,T] }
  \| (x_1,t_1) - (x_2,t_2) \|_{\R^2}^{(\alpha q- 2)}
  \, \lambda_{ \R^2 }( dx_1, dt_1 ) \, \lambda_{ \R^2 }( dx_2, dt_2 )
\\&\leq
  \frac{  
    C^q T \,
    \E\big[ 
      |Y|^q
    \big] 
  }{ 
    \nu^{\nicefrac{q}{2}} 
  }
\\&+
  \frac{ 
    \pi^{\alpha q}
    C^q T^2 \,
    \E\big[ 
      |Y|^q
    \big]
  }{
    \min\{1,\nu^{\nicefrac{q}{2}}\}
    (\nicefrac{1}{8}-\alpha)^{\nicefrac{q}{2}}
  }
  \left[ 
    \sup_{ x_1,x_2 \in (0,1) }
    \sup_{ t_1,t_2 \in [0,T] }
    \| (x_1,t_1) - (x_2,t_2) \|_{\R^2}^{(\alpha q- 2)}
  \right]
\\&=
  \frac{  
    C^q T \,
    \E\big[ 
      |Y|^q
    \big] 
  }{ 
    \nu^{\nicefrac{q}{2}} 
  }
\\&+
  \frac{ 
    \pi^{\alpha q}
    C^q T^2 \,
    \E\big[ 
      |Y|^q
    \big]
  }{
    \min\{1,\nu^{\nicefrac{q}{2}}\}
    (\nicefrac{1}{8}-\alpha)^{\nicefrac{q}{2}}
  }
  \left[ 
    \sup_{ x_1,x_2 \in (0,1) }
    \sup_{ t_1,t_2 \in [0,T] }
    \left( |x_1 - x_2|^2 + |t_1 - t_2|^2 \right)
  \right]^{(\frac{\alpha q}{2} - 1)}
\\&\leq
  \frac{  
    C^q T \,
    \E\big[ 
      |Y|^q
    \big] 
  }{ 
    \nu^{\nicefrac{q}{2}}  
  }
  +
  \frac{ 
    \pi^{\alpha q}
    C^q T^2 
    (
      1 + T^2
    )^{(\nicefrac{\alpha q}{2}-1)}
    \E\big[ 
      |Y|^q
    \big]
  }{
    \min\{1,\nu^{\nicefrac{q}{2}}\}
    (\nicefrac{1}{8}-\alpha)^{\nicefrac{q}{2}}
  }
\\&\leq
  \frac{ 
    \pi^{\alpha q}
    C^q \,
    \E\big[ 
      |Y|^q
    \big] \!
    \left[
      T
      +
      T^2
      (
	1 + T^2
      )^{(\nicefrac{\alpha q}{2}-1)}
    \right] 
  }{ 
    \min\{1,\nu^{\nicefrac{q}{2}}\}
    (\nicefrac{1}{8}-\alpha)^{\nicefrac{q}{2}}
  } .
\end{split}
\end{align}
In addition, observe that 
the assumption that
$
  \xi 
  \in
  \cup_{ r \in (\nicefrac{1}{4},\infty) }
  \L^p( \P; H_r )
$
implies that there exists a real number
$ \varepsilon \in (0,\infty) $
such that
\begin{equation}
\label{eq:O_sobo_xi}
  \xi
  \in
  \L^p( \P; H_{(\nicefrac{1}{4}+\varepsilon)} ) .
\end{equation}
The triangle inequality
and H{\"o}lder's inequality hence show that
\begin{align}
\begin{split}
 &\sup_{ N, M \in \N }
  \big\|
    {\textstyle \sup_{ t \in [0,T] } }
    \| \OMNN{t} \|_{ L^{\infty}(\lambda_{(0,1)};\R) }
  \big\|_{ \L^p( \P; \R ) }
\\&\leq 
  \sup_{ N, M \in \N }
  \big\|
    {\textstyle \sup_{ t \in [0,T] } }
    \| P_N e^{tA} \xi \|_{ L^{\infty}(\lambda_{(0,1)};\R) }
  \big\|_{ \L^p( \P; \R ) }
\\ &
\quad
  +
  \sup_{ N, M \in \N }
  \big\|
    {\textstyle \sup_{ t \in [0,T] } }
    \| \OMNNp{t} \|_{ L^{\infty}(\lambda_{(0,1)};\R) }
  \big\|_{ \L^p( \P; \R ) }
\\&\leq
  \left[ 
    \sup_{ v \in H_{(\nicefrac{1}{4}+\varepsilon)} \setminus \{0\} }
    \frac{ 
      \| v \|_{ L^{\infty}(\lambda_{(0,1)};\R) } 
    }{
      \| v \|_{ H_{(\nicefrac{1}{4}+\varepsilon)} }
    }
  \right] 
  \left[ 
    \sup_{ N \in \N }
    \big\|
      {\textstyle \sup_{ t \in [0,T] } }
      \| P_N e^{tA} \xi \|_{ H_{(\nicefrac{1}{4}+\varepsilon)} }
    \big\|_{ \L^p( \P; \R ) }
  \right]
\\&\quad+
  \sup_{ N, M \in \N }
  \big\|
    {\textstyle \sup_{ t \in [0,T] } }
    \| \OMNNp{t} \|_{ L^{\infty}(\lambda_{(0,1)};\R) }
  \big\|_{ \L^q( \P; \R ) } .
\end{split}
\end{align}
The fact that
$
  \forall \, N \in \N
  \colon 
  \| P_N \|_{ L(H) }
  \leq 
  1
$,
the fact that
\begin{equation}
  \forall \, t \in [0,\infty)
  \colon
  \| e^{tA} \|_{ L(H) }
  \leq 
  1
  ,
\end{equation}
the Sobolev embedding theorem,
\eqref{eq:O_sobo_7},
and~\eqref{eq:O_sobo_xi}
therefore prove that
\begin{align}
\begin{split}
 &\sup_{ N, M \in \N }
  \big\|
    {\textstyle \sup_{ t \in [0,T] } }
    \| \OMNN{t} \|_{ L^{\infty}(\lambda_{(0,1)};\R) }
  \big\|_{ \L^p( \P; \R ) }
\\&\leq 
  \left[ 
    \sup_{ v \in H_{(\nicefrac{1}{4}+\varepsilon)} \setminus \{0\} }
    \frac{ 
      \| v \|_{ L^{\infty}(\lambda_{(0,1)};\R) } 
    }{
      \| v \|_{ H_{(\nicefrac{1}{4}+\varepsilon)} }
    }
  \right] 
  \left[ 
    \sup_{ N \in \N }
    \big\|
      {\textstyle \sup_{ t \in [0,T] } }
      \| P_N \|_{ L(H) }
      \| e^{tA} \|_{ L(H) }
      \| \xi \|_{ H_{(\nicefrac{1}{4}+\varepsilon)} }
    \big\|_{ \L^p( \P; \R ) }
  \right]
\\&\quad+
  \sup_{ N, M \in \N }
  \left(
    \E\!\left[
      {\textstyle \sup_{ t \in [0,T] } }
      \| \OMNNp{t} \|_{ L^{\infty}( \lambda_{ (0,1) } ; \R ) }^q
    \right] 
  \right)^{\!\nicefrac{1}{q}}
\\&\leq
  \left[ 
    \sup_{ v \in H_{(\nicefrac{1}{4}+\varepsilon)} \setminus \{0\} }
    \frac{ 
      \| v \|_{ L^{\infty}(\lambda_{(0,1)};\R) } 
    }{
      \| v \|_{ H_{(\nicefrac{1}{4}+\varepsilon)} }
    }
  \right] 
  \|
    \xi
  \|_{ \L^p( \P; H_{(\nicefrac{1}{4}+\varepsilon)} ) }
\\&\quad+
  \frac{
    \pi^{\alpha} C \,
    \|Y\|_{ \L^q( \P; \R ) }
    \left[
      T
      +
      T^2
      (
	1 + T^2
      )^{(\nicefrac{\alpha q}{2}-1)}
    \right]^{\!\nicefrac{1}{q}}
  }{ 
    \min\{1,\sqrt{\nu}\}
    \sqrt{\nicefrac{1}{8}-\alpha}
  }
  <
  \infty .
\end{split}
\end{align}
This completes the proof
of Lemma~\ref{lem:O_sobo}.
\end{proof}

\begin{lemma}
\label{lem:O_prop_3}
Assume the setting in Section~\ref{sec:examples_setting},
let $ p,q \in [2,\infty) $, 
$ \xi \in \L^p( \P; L^q( \lambda_{ (0,1) }; \R ) ) $,
$ \theta \in [0,\nicefrac{1}{4}) $, 
and let $ O \colon [0,T] \times \Omega \rightarrow L^q( \lambda_{ (0,1) }; \R ) $ and
$ \OMNN{} \colon [0,T] \times \Omega \rightarrow P_N(H) $,
$ M, N \in \N $, be stochastic processes 
which satisfy for all $ t \in [0,T] $, $ M,N \in \N $ that
$
  [ 
    O_t - e^{tA} \xi
  ]_{ \P, \B(H) }
  =
  \int_0^t
  e^{(t-s)A} \, dW_s
$
and
$
  [ 
    \OMNN{t} - P_{N} e^{tA} \xi
  ]_{ \P, \B(H) }
  =
  \int_0^t
  P_{N}
  e^{(t-\fl{s})A} \, dW_s
$.
Then we have that 
\begin{equation}
  \sup_{ M,N \in \N }
  \sup_{ t \in [0,T] }
  \left[ 
    M^{\theta}
    \left(
      \E\Big[
        \| 
          P_N O_t - \OMNN{t} 
        \|_{ L^q( \lambda_{ (0,1) } ; \R ) }^p
      \Big]
    \right)^{\!\nicefrac{1}{p}}
  \right]
  <
  \infty .
\end{equation}
\end{lemma}
\begin{proof}[Proof of Lemma~\ref{lem:O_prop_3}]
Throughout this proof let $ Y \colon \Omega \rightarrow \R $
be a standard normal random variable and let
$ \tilde{p} = \max\{ p, q \} $, 
$ \varepsilon \in (\nicefrac{1}{2}+2\theta, 1) $,
$ (\mu_k)_{ k \in \N } \subseteq \R $
satisfy for all $ k \in \N $ that
$ \mu_k = \nu \pi^2 k^2 $.
Next note that 
H{\"o}lder's inequality,
Fubini's theorem,
and, e.g., Lemma~\ref{lem:normal_variables} 
imply that for all
$ t \in [0,T] $, $ M, N \in \N $ we have that
\begin{align}
\begin{split}
 &\E\!\left[
    \| 
      P_{N} 
      O_t
      -
      \OMNN{t}
    \|_{ L^{q}( \lambda_{(0,1)} ; \R ) }^{\tilde{p}}
  \right]
\\&=
  \E\!\left[
    \left|
      \int_0^1
      \big| 
	\underline{ 
	  P_{N} 
	  O_t
	}(x)
	-
	\underline{ \OMNN{t} }(x)
      \big|^{q} \, dx
    \right|^{\nicefrac{\tilde{p}}{q}}
  \right]
\\&\leq
  \E\!\left[
    \int_0^1
    \big| 
      \underline{ 
	P_{N} 
	O_t
      }(x)
      -
      \underline{ \OMNN{t} }(x)
    \big|^{\tilde{p}} \, dx
  \right]
\\&=
  \int_0^1
  \E\!\left[
    \big| 
      \underline{ 
	P_{N} 
	O_t
      }(x)
      -
      \underline{ \OMNN{t} }(x)
    \big|^{ \tilde{p} }
  \right] dx
\\&=
  \E\big[ 
    | Y |^{\tilde{p}}
  \big] 
  \int_0^1
  \left(
    \E\!\left[
      \big| 
	\underline{ 
	  P_{N} 
	  O_t
	}(x)
	-
	\underline{ \OMNN{t} }(x)
      \big|^2 
    \right] 
  \right)^{\!\nicefrac{\tilde{p}}{2}} dx .
\end{split}
\end{align}
This shows that for all 
$ t \in [0,T] $, $ M, N \in \N $
we have that
\begin{align}
\label{eq:OMNN_est_1}
\begin{split}
 &\E\!\left[
    \| 
      P_{N} 
      O_t
      -
      \OMNN{t}
    \|_{ L^{q}( \lambda_{(0,1)} ; \R ) }^{\tilde{p}}
  \right]
\\&=
  \E\big[ 
    | Y |^{\tilde{p}}
  \big] 
  \int_0^1
  \left(
    \E\!\left[
      \left|
        \sum_{ k=1 }^N
        \underline{ e_k }(x)
        \int_0^t
        \left( 
          e^{ -\mu_k(t-s) }
          -
          e^{ -\mu_k(t-\fl{s}) }
        \right) \left< e_k, dW_s \right>_H
      \right|^2 
    \right] 
  \right)^{\!\!\nicefrac{\tilde{p}}{2}} dx
\\&=
  \E\big[ 
    | Y |^{\tilde{p}}
  \big] 
  \int_0^1
  \left(
    \sum_{ k=1 }^N
    \big| \underline{ e_k }(x) \big|^2 \,     
    \E\!\left[
      \left|
        \int_0^t
        \left( 
          e^{ -\mu_k(t-s) }
          -
          e^{ -\mu_k(t-\fl{s}) }
        \right) \left< e_k, dW_s \right>_H
      \right|^2 
    \right] 
  \right)^{\!\!\nicefrac{\tilde{p}}{2}} dx 
\\&\leq 
  \E\big[ 
    | Y |^{\tilde{p}}
  \big] 
  \left(
    2
    \sum_{ k=1 }^N     
    \E\!\left[
      \left|
        \int_0^t
        \left( 
          e^{ -\mu_k(t-s) }
          -
          e^{ -\mu_k(t-\fl{s}) }
        \right) \left< e_k, dW_s \right>_H
      \right|^2 
    \right] 
  \right)^{\!\!\nicefrac{\tilde{p}}{2}} .
\end{split}
\end{align}
Moreover, It{\^o}'s isometry,
the fact that 
\begin{equation}
  \forall \,
  x \in [0, \infty) ,
  r \in [0,1]
  \colon 
  x^r e^{-x} \leq 1 ,
\end{equation}
and the fact that 
\begin{equation}
  \forall \,
  x \in (0, \infty) ,
  r \in [0,1]
  \colon 
  (1-e^{-x})/x^r \leq 1
\end{equation}
show that for all  
$ t \in (0,T] $, $ k, M \in \N $
we have that
\begin{align}
\label{eq:OMNN_est_2}
\begin{split}
 &\E\!\left[
    \left|
      \int_0^t
      \left( 
	e^{ -\mu_k(t-s) }
	-
	e^{ -\mu_k(t-\fl{s}) }
      \right) \left< e_k, dW_s \right>_H
    \right|^2 
  \right]
  =
  \int_0^t
  \left|
    e^{ -\mu_k(t-s) }
    -
    e^{ -\mu_k(t-\fl{s}) }
  \right|^2 ds
\\&=
  \int_0^t
  \big[
    e^{ -\mu_k(t-s) }
    ( 
      e^{-\mu_k(t-s)}
      -
      e^{ -\mu_k(t-\fl{s}) }
    ) 
    +
    e^{ -\mu_k(t-\fl{s}) }
    ( 
      e^{ -\mu_k(t-\fl{s}) }
      -
      e^{-\mu_k(t-s)}
    ) 
  \big] \, ds
\\&\leq
  2
  \int_0^t
  \left| 
    e^{-\mu_k(t-s)}
    -
    e^{ -\mu_k(t-\fl{s}) }
  \right| ds
  =
  2
  \int_0^t
  e^{-\mu_k(t-s)}
  \left(
    1
    -
    e^{ -\mu_k(s-\fl{s}) }
  \right) ds
\\&\leq
  2
  \int_0^t
  \left[ 
    \mu_k (t-s) 
  \right]^{-\varepsilon}
  \left[ 
    \mu_k (s-\fl{s}) 
  \right]^{2\theta} ds
  \leq
  \frac{ 
    2 T^{2\theta}
  }{ 
    (\mu_k)^{(\varepsilon-2\theta)}
    M^{2\theta}
  }
  \int_0^t
  \left( t-s \right)^{-\varepsilon} ds
\\&\leq
  \frac{ 
    2 T^{(1+2\theta-\varepsilon)}
  }{ 
    (\mu_k)^{(\varepsilon-2\theta)}
    ( 1 - \varepsilon )
    M^{2\theta}
  } .
\end{split}
\end{align}
Combining this and~\eqref{eq:OMNN_est_1}
demonstrates that for all $ t \in [0,T] $,
$ M,N \in \N $ we have that
\begin{align}
\begin{split}
  \E\!\left[
    \| 
      P_{N} 
      O_t
      -
      \OMNN{t}
    \|_{ L^{q}( \lambda_{(0,1)} ; \R ) }^{\tilde{p}}
  \right]
 &\leq
  \E\big[ 
    | Y |^{\tilde{p}}
  \big]
  \left(
    2
    \sum_{ k = 1 }^N
    \frac{ 
      2 T^{(1+2\theta-\varepsilon)}
    }{ 
      (\mu_k)^{(\varepsilon-2\theta)}
      ( 1 - \varepsilon )
      M^{2\theta}
    }
  \right)^{\!\!\nicefrac{\tilde{p}}{2}}
\\&=
  E\big[ 
    | Y |^{\tilde{p}}
  \big]
  \left(
    \frac{ 
      4 T^{(1+2\theta-\varepsilon)}
    }{ 
      ( 1 - \varepsilon )
      M^{2\theta}
    }
    \sum_{ k = 1 }^N
    \frac{1}{ (\mu_k)^{(\varepsilon-2\theta)} }
  \right)^{\!\!\nicefrac{\tilde{p}}{2}} .
\end{split}
\end{align}
Hence, we obtain that 
for all $ t \in [0,T] $,
$ M,N \in \N $ we have that
\begin{align}
\begin{split}
 &\| 
    P_{N} 
    O_t
    -
    \OMNN{t}
  \|_{ \L^p( \P; L^q( \lambda_{(0,1)} ; \R ) ) }
  \leq
  \| 
    Y
  \|_{ \L^{\tilde{p}}( \P; \R ) }
  \left[ 
    \frac{ 
      4 T^{(1+2\theta-\varepsilon)}
    }{ 
      ( 1 - \varepsilon )
      M^{2\theta}
    }
    \sum_{ k = 1 }^N
    \frac{1}{ (\mu_k)^{(\varepsilon-2\theta)} }
  \right]^{\!\nicefrac{1}{2}} .
\end{split}
\end{align}
H{\"o}lder's inequality
and the fact that 
$
  2(\varepsilon-2\theta) > 1
$
therefore prove that
\begin{align}
\label{eq:OMNN_est_3}
\begin{split}
 &\adjustlimits 
  \sup_{ M,N \in \N }
  \sup_{ t \in (0,T] }
  \left(
    M^{\theta} \,
    \| 
      P_{N} 
      O_t
      -
      \OMNN{t}
    \|_{ \L^p( \P; L^q( \lambda_{(0,1)} ; \R ) ) }
  \right)
\\&\leq
  \adjustlimits 
  \sup_{ M,N \in \N }
  \sup_{ t \in (0,T] }
  \left(
    M^{\theta} \,
    \| 
      P_{N} 
      O_t
      -
      \OMNN{t}
    \|_{ \L^{\tilde{p}}( \P; L^{q}( \lambda_{(0,1)} ; \R ) ) }
  \right)
\\&\leq
  \sup_{ M, N \in \N }
  \left(
    M^{\theta} \,
    \| 
      Y
    \|_{ \L^{\tilde{p}}( \P; \R ) }
    \left[
      \frac{ 
        4 T^{(1+2\theta-\varepsilon)}
      }{ 
        ( 1 - \varepsilon )
        M^{2\theta}
      }
      \sum_{ k=1 }^N
      \frac{ 
        1
      }{ 
        (\mu_k)^{(\varepsilon-2\theta)}
      }
    \right]^{\!\nicefrac{1}{2}}
  \right)
\\&=
  \frac{ 
    2 T^{(\nicefrac{1}{2}+\theta-\nicefrac{\varepsilon}{2})}
    \| 
      Y
    \|_{ \L^{\tilde{p}}( \P; \R ) }
  }{ 
    | \sqrt{\nu} \pi |^{(\varepsilon-2\theta)}
    \sqrt{ 1 - \varepsilon }
  }
  \left[
    \sum_{ k=1 }^{\infty}
    \frac{ 
      1
    }{ 
      k^{2(\varepsilon-2\theta)}
    }
  \right]^{\!\nicefrac{1}{2}}
  <
  \infty .
\end{split}
\end{align}
The proof of Lemma~\ref{lem:O_prop_3}
is thus completed.
\end{proof}

\begin{corollary}
\label{cor:O_properties}
Assume the setting in Section~\ref{sec:examples_setting}
and let $ p,q \in [2,\infty) $, 
$ \theta \in [\nicefrac{1}{4}-\nicefrac{1}{2q},\nicefrac{1}{4}) $, 
$ \xi \in \cup_{ r \in (\nicefrac{1}{4},\infty) \cap [2\theta,\infty) } \L^p( \P; H_r ) $.
Then there exist stochastic processes
$ O \colon [0,T] \times \Omega \rightarrow L^{q}( \lambda_{(0,1)} ; \R ) $ and
$ \OMNN{} \colon [0,T] \times \Omega \rightarrow P_N(H) $,
$ M, N \in \N $,
with continuous sample paths which satisfy 
\begin{enumerate}[(i)]
 \item\label{it:O_properties_1} that for all $ t \in [0,T] $
 we have that
 $
  [ 
    O_t - e^{tA} \xi
  ]_{ \P, \B(H) }
  =
  \int_0^t
  e^{(t-s)A} \, dW_s
 $,
 \item\label{it:O_properties_2} that for all $ M, N \in \N $, $ t \in [0,T] $
 we have that
 $
    [ 
      \OMNN{t} - P_{N} e^{tA} \xi
    ]_{ \P, \B(H) }
    =
    \int_0^t
    P_{N}
    e^{(t-\fl{s})A} \, dW_s
 $, 
 and
 \item\label{it:O_properties_3} that
 for all $ \gamma \in [0,2\theta] \cap [0,\nicefrac{1}{4}) $
 we have that
 \begin{align}
 \begin{split} 
  &\sup_{ M,N \in \N }
   \sup_{ t \in [0,T] }
   \left[
     M^{\theta} \!
     \left(
       \E\big[  
         \| 
           P_{N} (O_t - O_{\fl{t}} )      
         \|_{ L^{q}( \lambda_{(0,1)} ; \R ) }^p
         +
         \| 
           P_{N} O_t - \OMNN{t}     
         \|_{ L^{q}( \lambda_{(0,1)} ; \R ) }^p
       \big]
     \right)^{\!\nicefrac{1}{p}}
   \right]
\\&\quad+
  \sup_{ M,N \in \N }
  \sup_{ t \in [0,T] }
  \E\big[
    \| 
      \OMNN{t}
    \|_{ H_{\gamma} }^p
  \big]
  +
  \sup_{ N \in \N }
  \sup_{ t \in [0,T] }
  N^{2\theta} \!
  \left(
    \E\big[
      \| 
        O_t - P_{N} O_t
      \|_{ L^{q}( \lambda_{(0,1)} ; \R ) }^p
    \big]
  \right)^{\!\nicefrac{1}{p}}
\\&\quad+
  \sup_{ M,N \in \N }
  \E\big[
    {\textstyle \sup_{ t \in [0,T] } }
    \| \OMNN{t} \|_{ L^{\infty}( \lambda_{ (0,1) } ; \R ) }^p
  \big]
  <
  \infty .
 \end{split}
 \end{align}
\end{enumerate}
\end{corollary}
\begin{proof}[Proof of Corollary~\ref{cor:O_properties}]
First of all, note that
Lemma~\ref{lem:O_prop_1} implies
that there exists a stochastic process
$ O \colon [0,T] \times \Omega \rightarrow L^{q}( \lambda_{(0,1)} ; \R ) $ with
continuous sample paths which satisfies
for all $ t \in [0,T] $ that
$
  [ 
    O_t - e^{tA} \xi
  ]_{ \P, \B(H) }
  =
  \int_0^t
  e^{(t-s)A} \, dW_s
$
and
\begin{align} 
\label{eq:O_properties_1}
\begin{split}
 &\adjustlimits
  \sup_{ N \in \N }
  \sup_{ 0 \leq s < t \leq T }
  \frac{ 
  \| 
    P_{N} (O_t - O_{s} )      
  \|_{ \L^p( \P; L^{q}( \lambda_{(0,1)} ; \R ) ) }
  }{ (t-s)^{\theta} }
\\&\quad+
  \adjustlimits
  \sup_{ N \in \N }
  \sup_{ t \in [0,T] }
  \left(
    N^{2\theta} \,
    \| 
      O_t - P_{N} O_t
    \|_{ \L^p( \P; L^{q}( \lambda_{(0,1)} ; \R ) ) }
  \right)
  <
  \infty .
\end{split}
\end{align}
Hence, we obtain that
\begin{align}
\label{eq:O_properties_2}
\begin{split}
 &\sup_{ M, N \in \N }
  \sup_{ t \in [0,T] }
  \left[ 
    M^{\theta} \,
    \| P_N( O_t - O_{\fl{t}} ) \|_{ \L^p( \P; L^q( \lambda_{ (0,1) }; \R ) ) }
  \right]
\\&\leq
  \left[ 
    \sup_{ M \in \N }
    \sup_{ t \in [0, T] }
    M^{\theta} \,
    ( t - \fl{t} )^{\theta}
  \right]
  \left[
    \adjustlimits
    \sup_{ N \in \N }
    \sup_{ 0 \leq s < t \leq T }
    \frac{ 
    \| 
      P_{N} (O_t - O_{s} )      
    \|_{ \L^p( \P; L^{q}( \lambda_{(0,1)} ; \R ) ) }
    }{ (t-s)^{\theta} }
  \right]
\\&\leq
  T^{\theta}
  \left[
    \adjustlimits
    \sup_{ N \in \N }
    \sup_{ 0 \leq s < t \leq T }
    \frac{ 
    \| 
      P_{N} (O_t - O_{s} )      
    \|_{ \L^p( \P; L^{q}( \lambda_{(0,1)} ; \R ) ) }
    }{ (t-s)^{\theta} }
  \right]
  <
  \infty .
\end{split}
\end{align}
Next note that the
assumption that
$ 
  \xi 
  \in 
  \cup_{ r \in (\nicefrac{1}{4},\infty) \cap [2\theta,\infty) } 
  \L^p( \P; H_r ) 
$
and
Lemma~\ref{lem:O_prop_2}
(with $p=p$,
$\theta=\vartheta$,
$\xi = \Omega \ni \omega \mapsto \xi(\omega) \in H_{\vartheta} $
for $ \vartheta \in [0,2\theta] \cap [0,\nicefrac{1}{4}) $
in the notation of Lemma~\ref{lem:O_prop_2})
yield that there exist stochastic processes
$ \OMNN{} \colon [0,T] \times \Omega \rightarrow P_N(H) $,
$ M, N \in \N $,
with continuous sample paths which satisfy 
for all $ t \in [0,T] $, $ M,N \in \N $ that
\begin{equation}
\label{eq:O_properties_2a}
  [ 
    \OMNN{t} - P_{N} e^{tA} \xi
  ]_{ \P, \B(H) }
  =
  \int_0^t
  P_{N}
  e^{(t-\fl{s})A} \, dW_s
\end{equation}
and which satisfy for all 
$ \vartheta \in [0,2\theta] \cap [0,\nicefrac{1}{4}) $ that
\begin{align}
\label{eq:O_properties_3}
 &\sup_{ M,N \in \N }
  \sup_{ t \in [0,T] }
  \| 
    \OMNN{t}
  \|_{ \L^p( \P; H_{\vartheta} ) }
    <
  \infty .
\end{align}
Next observe that~\eqref{eq:O_properties_2a}
together with 
the assumption that
$ 
  \xi 
  \in 
  \cup_{ r \in (\nicefrac{1}{4},\infty) \cap [2\theta,\infty) } 
  \L^p( \P; H_r ) 
$
together with~\eqref{it:O_properties_2}
enables us to apply 
Lemma~\ref{lem:O_sobo}
to obtain that
\begin{align}
\label{eq:O_properties_5}
\begin{split}
 &\sup_{ M,N \in \N }
  \big\|
    {\textstyle \sup_{ t \in [0,T] } }
    \| \OMNN{t} \|_{ L^{\infty}( \lambda_{ (0,1) } ; \R ) }
  \big\|_{ \L^p(\P; \R) }
  <
  \infty .
\end{split}
\end{align}
Moreover, note 
that the fact that
$ \E\big[ \| \xi \|_{H_{(\nicefrac{1}{4}-\nicefrac{1}{2q})}}^p \big] < \infty $
and the fact that
\begin{equation}
  H_{(\nicefrac{1}{4}-\nicefrac{1}{2q})}
  \subseteq 
  W^{\nicefrac{1}{2}-\nicefrac{1}{q},2}((0,1),\R)
  \subseteq 
  W^{0,q}((0,1),\R)
  =
  L^{q}( \lambda_{ (0,1) } ; \R )
\end{equation}
continuously (cf., e.g., Da Prato \& Zabcyk~\cite[(A.46) in Section~A.5.2]{dz92}
and Lunardi~\cite{l09})
prove that
\begin{align}
\label{eq:O_properties_5a}
\begin{split}
 &\|
    \xi
  \|_{ \L^p( \P; L^{q}( \lambda_{ (0,1) } ; \R ) ) }
\\&\leq
  \left[ 
    \sup_{ v \in H_{(\nicefrac{1}{4}-\nicefrac{1}{2q})} \setminus \{0\} }
    \frac{ 
      \| v \|_{ L^{q}( \lambda_{ (0,1) } ; \R ) }
    }{ 
      \| v \|_{ H_{(\nicefrac{1}{4}-\nicefrac{1}{2q})} }
    }
  \right] 
  \|
    \xi
  \|_{ \L^p( \P; H_{(\nicefrac{1}{4}-\nicefrac{1}{2q})} ) }
  <
  \infty .
\end{split}
\end{align}
This together with~\eqref{eq:O_properties_1}
and~\eqref{eq:O_properties_2a}
allows us to apply Lemma~\ref{lem:O_prop_3}
to obtain that
\begin{align}
\label{eq:O_properties_6}
\begin{split}
 &\sup_{ M,N \in \N }
  \sup_{ t \in [0,T] }
  \left(
    M^{\theta}
    \| P_N O_t - \OMNN{t} \|_{ \L^p( \P; L^q( \lambda_{ (0,1) }; \R ) ) }
  \right)
  <
  \infty .
\end{split}
\end{align}
Combining this with~\eqref{eq:O_properties_1}--\eqref{eq:O_properties_5}
and~\eqref{eq:O_properties_5a}
completes the proof of Corollary~\ref{cor:O_properties}.
\end{proof}

\subsection{Strong convergence rates for 
numerical approximations of stochastic
Allen-Cahn equations}

\begin{lemma}
\label{lem:Ginzburg}
Assume the setting in
Section~\ref{sec:examples_setting},
let $ p \in [2,\infty) $, $ \vartheta \in (0,\infty) $, 
$ \theta \in [\nicefrac{1}{6}, \nicefrac{1}{4} )$,
$ 
  \xi 
  \in 
  \cup_{ r \in (\nicefrac{1}{4},\infty) \cap [2\theta,\infty) } 
  \L^{16p\max\{3,\vartheta\}}(\P; H_r ) 
$,
$ \gamma \in (\nicefrac{1}{6}, \nicefrac{1}{4}) $,
$ 
  \chi 
  \in 
  (0, \nicefrac{\gamma}{3}-\nicefrac{1}{18}
  ]
$,
let $ X \colon [0,T] \times \Omega \rightarrow L^{6}( \lambda_{(0,1)}; \R ) $
be a stochastic process with continuous sample paths
which satisfies for all $ t \in [0,T] $ that
$ 
  [ X_t - e^{tA}\xi - \int_0^t e^{(t-s)A} F(X_s) \, ds ]_{ \P, \B(H) } 
  = 
  \int_0^t e^{(t-s)A} \, dW_s 
$
and 
$
  \sup_{ s \in [0,T] }
  \E\big[
    \| X_s \|_{ L^{6}( \lambda_{(0,1)}; \R ) }^{12p}
  \big]
  <
  \infty
$,
and let $ \YMNN{} \colon [0,T] \times \Omega \rightarrow H_{\gamma}$,
$ M,N \in \N $, and 
$ \OMNN{} \colon [0,T] \times \Omega \rightarrow P_N(H) $, $ M,N \in \N $, 
be stochastic processes which satisfy for all
$ M,N \in \N $, $ t \in [0,T] $ that
$
  [ 
    \OMNN{t} 
    -
    P_N
    e^{tA}
    \xi
  ]_{ \P, \B(H) }
  =
  \int_0^t
  P_{ N }
  e^{ (t-\fl{s}) A } \, dW_s
$
and
\begin{align}
  \P\Big( 
    \YMNN{t} 
    =
    \textstyle\int_0^t
    P_N \,
    e^{(t-s)A} \,
    \one_{ 
      \{
	\| \YMNN{\fl{s}} \|_{ H_{\gamma} }
	+
	\| \OMNN{\fl{s}} \|_{ H_{\gamma} }
	\leq
	(M/T)^{\chi}
      \}
    }^{\Omega} \,
    F( \YMNN{\fl{s}} ) \, ds
    +
    \OMNN{t}
  \Big)
  =
  1 .
\end{align}
Then we have 
\begin{enumerate}[(i)]
\item\label{it:example_1}
that
$
  \sup_{ M,N \in \N }
  \sup_{ t \in [0,T] }
  \E\big[
    \| \YMNN{t} \|_{ H_{\gamma} }^{4p\max\{3,\vartheta\}}
  \big]
  <
  \infty
$ 
and
\item\label{it:example_2}
that there exists a real number
$ C \in \R $ such that for all
$ M,N \in \N $ we have that
\begin{align}
  \sup_{ t \in [0,T] }
  \left(
    \E\big[
      \|
        X_t - \YMNN{t}
      \|_{ H }^p
    \big]
  \right)^{\!\nicefrac{1}{p}}
  \leq
  C
  (
    M^{-\min\{\vartheta\chi,\theta\}}
    +
    N^{-2\theta}
  ) .
\end{align}
\end{enumerate}
\end{lemma}
\begin{proof}[Proof of Lemma~\ref{lem:Ginzburg}]
Throughout this proof let 
$ 
  (V, \left\| \cdot \right\|\!_V) 
  = 
  ( 
    L^{6}(\lambda_{(0,1)}; \R ), 
    \left\| \cdot \right\|\!_{ L^{6}(\lambda_{(0,1)}; \R ) } 
  ) 
$
and let $ \epsilon \in (0,1) $,
$ C 
  \in 
  [\nicefrac{32}{\epsilon}
  \max\{ \nicefrac{|a_2|^2}{(|a_3| + \one_{\{0\}}^{\R}(a_3))}, |a_3| \},
  \infty)
$ be real numbers.
Note that, e.g., 
\cite[Lemma~6.7]{BeckerJentzen2016}
proves that for all 
$ v, w \in L^{18}(\lambda_{ (0,1) }; \R) $
with $ v-w \in H_1 $ we have that
\begin{align}
\begin{split}
 &\langle 
    v-w,
    A(v-w) + F(v) - F(w)
  \rangle_H
\\&\leq
  2 
  \max\{ 1, \nicefrac{1}{(|a_3| + \one_{\{0\}}^{\R}(a_3))} \}
  \max\!\left\{ 
    1,
    \max_{ k \in \{ 1,2\} }
    \big[ k | a_k | \big]^{2}
  \right\}
  \| v-w \|_H^2 .
\end{split}
\end{align}
The fact that
$
  H_1 \subseteq L^{18}(\lambda_{ (0,1) }; \R)
$
therefore implies that for all 
$ v,w \in H_1 $ we have that 
\begin{align}
\begin{split}
 &\langle 
    v-w,
    Av + F(v) -Aw - F(w)
  \rangle_H
\\&\leq
  2 
  \max\{ 1, \nicefrac{1}{(|a_3| + \one_{\{0\}}^{\R}(a_3))} \}
  \max\!\left\{ 
    1,
    \max_{ k \in \{ 1,2\} }
    \big[ k | a_k | \big]^{2}
  \right\}
  \| v-w \|_H^2 .
\end{split}
\end{align}
Combining this,
Lemma~\ref{lem:coercivity_mix}
(with $\epsilon=\epsilon$,
$c=C$ in the notation of Lemma~\ref{lem:coercivity_mix})
and Lemma~\ref{lem:F_lipschitz}
ensures that there exit a real number $ c \in (0,\infty) $ such 
that for all $ N \in \N $, $ v,w \in P_N(H) $, 
$ x \in \C( [0,T], H_1 ) $, $ t \in [0,T] $ we have that
\begin{align}
\label{eq:example_monoticity}
 &\langle 
    v-w, Av + F(v) - Aw - F(w)
  \rangle_H
  \leq
  c
  \| v-w \|_H^2 ,
\end{align}
\begin{align}
\label{eq:example_coercivity}
\begin{split}
 &\langle 
    v, 
    P_N F(v + x_t)
  \rangle_{H_{\nicefrac{1}{2}}}
  +
  C
  \left[
    \sup_{ s \in [0,T] }
    \| x_s \|_{ L^{\infty}( \lambda_{(0,1)} ; \R) }^4
    +
    1
  \right]
  \langle 
    v, 
    F(v + x_t)
  \rangle_H
\\&\leq
  \epsilon
  \| v \|_{ H_1 }^2
  +
  c
  \| v \|_{ H_{\nicefrac{1}{2}} }^2
  +
  c \, C
  \left[
    \sup_{ s \in [0,T] }
    \| x_s \|_{ L^{\infty}( \lambda_{(0,1)} ; \R) }^4
    +
    1
  \right]
  \| v \|_{ H }^2
\\&\quad+
  c \!
  \left[
    \sup_{ s \in [0,T] }
    \| x_s \|_{ L^{\infty}( \lambda_{(0,1)} ; \R) }^{8}
    +
    1
  \right] ,
\end{split}
\end{align}
and
\begin{align}
\label{eq:example_Lipschitz}
  \| F(v) - F(w) \|_H^2
  \leq
  c
  \| v - w \|_V^2 \,
  \big( 
    1 
    + 
    \| v \|_V^{4}
    + 
    \| w \|_V^{4}
  \big) .
\end{align}
Moreover, note that Corollary~\ref{cor:O_properties}
(with $ p=16p\max\{3,\vartheta\} $,
$ q=6 $, $ \theta=\theta $, $ \xi=\xi $
in the notation of Corollary~\ref{cor:O_properties})
and the fact that $ 2\theta > \gamma $ imply
that there exist stochastic processes 
$ O \colon [0,T] \times \Omega \rightarrow V $
and $ \OMNNp{} \colon [0,T] \times \Omega \rightarrow P_N(H) $, $ M,N \in \N $,
with continuous sample paths which satisfy that
\begin{align}
\label{eq:example_O_properties_1}
\begin{split}
  &\sup_{ M,N \in \N }
   \sup_{ t \in [0,T] }
   \Big(
   M^{\theta}
   \big\| 
      \| 
	P_{N} (O_t - O_{\fl{t}} )      
      \|_{ V }
      +
      \| 
	P_{N} O_t 
	-
	\OMNNp{t}     
      \|_{ V }
    \big\|_{ \L^{16p\max\{3, \vartheta\}}( \P; \R) }
    \Big)
\\&+
  \sup_{ M,N \in \N }
  \sup_{ t \in [0,T] }
  \big\| 
    \| 
      \OMNNp{t}
    \|_{ H_{\gamma} }
    +
    N^{2\theta}
    \| 
      O_t
      -
      P_{N}
      O_t
    \|_{ V }
  \big\|_{ \L^{16p\max\{3, \vartheta\}}( \P; \R) }
\\&+
  \sup_{ M,N \in \N }
  \big\|
    {\textstyle \sup_{ t \in [0,T] } }
    \| \OMNNp{t} \|_{ L^{\infty}( \lambda_{ (0,1) } ; \R ) }
  \big\|_{ \L^{16p\max\{3, \vartheta\}}(\P; \R) }
  <
  \infty ,
\end{split}
\end{align}
that for all $ t \in [0,T] $ we have that
\begin{align}
\label{eq:example_O_def}
  [ O_t - e^{tA} \xi ]_{ \P, \B(H) }
 &=
  \textstyle\int_0^t
  e^{(t-s)A} \, dW_s,
\end{align}
and that for all 
$ t \in [0,T] $, $ M,N \in \N $ we have that
\begin{align}
  [ \OMNNp{t} - P_N \, e^{tA} \xi ]_{ \P, \B(H) }
 &=
  \textstyle\int_0^t
  P_N \,
  e^{(t-\fl{s})A} \, dW_s .
\end{align}
Observe that~\eqref{eq:example_O_def}
and the fact that 
$ 
  \forall \, t \in [0,T] 
$,
$
  \forall \, M,N \in \N 
  \colon
  \P( \OMNN{t} = \OMNNp{t} ) = 1 
$
guarantee for all
$ t \in [0,T] $, $ M,N \in \N $ that
\begin{align}
\label{eq:example_XY_probability}
\begin{split}
 &\P\Big( 
    X_t
    =
    \textstyle\int_0^t 
    e^{(t-s)A} F(X_s) \, ds
    +
    O_t
  \Big)
\\&=
  \P\Big( 
    \YMNN{t}
    =
    \textstyle\int_0^t
    P_N \,
    e^{(t-s)A} \,
    \one_{ 
      \{
	\| \YMNN{\fl{s}} \|_{ H_{\gamma} }
	+
	\| \OMNNp{\fl{s}} \|_{ H_{\gamma} }
	\leq
	(M/T)^{\chi}
      \}
    }^{\Omega} \,
    F( \YMNN{\fl{s}} ) \, ds
    +
    \OMNNp{t}
  \Big) 
\\&=
  1 .
\end{split}
\end{align}
Combining~\eqref{eq:example_monoticity}--\eqref{eq:example_Lipschitz}
and~\eqref{eq:example_O_properties_1}--\eqref{eq:example_XY_probability}
allows us to apply Corollary~\ref{cor:main}
(with 
$ H = H $, $ \H = \{ e_k \colon k \in \N \} $,
$ T = T $, $ c = c $,
$ \varphi = 4 $,
$ \epsilon = \epsilon $,
$ \rho = \nicefrac{1}{6} $,
$ \gamma = \gamma $, 
$ \chi = \chi $,
$ \D = \{ \{ e_1 \}, \{ e_1, e_2 \}, \{ e_1, e_2, e_3 \}, \ldots \} $,
$ \mu(e_N) = -\nu \pi^2 N^2 $,
$ A = A $, $ H_r = H_r $, $ V = V $, $ F = F $,
$ 
  \phi 
  = 
  C([0,T], H_1) 
  \ni 
  w 
  \mapsto 
  C
  [
    \sup_{ t \in [0,T] }
$
$
    \| w_t \|_{ L^{\infty}( \lambda_{(0,1)} ; \R) }^4
    +
    1
  ]
  \in [0,\infty)
$,
$ 
  \Phi 
  = 
  C([0,T], H_1) 
  \ni 
  w
  \mapsto 
  c
  [
    \sup_{ t \in [0,T] }
    \| w_t \|_{ L^{\infty}( \lambda_{(0,1)} ; \R) }^8
    +
    1
  ]
  \in [0,\infty)
$,
$ P_{\{ e_1, e_2, \ldots, e_N \}}(v) = \sum_{ k=1 }^N \langle e_k,v \rangle_H \, e_k $, 
$ (\Omega, \F, \P) = (\Omega, \F, \P) $,
$ X = X $, $ O = O $, 
$ \mathcal{O}^{M,\{ e_1, e_2, \ldots, e_N \}} = [0,T] \times \Omega \ni (\omega,t) \mapsto \OMNNp{t}(\omega) \in H_1 $, 
$ \mathcal{X}^{M,\{ e_1, e_2, \ldots, e_N \}} = \YMNN{} $,
$ \theta = \theta $, $ \vartheta = \vartheta $,
$ p = p $, $ \varrho = \theta $
for $ N \in \N $, $ r \in \R $, $ v \in H $
in the notation of Corollary~\ref{cor:main})
to obtain that~\eqref{it:example_1} holds and 
that there exists a real number $ K \in (0,\infty) $
such that for all $ M, N \in \N $ we have that
\begin{align}
\begin{split}
 &\sup_{ t \in [0,T] }
  \| X_t - \YMNN{t} \|_{ \L^p(\P; H) }
\\&\leq
  K
  \left[ 
    M^{ -\min\{ \vartheta \chi, \theta \} }
    +
    \| (-A)^{-\theta} (\Id_H - P_N) \|_{ L(H) }
    +
    \sup_{ t \in [0,T] }
    \| O_t - P_N O_t \|_{ \L^{2p}(\P; V) }
  \right] .
\end{split}
\end{align}
The fact that
\begin{equation}
  \forall \, N \in \N
  \colon 
  \| (-A)^{-\theta} (\Id_H - P_N) \|_{ L(H) }
  =
  (\nu \pi^2 (N+1)^2)^{-\theta}
\end{equation}
hence yields that for all $ M, N \in \N $ 
we have that
\begin{align}
\begin{split}
 &\sup_{ t \in [0,T] }
  \| X_t - \YMNN{t} \|_{ \L^p(\P; H) }
\\&\leq
  K
  \left[ 
    M^{ -\min\{ \vartheta \chi, \theta \} }
    +
    N^{-2\theta}
    \left(
      \frac{ 
        N^{2\theta} 
      }{
        \nu^{\theta}
        \pi^{2\theta}
        (N+1)^{2\theta}
      }
      +
      \sup_{ t \in [0,T] }
      \left[ 
        N^{2\theta}
        \| O_t - P_N O_t \|_{ \L^{2p}(\P; V) }
      \right]
    \right)
  \right] 
\\&\leq
  K
  \left[ 
    M^{ -\min\{ \vartheta \chi, \theta \} }
    +
    N^{-2\theta}
    \left(
      \frac{ 
        1
      }{
        \nu^{\theta}
        \pi^{2\theta}
      }
      +
      \sup_{ t \in [0,T] }
      \left[ 
        N^{2\theta}
        \| O_t - P_N O_t \|_{ \L^{2p}(\P; V) }
      \right]
    \right)
  \right] 
\\&\leq
  K
  \max\!\left\{ 
    1,
    \frac{1}{\nu^{\theta}}
    +
    \sup_{ t \in [0,T] }
    \left[ 
      N^{2\theta}
      \| O_t - P_N O_t \|_{ \L^{2p}(\P; V) }
    \right]
  \right\}
  \left[ 
    M^{ -\min\{ \vartheta \chi, \theta \} }
    +
    N^{-2\theta}
  \right] .  
\end{split}
\end{align}
Combining this with~\eqref{eq:example_O_properties_1} 
completes the proof of Lemma~\ref{lem:Ginzburg}.
\end{proof}

\begin{corollary}
\label{cor:Ginzburg3}
Assume the setting in
Section~\ref{sec:examples_setting},
let
$ \xi \in \cap_{ p\in[1,\infty) } \L^{p}(\P; H_{\nicefrac{1}{2}}) $,
$ \gamma \in (\nicefrac{1}{6}, \nicefrac{1}{4} )$,
$ 
  \chi 
  \in 
  (0,  
    \nicefrac{\gamma}{3} - \nicefrac{1}{18}
  ]
$,
and let $ \YMNN{} \colon [0,T] \times \Omega \rightarrow P_N(H) $,
$ M,N \in \N $, and 
$ \OMNN{} \colon [0,T] \times \Omega \rightarrow P_N(H) $,
$ M,N \in \N $,
be stochastic processes which satisfy for all
$ M,N \in \N $, $ t \in [0,T] $ that
$
  [ 
    \OMNN{t} 
    -
    P_N 
    e^{tA}
    \xi
  ]_{ \P, \B(H) }
  =
  \int_0^t
  P_{ N } 
  e^{ (t-\fl{s}) A } \, dW_s
$
and
\begin{align}
\label{eq:Ginzburg3_a}
  \P\Big( 
    \YMNN{t} 
    =
    \textstyle\int_0^t
    P_N \,
    e^{(t-s)A} \,
    \one_{ 
      \{
	\| \YMNN{\fl{s}} \|_{ H_{\gamma} }
	+
	\| \OMNN{\fl{s}} \|_{ H_{\gamma} }
	\leq
	(M/T)^{\chi}
      \}
    }^{\Omega} \,
    F( \YMNN{\fl{s}} ) \, ds
    +
    \OMNN{t}
  \Big)
  =
  1 .
\end{align}
Then 
\begin{enumerate}[(i)]
\item\label{it:Ginzburg3_1} we have that
there exists an up to indistinguishability
unique stochastic process
$
  X \colon [0,T] \times \Omega \rightarrow L^{6}( \lambda_{(0,1)}; \R )
$
with continuous sample paths
which satisfies for all $ t \in [0,T] $, $ p \in (0,\infty) $
that
$ 
  \sup_{ s \in [0,T] } 
  \E\big[
    \| X_s \|_{ L^{6}( \lambda_{(0,1)}; \R ) }^p
  \big]
  <
  \infty
$
and 
\begin{align} 
  \big[ 
    X_t - e^{tA} \xi - \textstyle\int_0^t e^{(t-s)A} F(X_s) \, ds 
  \big]_{ \P, \B(H) }
  =
  \textstyle\int_0^t
  e^{(t-s)A} \, dW_s ,
\end{align}
\item\label{it:Ginzburg3_2}
we have for all $ p \in (0,\infty) $ 
that
$
  \sup_{ r \in (-\infty,\gamma] }
  \sup_{ M, N \in \N }
  \sup_{ t \in [0,T] }
  \E\big[
    \| \YMNN{t} \|_{ H_{r} }^p
  \big]
  <
  \infty
$,
and
\item\label{it:Ginzburg3_3}
we have for all 
$ p \in (0,\infty) $,
$ r \in [0,\nicefrac{1}{4}) $
that there exists a real number
$ C \in \R $ such that for all
$ M,N \in \N $
it holds that
\begin{align}
  \sup_{ t \in [0,T] }
  \left(
    \E\big[
      \|
        X_t - \YMNN{t}
      \|_{ H }^p
    \big]
  \right)^{\!\nicefrac{1}{p}}
  \leq
  C
  (
    M^{-r}
    +
    N^{-2r}
  ) .
\end{align}
\end{enumerate} 
\end{corollary}
\begin{proof}[Proof of Corollary~\ref{cor:Ginzburg3}]
Note that under the assumptions of Corollary~\ref{cor:Ginzburg3}
it is well known
(cf., e.g., \cite[Theorem~3.4.1~(ii) in Section~3.4, Lemma~2.4.2 in Section~2, and Definition~2.7 in Section~2]{MantheyZausinger99},
\cite[Lemma~28 in Section~3.2]{j09b}, Lemma~\ref{lem:F_lipschitz},
the hypothesis that
$ \xi \in \cap_{ p\in[1,\infty) } \L^{p}(\P; H_{\nicefrac{1}{2}}) $,
and \cite[(A.46) in Section~A.5.2]{dz92})
that there exist stochastic processes
$
  \tilde{X}_q \colon [0,T] \times \Omega \rightarrow L^{q}( \lambda_{ (0,1) }; \R )
$,
$ q \in \{6, 7, 8,\ldots\} $,
with continuous sample paths which satisfy
for all $ t \in [0,T] $, $ q \in \{6, 7, 8,\ldots\} $ that
$
  \sup_{ s \in [0,T] }
  \E\big[ \| \tilde{X}_{q,s} \|_{ L^{q}( \lambda_{ (0,1) }; \R ) }^q \big]
  <
  \infty
$
and
\begin{align} 
\label{eq:exact_def}
  \big[ 
    \tilde{X}_{q,t} 
    - 
    e^{tA} \xi 
    - 
    \textstyle\int_0^t 
    e^{(t-s)A}
    F(\tilde{X}_{q,s}) \, ds
  \big]_{ \P, \B(H) }
  =
  \textstyle\int_0^t
  e^{(t-s)A} \, dW_s .
\end{align}
Combining this with~\eqref{eq:Ginzburg3_a}, 
the assumption that 
$ \xi \in \cap_{ p\in[1,\infty) } \L^{p}(\P; H_{\nicefrac{1}{2}}) $,
and, e.g., Lemma~2.2 in Andersson et al.~\cite{AnderssonJentzenKurniawan2015}
allows us to apply Lemma~\ref{lem:Ginzburg} 
(with $ p = p $, 
$ \vartheta = \vartheta $,
$ \theta = r $,
$ 
  \xi 
  = 
  \xi
$,
$ \gamma = \gamma $,
$ \chi = \chi $,
$ X = [0,T] \times \Omega \ni (t,\omega) \mapsto \tilde{X}_{p,t}(\omega) \in L^{6}( \lambda_{ (0,1) }; \R ) $,
$ \YMNN{} = [0,T] \times \Omega \ni (t,\omega) \mapsto \YMNN{t}(\omega) \in H_{\gamma} $,
$ \OMNN{} = \OMNN{} $
for 
$ p \in \{6, 7, 8, \ldots \} $,
$ \vartheta \in [\nicefrac{r}{\chi}, \infty) $,
$ r \in [\nicefrac{1}{6}, \nicefrac{1}{4}) $,
$ M,N \in \N $
in the notation of Lemma~\ref{lem:Ginzburg})
to obtain that
there exists a function
$ C \colon [2,\infty) \times [\nicefrac{1}{6},\nicefrac{1}{4}) \rightarrow \R $ 
such that for all
$ p \in \{6, 7, 8, \ldots \} $,
$ r \in [\nicefrac{1}{6},\nicefrac{1}{4}) $,
$ M,N \in \N $
we have that
\begin{align}
\label{eq:specific convergence}
  \sup_{ t \in [0,T] }
  \|
    \tilde{X}_{p,t} - \YMNN{t}
  \|_{ \L^p(\P;H) }
  \leq
  C_{p,r}
  (
    M^{-r}
    +
    N^{-2r}
  ) .
\end{align}
Next observe that
the triangle inequality
ensures that for all 
$ p_1, p_2 \in \{6, 7, 8, \ldots \} $
with $ p_1 \leq p_2 $ we have that
\begin{align}
\begin{split}
 &\sup_{ t \in [0,T] }
  \|
    \tilde{X}_{p_1,t}
    -
    \tilde{X}_{p_2,t}
  \|_{ \L^{p_1}(\P;H) }
\\&=
  \limsup_{ M \rightarrow \infty }
  \limsup_{ N \rightarrow \infty }
  \sup_{ t \in [0,T] }
  \|
    \tilde{X}_{p_1,t}
    -
    \YMNN{t}
    +
    \YMNN{t}
    -
    \tilde{X}_{p_2,t}
  \|_{ \L^{p_1}(\P;H) }
\\&\leq
  \limsup_{ M \rightarrow \infty }
  \limsup_{ N \rightarrow \infty }
  \sup_{ t \in [0,T] }
  \left[ 
    \|
      \tilde{X}_{p_1,t}
      -
      \YMNN{t}
    \|_{ \L^{p_1}(\P;H) }
    +
    \|
      \tilde{X}_{p_2,t}
      -
      \YMNN{t}
    \|_{ \L^{p_1}(\P;H) }
  \right] .
\end{split}
\end{align}
H{\"o}lder's inequality and~\eqref{eq:specific convergence}
hence prove that for all 
$ p_1, p_2 \in \{6,7,8, \ldots \} $,
$ r \in (\nicefrac{1}{6},\nicefrac{1}{4}) $
with $ p_1 \leq p_2 $ we have that
\begin{align}
\begin{split}
 &\sup_{ t \in [0,T] }
  \|
    \tilde{X}_{p_1,t}
    -
    \tilde{X}_{p_2,t}
  \|_{ \L^{p_1}(\P;H) }
\\&\leq
  \limsup_{ M \rightarrow \infty }
  \limsup_{ N \rightarrow \infty }
  \sup_{ t \in [0,T] }
  \left[ 
    \|
      \tilde{X}_{p_1,t}
      -
      \YMNN{t}
    \|_{ \L^{p_1}(\P;H) }
    +
    \|
      \tilde{X}_{p_2,t}
      -
      \YMNN{t}
    \|_{ \L^{p_2}(\P;H) }
  \right]
\\&\leq
  C_{p_1,r}
  \limsup_{ M \rightarrow \infty }
  \limsup_{ N \rightarrow \infty }
  \left[ 
    M^{-r}
    +
    N^{-2r}  
  \right]
  +
  C_{p_2,r}
  \limsup_{ M \rightarrow \infty }
  \limsup_{ N \rightarrow \infty }
  \left[ 
    M^{-r}
    +
    N^{-2r}  
  \right]
  =
  0 .
\end{split}
\end{align}
This implies that for all 
$
  q_1, q_2
  \in
  \{ 6,7,8, \ldots \}
$,
$ t \in [0,T] $
we have that
\begin{align}
\label{eq:Ginzburg3_P}
  \P\big( 
    \tilde{X}_{q_1,t} = \tilde{X}_{q_2,t} 
  \big) = 1 . 
\end{align}
In the next step let 
$ \tilde{\Omega} \subseteq \Omega $
be the set given by
\begin{align}
\label{eq:Ginzburg3_b}
\begin{split}
  &\tilde{\Omega}
   =
   \big\{
     \omega \in \Omega
     \colon
     \big( \,
       \forall \, q_1, q_2 \in \{ 6,7,8, \ldots \},
       t \in [0,T] 
       \colon 
       \tilde{X}_{q_1,t}(\omega) = \tilde{X}_{q_2,t}(\omega)
     \big)
   \big\}
\end{split}
\end{align}
and let 
$ X \colon [0,T] \times \Omega \rightarrow L^{6}( \lambda_{ (0,1) }; \R ) $
be the function which satisfies for all $ t \in [0,T], \omega \in \Omega$
that 
\begin{align}
\label{eq:Ginzburg3_c}
  X_t(\omega) 
  =
  \begin{cases}
    \tilde{X}_{6,t}(\omega) & \colon \omega \in \tilde{\Omega} \\
    0 & \colon \omega \in \Omega \setminus \tilde{\Omega} .
  \end{cases}
\end{align}
Note that the fact that every 
$ q \in \{ 6,7,8, \ldots \} $
we have that $ \tilde{X}_q \colon [0,T] \times \Omega \rightarrow L^{q}( \lambda_{ (0,1) }; \R ) $ 
has continuous sample paths 
shows that
\begin{align}
\begin{split}
   \tilde{\Omega}
  &=
   \big\{
     \omega \in \Omega
     \colon
     \big( \,
       \forall \, q_1, q_2 \in \{ 6,7,8, \ldots \} ,
       t \in [0,T] \cap \mathbb{Q} 
       \colon 
       \tilde{X}_{q_1,t}(\omega) = \tilde{X}_{q_2,t}(\omega)
     \big)
   \big\}
\\&=
  \cap_{ q_1, q_2 \in \{ 6,7,8, \ldots \} }
  \cap_{ t \in [0,T] \cap \mathbb{Q} }
  \big\{ 
    \tilde{X}_{q_1,t}(\omega) = \tilde{X}_{q_2,t}(\omega)
  \big\} .
\end{split}
\end{align}
Combining this with~\eqref{eq:Ginzburg3_P} ensures that
that 
\begin{equation}
\label{eq:Ginzburg3_d}
 \tilde{\Omega}
 \in 
 \F
 \qquad 
 \text{and}
 \qquad 
 \P( \tilde{\Omega} ) = 1 .
\end{equation}
Next observe that the fact that
$ \tilde{X} \colon [0,T] \times \Omega \rightarrow L^{6}( \lambda_{ (0,1) }; \R ) $
has continuous sample paths demonstrates that $ X $ has continuous sample paths.
Moreover, note that~\eqref{eq:Ginzburg3_b}, \eqref{eq:Ginzburg3_c},
and~\eqref{eq:Ginzburg3_d} ensure that
for all $ q \in \{ 6, 7, 8, \ldots \} $ we have that
\begin{equation}
 \P( \forall \, t \in [0,T] \colon \tilde{X}_{q,t} = X_t ) = 1 .
\end{equation}
Combining this with~\eqref{eq:exact_def}
demonstrates that for all $ t \in [0,T] $, $ p \in (0,\infty) $ 
we have that
\begin{equation}
  \sup_{ s \in [0,T] } 
  \E\big[
    \| X_s \|_{ L^{6}( \lambda_{(0,1)}; \R ) }^p
  \big]
  <
  \infty
\end{equation}
and 
\begin{align} 
  \big[ 
    X_t - e^{tA} \xi - \textstyle\int_0^t e^{(t-s)A} F(X_s) \, ds 
  \big]_{ \P, \B(H) }
  =
  \textstyle\int_0^t
  e^{(t-s)A} \, dW_s .
\end{align}
This, the fact that $ X \colon [0,T] \times \Omega \rightarrow L^{6}( \lambda_{ (0,1) }; \R ) $
has continuous sample paths, and again Lemma~\ref{lem:Ginzburg}
(with $ p = p $, 
$ \vartheta = \vartheta $,
$ \theta = \theta $,
$ 
  \xi 
  = 
  \xi
$,
$ \gamma = \gamma $,
$ \chi = \chi $,
$ X = X $,
$ \YMNN{} = [0,T] \times \Omega \ni (t,\omega) \mapsto \YMNN{t}(\omega) \in H_{\gamma} $,
$ \OMNN{} = \OMNN{} $
for 
$ p \in [2, \infty) $,
$ \vartheta \in [\nicefrac{\theta}{\chi}, \infty) $,
$ \theta \in [\nicefrac{1}{6}, \nicefrac{1}{4}) $,
$ M,N \in \N $
in the notation of Lemma~\ref{lem:Ginzburg})
complete the proof of Corollary~\ref{cor:Ginzburg3}.
\end{proof}

\section{Lower and upper bounds for
strong approximation errors of numerical approximations
of linear stochastic heat equations}
\label{sec:lower_bounds_section}

\subsection{Setting}
\label{sec:lower_bounds_setting}
Consider the notation in Section~\ref{sec:notation},
let $ T, \nu \in (0,\infty) $,
$ 
  ( H, \langle \cdot, \cdot \rangle_H, \left\| \cdot \right\|\!_H ) 
  =
  ( 
    L^2(\lambda_{(0,1)}; \R), 
    \langle \cdot, \cdot \rangle_{ L^2(\lambda_{(0,1)}; \R) },
$
$
    \left\| \cdot \right\|\!_{ L^2(\lambda_{(0,1)}; \R) }
  )
$,
$ (e_n)_{ n \in \N } \subseteq H $,
$ (P_n)_{ n \in \N \cup \{ \infty \} } \subseteq L(H) $
satisfy for all $ m \in \N $, $ n \in \N \cup \{ \infty \} $, 
$ v \in H $ that
$ e_m = [ (\sqrt{2}\sin(m \pi x))_{ x \in (0,1) } ]_{ \lambda_{(0,1)}, \B(\R) } $
and 
$ P_n(v) = \sum_{ k=1 }^n \langle e_k, v \rangle_H \, e_k $,
let $ A \colon D(A) \subseteq H \rightarrow H $
be the Laplacian with Dirichlet boundary conditions on $ H $
times the real number $ \nu $,
let $ ( \Omega, \F, \P ) $ be a probability space,
let $ (W_t)_{ t \in [0,T] } $ be an $ \Id_H $-cylindrical 
Wiener process, 
and let $ O \colon [0,T] \times \Omega \rightarrow H $
and $ \OMNN{} \colon [0,T] \times \Omega \rightarrow H $,
$ M, N \in \N $, be stochastic processes which satisfy
for all $ t \in [0,T] $, $ M \in \N $, $N \in \N \cup \{ \infty \} $ that
$
  [ O_t ]_{ \P, \B(H) } 
  =
  \int_0^t
  e^{(t-s)A} \, dW_s
$
and 
$
  [ \OMNN{t} ]_{ \P, \B(H) } 
  =
  \int_0^t
  P_N
  e^{(t-\fl{s})A} \, dW_s
$.

\subsection{Lower and upper bounds for
Hilbert-Schmidt norms of Hilbert-Schmidt operators}

\begin{lemma}
\label{lem:eA_monoton}
Assume the setting in Section~\ref{sec:lower_bounds_setting}
and let
$ N \in \N \cup \{ \infty \} $,
$ s_1, s_2, t \in [0,\infty) $ with $ s_1 \leq s_2 $.
Then 
\begin{enumerate}[(i)]
 \item\label{it:eA_monoton_1} we have that
\begin{equation}
  \left( 
    \sum_{ n = 1 }^{\infty}
    \|
      P_N e^{s_1 A}
      ( \Id_H - e^{tA} ) \,
      e_n
    \|_{ H }^2
  \right)^{\!\nicefrac{1}{2}}
  \geq
  \left( 
    \sum_{ n = 1 }^{\infty}
    \|
      P_N e^{s_2 A}
      ( \Id_H - e^{tA} ) \,
      e_n
    \|_{ H }^2
  \right)^{\!\nicefrac{1}{2}}
\end{equation}
and 
\item\label{it:eA_monoton_2} we have that
\begin{equation}
  \left( 
    \sum_{ n = 1 }^{\infty}
    \|
      P_N e^{t A}
      ( \Id_H - e^{s_1 A} ) \,
      e_n
    \|_{ H }^2
  \right)^{\!\nicefrac{1}{2}}
  \leq 
  \left( 
    \sum_{ n = 1 }^{\infty}
    \|
      P_N e^{t A}
      ( \Id_H - e^{s_2 A} ) \,
      e_n
    \|_{ H }^2
  \right)^{\!\nicefrac{1}{2}} .
\end{equation}
\end{enumerate}

\end{lemma}

\begin{proof}[Proof of Lemma~\ref{lem:eA_monoton}]
Throughout this proof let 
$ ( \mu_n )_{ n \in \N } \subseteq \R $
satisfy for all $ n \in \N $ that
$ \mu_n = \nu \pi^2 n^2 $.
Next observe that
\begin{align}
\begin{split}
 &\sum_{ n = 1 }^{\infty}
  \|
    P_N e^{s_1 A}
    ( \Id_H - e^{tA} ) \, 
    e_n
  \|_{ H }^2
\\&=
  \sum_{ n = 1 }^N
  \| 
    e^{s_1 A}
    ( \Id_H - e^{tA} ) \,
    e_n
  \|_H^2
  =
  \sum_{ n = 1 }^N
  \| 
    e^{-\mu_n s_1}
    ( 1 - e^{-\mu_n t} ) \,
    e_n
  \|_H^2
\\&=
  \sum_{ n = 1 }^N
  | 
    e^{-\mu_n s_1}
    ( 1 - e^{-\mu_n t} )
  |^2
  \geq
  \sum_{ n = 1 }^N
  | 
    e^{-\mu_n s_2}
    ( 1 - e^{-\mu_n t} )
  |^2
\\&=
  \sum_{ n = 1 }^N
  \| 
    e^{s_2 A}
    ( \Id_H - e^{tA} ) \,
    e_n
  \|_H^2
  =
  \sum_{ n = 1 }^{\infty}
  \|
    P_N e^{s_2 A}
    ( \Id_H - e^{tA} ) \,
    e_n
  \|_{ H }^2 .
\end{split} 
\end{align}
This establishes~\eqref{it:eA_monoton_1}.
Moreover, note that
\begin{align}
\begin{split}
 &\sum_{ n = 1 }^{\infty}
  \|
    P_N e^{t A}
    ( \Id_H - e^{s_1 A} ) \,
    e_n
  \|_{ H }^2
\\&=
  \sum_{ n = 1 }^N
  \| 
    e^{t A}
    ( \Id_H - e^{s_1 A} ) \,
    e_n
  \|_H^2
  =
  \sum_{ n = 1 }^N
  \| 
    e^{-\mu_n t}
    ( 1 - e^{-\mu_n s_1} ) \,
    e_n
  \|_H^2
\\&=
  \sum_{ n = 1 }^N
  | 
    e^{-\mu_n t}
    ( 1 - e^{-\mu_n s_1} )
  |^2
  \leq
  \sum_{ n = 1 }^N
  | 
    e^{-\mu_n t}
    ( 1 - e^{-\mu_n s_2} )
  |^2
\\&=
  \sum_{ n = 1 }^N
  \| 
    e^{t A}
    ( \Id_H - e^{s_2 A} ) \,
    e_n
  \|_H^2
  =
  \sum_{ n = 1 }^{\infty}
  \|
    P_N e^{t A}
    ( \Id_H - e^{s_2 A} ) \,
    e_n
  \|_{ H }^2 . 
\end{split}
\end{align}
The proof of Lemma~\ref{lem:eA_monoton}
is thus completed.
\end{proof}

\begin{lemma}
\label{lem:eA_bounds}
Assume the setting in Section~\ref{sec:lower_bounds_setting}
and let $ N \in \N \cup \{ \infty \} $, 
$ t \in (0,T] $. Then
\begin{multline}
  \left[ 
    \int_0^{\max\{0, t(N+1)^2-(1+\sqrt{t})^2\}} 
    \frac{ 
      ( 1 - e^{-\nu\pi^2 \min\{1,tN^2\}} )^2
    }{
      2 \nu \pi^2 (x+[1+\sqrt{T}]^2)^{\nicefrac{3}{2}} 
    } \, dx
  \right]^{\nicefrac{1}{2}}
\\
  \leq
  \|
    P_N
    (-\sqrt{t}A)^{-\nicefrac{1}{2}}
    ( \Id_H - e^{tA} )
  \|_{ HS(H) }
  \leq
  \left[ 
    \tfrac{ 
      1
    }{
      \pi \sqrt{\nu}
    }
    +
    \tfrac{ 
      1 
    }{
      \nu \pi^2
    }
    +
    4
    \pi
    \sqrt{\nu}
  \right]^{\nicefrac{1}{2}} .
\end{multline}

\end{lemma}
\begin{proof}[Proof of Lemma~\ref{lem:eA_bounds}]
Observe that
\begin{align}
\begin{split}
 &\tfrac{1}{\sqrt{t}}
  \|
    P_N
    (-A)^{-\nicefrac{1}{2}}
    ( \Id_H - e^{tA} )
  \|_{ HS(H) }^2
\\&=
  \tfrac{1}{\sqrt{t}}
  \sum_{ k=1 }^N
  \| 
    (-A)^{-\nicefrac{1}{2}}
    ( \Id_H - e^{tA} ) e_k
  \|_H^2
  =
  \tfrac{1}{\sqrt{t}}
  \sum_{ k=1 }^N
  \| 
    ( \nu \pi^2 k^2 )^{-\nicefrac{1}{2}}
    ( 1 - e^{-\nu\pi^2 k^2 t} ) e_k
  \|_H^2
\\&=
  \sum_{ k=1 }^N
  \frac{ 
    ( 1 - e^{-\nu\pi^2 k^2 t} )^2
  }{
    \nu \pi^2 k^2 \sqrt{t}
  }
  =
  \sum_{ k=1 }^N
  \int_k^{k+1} 
  \frac{ 
    ( 1 - e^{-\nu\pi^2 k^2 t} )^2
  }{
    \nu \pi^2 k^2 \sqrt{t}
  } \, dx
\\&\geq 
  \sum_{ k=1 }^N
  \int_k^{k+1} 
  \frac{ 
    ( 1 - e^{-\nu\pi^2 (x-1)^2 t} )^2
  }{
    \nu \pi^2 x^2 \sqrt{t}
  } \, dx 
\\&=
  \int_1^{N+1} 
  \frac{ 
    ( 1 - e^{-\nu\pi^2 (x-1)^2 t} )^2
  }{
    \nu \pi^2 x^2 \sqrt{t}
  } \, dx 
  \geq
  \int_{1+\min\{\nicefrac{1}{\sqrt{t}}, N\}}^{N+1} 
  \frac{ 
    ( 1 - e^{-\nu\pi^2 (x-1)^2 t} )^2
  }{
    \nu \pi^2 x^2 \sqrt{t}
  } \, dx .
\end{split}
\end{align}
This and the integral transformation theorem imply that
\begin{align}
\label{eq:eA_bounds_1}
\begin{split}
 &\tfrac{1}{\sqrt{t}}
  \|
    P_N
    (-A)^{-\nicefrac{1}{2}}
    ( \Id_H - e^{tA} )
  \|_{ HS(H) }^2
\\&\geq
  \int_{1+\min\{\nicefrac{1}{\sqrt{t}}, N\}}^{N+1} 
  \frac{ 
    ( 1 - e^{-\nu\pi^2 \min\{1,tN^2\}} )^2
  }{
    \nu \pi^2 x^2 \sqrt{t}
  } \, dx
\\&=
  \int_{(1+\min\{\nicefrac{1}{\sqrt{t}}, N\})^2}^{(N+1)^2} 
  \frac{ 
    ( 1 - e^{-\nu\pi^2\min\{1,tN^2\}} )^2
  }{
    2 \nu \pi^2 x \sqrt{xt} 
  } \, dx
\\&=
  \int_{t(1+\min\{\nicefrac{1}{\sqrt{t}}, N\})^2}^{t(N+1)^2}
  \frac{ 
    ( 1 - e^{-\nu\pi^2 \min\{1,tN^2\}} )^2
  }{
    2 \nu \pi^2 x \sqrt{x} 
  } \, dx
\\&=
  \int_{\min\{(1+\sqrt{t})^2,t(N+1)^2\}}^{t(N+1)^2}
  \frac{ 
    ( 1 - e^{-\nu\pi^2 \min\{1,tN^2\}} )^2
  }{
    2 \nu \pi^2 x \sqrt{x} 
  } \, dx
\\&=
  \int_0^{t(N+1)^2-\min\{(1+\sqrt{t})^2,t(N+1)^2\}} 
  \frac{ 
    ( 1 - e^{-\nu\pi^2 \min\{1,tN^2\}} )^2
  }{
    2 \nu \pi^2 (x+\min\{(1+\sqrt{t})^2,t(N+1)^2\})^{\nicefrac{3}{2}} 
  } \, dx 
\\&\geq
  \int_0^{\max\{0, t(N+1)^2-(1+\sqrt{t})^2\}} 
  \frac{ 
    ( 1 - e^{-\nu\pi^2 \min\{1,tN^2\}} )^2
  }{
    2 \nu \pi^2 (x+[1+\sqrt{T}]^2)^{\nicefrac{3}{2}} 
  } \, dx .
\end{split}
\end{align}
Moreover, note that
\begin{align}
\begin{split}
 &\tfrac{1}{\sqrt{t}}
  \|
    P_N
    (-A)^{-\nicefrac{1}{2}}
    ( \Id_H - e^{tA} )
  \|_{ HS(H) }^2
\\&=
  \tfrac{1}{\sqrt{t}}
  \sum_{ k=1 }^N
  \| 
    (-A)^{-\nicefrac{1}{2}}
    ( \Id_H - e^{tA} ) e_k
  \|_H^2
  =
  \tfrac{1}{\sqrt{t}}
  \sum_{ k=1 }^N
  \| 
    ( \nu \pi^2 k^2 )^{-\nicefrac{1}{2}}
    ( 1 - e^{-\nu\pi^2 k^2 t} ) e_k
  \|_H^2
\\&=
  \sum_{ k=1 }^N
  \frac{ 
    ( 1 - e^{-\nu\pi^2 k^2 t} )^2
  }{
    \nu \pi^2 k^2 \sqrt{t}
  } 
  =
  \frac{ 
    ( 1 - e^{-\nu\pi^2 t} )^2
  }{
    \nu \pi^2 \sqrt{t}
  }
  +
  \sum_{ k=2 }^N
  \int_{k-1}^k
  \frac{ 
    ( 1 - e^{-\nu\pi^2 k^2 t} )^2
  }{
    \nu \pi^2 k^2 \sqrt{t}
  } \, dx .
\end{split}
\end{align}
The fact that 
\begin{equation}
  \forall \, x \in (0,\infty)
  ,
  r \in [0,1]
  \colon
  x^{-r} 
  (1-e^{-x}) 
  \leq 
  1
  ,
\end{equation}
the fact that
\begin{equation}
  \forall \,
  x \in [1,\infty)
  \colon
  (x+1)^2 \leq 4 x^2
  ,
\end{equation}
and the integral transformation theorem
hence yield that
\begin{align}
\begin{split}
 &\tfrac{1}{\sqrt{t}}
  \|
    P_N
    (-A)^{-\nicefrac{1}{2}}
    ( \Id_H - e^{tA} )
  \|_{ HS(H) }^2
\\&\leq
  \frac{ 
    ( 1 - e^{-\nu\pi^2 t} )^{\nicefrac{3}{2}}
  }{
    \pi \sqrt{\nu} 
  }
  +
  \sum_{ k=2 }^N
  \int_{k-1}^k
  \frac{ 
    ( 1 - e^{-\nu\pi^2 (x+1)^2 t} )^2
  }{
    \nu \pi^2 x^2 \sqrt{t}
  } \, dx
\\&\leq 
  \frac{ 
    1
  }{
    \pi \sqrt{\nu}
  }
  +
  \int_{1}^{N}
  \frac{ 
    ( 1 - e^{-4\nu\pi^2 x^2 t} )^2
  }{
    \nu \pi^2 x^2 \sqrt{t}
  } \, dx
\\&=
  \frac{ 
    1
  }{
    \pi \sqrt{\nu}
  }
  +
  \int_{1}^{N^2}
  \frac{ 
    ( 1 - e^{-4\nu\pi^2 x t} )^2
  }{
    2 \nu \pi^2 x \sqrt{x t}
  } \, dx
  =
  \frac{ 
    1
  }{
    \pi \sqrt{\nu}
  }
  +
  \int_{t}^{tN^2}
  \frac{ 
    ( 1 - e^{-4\nu\pi^2 x} )^2
  }{
    2 \nu \pi^2 x \sqrt{x}
  } \, dx .
\end{split}
\end{align}
Again the fact that 
\begin{equation}
  \forall \, x \in (0,\infty)
  ,
  r \in [0,1]
  \colon
  x^{-r} 
  (1-e^{-x}) 
  \leq 
  1
\end{equation}
therefore ensures that
\begin{align}
\begin{split}
 &\tfrac{1}{\sqrt{t}}
  \|
    P_N
    (-A)^{-\nicefrac{1}{2}}
    ( \Id_H - e^{tA} )
  \|_{ HS(H) }^2
\\&\leq 
  \frac{ 
    1
  }{
    \pi \sqrt{\nu}
  }
  +
  \int_{0}^{\infty}
  \frac{ 
    ( 1 - e^{-4\nu\pi^2 x} )^2
  }{
    2 \nu \pi^2 x \sqrt{x}
  } \, dx
\\  & \leq
  \frac{ 
    1
  }{
    \pi \sqrt{\nu}
  }
  +
  2
  \int_{0}^{1}
  \frac{ 
    ( 1 - e^{-4\nu\pi^2 x} )
  }{
    \sqrt{x}
  } \, dx
  +
  \int_{1}^{\infty}
  \frac{ 
    1
  }{
    2 \nu \pi^2 x \sqrt{x}
  } \, dx
\\&\leq
  \frac{ 
    1
  }{
    \pi \sqrt{\nu}
  }
  +
  4
  \pi
  \sqrt{\nu}
  \int_{0}^{1}
  \sqrt{ 1 - e^{-4\nu\pi^2 x} } \, dx
  +
  \left[ 
    \frac{ 
      -1 
    }{
      \nu \pi^2 \sqrt{x}
    }
  \right]_{ x=1 }^{ x = \infty }
\\ &
\leq
  \frac{ 
    1
  }{
    \pi \sqrt{\nu}
  }
  +
  4
  \pi
  \sqrt{\nu}
  +
  \frac{ 
    1 
  }{
    \nu \pi^2
  } .
\end{split}
\end{align}
Combining this and~\eqref{eq:eA_bounds_1}
completes the proof of Lemma~\ref{lem:eA_bounds}.
\end{proof}

\subsection{Lower and upper bounds for 
strong approximation errors of temporal discretizations
of linear stochastic heat equations}

\begin{lemma}
\label{lem:lower_bound_1}
Assume the setting in Section~\ref{sec:lower_bounds_setting}
and let $ M \in \N $, $ N \in \N \cup \{ \infty \} $. Then
\begin{align}
\begin{split}
 &\frac{1}{M^{\nicefrac{1}{4}}}
  \left[ 
    \int_0^{\max\left\{0, \frac{T(N+1)^2}{2M}-\left[1+\frac{\sqrt{T}}{\sqrt{2M}}\right]^{2}\right\}} 
    \frac{ 
      \sqrt{T}
      \left[ 1 - e^{-\nu \pi^2 T} \right] \!
      \left[ 1 - \exp(-\nu\pi^2 \min\{1,\frac{TN^2}{2M}\}) \right]^2
    }{
      8 \nu \pi^2 \sqrt{2} 
      (x+[1+\sqrt{T}]^2)^{\nicefrac{3}{2}} 
    } \, dx
  \right]^{\nicefrac{1}{2}}
\\&\leq 
 \|
   P_N O_T - \OMNN{T} 
 \|_{ \L^2(\P; H ) }
 =
 \sup_{ t \in [0,T] }
 \| 
   P_N O_t - \OMNN{t} 
 \|_{ \L^2(\P; H ) }
\\&=
 \sup_{ t \in [0,T] }
 \left[
   \int_0^t
   \|
     P_N
     e^{(t-s)A}
     ( 
       \Id_H - e^{(s-\fl{s})A} 
     )
   \|_{ HS(H) }^2 \, ds
 \right]^{\nicefrac{1}{2}}
\\&\leq
 \frac{1}{M^{\nicefrac{1}{4}}}
 \left[
    \frac{\sqrt{T}}{ 2 }
    \left( 
      \frac{ 
	1
      }{
	\pi \sqrt{\nu}
      }
      +
      \frac{ 
	1 
      }{
	\nu \pi^2
      }
      +
      4
      \pi
      \sqrt{\nu}
    \right)
  \right]^{\nicefrac{1}{2}} .
\end{split}
\end{align}

\end{lemma}
\begin{proof}[Proof of Lemma~\ref{lem:lower_bound_1}]
Throughout this proof let
$ (\mu_n)_{ n \in \N } \subseteq \R $ satisfy
for all $ n \in \N $ that 
$ \mu_n = \nu \pi^2 n^2 $
and let $ \cl[h]{\cdot} \colon \R \rightarrow \R $, 
$ h \in (0,\infty) $,
be the functions which satisfy
for all $ h \in (0,\infty) $, $ t \in \R $ that 
$
  \cl[h]{t}
  =
  \min\!\left(
    \{ 0, h, -h, 2h, -2h, \ldots \} \cap [t,\infty)
  \right)
$.
Observe that Lemma~\ref{lem:eA_monoton}~\eqref{it:eA_monoton_1} 
ensures for all 
$ t \in [0,T) $ that
\begin{align}
\begin{split}
 &2
  \int_{\fl{t}}^{\fl{t}+\frac{T}{M}}
  \one_{[\fl{s}, \fl{s}+\frac{T}{2M}]}^{\R}(s) \,
  \|
    P_N
    e^{sA}
    ( 
      \Id_H - e^{\frac{T}{2M}A} 
    )
  \|_{ HS(H) }^2 \, ds
\\&=
  2
  \int_{\fl{t}}^{\fl{t}+\frac{T}{2M}}
  \|
    P_N
    e^{sA}
    ( 
      \Id_H - e^{\frac{T}{2M}A} 
    )
  \|_{ HS(H) }^2 \, ds
\\&\geq 
  \int_{\fl{t}}^{\fl{t}+\frac{T}{2M}}
  \|
    P_N
    e^{sA}
    ( 
      \Id_H - e^{\frac{T}{2M}A} 
    )
  \|_{ HS(H) }^2 \, ds
  +
  \int_{\fl{t}+\frac{T}{2M}}^{\fl{t}+\frac{T}{M}}
  \|
    P_N
    e^{sA}
    ( 
      \Id_H - e^{\frac{T}{2M}A} 
    )
  \|_{ HS(H) }^2 \, ds
\\&=
  \int_{\fl{t}}^{\fl{t}+\frac{T}{M}}
  \|
    P_N
    e^{sA}
    ( 
      \Id_H - e^{\frac{T}{2M}A} 
    )
  \|_{ HS(H) }^2 \, ds .
\end{split}
\end{align}
Therefore, we obtain that
\begin{align}
\label{eq:lower_bound_1_pre}
\begin{split}
 &2
  \int_{0}^{T}
  \one_{[\fl{s}, \fl{s}+\frac{T}{2M}]}^{\R}(s) \,
  \|
    P_N
    e^{sA}
    ( 
      \Id_H - e^{\frac{T}{2M}A} 
    )
  \|_{ HS(H) }^2 \, ds
\\&\geq 
  \int_{0}^{T}
  \|
    P_N
    e^{sA}
    ( 
      \Id_H - e^{\frac{T}{2M}A} 
    )
  \|_{ HS(H) }^2 \, ds .
\end{split}
\end{align}
Next note that It{\^o}'s isometry implies for all 
$ t \in [0,T] $ that
\begin{align}
\label{eq:lower_bound_1}
\begin{split}
 &\| 
    P_N O_t - \OMNN{t} 
  \|_{ \L^2(\P; H ) }^2
\\&=
  \E\!\left[ 
    \|
      P_N O_t - \OMNN{t}
    \|_H^2
  \right]
  =
  \E\!\left[ 
    \left\|
      \int_0^t
      P_N
      e^{(t-s)A}
      ( 
        \Id_H - e^{(s-\fl{s})A} 
      ) \, dW_s
    \right\|_H^2
  \right]
\\&=
  \int_0^t
  \|
    P_N
    e^{(t-s)A}
    ( 
      \Id_H - e^{(s-\fl{s})A} 
    )
  \|_{ HS(H) }^2 \, ds .
\end{split}
\end{align}
This, the fact that
$ 
  \forall \, s \in [0,T]
  \colon
  T - \fl{T-s} = \cl{s} 
$,
and Lemma~\ref{lem:eA_monoton}~\eqref{it:eA_monoton_2}
ensure that
\begin{align}
\begin{split}
 &\| 
    P_N O_T - \OMNN{T} 
  \|_{ \L^2(\P; H ) }^2
\\&=
  \int_0^T
  \|
    P_N
    e^{sA}
    ( 
      \Id_H - e^{(T-s-\fl{T-s})A} 
    )
  \|_{ HS(H) }^2 \, ds
\\&=
  \int_0^T
  \|
    P_N
    e^{sA}
    ( 
      \Id_H - e^{(\cl{s}-s)A} 
    )
  \|_{ HS(H) }^2 \, ds
\\&\geq
  \int_0^T
  \one_{[\fl{s}, \fl{s}+\frac{T}{2M}]}^{\R}(s) \,
  \|
    P_N
    e^{sA}
    ( 
      \Id_H - e^{(\cl{s}-s)A} 
    )
  \|_{ HS(H) }^2 \, ds
\\&\geq
  \int_0^T
  \one_{[\fl{s}, \fl{s}+\frac{T}{2M}]}^{\R}(s) \,
  \|
    P_N
    e^{sA}
    ( 
      \Id_H - e^{\frac{T}{2M}A} 
    )
  \|_{ HS(H) }^2 \, ds .
\end{split}
\end{align}
Inequality~\eqref{eq:lower_bound_1_pre} 
hence proves that
\begin{align}
\begin{split}
 &\| 
    P_N O_T - \OMNN{T} 
  \|_{ \L^2(\P; H ) }^2
\\&\geq
  \frac{1}{2}
  \left[ 
    2
    \int_0^T
    \one_{[\fl{s}, \fl{s}+\frac{T}{2M}]}^{\R}(s) \,
    \|
      P_N
      e^{sA}
      ( 
	\Id_H - e^{\frac{T}{2M}A} 
      )
    \|_{ HS(H) }^2 \, ds
  \right] 
\\&\geq
  \frac{1}{2}
  \int_0^T
  \|
    P_N
    e^{sA}
    ( 
      \Id_H - e^{\frac{T}{2M}A} 
    )
  \|_{ HS(H) }^2 \, ds 
  =
  \frac{1}{2}
  \int_0^T
  \sum_{ k=1 }^N
  \|
    e^{sA}
    ( 
      \Id_H - e^{\frac{T}{2M}A} 
    ) e_k
  \|_H^2 \, ds
\\&=
  \frac{1}{2}
  \int_0^T
  \sum_{ k=1 }^N
  |
    e^{-\mu_k s}
    ( 
      1 - e^{-\mu_k\frac{T}{2M}} 
    )
  |^2 \, ds
  =
  \frac{1}{2}
  \sum_{ k=1 }^N
  \frac{ 
    ( 1 - e^{-2\mu_k T} ) 
  }{ 2 \mu_k } \,
  | 1 - e^{-\mu_k\frac{T}{2M}} |^2 .
\end{split}
\end{align}
Lemma~\ref{lem:eA_bounds} 
therefore implies that
\begin{align}
\label{eq:lower_bound_2}
\begin{split}
 &\| 
    P_N O_T - \OMNN{T} 
  \|_{ \L^2(\P; H ) }^2
\\&\geq
  \frac{1}{4}
  ( 1 - e^{-\mu_1 T} ) 
  \sum_{ k=1 }^N
  \left| 
    \frac{ 
      (1 - e^{-\mu_k\frac{T}{2M}})
    }{ \sqrt{\mu_k} }
  \right|^2
  =
  \frac{1}{4}
  ( 1 - e^{-\mu_1 T} ) 
  \sum_{ k=1 }^N
  \| 
    (-A)^{-\nicefrac{1}{2}}
    ( 
      \Id_H - e^{\frac{T}{2M}A} 
    ) e_k
  \|_H^2
\\&=
  \frac{1}{4}
  ( 1 - e^{-\mu_1 T} ) 
  \| 
    P_N
    (-A)^{-\nicefrac{1}{2}}
    ( 
      \Id_H - e^{\frac{T}{2M}A} 
    ) 
  \|_{ HS(H) }^2
\\&\geq
  \frac{
    \sqrt{ T }
    ( 1 - e^{-\mu_1 T} ) 
  }{4 \sqrt{2M} }
  \left[ 
    \int_0^{\max\left\{0, \frac{T(N+1)^2}{2M}-\left[1+\frac{\sqrt{T}}{\sqrt{2M}}\right]^2\right\}} 
    \frac{ 
      \left[ 1 - \exp(-\nu\pi^2 \min\{1,\frac{TN^2}{2M}\}) \right]^2
    }{
      2 \nu \pi^2 (x+[1+\sqrt{T}]^2)^{\nicefrac{3}{2}} 
    } \, dx
  \right] .
\end{split}
\end{align}
In the next step observe
that~\eqref{eq:lower_bound_1} 
and Lemma~\ref{lem:eA_monoton}~\eqref{it:eA_monoton_2}
assure that
\begin{align}
\begin{split}
\sup_{ t \in [0,T] }
  \| 
    P_N O_t - \OMNN{t} 
  \|_{ \L^2(\P; H ) }^2
&\leq
  \sup_{ t \in [0,T] }
  \int_0^t
  \|
    P_N
    e^{(t-s)A}
    ( \Id_H - e^{\frac{T}{M}A} )
  \|_{ HS(H) }^2 \, ds
\\ & =
  \sup_{ t \in [0,T] }
  \int_0^t
  \|
    P_N
    e^{sA}
    ( \Id_H - e^{\frac{T}{M}A} )
  \|_{ HS(H) }^2 \, ds
\\&=
  \int_0^T
  \|
    P_N
    e^{sA}
    ( \Id_H - e^{\frac{T}{M}A} )
  \|_{ HS(H) }^2 \, ds
\\ & 
=
  \int_0^T
  \sum_{ k=1 }^N
  \|
    e^{sA}
    ( \Id_H - e^{\frac{T}{M}A} )
    e_k
  \|_H^2 \, ds .
\end{split}
\end{align}
Lemma~\ref{lem:eA_bounds}
hence yields that	
\begin{align}
\begin{split}
 &\sup_{ t \in [0,T] }
  \| 
    P_N O_t - \OMNN{t} 
  \|_{ \L^2(\P; H ) }^2
\\&\leq
  \int_0^T
  \sum_{ k=1 }^N
  | 
    e^{-\mu_k s}
    (
      1 - e^{-\mu_k \frac{T}{M} }
    )
  |^2 \, ds
  =
  \sum_{ k=1 }^N
  \frac{ (1-e^{-2\mu_k T} ) }{ 2 \mu_k }
  | 1 - e^{-\mu_k \frac{T}{M} } |^2
\\ &
\leq
  \frac{1}{2}
  \sum_{ k=1 }^N
  \left|
    \frac{ 
      ( 1 - e^{-\mu_k \frac{T}{M} } )
    }{ \sqrt{ \mu_k } }
  \right|^2
  =
  \frac{1}{2}
  \sum_{ k=1 }^N
  \|
    (-A)^{-\nicefrac{1}{2} } 
    ( \Id_H - e^{\frac{T}{M}A } ) e_k
  \|_H^2
\\&
  =
  \frac{1}{2}
  \|
    P_N
    (-A)^{-\nicefrac{1}{2} } 
    ( \Id_H - e^{\frac{T}{M}A } )
  \|_{ HS(H) }^2
\leq
  \frac{\sqrt{T}}{ 2 \sqrt{M} }
  \left[ 
    \frac{ 
      1
    }{
      \pi \sqrt{\nu}
    }
    +
    \frac{ 
      1 
    }{
      \nu \pi^2
    }
    +
    4
    \pi
    \sqrt{\nu}
  \right] .
\end{split}
\end{align}
Combining this with~\eqref{eq:lower_bound_1} and~\eqref{eq:lower_bound_2}
completes the proof of Lemma~\ref{lem:lower_bound_1}.
\end{proof}

In the next result, Corollary~\ref{cor:lower_bound_1},
we specialize Lemma~\ref{lem:lower_bound_1} to the case $ N = \infty $
where no spatial discretization is applied to the stochastic process 
$
  O \colon [0,T] \times \Omega \rightarrow H 
$.

\begin{corollary}
\label{cor:lower_bound_1}
Assume the setting in Section~\ref{sec:lower_bounds_setting}
and let $ M \in \N $. Then
\begin{align}
\begin{split}
 &\frac{1}{M^{\nicefrac{1}{4}}}
  \left[ 
    \int_0^{\infty} 
    \frac{ 
      \sqrt{T}
      ( 1 - e^{-\nu \pi^2 T} ) 
      ( 1 - e^{-\nu\pi^2} )^2
    }{
      8 \nu \pi^2 \sqrt{2} (x+[1+\sqrt{T}]^2)^{\nicefrac{3}{2}} 
    } \, dx
  \right]^{\nicefrac{1}{2}}
\\&\leq
  \liminf_{ N \rightarrow \infty }
  \| 
    P_N O_T - \mathcal{O}_T^{M,N} 
  \|_{ \L^2(\P; H ) }
  =
  \limsup_{ N \rightarrow \infty }
  \| 
    P_N O_T - \mathcal{O}_T^{M,N} 
  \|_{ \L^2(\P; H ) }
\\&=
  \| 
    O_T - \mathcal{O}_T^{M,\infty} 
  \|_{ \L^2(\P; H ) }
  =
  \adjustlimits\liminf_{ N \rightarrow \infty }
  \sup_{ t \in [0,T] }
  \| 
    P_N O_t - \mathcal{O}_t^{M,N} 
  \|_{ \L^2(\P; H ) }
\\&=
  \adjustlimits\limsup_{ N \rightarrow \infty }
  \sup_{ t \in [0,T] }
  \| 
    P_N O_t - \mathcal{O}_t^{M,N} 
  \|_{ \L^2(\P; H ) }
  =
  \sup_{ t \in [0,T] }
  \| 
    O_t - \mathcal{O}_t^{M,\infty} 
  \|_{ \L^2(\P; H ) }
\\&\leq
  \frac{1}{M^{\nicefrac{1}{4}}}
  \left[
    \frac{\sqrt{T}}{ 2 }
    \left( 
      \frac{ 
	1
      }{
	\pi \sqrt{\nu}
      }
      +
      \frac{ 
	1 
      }{
	\nu \pi^2
      }
      +
      4
      \pi
      \sqrt{\nu}
    \right)
  \right]^{\nicefrac{1}{2}} .
\end{split}
\end{align}

\end{corollary}

\subsection{Lower and upper bounds for 
strong approximation errors of spatial discretizations
of linear stochastic heat equations}

\begin{lemma}
\label{lem:PN_eA_convergence}
Assume the setting 
in Section~\ref{sec:lower_bounds_setting}.
Then
\begin{align}
 &\limsup_{ M \rightarrow \infty }
  \sup_{ N \in \N }
  \sup_{ t \in [0,T] }
  \left\|
    \int_0^t
    \left(
      P_N
      e^{(t-s)A}
      -
      P_N
      e^{(t-\fl{s})A}
    \right) dW_s
  \right\|_{ L^2(\P; H) }
  =
  0 .
\end{align}
\end{lemma}
\begin{proof}[Proof of Lemma~\ref{lem:PN_eA_convergence}]
Throughout this proof 
let $ \alpha \in (0, \nicefrac{1}{4}) $
and let
$ \beta \in (\nicefrac{1}{4}, \nicefrac{1}{2}-\alpha) $.
Note that
the fact that
$
  4\beta > 1
$
shows that
\begin{align}
\label{eq:PN_eA_convergence_1}
\begin{split}
&
  \sup_{ N \in \N }
  \|
    P_N
  \|_{ HS(H, H_{-\beta}) }^2
\\ &=
  \sup_{ N \in \N }
  \left[ 
    \sum_{ k=1 }^{ N }
    \|
      e_k
    \|_{ H_{-\beta} }^2
  \right]
  =
  \sum_{ k=1 }^{ \infty }
  \|
    (-A)^{-\beta}
    e_k
  \|_{ H }^2
\\&
  =
  \sum_{ k=1 }^{ \infty }
  |
    (\nu \pi^2 k^2)^{-\beta}
  |^2
=
  \sum_{ k=1 }^{ \infty }
  \frac{ 1 }{ ( \sqrt{\nu} \pi k )^{4\beta} }
  <
  \infty .
\end{split}
\end{align}
Next observe that
for all $ M, N \in \N $,
$ t \in [0,T] $
we have that
\begin{align}
\begin{split}
 &\int_0^t
  \|
    P_N
    e^{(t-s)A}
    (
      \Id_H
      -
      e^{(s-\fl{s})A}
    )
  \|_{ HS(H) }^2 \, ds
\\&\leq
  \|
    (-A)^{-\beta}
    P_N
  \|_{ HS(H) }^2
  \int_0^t
  \|
    (-A)^{\beta}
    e^{(t-s)A}
    (
      \Id_H
      -
      e^{(s-\fl{s})A}
    )
  \|_{ L(H) }^2 \, ds
\\&=
  \|
    P_N
  \|_{ HS(H, H_{-\beta}) }^2
  \int_0^t
  \|
    (-A)^{(\alpha+\beta)}
    e^{(t-s)A}
    (-A)^{-\alpha}
    (
      \Id_H
      -
      e^{(s-\fl{s})A}
    )
  \|_{ L( H ) }^2 \, ds
\\&\leq
  \|
    P_N
  \|_{ HS(H, H_{-\beta}) }^2
  \int_0^t
  \|
    (-A)^{(\alpha+\beta)}
    e^{(t-s)A}
  \|_{ L(H) }^2
  \|
    (-A)^{-\alpha}
    (
      \Id_H
      -
      e^{(s-\fl{s})A}
    )
  \|_{ L( H ) }^2 \, ds .
\end{split}
\end{align}
The fact that 
\begin{equation}
  \forall \,
  s \in [0,\infty) 
  ,
  r \in [0,1]
  \colon 
  \| (-sA)^r e^{sA} \|_{ L(H) }
  \leq 
  1
\end{equation}
and the fact that
\begin{equation}
  \forall \,
  s \in (0,\infty) 
  ,
  r \in [0,1]
  \colon 
  \| (-sA)^{-r} (\Id_H - e^{sA} ) \|_{ L(H) }
  \leq 
  1
\end{equation}
hence prove for all  
$ M, N \in \N $,
$ t \in [0,T] $ that
\begin{align}
\label{eq:PN_eA_convergence_2}
\begin{split}
 &\int_0^t
  \|
    P_N
    e^{(t-s)A}
    (
      \Id_H
      -
      e^{(s-\fl{s})A}
    )
  \|_{ HS(H) }^2 \, ds
\\&\leq
  \|
    P_N
  \|_{ HS(H, H_{-\beta}) }^2
  \int_0^t
  (t-s)^{-2(\alpha+\beta)}
  (s-\fl{s})^{2\alpha} \, ds
\\&\leq
  \frac{ T^{2\alpha} }{ M^{2\alpha} }
  \|
    P_N
  \|_{ HS(H, H_{-\beta}) }^2
  \int_0^t
  (t-s)^{-2(\alpha+\beta)} \, ds
\\ &
  =
  \frac{ t^{(1-2\alpha-2\beta)} T^{2\alpha}  }{ (1-2\alpha-2\beta) M^{2\alpha} }
  \|
    P_N
  \|_{ HS(H, H_{-\beta}) }^2 .
\end{split}
\end{align}
It{\^o}'s isometry
therefore
ensures for all $ M \in \N $
that
\begin{align}
\begin{split}
 &\sup_{ N \in \N }
  \sup_{ t \in [0,T] }
  \left\|
    \int_0^t
    \left(
      P_N
      e^{(t-s)A}
      -
      P_N
      e^{(t-\fl{s})A}
    \right) dW_s
  \right\|_{ L^2(\P; H) }^2
\\&=
  \sup_{ N \in \N }
  \sup_{ t \in [0,T] }
  \int_0^t
  \|
    P_N
    e^{(t-s)A}
    -
    P_N
    e^{(t-\fl{s})A}
  \|_{ HS(H) }^2 \, ds
\\&\leq
  \frac{ T^{(1-2\beta)} }{ (1-2\alpha-2\beta) M^{2\alpha} }
  \left[
    \sup_{ N \in \N }
    \|
      P_N
    \|_{ HS(H, H_{-\beta}) }^2
  \right] .
\end{split}
\end{align}
Combining this with~\eqref{eq:PN_eA_convergence_1}
completes the 
proof of Lemma~\ref{lem:PN_eA_convergence}.
\end{proof}

\begin{lemma}
\label{lem:lower_bound_2}
Assume the setting in Section~\ref{sec:lower_bounds_setting}
and let $ N \in \N $. Then
\begin{align}
\begin{split}
 &\left[
    \frac{ \sqrt{ 1 - e^{-\nu T} } }{ 2 \pi \sqrt{\nu} }  
  \right]
  \frac{ 1 }{ \sqrt{ N } }
  \leq
  \liminf_{ M \rightarrow \infty }
  \| O_T - \OMNN{T} \|_{ \L^2(\P; H) }
  =
  \limsup_{ M \rightarrow \infty }
  \| O_T - \OMNN{T} \|_{ \L^2(\P; H) }
\\&=
  \adjustlimits\liminf_{ M \rightarrow \infty }
  \sup_{ t \in [0,T] }
  \| O_t - \OMNN{t} \|_{ \L^2(\P; H) }
  =
  \adjustlimits\limsup_{ M \rightarrow \infty }
  \sup_{ t \in [0,T] }
  \| O_t - \OMNN{t} \|_{ \L^2(\P; H) }
\\&=
  \|
    O_T
    -
    P_N
    O_T
  \|_{ \L^2(\P; H) }
  =
  \sup_{ t \in [0,T] }
  \|
    O_t - P_N O_t
  \|_{ \L^2(\P; H) }
  \leq
  \left[
    \frac{ 1 }{ \pi \sqrt{ 2 \nu } }  
  \right]
  \frac{ 1 }{ \sqrt{ N } } .
\end{split}
\end{align}

\end{lemma}
\begin{proof}[Proof of Lemma~\ref{lem:lower_bound_2}]
Throughout this proof 
let $ (\mu_n)_{ n\in\N } \subseteq \R $
satisfy for all $ n \in \N $ that
\begin{equation}
  \mu_n = \nu \pi^2 n^2 
  .
\end{equation}
Note that Parseval's identity
shows that
for all $ t \in [0,T] $ we have that
\begin{align}
\begin{split}
 &\|
    O_t - P_N O_t
  \|_{ \L^2(\P; H) }^2
\\&=
  \E\!\left[
    \| 
      O_t 
      -
      P_N O_t
    \|_H^2
  \right]
  =
  \E\!\left[
    \sum_{ k = N+1 }^{\infty}
    |
      \langle e_k, O_t \rangle_H
    |^2
  \right]
  =
  \sum_{ k = N+1 }^{\infty}
  \E\!\left[
    |
      \langle e_k, O_t \rangle_H
    |^2
  \right]
\\&=
  \sum_{ k = N+1 }^{\infty}
  \E\!\left[
    \left|
      \int_0^t
      \langle e_k, e^{(t-s)A} dW_s \rangle_H
    \right|^2
  \right]
  =
  \sum_{ k = N+1 }^{\infty}
  \E\!\left[
    \left|
      \int_0^t
      \langle e^{(t-s)A} e_k, dW_s \rangle_H
    \right|^2
  \right] .
\end{split}
\end{align}
It{\^o}'s isometry hence proves for all 
$ t \in [0,T] $ that
\begin{align}
\label{eq:lower_bound_3}
\begin{split}
&
  \|
    O_t - P_N O_t
  \|_{ \L^2(\P; H) }^2
\\
&=
  \sum_{ k = N+1 }^{\infty}
  \E\!\left[
    \left|
      \int_0^t
      e^{-\mu_k(t-s)}
      \, \langle e_k, dW_s \rangle_H
    \right|^2
  \right]
\\&=
  \sum_{ k = N+1 }^{\infty}
  \int_0^t
  e^{-2\mu_k(t-s)} \, ds
  =
  \sum_{ k = N+1 }^{\infty}
  \int_0^t
  e^{-2\mu_k s} \, ds
  =
  \sum_{ k = N+1 }^{\infty}
  \frac{ \left( 1 - e^{-2\mu_k t} \right) }{ 2 \mu_k } .
\end{split}
\end{align}
This shows that
\begin{align}
\label{eq:lower_bound_4}
\begin{split}
 &\sup_{ t \in [0,T] }
  \|
    O_t - P_N O_t
  \|_{ \L^2(\P; H) }^2
  =
  \|
    O_T - P_N O_T
  \|_{ \L^2(\P; H) }^2
\\&=
  \sum_{ k = N+1 }^{\infty}
  \frac{ \left( 1 - e^{-2\mu_k T} \right) }{ 2 \mu_k }
  =
  \sum_{ k = N+1 }^{\infty}
  \frac{ \big( 1 - e^{-2 \nu \pi^2 k^2 T} \big) }{ 2 \nu \pi^2 k^2 }
  \geq 
  \left[
    \frac{ 1 - e^{-\nu T} }{ 2 \nu \pi^2 }  
  \right]
  \left[ 
    \sum_{ k = N+1 }^{\infty}
    \frac{ 1 }{ k^2 }
  \right]
\\&\geq
  \left[
    \frac{ 1 - e^{-\nu T} }{ 2 \nu \pi^2 }  
  \right]
  \left[ 
    \sum_{ k = N+1 }^{\infty}
    \int_k^{k+1}
    \frac{1}{x^2} \, dx
  \right]
  =
  \left[
    \frac{ 1 - e^{-\nu T} }{ 2 \nu \pi^2 }  
  \right]
  \left[
    \int_{N+1}^{\infty}
    \frac{1}{x^2} \, dx
  \right]
  =
  \left[
    \frac{ 1 - e^{-\nu T} }{ 2 \nu \pi^2 }  
  \right]
  \left[
    - \frac{1}{x}
  \right]_{ x = N+1 }^{ x = \infty } 
\\&=
  \left[
    \frac{ 1 - e^{-\nu T} }{ 2 \nu \pi^2 }  
  \right]
  \frac{ 1 }{ ( N + 1 ) }
  \geq
  \left[
    \frac{ 1 - e^{-\nu T} }{ 2 \nu \pi^2 }  
  \right]
  \frac{ 1 }{ ( N + N ) }
  =
  \left[
    \frac{ 1 - e^{-\nu T} }{ 4 \nu \pi^2 }  
  \right]
  \frac{ 1 }{ N } .
\end{split}
\end{align}
This
implies that
\begin{align}
\label{eq:lower_bound_5}
\begin{split}
 &\sup_{ t \in [0,T] }
  \|
    O_t - P_N O_t
  \|_{ \L^2(\P; H) }^2
\\&=
  \sum_{ k = N+1 }^{\infty}
  \frac{ \big( 1 - e^{-2 \nu \pi^2 k^2 T} \big) }{ 2 \nu \pi^2 k^2 }
  \leq 
  \left[
    \frac{ 1 }{ 2 \nu \pi^2 }  
  \right]
  \left[ 
    \sum_{ k = N+1 }^{\infty}
    \frac{ 1 }{ k^2 }
  \right]
  \leq
  \left[
    \frac{ 1 }{ 2 \nu \pi^2 }  
  \right]
  \left[
    \sum_{ k = N+1 }^{\infty}
    \int_{ k-1 }^k
    \frac{ 1 }{ x^2 } \, dx
  \right]
\\&=
  \left[
    \frac{ 1 }{ 2 \nu \pi^2 }  
  \right]
  \left[ 
    \int_{ N }^{\infty}
    \frac{ 1 }{ x^2 } \, dx
  \right]
  =
  \left[
    \frac{ 1 }{ 2 \nu \pi^2 }  
  \right]
  \left[
    - \frac{1}{x}
  \right]_{ x = N }^{ x = \infty }
  =
  \left[
    \frac{ 1 }{ 2 \nu \pi^2 }  
  \right]
  \frac{ 1 }{ N } .
\end{split}
\end{align}
In addition, note that
the triangle inequality
and 
Lemma~\ref{lem:PN_eA_convergence}
prove that
\begin{align}
\label{eq:lower_bound_6}
\begin{split}
 &\adjustlimits\limsup_{ M \rightarrow \infty }
  \sup_{ t \in [0,T] }
  \| O_t - \OMNN{t} \|_{ \L^2(\P; H) }
\\&=
  \adjustlimits\limsup_{ M \rightarrow \infty }
  \sup_{ t \in [0,T] }
  \left\|
    \int_0^t
    ( e^{(t-s)A} - P_N e^{(t-\fl{s})A} ) \, dW_s
  \right\|_{ L^2(\P; H) }
\\&=
  \adjustlimits\limsup_{ M \rightarrow \infty }
  \sup_{ t \in [0,T] }
  \left\|
    \int_0^t
    ( e^{(t-s)A} - P_N e^{(t-s)A} ) \, dW_s
    +
    \int_0^t
    ( P_N e^{(t-s)A} - P_N e^{(t-\fl{s})A} ) \, dW_s
  \right\|_{ L^2(\P; H) }
\\&\leq
  \adjustlimits\limsup_{ M \rightarrow \infty }
  \sup_{ t \in [0,T] }
  \left\|
    \int_0^t
    ( e^{(t-s)A} - P_N e^{(t-s)A} ) \, dW_s
  \right\|_{ L^2(\P; H) }
\\&\quad+ 
  \adjustlimits\limsup_{ M \rightarrow \infty }
  \sup_{ t \in [0,T] }
  \left\|
    \int_0^t
    ( P_N e^{(t-s)A} - P_N e^{(t-\fl{s})A} ) \, dW_s
  \right\|_{ L^2(\P; H) }
\\&=
  \sup_{ t \in [0,T] }
  \left\|
    \int_0^t
    ( e^{(t-s)A} - P_N e^{(t-s)A} ) \, dW_s
  \right\|_{ L^2(\P; H) }
  =
  \sup_{ t \in [0,T] }
  \| O_t - P_N O_t \|_{ \L^2(\P; H) } .
\end{split}
\end{align}
Furthermore, observe that 
the triangle inequality,
Lemma~\ref{lem:PN_eA_convergence},
and~\eqref{eq:lower_bound_4}
ensure that 
\begin{align}
\begin{split}
 &\liminf_{ M \rightarrow \infty }
  \| O_T - \OMNN{T} \|_{ \L^2(\P; H) }
  =
  \liminf_{ M \rightarrow \infty }
  \left\|
    \int_0^T
    ( e^{(T-s)A} - P_N e^{(T-\fl{s})A} ) \, dW_s
  \right\|_{ L^2(\P; H) }
\\&=
  \liminf_{ M \rightarrow \infty }
  \left\|
    \int_0^T
    ( e^{(T-s)A} - P_N e^{(T-s)A} ) \, dW_s
    +
    \int_0^T
    ( P_N e^{(T-s)A} - P_N e^{(T-\fl{s})A} ) \, dW_s
  \right\|_{ L^2(\P; H) }
\\&\geq
  \liminf_{ M \rightarrow \infty }
  \left\|
    \int_0^T
    ( e^{(T-s)A} - P_N e^{(T-s)A} ) \, dW_s
  \right\|_{ L^2(\P; H) }
\\&\quad-
  \liminf_{ M \rightarrow \infty }
  \left\|
    \int_0^T
    ( P_N e^{(T-s)A} - P_N e^{(T-\fl{s})A} ) \, dW_s
  \right\|_{ L^2(\P; H) }
\\&=
  \left\|
    \int_0^T
    ( e^{(T-s)A} - P_N e^{(T-s)A} ) \, dW_s
  \right\|_{ L^2(\P; H) }
  =
  \| O_T - P_N O_T \|_{ \L^2(\P; H) }
\\&=
  \sup_{ t \in [0,T] }
  \| O_t - P_N O_t \|_{ \L^2(\P; H) } .
\end{split}
\end{align}
Combining this with~\eqref{eq:lower_bound_4}--\eqref{eq:lower_bound_6}
completes the proof of Lemma~\ref{lem:lower_bound_2}.
\end{proof}

\subsection{Lower and upper bounds for 
strong approximation errors of full discretizations
of linear stochastic heat equations}

\begin{corollary}
\label{cor:lower_bounds}
Assume the setting in Section~\ref{sec:lower_bounds_setting}
and let $ M, N \in \N $. Then
\begin{align}
\begin{split}
 &\frac{1}{M^{\nicefrac{1}{4}}}
  \left[ 
    \int_0^{\max\left\{0, \frac{T(N+1)^2}{2M}-\left[1+\frac{\sqrt{T}}{\sqrt{2M}}\right]^2\right\}} 
    \frac{
      \sqrt{T}
      \left[ 1 - e^{-\nu \pi^2 T} \right]
      \left[ 1 - \exp(-\nu\pi^2 \min\{1,\frac{TN^2}{2M}\}) \right]^2
    }{
      32 \nu \pi^2 \sqrt{2} 
      (x+[1+\sqrt{T}]^2)^{\nicefrac{3}{2}} 
    } \, dx
  \right]^{\nicefrac{1}{2}}
\\&\quad+
  \frac{ 1 }{ N^{\nicefrac{1}{2}} }
  \left[
    \frac{ \sqrt{ 1 - e^{-\nu T} } }{ 4 \pi \sqrt{\nu} }  
  \right]
\\&\leq
  \|
    O_T
    -
    \OMNN{T}
  \|_{ \L^2(\P; H) }
  \leq 
  \sup_{ t \in [0,T] }
  \|
    O_t
    -
    \OMNN{t}
  \|_{ \L^2(\P; H) }
\\&\leq  
  \frac{1}{M^{\nicefrac{1}{4}}}
  \left[
    \frac{\sqrt{ T }}{ 2 }
    \left( 
      \frac{ 
	1
      }{
	\pi \sqrt{\nu}
      }
      +
      \frac{ 
	1 
      }{
	\nu \pi^2
      }
      +
      4
      \pi
      \sqrt{\nu}
    \right)
  \right]^{\nicefrac{1}{2}}
  +
  \frac{ 1 }{ N^{\nicefrac{1}{2}} }
  \left[
    \frac{ 1 }{ \pi \sqrt{ 2 \nu } }  
  \right] .
\end{split}
\end{align}
\end{corollary}
\begin{proof}[Proof of Corollary~\ref{cor:lower_bounds}]
Observe that the fact that $ P_N $ is self-adjoint
ensures for all 
$ x \in H $, $ y \in P_N(H) $ that
\begin{align}
\begin{split}
&
  \langle
    x - P_N(x),
    P_N(x) - y
  \rangle_H
\\
&=
  \langle
    x - P_N(x),
    P_N(x) - P_N(y)
  \rangle_H
  =
  \langle
    x - P_N(x),
    P_N(x - y)
  \rangle_H
\\&=
  \langle
    P_N(x - P_N(x)),
    x - y
  \rangle_H
  =
  \langle
    P_N(x) - P_N(x),
    x - y
  \rangle_H
\\&=
  \langle
    0,
    x - y
  \rangle_H
  =
  0 .
\end{split}
\end{align}
This implies for all $ t \in [0,T] $ that
\begin{align}
\begin{split}
 &\|
    O_t
    -
    \OMNN{t}
  \|_{ \L^2(\P; H) }^2
\\&=
  \E\!\left[ 
    \| 
      O_t 
      - 
      \OMNN{t}
    \|_H^2
  \right] 
  =
  \E\!\left[
    \| 
      O_t 
      -
      P_N O_t
      +
      P_N O_t
      - 
      \OMNN{t}
    \|_H^2
  \right] 
\\&=
  \E\!\left[
    \| 
      O_t 
      -
      P_N O_t
    \|_H^2
  \right]
  +
  2 \,
  \E\!\left[
    \langle 
      O_t 
      -
      P_N O_t,
      P_N O_t
      - 
      \OMNN{t}
    \rangle_H
  \right]
  +
  \E\!\left[
    \| 
      P_N O_t
      - 
      \OMNN{t}
    \|_H^2
  \right]
\\&=
  \|
    O_t 
    -
    P_N O_t
  \|_{ \L^2(\P; H) }^2
  +
  \|
    P_N O_t
    -
    \OMNN{t}
  \|_{ \L^2(\P; H) }^2 .
\end{split}
\end{align}
Combining this with Lemma~\ref{lem:lower_bound_1}, Lemma~\ref{lem:lower_bound_2},
and the fact that 
\begin{equation}
  \forall \, x,y \in [0,\infty)
  \colon
  \nicefrac{\sqrt{x}}{2}
  +
  \nicefrac{\sqrt{y}}{2}
  \leq 
  \max\{\sqrt{x},\sqrt{y}\}
  \leq
  \sqrt{x+y}
  \leq 
  \sqrt{x} + \sqrt{y}
\end{equation}
completes the proof of Corollary~\ref{cor:lower_bounds}.
\end{proof}

\bibliographystyle{acm}
\bibliography{bibfile}

\end{document}